\documentclass[11pt,a4paper]{article}%
\usepackage[centertags]{amsmath}
\usepackage{amsfonts}
\usepackage{amssymb}
\usepackage{amsthm}
\usepackage{epsfig}
\usepackage{setspace}
\usepackage{ae}
\usepackage{eucal}
\usepackage[usenames]{color}%
\usepackage{graphicx}

\theoremstyle{plain}
\newtheorem{thm}{Theorem}[section]
\newtheorem{lem}[thm]{Lemma}
\newtheorem{cor}[thm]{Corollary}

\theoremstyle{definition}
\newtheorem{defi}[thm]{Definition}

\theoremstyle{remark}

\newtheorem{rem}[thm]{Remark}

\renewcommand{\div}{\operatorname{div}}

\begin{document}

\title{Some infinite matrix analysis, a Trotter product formula for dissipative operators, and an algorithm for the incompressible Navier-Stokes equation}
\author{J\"org Kampen }
\maketitle

\begin{abstract}
We introduce a global scheme on the $n$-torus of a controlled auto-controlled incompressible Navier-Stokes equation in terms of a coupled controlled infinite ODE-system of Fourier-modes with smooth data. We construct a scheme of global approximations related to linear partial integro-differential equations in dual space which are uniformly bounded in dual Sobolev spaces with polynomially decaying modes. Global boundedness of solution can be proved using an auto-controlled form of the scheme with damping via a time dilatation transformation. The scheme is based on some infinite matrix algebra related to weakly singular integrals, and a Trotter-product formula for dissipative operators which leads to rigorous existence results and uniform global upper bounds. For data with polynomial decay of the modes (smooth data) global existence follows from the preservation of upper bounds at each time step, and a compactness property in strong dual Sobolev spaces (of polynomially decaying Fourier modes). The main difference to schemes considered in \cite{KAC,KNS,K3,KB1,KB2, KB3, KHyp} is that we have a priori estimates for solutions of the iterated global linear partial integro-differential equations. Furthermore, the external control function is optional, controls only the non-dissipative zero modes and is, hence, much simpler. The analysis of the scheme leads to an algorithmic scheme  of the Navier-Stokes equation on the torus for regular data based on the Trotter-type product formula for specific dissipative operators. The infinite systems are written also in a real sinus and cosinus basis for numerical purposes. 
The main intention here is to define an algorithm which uses the damping via dissipative modes and which converges in strong norms for arbitrary viscosity constants $\nu>0$. Finally, we discuss some notions of the concept of 'turbulence', which may be a concept to be described by dynamical properties, as has been suggested by other authors for a long time (cf. \cite{CFNT, FJRT,FJKT}). An appendix contains a detailed analysis of Euler type Trotter product schemes and criteria of convergence. In this latest version of the paper convergence is reconsidered in detail in the appendix.
\end{abstract}


\section{Introduction}
The formal transformation of partial differential equations to infinite systems of ordinary differential equations via representations in a Fourier basis may be useful if the problem is posed on a torus. Formal representations of solutions are simplified at least if the equations are linear. Well, even nonlinear infinite ODE-representations may be useful in order to define solution schemes. However, in order to make sense of an exponential function of an infinite matrix applied to an infinite vector the sequence spaces involved have to be measured in rather strong norms. Schur \cite{S} may have been the first who provided necessary and sufficient conditions such that infinite linear transformations make sense. His criteria were in the sense of $l^1$-sequence spaces and are far to weak for our purposes here. However, dealing with the incompressible Navier-Stokes equation we shall assume that the initial vector-valued data function $\mathbf{h}=\left(h_1,\cdots ,h_n\right)^T$ satisfies
\begin{equation}
h_i\in C^{\infty}\left({\mathbb T}^n\right),
\end{equation}
where ${\mathbb T}^n$ denotes the $n$-torus (maybe the only flat compact manfold of practical interest when studying the Navier-Stokes equation).
In the following let ${\mathbb Z}$ be the set of integers, let ${\mathbb N}$ be the set of natural numbers including zero, let ${\mathbb R}$ denote the field of real numbers, and let $\alpha,\beta\in {\mathbb Z}^n$ denote multiindices of the set ${\mathbb Z}^n$ of $n$-tuples of integers. Differences of multiindices $\alpha-\beta\in {\mathbb Z}^n$ are built componentwise, and as usual for a multiindex $\gamma$ with nonnegative entries we denote $|\gamma|=\sum_{i=1}^n|\gamma_i|$. Consider a smooth function $\mathbf{h}=\left(h_1,\cdots ,h_n \right)^T$ with $h_i\in C^{\infty}\left({\mathbb T}^n_l\right)$, where ${\mathbb T}^n_l={\mathbb R}^n/l{\mathbb Z}^n$ is the torus of dimension $n$ and size $l>0$. Note that the smoothness of the functions $h_i$ means that there is a polynomial decay of the Fourier modes, i.e., for a torus of size $l=1$ for the modes of multivariate differentials of the function $h_i$ we have
\begin{equation}\label{alphamodes}
\begin{array}{ll}
D^{\gamma}h_{i\alpha}:=\int_{{\mathbb T}^n}D^{\gamma}_xh_i(x)\exp\left(-2\pi i\alpha x \right)dx\\
\\
=(-1)^{|\gamma|}\int_{{\mathbb T}^n}h_i(x)\Pi_{i=1}^n(-2\pi i\alpha_i)^{\gamma_i}\exp\left(-2\pi i\alpha x \right)dx\\
\\
=\Pi_{i=1}^n(2\pi i\alpha_i)^{\gamma_i}h_{i\alpha},
\end{array}
\end{equation}
where as usual for a multiindex $\gamma$ with nonnegative entries $\gamma_i\geq 0$ the symbol $D^{\gamma}_x$ denotes the multivariate derivative with respect to $x$ and of order $\gamma_i$ with respect to the $i$th component of $x$.
Here $h_{i\alpha}$ is the $\alpha$th Fourier mode of the function $h_i$, and $D^{\gamma}h_{i\alpha}$ is the $\alpha$th Fourier mode of the function $D^{\gamma}_xh_i$.
Since $D^{\gamma}_xh_i$ is smooth this function has a Fourier decomposition on ${\mathbb T}^n$.
 The sequence $\left( D^{\gamma}h_{i\alpha} \right)_{\alpha\in {\mathbb Z}^n}$ of the associated $\alpha$-modes exists in the space of square integrable sequences $l^2\left({\mathbb Z}^n\right)$. Hence, we have polynomial decay of the modes $h_{i\alpha}$ as $|\alpha|\uparrow \infty$. This polynomial decay is important in the following in order to make sense of certain matrix operations with certain multiindexed infinite matrices and multiplications of these matrices with infinite vectors, because we are interested in regular solutions. In the following we have in mind that a multiindexed vector
\begin{equation}
\mathbf{u}^F:=\left( u_{\alpha} \right)^T_{\alpha\in {\mathbb Z}^n} 
\end{equation}
(the superscript $T$ meaning 'transposed') with (possibly complex) constants $u_{\alpha}$ (although we consider only real solutions in this paper) which decay fast enough as $|\alpha|\uparrow \infty$ corresponds to a function
 \begin{equation}\label{uclass}
  u\in C^{\infty}\left({\mathbb T}^n_l\right), 
 \end{equation}
where
\begin{equation}
u(x):=\sum_{\alpha\in {\mathbb Z}^n}u_{\alpha}\exp{\left( \frac{2\pi i\alpha x}{l}\right) }.
\end{equation}
We are interested in real functions, and real Fourier systems will be considered along with complex Fourier systems. The advantage of complex Fourier systems is that the notation os more succinct, and they lead to real solutions as well. However, form the numerical point of view it may be an advantage to control the erroer on real spaces.   Therefore it make sense to deal with complex Fourier systems and compare them to real Fourier systems on an algorithmic or numerical level at decisive steps where in order to ensure that the approximative solution constructed is real or that the error of the approximative solution is real. More precisely, if the latter representation can be rewritten with real functions $\cos\left( \frac{2\pi i\alpha x}{l}\right)$ and $\sin\left( \frac{2\pi i\alpha x}{l}\right)$ and with real coefficients of these functions, then the function in (\ref{uclass}) has values in ${\mathbb R}$, and this is the case we have in mind in this paper. Later we shall write a system in  the basis
\begin{equation}\label{realbasis*}
\left\lbrace \sin\left(\frac{\pi \alpha x}{l}\right) \cos\left(\frac{\pi \alpha x}{l}\right) \right\rbrace_{\alpha\in{\mathbb N}^n},
\end{equation}
where ${\mathbb N}$ denotes the set of natural numbers including zero, i.e., we have $0\in {\mathbb N}$. We shall see that the analysis obtained in the complex notation can be transferred to real analysis with respect to the basis (\ref{realbasis*}).

We say that the infinite vector $\mathbf{u}^F$ is in the dual Sobolev space of order $s\in {\mathbb R}$, i.e.,
\begin{equation}\label{hs}
\mathbf{u}^F\in h^s\left({\mathbb Z}^n\right) 
\end{equation}
 in symbols, if the corresponding function $u$ defined in (\ref{uclass}) is in the Sobolev space $H^s\left({\mathbb T}^n_l\right)$ of order $s\in {\mathbb R}$. We may also define the dual space $h^s$ directly defining
 \begin{equation}
 \mathbf{u}^F\in h^s\left({\mathbb Z}^n\right)\leftrightarrow \sum_{\alpha \in{\mathbb Z}^n}|u_{\alpha}|^2\left\langle \alpha\right\rangle^{2s}< \infty, 
 \end{equation}
where
\begin{equation}
 \left\langle \alpha\right\rangle :=\left(1+|\alpha|^2 \right)^{1/2}. 
\end{equation}
The two definitions are clearly equivalent.
  
  Note that being an element of the function space in (\ref{hs}) for nonnegative $s$ means that we have some decay of the $\alpha$-modes. The less the order of deacy the less regularity we have. For example for constant $\alpha$-modes the corresponding function in classical $H^s$-space
 \begin{equation}
 \sum_{\alpha\in {\mathbb Z}^n}C\exp(\frac{2\pi i\alpha x}{l})
 \end{equation}
is a formal expression of $C$ times the $\delta$ distribution and we have
\begin{equation}
 \delta\in H^s~~\mbox{if}~~s<-\frac{n}{2}. 
\end{equation}
Now we may rephrase the ideas in \cite{KAC}, \cite{KNS}, \cite{K3}, \cite{KB2}, and \cite{KB3} in this context. We shall deviate from these schemes as we simplify the control function and define an iterated scheme of global equations corresponding to linear partial integro-differential equations approximating the (controlled) incompressible Navier Stokes equation. We also apply the auto-control transformation in order to sharpen analytical results and get global smooth solutions which are uniformly bounded in time. Note that the latter idea can be made consitent with efficiency if we consider stepsizes of size $0.5$ in the time dilatation transformation considered below, which then stretches a time interval $\left[0,0,5 \right]$ to a time interval of length $\left[0,\frac{1}{\sqrt{3}} \right]$ which is a small increment of amount in time compared to the advantage of a strong potential damping term stabilizing the scheme. Formally, the modes $(v_{i\alpha})_{\alpha\in {\mathbb Z}^n},~1\leq i\leq n$, of the velocity function $v_i,~1\leq i\leq n,$ of the incompressible Navier-Stokes equation satisfy the infinite ODE-system (derivation below)
\begin{equation}\label{navode200first}
\begin{array}{ll}
\frac{d v_{i\alpha}}{dt}=\sum_{j=1}^n\nu \left( -\frac{4\pi^2 \alpha_j^2}{l^2}\right)v_{i\alpha}
-\sum_{j=1}^n\sum_{\gamma \in {\mathbb Z}^n}\frac{2\pi i \gamma_j}{l}v_{j(\alpha-\gamma)}v_{i\gamma}\\
\\
+2\pi i\alpha_i1_{\left\lbrace \alpha\neq 0\right\rbrace}\frac{\sum_{j,k=1}^n\sum_{\gamma\in {\mathbb Z}^n}4\pi^2 \gamma_j(\alpha_k-\gamma_k)v_{j\gamma}v_{k(\alpha-\gamma)}}{\sum_{i=1}^n4\pi^2\alpha_i^2},
\end{array} 
\end{equation}
where the modes  $v_{i\alpha}$ depend on time $t$ and such that for all $1\leq i\leq n$ and all $\alpha\in {\mathbb Z}^n$ we have
\begin{equation}
v_{i\alpha}(0)=h_{i\alpha}.
\end{equation}
In the infinite system and for fixed $\alpha$ we call the equation in (\ref{navode200first}) with left side $\frac{d v_{i\alpha}}{dt}$ the $\alpha$-mode equation.
Some simple but important observations are in order here. First note that the damping
\begin{equation}
\sum_{j=1}^n\nu \left( -\frac{4\pi^2 \alpha_j^2}{l^2}\right)
\end{equation}
is not equal to zero unless $|\alpha|=\sum_{i=1}^n|\alpha_i|=0$. Hence we have damping except for the zero-modes $v_{i0}$ (where the subscript $0$ denotes the $n$-tuple of zeros). Second, note that the zero modes $v_{i0}$  contribute to the $\alpha$-mode equation for $\alpha\neq 0$ only via the convection terms. This is because for the second term on the right side of (\ref{navode200first}) only the term with $\gamma=\alpha$, i.e., the summand 
\begin{equation}
-\sum_{j=1}^n\frac{2\pi i \alpha_j}{l}v_{j0}v_{i\alpha}
\end{equation}
contains a zero mode. Third, note that the pressure term in (\ref{navode200first}) has no zero mode summands since
\begin{equation}
\begin{array}{ll}
2\pi i\alpha_i1_{\left\lbrace \alpha\neq 0\right\rbrace}\frac{\sum_{j,k=1}^n\sum_{\gamma\in {\mathbb Z}^n}4\pi^2 \gamma_j(\alpha_k-\gamma_k)v_{j\gamma}v_{k(\alpha-\gamma)}}{\sum_{i=1}^n4\pi^2\alpha_i^2}=\\
\\
2\pi i\alpha_i1_{\left\lbrace \alpha\neq 0\right\rbrace}\frac{\sum_{j,k=1}^n\sum_{\gamma\in {\mathbb Z}^n\setminus \left\lbrace \alpha,0\right\rbrace} 4\pi^2 \gamma_j(\alpha_k-\gamma_k)v_{j\gamma}v_{k(\alpha-\gamma)}}{\sum_{i=1}^n4\pi^2\alpha_i^2}.
\end{array}
\end{equation}
More precisely, note that the $0$-mode equation consists only of terms corresponding to convection terms, i.e., we have
\begin{equation}\label{zeromode}
\begin{array}{ll}
\frac{d v_{i0}}{dt}=
-\sum_{j=1}^n\sum_{\gamma \in {\mathbb Z}^n\setminus \left\lbrace 0\right\rbrace }\frac{2\pi i \gamma_j}{l}v_{j(-\gamma)}v_{i\gamma}.
\end{array} 
\end{equation}
Furthermore we observe that in (\ref{zeromode}) we have no zero modes on the right side but only non zero-modes of coupled equations. Furthermore all non-zero modes involve a damping term, because $\nu \left( -\frac{4\pi^2 \alpha_j^2}{l^2}\right)<0$ for $\nu>0$ and $|\alpha|\neq 0$. Note that this damping becomes stronger as the order of the modes $|\alpha|\neq 0$ increases - an important difference to the incompressible Euler equation, where we  have no viscosity damping indeed. This is also a first hint that it may useful to define a controlled incompressible Navier-Stokes equation on the torus where a simple control functions controls just the zero modes. In order to sharpen analytical results we shall consider extended auto-controlled schemes where we introduce a damping term via time dilatation similar as in \cite{KAC}. This is a time transformation at each time step, and as at each local time step we deal with time-dependent operators this additional idea fits quite well with the scheme on a local level. On a global time-level it guarantees the preservation of global upper bounds. 
This idea is implemented in local time where we may consider the coordinate transformation
\begin{equation}\label{timedil}
(\tau (t) ,x)=\left(\frac{t}{\sqrt{1-t^2}},x\right),
\end{equation}
which is a time dilatation effectively and leaves the spatial coordinates untouched.
Then on a time local level, i.e., for some parameter $\lambda >0$ and $t\in [0,1)$ the function $u_i,~1\leq i\leq n$ with 
\begin{equation}\label{uvlin}
\lambda(1+t)u_i(\tau,x)=v_i(t ,x),
\end{equation}
carries all information of the velocity function on this interval, and satisfies
\begin{equation}
\frac{\partial}{\partial t}v_i(t,x)=\lambda u_i(\tau,x)+\lambda(1+t)\frac{\partial}{\partial \tau}u_i(\tau,x)\frac{d \tau}{d t},
\end{equation}
where 
\begin{equation}
 \frac{d\tau}{dt}=\frac{1}{\sqrt{1-t^2}^3}.
\end{equation}
We shall choose $0<\lambda<1$ in order to make the nonlinear terms smaller on the time interval compared to the damping term (they get an additional factor $\lambda$. The price to pay are larger initial data but we shall observe that we can compensate this in an appropriate scheme. 
\begin{rem}
Alternatively, we can use localized transformations of the form
\begin{equation}\label{uvloc}
\lambda(1+(t-t_0))u^{t_0}_i(\tau,x)=v_i(t ,x),~\mu>0
\end{equation}
for $t_0\geq 0$ and where 
\begin{equation}\label{timedil}
(\tau (t) ,x)=\left(\frac{t-t_0}{\sqrt{1-(t-t_0)^2}},x\right),
\end{equation}
\begin{equation}
\frac{\partial}{\partial t}v_i(t,x)=\lambda u_i(\tau,x)+\lambda(1+(t-t_0))\frac{\partial}{\partial \tau}u_i(\tau,x)\frac{t-t_0}{\sqrt{1-(t-t_0)^2}^3}.
\end{equation}
Using this localized term we can avoid a weaker damping as time increases. In the following we get the equations for the transformations of the form (\ref{uvloc}) if we replace $t$ by $t-t_0$ in the coeffcients. 
\end{rem}

 We denote the inverse of $\tau(t)$ by $t(\tau)$. For the modes of $u_i,~1\leq i\leq n$ we get the equation   
\begin{equation}\label{navode200firsttimedil}
\begin{array}{ll}
\frac{d u_{i\alpha}}{d\tau}=\sqrt{1-t(\tau)^2}^3
\sum_{j=1}^n\nu \left( -\frac{4\pi \alpha_j^2}{l^2}\right)u_{i\alpha}-\\
\\
\lambda(1+t(\tau))\sqrt{1-t(\tau)^2}^3\sum_{j=1}^n\sum_{\gamma \in {\mathbb Z}^n}\frac{2\pi i \gamma_j}{l}u_{j(\alpha-\gamma)}u_{i\gamma}+\\
\\
\lambda(1+t(\tau))\sqrt{1-t(\tau)^2}^3\frac{2\pi i\alpha_i1_{\left\lbrace \alpha\neq 0\right\rbrace}\sum_{j,k=1}^n\sum_{\gamma\in {\mathbb Z}^n}4\pi^2 \gamma_j(\alpha_k-\gamma_k)u_{j\gamma}u_{k(\alpha-\gamma)}}{\sum_{i=1}^n4\pi^2\alpha_i^2}\\
\\
-\sqrt{1-t^2(\tau)}^3(1+t(\tau))^{-1}u_{i\alpha}.
\end{array} 
\end{equation}
Note that in the latter equation the nonlinear terms have the factor $\lambda$ (which may be chosen to be small) while the damping potential term has no factor $\rho$ and may dominate the nonlinear terms. Note that $\lambda>0$ is another scaling factor. Especially, in the transformation above the Laplacian has no coefficient $\lambda>0$. Using a time dilatation transformation at each time step it becomes easier to prove the existence of global upper bounds. On the other hand, used in the present form, at each time step we blow up a time step interval of size $1$ to size infinity, and this is certainly more interesting from an analytical than from a numerical or computational point of view. Well, in general it may be useful to have stability of computations and also in this situation it may be interesting to use time dilatation transformations in order to obtain stability via damping and pay the price of additional computation costs for this stability. So it seems prima facie. However on closer inspection we observe that we may use a time dilatation algorithm on a smaller interval (not on unit time intervals $[l-1,l]$ but on half unit time intervals $\left[l-1,l-\frac{1}{2}\right]$ for exmple and repeat the subscheme twice. The time step size does not increase very much then (by the nature of the time transformation) but we have a damping term now. We shall discuss this more closely in the section about algorithms.    

Next we make this idea of our schemes more precise by defining the schemes for computing the modes. We forget the time dilatation for a moment since this is a local operation and consider the global equation in time coordinates $t\geq 0$ again. For purposes of global existence it can make sense to start with the the multivariate Burgers equation, because we know that we have a unique global regular solution for this equation on the $n$-torus. As we shall see below in more detail starting with an assumed solution or the multivariate Burgers equation is also advantageous from an analytical point of view as we know the existences of regular solutions (polynomial decay of modes) and we have to take care only of the additional Leray projection term. On the other hand, and from a numerical point of view we have no explicit formula for solutions of multivariate Burgers equation (except in special cases) and, hence, in a numerical scheme we should better construct the solution. The existence of global solutions $u^{b}$ can be derived by the a priori estimates 
\begin{equation}\label{aprioriestburg}
\frac{\partial}{\partial t}\|u^b(t,.)\|_{H^s}\leq \|u^b(t,.)\|_{H^{s+1}}\sum_{i,j}\sum_{|\alpha|+|\beta|\leq s}\|D^{\alpha}u^b_iD^{\beta}u_j\|_{L^2}-2\|\nabla u^b\|^2_{H^s},
\end{equation}
which hold on the $n$-torus, or by the arguments which we discussed in \cite{KB1} and \cite{KB2}. In our context this means that for positive viscosity $\nu>0$ and smooth data (at least formally) we have a global regular solution for the system for $1\leq i\leq n$ and $\alpha \in {\mathbb Z}^n$ of the form
\begin{equation}\label{navode200b}
\begin{array}{ll}
\frac{d u^b_{i\alpha}}{dt}=\sum_{j=1}^n\nu \left( -\frac{4\pi \alpha_j^2}{l^2}\right)u^b_{i\alpha}
-\sum_{j=1}^n\sum_{\gamma \in {\mathbb Z}^n}\frac{2\pi i \gamma_j}{l}u^b_{j(\alpha-\gamma)}u^b_{i\gamma},
\end{array} 
\end{equation}
where $u^b_{i\alpha}$ depend on time $t$ and such that for all $1\leq i\leq n$ and all $\alpha\in {\mathbb Z}^n$ we have
\begin{equation}
u^b_{i\alpha}(0)=h_{i\alpha}.
\end{equation}
This suggests the following first approach concerning a (formal) scheme for the purpose of global existence.
We compute first $v^0_{i\alpha}=u_{i\alpha}$ for $1\leq i\leq n$ and $\alpha\in {\mathbb Z}^n$ and then iteratively
\begin{equation}
v^k_{i\alpha}:=v^0_{i\alpha}+\sum_{p=1}^k\delta v^p_{i\alpha},
\end{equation}
where $\delta v^p_{i\alpha}:=v^p_{i\alpha}-v^{p-1}_{i\alpha}$ for all $1\leq i\leq n$ and $\alpha\in {\mathbb Z}^n$ and for $p\geq 1$ we have 
\begin{equation}\label{navode200a}
\begin{array}{ll}
\frac{d v^p_{i\alpha}}{dt}=\sum_{j=1}^n\nu \left( -\frac{4\pi \alpha_j^2}{l^2}\right)v^p_{i\alpha}
-\sum_{j=1}^n\sum_{\gamma \in {\mathbb Z}^n}\frac{2\pi i \gamma_j}{l}v^{p}_{j(\alpha-\gamma)}v^p_{i\gamma}\\
\\
+2\pi i\alpha_i1_{\left\lbrace \alpha\neq 0\right\rbrace}\frac{\sum_{j,k=1}^n\sum_{\gamma\in {\mathbb Z}^n}4\pi \gamma_j(\alpha_k-\gamma_k)v^{p}_{j\gamma}v^{p-1}_{k(\alpha-\gamma)}}{\sum_{i=1}^n4\pi\alpha_i^2},
\end{array} 
\end{equation}
where $v_{i\alpha}$ depend on time $t$ and such that for all $1\leq i\leq n$ and all $\alpha\in {\mathbb Z}^n$ we have
\begin{equation}
v_{i\alpha}(0)=h_{i\alpha}.
\end{equation}
Alternatively, we may start with the initial data $h_i$ and their modes $h_{i\alpha}$ as first order coefficients of the first approximating equation. In this case we have for an iteration number $p\geq 0$ the approximation at the $p$th stage via the linear equation
\begin{equation}\label{navode200lin}
\begin{array}{ll}
\frac{d v^p_{i\alpha}}{dt}=\sum_{j=1}^n\nu \left( -\frac{4\pi \alpha_j^2}{l^2}\right)v^p_{i\alpha}
-\sum_{j=1}^n\sum_{\gamma \in {\mathbb Z}^n}\frac{2\pi i \gamma_j}{l}v^{p-1}_{j(\alpha-\gamma)}v^{p}_{i\gamma}\\
\\
+2\pi i\alpha_i1_{\left\lbrace \alpha\neq 0\right\rbrace}\frac{\sum_{j,k=1}^n\sum_{\gamma\in {\mathbb Z}^n}4\pi \gamma_j(\alpha_k-\gamma_k)v^{p}_{j\gamma}v^{p-1}_{k(\alpha-\gamma)}}{\sum_{i=1}^n4\pi\alpha_i^2},
\end{array} 
\end{equation}
where again for all $1\leq i\leq n$ and all $\alpha\in {\mathbb Z}^n$ we have
\begin{equation}
v_{i\alpha}(0)=h_{i\alpha}.
\end{equation}
For $p=0$ we then have $v^{p-1}_{j(\alpha-\gamma)}=v^{-1}_{j(\alpha-\gamma)}:=h_{j(\alpha-\gamma)}$.
The latter equation in (\ref{navode200lin}) still corresponds to a partial integro-differential equation. However, this scheme has the advantage that we can built an algorithm on it via a Trotter product formula for infinite matrices (because we know the data $h_{i\alpha}$ and for $p\geq 1$ the data $v^{p-1}_{i\alpha}$ from the previous iteration step).

\begin{rem}
The equation in (\ref{navode200lin}) corresponds to a linear integro-differential equation in classical space and differs in this respect from the approximations we considered in \cite{KB2} and \cite{KB3} and also in \cite{KHyp,K3,KNS} (although we mentioned this type of global scheme in \cite{KNS}). The analysis in all these papers was based on local equations because a priori estimates are easier at hand - as is the trick with the adjoint. In dual spaces of Fourier basis representation considered in this paper spatially global equations are easier to handle. We could also use the corresponding local equations. The corresponding system is
\begin{equation}\label{navode200loc}
\begin{array}{ll}
\frac{d v^p_{i\alpha}}{dt}=\sum_{j=1}^n\nu \left( -\frac{4\pi \alpha_j^2}{l^2}\right)v^p_{i\alpha}
-\sum_{j=1}^n\sum_{\gamma \in {\mathbb Z}^n}\frac{2\pi i \gamma_j}{l}v^{p-1}_{j(\alpha-\gamma)}v^p_{i\gamma}\\
\\
+2\pi i\alpha_i1_{\left\lbrace \alpha\neq 0\right\rbrace}\frac{\sum_{j,k=1}^n\sum_{\gamma\in {\mathbb Z}^n}4\pi \gamma_j(\alpha_k-\gamma_k)v^{p-1}_{j\gamma}v^{p-1}_{k(\alpha-\gamma)}}{\sum_{i=1}^n4\pi\alpha_i^2},
\end{array} 
\end{equation}
but there is no real advantage analyzing (\ref{navode200loc}) instead of (\ref{navode200lin}). Note that may even 'linearize' the Burgers term and substitute the Burgers term (\ref{navode200loc}) by the term $-\sum_{j=1}^n\sum_{\gamma \in {\mathbb Z}^n}\frac{2\pi i \gamma_j}{l}v^{p-1}_{j(\alpha-\gamma)}v^{p-1}_{i\gamma}$. We considered this in \cite{KB3} in classical spaces in order to avoid the use of the adjoint of fundamental solutions and work with estimates of convolutions directly. Again, there is no great advantage to do the same in dual spaces, and it is certainly not preferable from an algorithmic point of view. 
\end{rem}
Note that the 'global' term in (\ref{navode200a}) and (\ref{navode200lin}) is of the form
\begin{equation}
2\pi i\alpha_i1_{\left\lbrace \alpha\neq 0\right\rbrace}\frac{\sum_{j,k=1}^n\sum_{\gamma\in {\mathbb Z}^n}4\pi \gamma_j(\alpha_k-\gamma_k)v^{p}_{j\gamma}v^{p-1}_{k(\alpha-\gamma)}}{\sum_{i=1}^n4\pi\alpha_i^2},
\end{equation}
and it looks more like its 'local' companions (especially the Burgers term) in this discrete Fourier-based representation than is the case for representations in classical spaces.
Formally,
for all $1\leq i\leq n$ and $\alpha\in {\mathbb Z}^n$ and for $p\geq 1$ for $\delta v^p_{i\alpha}=v^p_{i\alpha}-v^{p-1}_{i\alpha},~\alpha\in {\mathbb Z}^n$ we have  
\begin{equation}\label{navode2000}
\begin{array}{ll}
\frac{d \delta v^p_{i\alpha}}{dt}=\sum_{j=1}^n\nu \left( -\frac{4\pi \alpha_j^2}{l^2}\right)\delta v^p_{i\alpha}
-\sum_{j=1}^n\sum_{\gamma \in {\mathbb Z}^n}\frac{2\pi i \gamma_j}{l}v^{p-1}_{j(\alpha-\gamma)}\delta v^p_{i\gamma}\\
\\
-\sum_{j=1}^n\sum_{\gamma \in {\mathbb Z}^n}\frac{2\pi i \gamma_j}{l}\delta v^{p-1}_{j(\alpha-\gamma)} v^{p-1}_{i\gamma}\\
\\
+2\pi i\alpha_i1_{\left\lbrace \alpha\neq 0\right\rbrace}\frac{\sum_{j,k=1}^n\sum_{\gamma\in {\mathbb Z}^n}4\pi \gamma_j(\alpha_k-\gamma_k)v^{p}_{j\gamma}v^{p-1}_{k(\alpha-\gamma)}}{\sum_{i=1}^n4\pi\alpha_i^2}\\
\\
-2\pi i\alpha_i1_{\left\lbrace \alpha\neq 0\right\rbrace}\frac{\sum_{j,k=1}^n\sum_{\gamma\in {\mathbb Z}^n}4\pi \gamma_j(\alpha_k-\gamma_k)v^{p-1}_{j\gamma}v^{p-2}_{k(\alpha-\gamma)}}{\sum_{i=1}^n4\pi\alpha_i^2},
\end{array} 
\end{equation}
where $v_{i\alpha}$ depend on time $t$ and such that for all $1\leq i\leq n$ and all $\alpha\in {\mathbb Z}^n$ we have
\begin{equation}
\delta v_{i\alpha}(0)=0.
\end{equation}
It is well-known that the viscosity parameter may be chosen arbitrarily. This fact facilitates the analysis a bit, but for the proof of global existence it is not essential. It is important that the viscosity is strictly positive ($\nu>0$). We need this for the Trotter product formula below, and without it the arguments breaks down. 

\begin{rem}
In order to have a certin contraction property of iterated weakly singular elliptic integrals for all modes and in order to have a stronger diffusion damping we
consider parameter transformations of the form
\begin{equation}
v_i(t,x)=r^{\lambda}v^*_i(r^{\nu_0}t,r^{\mu}x),~p(t,x)=r^{\delta}p^*(r^{\nu'}t,r^{\mu}x)
\end{equation}
long with $\tau=r^{\nu}t$ and $y=r^{\mu}x$ for some positive real number $r>0$ and some positive real parameters $\lambda,\mu,\nu_0,\nu'$. We have (using Einstein notation for spatial variables)
\begin{equation}
\begin{array}{ll}
\frac{\partial}{\partial t}v_i(t,x)=r^{\lambda+\nu_0}\frac{\partial}{\partial \tau}v^*_i(\tau,y),
~v_{i,j}(t,x)=r^{\lambda+\mu}v^*_{i,j}(\tau,y),\\
\\
v_{i,j,k}(t,x)=r^{\lambda+2\mu}\frac{\partial^2}{\partial x_j\partial x_k}v^*_i(\tau,y),~p_{,i}(t,x)=r^{\delta+\mu}p^*_{,i}(\tau,y).
\end{array}
\end{equation}
Hence 
\begin{equation}
r^{\lambda+\nu_0}\frac{\partial}{\partial \tau}v^*_{i}(\tau,y)=\nu r^{\lambda+2\mu}\Delta v^*_i(\tau,y)-\sum_{j=1}^nr^{2\lambda +\mu}v^*_j(\tau,y)\frac{\partial v^*_i}{\partial x_j}(\tau,y)-r^{\delta+\mu}p^*_{,i},
\end{equation}
or 
\begin{equation}
\begin{array}{ll}
\frac{\partial}{\partial \tau}v^*_{i}(\tau,y)=\nu r^{\lambda+2\mu-\lambda -\nu_0}\Delta v^*_i(\tau,y)\\
\\
-\sum_{j=1}^nr^{2\lambda +\mu-\lambda -\nu_0}v^*_j(\tau,y)\frac{\partial v^*_i}{\partial x_j}(\tau,y)-r^{\delta+\mu-\lambda -\nu_0}p^*_{,i},
\end{array}
\end{equation}
which (for example) for the parameter constellations 
\begin{equation}\label{paracon}
\mu,\nu,\lambda,\delta \mbox{ with }\lambda +\mu-\nu_0=0 \mbox{ and } \delta+\mu-\lambda-\nu_0=0
\end{equation}
becomes
\begin{equation}\label{v*eq}
\begin{array}{ll}
\frac{\partial}{\partial \tau}v^*_{i}(\tau,y)=\\
\\
\nu r^{2\mu-\nu_0}\Delta v^*_i(\tau,y)-\sum_{j=1}^nr^{\lambda +\mu-\nu_0}v^*_j(\tau,y)\frac{\partial v^*_i}{\partial x_j}(\tau,y)-r^{\delta+\mu-\lambda-\nu_0}p_{,i}\\
\\
=\nu r^{2\mu-\nu_0}\Delta v^*_i(\tau,y)-\sum_{j=1}^nv^*_j(\tau,y)\frac{\partial v^*_i}{\partial x_j}(\tau,y)-p^*_{,i}.
\end{array}
\end{equation}
Note that we may choose (for example) $2\mu-\nu'\in {\mathbb R}_+$ (${\mathbb R}_+$ being the set of strictly positive real numbers) freely and still satisfy the conditions (\ref{paracon}). Note that we can take $\nu'=\nu_0=\nu$. Especially, we are free to assume any viscosity $\nu>0$. Note that the Leray projection term is determined via the relation $\Delta p^*=\sum_{j,m}(v^*_{j,m}v^*_{m,j})$, and for some parmeter constellations it gets the same parameter coeffcient as the the burgers term (cf. Appendix).  Furthermore, choosing $\lambda=\delta=0$ and $\mu<\nu_0<2\mu$ with $\mu$ large viscosity damping coefficient becomes large compared to the parmenter coefficients of the nonlinear terms as we get the equation
\begin{equation}\label{v*eq}
\begin{array}{ll}
\frac{\partial}{\partial \tau}v^*_{i}(\tau,y)=\\
\\
\nu r^{2\mu-\nu_0}\Delta v^*_i(\tau,y)-\sum_{j=1}^nr^{ +\mu-\nu_0}v^*_j(\tau,y)\frac{\partial v^*_i}{\partial x_j}(\tau,y)-r^{\mu-\nu_0}p_{,i}.
\end{array}
\end{equation}
Indeed for large $\mu$ we have a strong viscosity damping and small coefficients of the Burgers and the Leray projection term. Especially for large $\mu$ in this parameter constellation we have a contraction property of elliptic integrals which serve as natural upper bounds of the nonlinear growth terms. For example, in  natural iteration schemes we may use the relation  
\begin{equation}\label{lerayell}
\sum_{\beta\in {\mathbb Z}^n\setminus \left\lbrace 0,\alpha\right\rbrace }\frac{C}{|\alpha-\beta|^{s}}\frac{C}{\beta^r}\leq \frac{cC^2}{1+|\alpha|^{r+s-n}}
\end{equation}
in order to estimate the growth contribution of the Leray projection term at one time step, where for $r\geq n$ and $s\geq n+1$ we have $r+s-n\geq n+2$. The $c$ in (\ref{lerayell}) depends only on the dimension, in with $\mu$ large enough we obatin a contraction property for all modes $\alpha$ different from $0$.
For the numerical and computational analysis it is useful to refine this a bit and consider in addition certain  bi-parameter transformation with respect to the viscosity constant $\nu >0$ and with respect to the diameter $l$ of the $n$-torus.
From an analytic perspective we may say that it is sufficient to prove global existence for specific parameters $\nu>0$ and $l>0$. Let us consider coordinate transformations with parameter $l>0$ first.   
 If $\mathbf{v}^1, p^1$ is solution on the domain $[0,\infty)\times {\mathbb T}^n_1$, then the function pair $(t,x)\rightarrow \mathbf{v}^l(t,x)$, and $(t,x)\rightarrow p^l(t,x)$  on $[0,\infty)\times {\mathbb T}^n_l$ along with
\begin{equation}
\mathbf{v}^l(t,x):=\frac{1}{l}\mathbf{v}^1\left(\frac{t}{l^2},\frac{x}{l} \right), 
\end{equation}
and
\begin{equation}
p^l(t,x):=\frac{1}{l^2}p^1\left(\frac{t}{l^2},\frac{x}{l} \right) 
\end{equation}
is a solution pair on the $n$-torus of size $l>0$ with initial data
\begin{equation}\label{factorl}
\mathbf{h}^l(x):=\frac{1}{l}\mathbf{h}^1(\frac{x}{l})~\mbox{on}~{\mathbb T}^n_l.
\end{equation}
Hence if we can construct solutions to the Navier-Stokes solution for $h_l\in C^{\infty}\left({\mathbb T}^n_l\right)$ without further restrictions on the data $\mathbf{h}_l$, then we can construct global solutions for the problem in ${\mathbb T}^n_1$ with  $\mathbf{h}=\mathbf{h}_1\in C^{\infty}\in \left({\mathbb T}^n\right)_1$ with factor $l$ as in (\ref{factorl}). We only need to produce a fixed number $l>0$ such that arbitrary data are allowed. Second if $v^{\nu}_i,~1\leq i\leq n$ is a solution to the incompressible Navier Stokes equation with parameter $\nu>0$ then via the time transformation
\begin{equation}
v_i(t,x):=(r\nu )^{-1}v^{\nu}_i((r\nu)^{-1}t,x),~p(t,x):=(r\nu)^{-2}p^{\nu}((r\nu)^{-1}t,x)
\end{equation}
(for all $t\geq 0$) we observe that $v_i$ is solution of the incompressible Navier-Stokes equation with viscosity parameter which may be chosen.  Concatenation of the previous transformations shows that we can indeed choose specific values for $\nu>0$ and $l>0$, and this may be useful for designing algorithms. For example, it is useful to observe that in (\ref{navode2000}) the modulus of the diagonal coefficients
\begin{equation}\label{dampterm}
\sum_{j=1}^n\nu \left( -\frac{4\pi \alpha_j^2}{l^2}\right)
\end{equation}
becomes large for $|\alpha|\neq 0$ if $\nu$ is large or $l$ is small in comparison to the other terms of the iteration in (\ref{navode2000}), but it is even more intersting that we may choose $\nu$ large and $l>0$ large on a scale such that the convection terms are small in comparison to the diagonal damping terms in (\ref{dampterm}). 
\end{rem}  
As we observed in other articles for the time-local solution of the Navier-Stokes equation we may set up a scheme involving simple scalar linear parabolic equations of the form (\ref{parascalar}) below  (cf. \cite{KAC, KNS, K3,KB1,KB2,KB3, KHyp}). We maintained that in classical space a global controlled scheme for an equivalent equation can be obtained  if in a time-local scheme the  approximate functions  in a functional series (evaluated at arbitrary time $t$) inherit the property of being in $C^k\cap H^k$ for $k\geq 2$ together with polynomial decay of order $k$. This is remarkable since you will not expect this from standard local a priori estimates of the Gaussian. For the equivalent controlled scheme in \cite{KB3} we only needed to prove this for the higher order correction terms of the time-local scheme. There it is true for some order of decay because the representations involve convolution integrals with products of approximative value functions where each factor is inductively of polynomial decay. The growth of the Leray projection term for the scheme in \cite{KB3} is linearly bounded on a certain time scale. In \cite{KAC, KNS, K3, KHyp} we proposed auto-controlled schemes and more complicated equivalent controlled schemes with a bounded regular control function which controls the growth of the controlled solution function such that solutions turn out to be uniformly bounded for all time. 
Actually, the simple schemes may be reconsidered in terms of the Trotter product formula representations of solutions to scalar parabolic equations of the form 
\begin{equation}\label{parascalar}
\frac{\partial u}{\partial t}-\nu \Delta u+Wu=g,
\end{equation}
with initial data $u(0,x)=f\in \cap_{s\in {\mathbb R}^n}H^s$ and some dynamically generated source terms $g$, and where $W=\sum_{i=1}^nw_i(x)\frac{\partial}{\partial x_i}$ is a vector field with bounded and uniformly Lipschitz continuous coefficients. The $w_i$ then are replaced by approximative value function components at each iteration step of course. 
Local application of the Trotter product formula leads to representations of the from
\begin{equation}\label{trottappl}
\exp\left(t\left(\nu \Delta +W \right)  \right)f=\lim_{k\uparrow \infty}\left(\exp\left(\frac{t}{k}W\right)\exp\left(\frac{t}{k}\nu\Delta\right)\right)^kf.
\end{equation}
Then in a second step having obtained time-local representations of solutions to the incompressible Navier-Stokes equation one can use this reduction from a system to a scalar level and apply the Trotter product formula on a global time scale. Time dilatation transformations as in \cite{KAC} may be used to prove the existence of upper bounds.  From this point of view the local contraction results considered in \cite{KB3} or in \cite{KB1, KB2, KHyp} (with the use of adjoint equations) are essential. 
In dual spaces the use of a Trotter product formula seems to be not only useful but mandatory in two respects. First, consider the infinite matrix
\begin{equation}\label{damptermmatrix}
D:=\left(d_{\alpha\beta} \right)_{\alpha,\beta\in {\mathbb Z}^n}
:=\left( \delta_{\alpha\beta}\sum_{j=1}^n\nu \left( -\frac{4\pi \alpha_j^2}{l^2}\right)\right)_{\alpha,\beta\in {\mathbb Z}^n}
\end{equation}
with the infinite Kronecker delta function $\delta_{\alpha\beta}$, i.e,
\begin{equation}
\delta_{\alpha\beta}=\left\lbrace \begin{array}{ll}
1\mbox{ if }\alpha=\beta\\
\\
0 \mbox{ if }\alpha\neq \beta.
\end{array}\right.
\end{equation}

If we measure regularity with respect to the degree of decay of entries as the order of the modes increases, then this matrix lives in a space of rather weak regularity. Worse then this, iterations of the matrix, i.e., matrices of the form
\begin{equation}
\begin{array}{ll}
D^2:=DD=:\left( d^{(2)}_{\alpha\beta}\right)_{\alpha,\beta\in {\mathbb Z}^n}:=\left(\sum_{\gamma\in {\mathbb Z}^n}d_{\alpha\gamma}d_{\gamma\beta} \right)_{\alpha,\beta\in {\mathbb Z}^n}\\
\\
D^m:=DD^{m-1}=:\left( d^{(m)}_{\alpha\beta}\right)_{\alpha,\beta\in {\mathbb Z}^n}:=\left(\sum_{\gamma\in {\mathbb Z}^n}d_{\alpha\gamma}d^{(m-1)}_{\gamma\beta} \right)_{\alpha,\beta\in {\mathbb Z}^n}
\end{array}
\end{equation}
 live in matrix spaces of lower and lower regularity as $m$ increases (if we measure regularity in relation to decay with respect to the order of modes). However, since the matrix $D$ has a minus sign we may make sense of the matrix
\begin{equation}\label{expD}
\exp(D)=\left(\delta_{\alpha\beta}\exp\left(d_{\alpha\alpha}\right)\right)_{\alpha,\beta\in {\mathbb Z}^n}=\left(\delta_{\alpha\beta}\exp\left(\sum_{j=1}^n\nu \left( -\frac{4\pi \alpha_j^2}{l^2}\right)\right)\right)_{\alpha,\beta\in {\mathbb Z}^n},
\end{equation}
and this matrix makes perfect sense in terms of the type of regularity mentioned, i.e., the type of regularity expressed by dual Sobolev spaces which measure the order of polynomial decay of the modes. Since we consider infinite ODEs involving a Leray projection term we shall have correlations between the components $1\leq i\leq n$ even at each approximation step. Hence matrices which correspond to linear approximations of the Navier Stokes equations will involve big diagonal matrices of the form
\begin{equation}
\left(\delta_{ij}\exp(D)\right)_{1\leq i,j\leq n},
\end{equation}
where $\delta_{ij}$ denotes the usual Kronecker $\delta$ for $1\leq i,j\leq n$ and $\exp(D)$ is as in (\ref{expD}) above. However, the idea that such a dissipative matrix lives in a regular ($\equiv$'polynomially decaying modes as the order of modes increases') infinite matrix space is the same. So much for the diffusion term. Note that we have some linear growth with respect to the modes of a matrix related to the convection term, and there may be some concern that iteration of this matrix lead to divergences. However, the initial data at each time step have sufficient decay such that they can serve as analytic vectors, and weakly elliptic integrals will show that these decay properties are preserved by the scheme. If the data at each substep are sufficiently regular (where 'sufficiency' may be determined by upper bounds obtained via  weakly singular elliptic integrals, cf. below), then on a time local level algorithms via Trotter product formula approximate limits of alternate local solutions on $[t_0,t_1]\times h^s\left({\mathbb Z}^n\right)$ of subproblems of the form 
\begin{equation}\label{navode200locsub1}
\left\lbrace \begin{array}{ll}
\frac{d v^p_{0i\alpha}}{dt}=\sum_{j=1}^n\nu \left( -\frac{4\pi \alpha_j^2}{l^2}\right)v^p_{0i\alpha},\\
\\
v^p_{0i\alpha}(t_0)=v^p_{i\alpha}(t_0),~\alpha\in {\mathbb Z}^n
\end{array} 
\right.
\end{equation}
and
\begin{equation}\label{navode200locsub2}
\left\lbrace \begin{array}{ll}
\frac{d v^p_{1i\alpha}}{dt}=
-\sum_{j=1}^n\sum_{\gamma \in {\mathbb Z}^n}\frac{2\pi i \gamma_j}{l}v^{p-1}_{1j(\alpha-\gamma)}v^p_{1i\gamma}\\
\\
+2\pi i\alpha_i1_{\left\lbrace \alpha\neq 0\right\rbrace}\frac{\sum_{j,k=1}^n\sum_{\gamma\in {\mathbb Z}^n}4\pi \gamma_j(\alpha_k-\gamma_k)v^{p}_{1j\gamma}v^{p-1}_{1k(\alpha-\gamma)}}{\sum_{i=1}^n4\pi\alpha_i^2},\\
\\
v^p_{1i\alpha}(t_0)=v_{i\alpha}(t_0),~\alpha\in {\mathbb Z}^n,
\end{array}\right.
\end{equation}
and where for $p=1$ we may have $v^{p-1}_{i\alpha}=v^0_{\alpha}=v_{i\alpha}(t_0)$. Here, for $\alpha\in {\mathbb Z}^n$ and $1\leq i\leq n$ the time dependent functions $v_{i\alpha}$ are the modes of the vector $v_{i}$ which is assumed to be known for time $t_0$ (by a previous time step or by the initial data at time $t_0=0$.
If we can ensure that a certain regularity of the data $v^p_{i\alpha}(t_0)$, say $ v_{i}(t_0) \in h^s\left({\mathbb Z}^n\right) $ for $s>n+2$ then we shall observe that $v^p_{i}(t)\in h^s\left({\mathbb Z}^n\right) $ for some $s>n+2$ and time $t\in [t_0,t]$. The inheritage of this regularity then leads to a sequence of solutions $v^p_{i}\in h^s\left({\mathbb Z}^n\right)$ with $s>n+2$, where the limit solves the local incompressible Navier Stokes equation. Writing (\ref{navode200locsub2}) in the form
\begin{equation}\label{navode200locsub2mat}
\left\lbrace \begin{array}{ll}
\frac{d \mathbf{v}^p_{1}}{dt}=B\mathbf{v}^p_1\\
\\
\mathbf{v}^p_{1}(t_0)=\mathbf{v}(t_0),
\end{array}\right.
\end{equation}
where $\mathbf{v}^p_1=\left(v^p_{11},\cdots, v^p_{1n} \right)^T$ and $\mathbf{v}(t_0)=(v_1(t_0),\cdots,v_n(t_0))^T$ are lists of infinite vectors and $B$ is a $n{\mathbb Z}^n\times n{\mathbb Z}^n$ matrix such that the system  (\ref{navode200locsub2mat}) is equivalent to the system in (\ref{navode200locsub2}). Note that the matrix $B$ depends only on data which are known from the previous iteration step $p-1$.
For this subproblem the local solution can be represented in the form
\begin{equation}\label{dysonobs}
\begin{array}{ll}
\mathbf{v}^p_1=v(t_0)+\\
\\
\sum_{m=1}^{\infty}\frac{1}{m!}\int_0^tds_1\int_0^tds_2\cdots \int_0^tds_m T_m\left(B(t_1)B(t_2)\cdot \cdots \cdot B(t_m) \right)\mathbf{v}(t_0),
\end{array} 
\end{equation}
where for $m=2$ we define
\begin{equation}
T_2\left(B(t_1)B(t_2)\right)=\left\lbrace \begin{array}{ll}
B(t_1)B(t_2)~~\mbox{ if }~t_1\geq t_2\\
\\
B(t_2)B(t_1)~~\mbox{ if }~t_2>t_1,
\end{array}\right.
\end{equation}
and for $m>2$ the time order operator $T_m$ may be defined recursively. Let ${\cal T}_m:\left\lbrace t_1,\cdots,t_m|t_i\geq 0~\mbox{for}~1\leq i\leq m\right\rbrace$.
Having defined $T_m$ for some $m\geq 2$ for $m+1\geq 3$ we first define
\begin{equation}
T_{m+1,\leq t_{m+1}}\left(B(t_1)B(t_2)\cdot \cdots \cdot B(t_{m+1})\right)=T_k\left(B_{t_{k_1}},\cdots ,B_{t_{k_j}} \right), 
\end{equation}
where $\left\lbrace t_{k_1},\cdots ,t_{k_j}\right\rbrace:=\left\lbrace t_i \in {\cal T}_m|t_i\leq t_m~\mbox{and}~i\neq m\right\rbrace $, and
\begin{equation}
T_{m+1,>t_{m+1}}\left(B(t_1)B(t_2)\cdot \cdots \cdot B(t_{m+1})\right)=T_l\left(B_{t_{l_1}},\cdots ,B_{t_{l_i}} \right), 
\end{equation}
where $\left\lbrace t_{l_1},\cdots ,t_{l_i}\right\rbrace:=\left\lbrace t_p \in {\cal T}_m|t_i> t_m\right\rbrace $. Then
\begin{equation}
\begin{array}{ll}
T_{m+1}\left(B(t_1)B(t_2)\cdot \cdots\cdot B(t_{m+1)} \right)\\
\\
=T_{m+1,\leq t_{m+1}}\left(B(t_1)B(t_2)\cdot \cdots \cdot B(t_{m+1})\right)\times\\
\\
\times B(t_{m+1}) T_{m+1,>t_{m+1}}\left(B(t_1)B(t_2)\cdot \cdots \cdot B(t_{m+1})\right) .
\end{array}
\end{equation}

We shall observe that for data $v_i(t_0)\in h^s\left({\mathbb Z}^n\right)$ for $s>n+2$ and $1\leq i\leq n$ a natural iteration scheme preserves regularity in the sense that $v^p_{1i}\in h^s\left({\mathbb Z}^n\right)$ for all $1\leq i\leq n$. and $p\geq 1$. Next we define some more basic parts of a natural time-local solution scheme.
Consider the first step in our scheme for the vectors $\left( v^p_{i\alpha}\right)_{\alpha\in {\mathbb Z}^n}$. For $p=0$ the modes $\left( v^0_{i\alpha}\right)_{\alpha\in {\mathbb Z}^n}$ are equal to the modes of the (auto-)controlled scheme we describe below and the corresponding equation can be written in the form 
\begin{equation}\label{navode200linstage0}
\begin{array}{ll}
\frac{d v^0_{i\alpha}}{dt}=\sum_{j=1}^n\nu \left( -\frac{4\pi \alpha_j^2}{l^2}\right)v^0_{i\alpha}
-\sum_{j=1}^n\sum_{\gamma \in {\mathbb Z}^n}\frac{2\pi i \gamma_j}{l}h_{j(\alpha-\gamma)}v^0_{i\gamma}\\
\\
+2\pi i\alpha_i1_{\left\lbrace \alpha\neq 0\right\rbrace}\frac{\sum_{j,k=1}^n\sum_{\gamma\in {\mathbb Z}^n}4\pi \gamma_j(\alpha_k-\gamma_k)v^{0}_{j\gamma}h_{k(\alpha-\gamma)}}{\sum_{i=1}^n4\pi\alpha_i^2}.
\end{array} 
\end{equation}
If we consider the $n$-tuple $\mathbf{v}^F=\left( \mathbf{v}^F_1,\cdots,\mathbf{v}^F_n\right)^T$ as an infinite vector with $\mathbf{v}^F_i:=\left(v_{i\alpha}\right)_{\alpha\in {\mathbb Z}^n}$, then with the usual identifications 
the equation (\ref{navode200linstage0}) is equivalent to an infinite linear ODE
\begin{equation}
\frac{d \mathbf{v}^{0,F}}{dt}=A_0\mathbf{v}^{0,F},
\end{equation}
where the matrix $A_0$ is implicitly defined by  (\ref{navode200linstage0}) and will be given explicitly in our more detailed description below. Together with the initial data 
\begin{equation}
\mathbf{v}^{0,F}(0)=\mathbf{h}^F(0)
\end{equation}
 this is an equivalent formulation of the equation for the first step of our scheme. We shall define a dissipative diagonal matrix $D_0$ and a matrix $B_0$ related to the convection and Leray projection terms such that $A_0=D_0+B_0$, and prove a Trotter product formula which allows us to make sense of the formal solution
\begin{equation}
\begin{array}{ll}
\mathbf{v}^{0,F}(t)=\exp(A_0t)\mathbf{h}^F\\
\\
:=\lim_{l\uparrow \infty}\lim_{k\uparrow \infty}\left( \exp\left(P_{M^l}(D_0)\frac{t}{k}\right)\exp\left(P_{M^l}B_0\frac{t}{k} \right)\right)^k\mathbf{h}^F.
\end{array}
\end{equation}
Here $P_{M^l}$ denotes a projection to the finite modes of order less or equal to $l>0$. Here and in the following we may consider some order of the multiindices and assume that this order is preserved by the projection operators.
Note that at this first stage the modes $h_{i\alpha}$ are not time-dependent. Hence $A_0$ is defined as a matrix which is independent of time, and we have no need of a Dyson formalism at this stage. 
For the higher stages of approximation we have to deal with time dependence of the related infinite matrices $A_p$ which define the infinite ODEs at iteration step $p\geq 0$ for the modes $v^{p}_{i\alpha}$ and $v^{r,p}_{i\alpha}$ in the presence of a control function $r$. 
The formula (\ref{trottappl}) makes clear why strong contraction estimates of the time-local expressions
\begin{equation}\label{timeloc}
\exp\left(\frac{t}{k}W\right)\exp\left(\frac{t}{k}\nu\Delta\right)f
\end{equation}
are important. As we said, we pointed out this in \cite{KB2} and \cite{KB3} from a different point of view. In dual spaces we may approximate such expressions via matrix equations of finite modes (projections of expressions of the form (\ref{timeloc}) to approximating equations of finite modes).  
For the finite mode approximation of the first approximation sequence $v^0_{i\alpha},~\alpha\in {\mathbb Z}^n$ first order coefficients are time independent and we may use the Baker-Campbell-Hausdorff formula for finite  matrices $A$ and $B$ of the form
\begin{equation}\label{BCH}
 \exp(A)\exp(B)=\exp(C),
\end{equation}
where
\begin{equation}\label{Ceq}
\begin{array}{ll}
C=A+B+\frac{1}{2}\left[A,B\right]+\frac{1}{12}\left[A,\left[A,B\right]\right]+\frac{1}{12}\left[\left[A,B\right],B\right]\\
\\
+\mbox{ higher order terms}.
\end{array}
\end{equation}
Here $\left[.,. \right]$ denote Lie brackets and the expression 'higher order terms' refers to all terms of multiple commutators of $A$ and $B$. We shall see that for dissipative operators as in the case of the incompressible Navier-Stokes equation and its linear approximations we can prove extensions of the formula in (\ref{BCH}), or  we can apply the formula in (\ref{BCH}) considering limits to infinite matrices applied to infinite vectors with polynomial decay of some order, where we restrict our investigation to the special cases which fit for the analysis of some infinite linear ODEs approximating the incompressible Navier-Stokes equation written in dual space. Note that for higher order approximations $v^p_{i\alpha},~\alpha\in {\mathbb Z}^n,~1\leq i\leq n$ with $p\geq 1$ we have time dependence of the coefficients and this means that we have to apply a time order operator as in  Dyson's formalism in order to solve the related linear approximating equations formally. Using our observations on well defined infinite matrix operations for function spaces with appropriate polynomial decay these formal Dyson formalism solutions can be justified rigorously. Note that (\ref{Ceq}) is closely connected to the H\"{o}rmander condition, and this was one of the indicators which lead to the expectation in \cite{K3} that global smooth existence is true if the H\"{o}rmander condition is satisfied. We considered this \cite{KHyp}.

In this paper we shall see that we can simplify the schemes formulated in classical spaces in our formulation on dual spaces in some respects. The first simplification is that we may define an iteration scheme on a global time scale, where the growth of the scheme may be estimated via an autocontrolled subscheme, i.e., a scheme where damping potential terms are introduced via time dilatation. The second simplification is that the estimates in dual spaces become estimates of discrete infinite sums which can be done on a very elementary level. Furthermore, it is useful to have a control function which ensures that the scheme for the controlled equation is a scheme for non-zero modes (ensuring that the damping factors associated with strictly positive viscosity or dissipative effects are active for all modes of the controlled scheme). For numerical and analytical purposes we may choose specific $\nu$ and $l$ and define a controlled scheme for a controlled incompressible Navier-Stokes equation such that the diagonal matrix elements corresponding to the Laplacian terms become dominant and such that there exists a global iteration in an appropriate function space. For analytical purposes it is useful to estimate the growth via comparison with a time dilated systems iteratively and locally in time. The control function is a scalar univariate function and is much simpler than the control functions considered in \cite{KNS} and \cite{K3}. Indeed, the control function in the present paper will only control the zero modes $v_{i0}$ of the value function modes. 

Next we define a controlled scheme and an extended controlled scheme with an autocontrol. We may start with one of the two possibilities mentioned above. We may start with the solution scheme for the multivariate Burgers equation, which is given in dual representation by the infinite ODE
\begin{equation}\label{navode200**}
\begin{array}{ll}
\frac{d u_{i\alpha}}{dt}=\sum_{j=1}^n\nu \left( -\frac{4\pi \alpha_j^2}{l^2}\right)u_{i\alpha}
-\sum_{j=1}^n\sum_{\gamma \in {\mathbb Z}^n}\frac{2\pi i \gamma_j}{l}u_{j(\alpha-\gamma)}u_{i\gamma},
\end{array} 
\end{equation}
where the modes $u_{i\alpha}$ depend on time $t$ and such that for all $1\leq i\leq n$ and all $\alpha\in {\mathbb Z}^n$ we have
\begin{equation}
u_{i\alpha}(0)=h_{i\alpha}.
\end{equation}
This is a slight variation of the linearized scheme mentioned above. Indeed for algorithmic purposes we may start with a linearization of (\ref{navode200**}) and call this $u_{i\alpha}$ as well. It does not really matter for analytical purposes such as existence and regularity. Anyway, we may define $v^{0}_{i\alpha}=u_{i\alpha}$ for $1\leq i\leq n$ and $\alpha\in {\mathbb Z}^n$ and then iteratively
\begin{equation}
v^{k}_{i\alpha}:=v^{0}_{i\alpha}+\sum_{p=1}^k\delta v^{p}_{i\alpha},
\end{equation}
where for $p\geq 1$ we define $\delta v^{p}_{i\alpha}:=v^{p}_{i\alpha}-v^{(p-1)}_{i\alpha}$ for all $1\leq i\leq n$ and $\alpha\in {\mathbb Z}^n$ such that
\begin{equation}\label{navode20001}
\begin{array}{ll}
\frac{d \delta v^p_{i\alpha}}{dt}=\sum_{j=1}^n\nu \left( -\frac{4\pi \alpha_j^2}{l^2}\right)\delta v^p_{i\alpha}
-\sum_{j=1}^n\sum_{\gamma \in {\mathbb Z}^n\setminus \left\lbrace \alpha\right\rbrace }\frac{2\pi i \gamma_j}{l}v^{p-1}_{j(\alpha-\gamma)}\delta v^p_{i\gamma}\\
\\
-\sum_{j=1}^n\frac{2\pi i \alpha_j}{l}v^{p-1}_{j0}\delta v^p_{i\alpha}
-\sum_{j=1}^n\sum_{\gamma \in {\mathbb Z}^n\setminus \left\lbrace \alpha\right\rbrace}\frac{2\pi i \gamma_j}{l}\delta v^{p-1}_{j(\alpha-\gamma)} v^{p-1}_{i\gamma}\\
\\
-\sum_{j=1}^n\frac{2\pi i \alpha_j}{l}\delta v^{p-1}_{j0} v^{p-1}_{i\alpha}\\
\\
+2\pi i\alpha_i1_{\left\lbrace \alpha\neq 0\right\rbrace}\frac{\sum_{j,k=1}^n\sum_{\gamma\in {\mathbb Z}^n}4\pi \gamma_j(\alpha_k-\gamma_k) v^{p-1}_{k(\alpha-\gamma)}\delta v^{p}_{j\gamma}}{\sum_{i=1}^n4\pi\alpha_i^2}\\
\\
+2\pi i\alpha_i1_{\left\lbrace \alpha\neq 0\right\rbrace}\frac{\sum_{j,k=1}^n\sum_{\gamma\in {\mathbb Z}^n}4\pi \gamma_j(\alpha_k-\gamma_k)\delta v^{p-1}_{k(\alpha-\gamma)} v^{p-1}_{j\gamma}}{\sum_{i=1}^n4\pi\alpha_i^2}
\end{array} 
\end{equation}
(where for $p=0$ and $p=-1$ we define $\delta v^{p}_{i\alpha}=0$). Note that we extracted the zero modes from the sums in (\ref{navode20001}). The reason is that we want to extend this scheme to a controlled scheme which is equivalent but has no zero modes. The present local scheme suggests this because the Leray projection terms do not contain zero modes. Note that for all $1\leq i\leq n$ and all $\alpha\in {\mathbb Z}^n$ we have
\begin{equation}
\delta v^{p}_{i\alpha}(0)=0.
\end{equation}
Global smooth existence at stage $p$ of the construction means that the $v^p_{i\alpha}$ are defined for all time ${\mathbb R}_+=\left\lbrace t\in {\mathbb R}|t\geq 0\right\rbrace$, and such that polynomial decay of the modes is preserved throughout time.  
Next we introduce the idea of a simplified control function was introduced in \cite{KNS} and \cite{K3}, and \cite{KB3} in a simplified but still complicated form.
Here we introduce a control function for the zero modes. We define
\begin{equation}\label{controlstart}
v^{r,p}_{i\alpha}(t)=v^{p}_{i\alpha }(t)+r^p_0(t),
\end{equation}
where 
\begin{equation}
r^p_0:[0,\infty)\rightarrow {\mathbb R}
\end{equation}
is defined by
\begin{equation}
r^p_0(t)=-v^{p}_{i0 }(t).
\end{equation}
We have to show then that $r_0:[0,\infty)\rightarrow {\mathbb R}$ is well-defined (especially bounded), i.e., we have to show that there is a limit
\begin{equation}
r_0(t):=\lim_{p\uparrow \infty} r^p_0(t)=r^0_0(t)+\sum_{p=1}^{\infty}\delta r^p_0(t)
\end{equation}
along with $\delta r^p_0(t)=r^p_0(t)-r^{p-1}_0(t)$. 
This leads to the following scheme. 
We start with the solution for the multivariate Burgers equation, or a linearized version of the scheme with first order coefficients related to the initial data. Then we annihilate the zero modes, i.e. we define
\begin{equation}
v^{r,0}_{i\alpha}=v^0_{i\alpha}=u_{i\alpha}
\end{equation}
for $\alpha\neq 0$ and
\begin{equation}
v^{r,0}_{i0}=v^{0}_{i0}+r^0_0=0,
\end{equation}
where for all $t\geq 0$
\begin{equation}
r^0_0(t)=-u_{i0}(t).
\end{equation}
For $p\geq 1$ and for $1\leq i\leq n$ and $\alpha\in {\mathbb Z}^n\setminus \left\lbrace 0 \right\rbrace$  we define
\begin{equation}
v^{r,k}_{i\alpha}:=v^{r,0}_{i\alpha}+\sum_{p=1}^k\delta v^{r,p}_{i\alpha},
\end{equation}
where $\delta v^{r,p}_{i\alpha}:=\delta v^{p}_{i\alpha}+r^p_{i\alpha}$ 
for all $1\leq i\leq n$ and $\alpha\in {\mathbb Z}^n\setminus \left\lbrace 0\right\rbrace $, and 
\begin{equation}\label{navode20001r}
\begin{array}{ll}
\frac{d \delta v^{r,p}_{i\alpha}}{dt}=\sum_{j=1}^n\nu \left( -\frac{4\pi \alpha_j^2}{l^2}\right)\delta v^{r,p}_{i\alpha}
-\sum_{j=1}^n\sum_{\gamma \in {\mathbb Z}^n\setminus \left\lbrace \alpha\right\rbrace }\frac{2\pi i \gamma_j}{l}v^{r,p-1}_{j(\alpha-\gamma)}\delta v^{r,p}_{i\gamma}\\
\\
-\sum_{j=1}^n\sum_{\gamma \in {\mathbb Z}^n\setminus \left\lbrace \alpha\right\rbrace}\frac{2\pi i \gamma_j}{l}\delta v^{r,p-1}_{j(\alpha-\gamma)} v^{r,p-1}_{i\gamma}\\
\\
+2\pi i\alpha_i1_{\left\lbrace \alpha\neq 0\right\rbrace}\frac{\sum_{j,k=1}^n\sum_{\gamma\in {\mathbb Z}^n}4\pi \gamma_j(\alpha_k-\gamma_k)v^{r,p-1}_{j\gamma}\delta v^{r,p}_{k(\alpha-\gamma)}}{\sum_{i=1}^n4\pi\alpha_i^2}\\
\\
+2\pi i\alpha_i1_{\left\lbrace \alpha\neq 0\right\rbrace}\frac{\sum_{j,k=1}^n\sum_{\gamma\in {\mathbb Z}^n}4\pi \gamma_j(\alpha_k-\gamma_k)\delta v^{r,p-1}_{j\gamma}v^{r,p-1}_{k(\alpha-\gamma)}}{\sum_{i=1}^n4\pi\alpha_i^2}.
\end{array} 
\end{equation}
Furthermore, for all $1\leq i\leq n$ and $\alpha=0$ we shall ensure that
\begin{equation}
\delta v^{r,p}_{i\alpha}(0)=0.
\end{equation}
Note that in the equation (\ref{navode20001r}) the terms
\begin{equation}
\begin{array}{ll}
-\sum_{j=1}^n\frac{2\pi i \alpha_j}{l}v^{p-1}_{j0}\delta v^p_{i\alpha}-\sum_{j=1}^n\frac{2\pi i \alpha_j}{l}\delta v^{p-1}_{j0} v^{p-1}_{i\alpha}
\end{array}
\end{equation}
on the right side of (\ref{navode20001}) are cancelled. This is because we define the control function $r_0$ such that the zero modes become zero. This is done as follows.
First we note that for $p=0$ and $p=-1$ we may define $\delta v^{r,p}_{i\alpha}=0$ for all $\alpha \in {\mathbb Z}^n$.
For $\alpha=0$ we define first the increment $\delta v^{*,r,p}_{i0}$ for $p\geq 1$ via the equation
\begin{equation}\label{navode200011}
\begin{array}{ll}
\frac{d \delta v^{*,r,p}_{i0}}{dt}=\sum_{j=1}^n\nu \left( -\frac{4\pi \alpha_j^2}{l^2}\right)\delta v^{*,r,p}_{i0}
-\sum_{j=1}^n\sum_{\gamma \in {\mathbb Z}^n\setminus \left\lbrace \alpha\right\rbrace }\frac{2\pi i \gamma_j}{l}v^{*,r,p-1}_{j(-\gamma)}\delta v^{*,r,p}_{i\gamma}\\
\\
-\sum_{j=1}^n\sum_{\gamma \in {\mathbb Z}^n\setminus \left\lbrace \alpha\right\rbrace}\frac{2\pi i \gamma_j}{l}\delta v^{*,r,p-1}_{j(-\gamma)} v^{*,r,p-1}_{i\gamma}.
\end{array} 
\end{equation}
Then we define
\begin{equation}\label{controlend}
\delta r^p_0(t)=-\delta v^{*,r,p}_{i0},
\end{equation}
closing the recursion. Note that this ensures
\begin{equation}
\delta v^{r,p}_{i0}=0
\end{equation}
at all stages $p\geq 1$ for the zero modes (provided that we can ensure that $r^p_0$ is globally well-defined).
The introduction of the control function for the zero modes makes the system 'autonomous' for the modes $\alpha\in {\mathbb Z}^n\setminus \left\lbrace 0\right\rbrace $. This means that all modes have damping factors (involving $\nu>0$), i.e., we have the stabilizing diagonal factor terms
\begin{equation}
\delta_{\alpha\beta}\nu\left(-\sum_{j=1}^n\frac{4\pi \alpha_j^2}{l^2} \right),
\end{equation}
which contribute for
\begin{equation}
\sum_{j=1}^n\alpha_j^2\neq 0,~\mbox{or}~\alpha_i\neq 0~\forall~1\leq i\leq n.
\end{equation}
We observed that the parameter $\nu>0$ may be large without loss of generality. This alone may indicate that the matter of global smooth existence should be detached from the subject of turbulence as simulation indicate turbulent phenomena for high Reynold numbers. 
Solutions of each approximation step of the scheme can be represented in terms of fundamental solutions of scalar equations, where each iteration step of higher order involves fundamental solutions of linear scalar parabolic equations. This is an idea which we considered earlier in \cite{KB1}, \cite{KB2},  \cite{K3}, and \cite{KNS}. In this paper we consider 'fundamental solutions' in a dual space of Fourier modes. They may be called 'generalized $\theta$-functions' but in general they live only in distributional space of negative Sobolev norm. However, we shall see that for certain initial data we can make sense of the related analytic vectors, and show that there spatial convergence radius is global (holds for the full size of the $n$-torus).

In this paper we observe that for all data $\mathbf{h}=\left(h_1,\cdots ,h_n \right)$ along with $h_i\in C^{\infty}$ and all positive real numbers $\nu >0$ the global extended iteration scheme converges globally to a solution of an infinite ODE system equivalent to a controlled Navier-Stokes equation in a strong norm, i.e. for all $t\geq 0$ the sequences $v^p_{i\alpha},~\alpha\in {\mathbb Z}^n$ converge in the dual Sobolev space $h^s\left( {\mathbb Z}^n\right)$ for $s\geq n+2$, where the choice $n+2$ corresponds naturally to certain requirements of infinite matrix operations. Clearly and a fortiori, the sequences converge also in $h^s$ for $s\leq 2$ (as we noted in a former version of this paper). Moreover, the control function is a globally well-defined univariate bounded differentiable  function such that controlled incompressible Navier-Stokes equation is equivalent to the usual incompressible Navier-Stokes equation system. 
More precisely, we have
\begin{thm}\label{mainthm1}
Given real numbers $l>0$ and $\nu>0$ such for all $1\leq p\leq n$ and data
$\mathbf{h}\in \left[C^{\infty}\left({\mathbb T}^n_l\right)\right]^n$ for all $s\in {\mathbb R}$
we have
\begin{equation}
v^{r}_{i\alpha}\left( (t)\right)_{\alpha\in {\mathbb Z}^n}=\lim_{m\uparrow \infty}\left( v^{r,m}_{i\alpha}(t)\right)_{\alpha\in {\mathbb Z}^n}\in h^s_l\left({\mathbb Z}^n\right).
\end{equation}
for the scheme described in the introduction. Moreover, $t\rightarrow \left( v^{r}_{i\alpha}(t)\right)_{\alpha\in {\mathbb Z}^n}$ satisfies the infinite ODE system in (\ref{navode20001}).
This implies that $v_{i\alpha}(t)=v^r_{i\alpha}(t)-\delta_{0\alpha}r_i(t)$ satisfies the infinite ODE-system corresponding to the incompressible Navier-Stokes equation, where $\delta_{0\alpha}=0$ if $\alpha\neq 0$ and $\delta_{\alpha 0}=1$ if $\alpha=0$, i.e., the function
\begin{equation}\label{vi}
v_j:=\sum_{\alpha\in {\mathbb Z}^n}v_{j\alpha}\exp{(2\pi i\alpha x)},~1\leq j\leq n,
\end{equation}
is a global classical solution of the Navier stokes equation system (\ref{nav}) below. Note that $v_{j\alpha},~\alpha\in {\mathbb Z}^n$ are functions of time as is $v_j$ in (\ref{vi}). Moreover the modes $v_{j\alpha}$ in (\ref{vi}) are real, i.e. the solution in (\ref{vi}) has a representation
\begin{equation}
v_j:=\sum_{\alpha\in {\mathbb N}^n}\left( v_{sj\alpha}\sin{(2\pi\alpha x)}+v_{cj\alpha}\cos{(2\pi\alpha x)}\right) ,~1\leq j\leq n,
\end{equation}
where $v_{sj\alpha}(t)\in {\mathbb R}$ and $v_{cj\alpha}(t)\in {\mathbb R}$ for all $t\geq 0$. Furthermore the solution scheme implies that for regular data $h_i\in H^s$ for $s>n+2+r$ and all $1\leq i\leq n$ we have global regular solutions $v_i\in H^s,~1\leq i\leq n$ of the incompressible Navier Stokes equation with the same $s>n+2+r$.    
\end{thm}

Note that in the latter statement this degree of regularity is the optimal regularity iff we consider the data as part of the solution. 
Furthermore, we say that $\mathbf{v}=(v_1,\cdots ,v_n)^T$ is a global classical solution of the incompressible Navier-Stokes equation system
\begin{equation}\label{nav}
\left\lbrace \begin{array}{ll}
    \frac{\partial\mathbf{v}}{\partial t}-\nu \Delta \mathbf{v}+ (\mathbf{v} \cdot \nabla) \mathbf{v} = - \nabla p~~~~t\geq 0,\\
    \\
 \nabla \cdot \mathbf{v} = 0,~~~~t\geq 0,\\
 \\
 \mathbf{v}(0,x)=\mathbf{h}(x),
\end{array}\right.,
\end{equation}
on some domain $\Omega$ (which is the $n$-torus ${\mathbb T}^n$ in this paper) if $\mathbf{v}$ solves the equivalent Navier-Stokes equation system in its Leray projection form,
where $p$ is eliminated by the Poisson equation
\begin{equation}\label{press}
-\Delta p=\sum_{j,k=1}^nv_{j,k}v_{k,j}.
\end{equation}
Note that $v_{j,k}$ denotes the first partial derivative of $v_j$ with respect to the component $x_k$ etc. (Einstein notation). This means that we construct a solution to an equation of the form 
 \begin{equation}\label{navleraya}
\left\lbrace \begin{array}{ll}
    \frac{\partial\mathbf{v}}{\partial t}-\nu \Delta \mathbf{v}+ (\mathbf{v} \cdot \nabla) \mathbf{v} =\sum_{j,k=1}^n\int \nabla K_{{\mathbb T}^n}(x-y)v_{j,k}v_{k,j}(t,y)dy,\\ 
 \\
 \mathbf{v}(0,x)=\mathbf{h}(x),
\end{array}\right.
\end{equation}
along with (\ref{press}). Here $K_{{\mathbb T}^n}$ is a kernel on the torus which can be de determinde via Fourier tranformation.
For a fixed function $\mathbf{v}$ which solves (\ref{navleraya}) with (\ref{press}) the Cauchy problem for divergence
\begin{equation}
\begin{array}{ll}
\frac{\partial}{\partial t}\mbox{div} \mathbf{v}-\nu \Delta \div \mathbf{v}+\sum_{j=1}^nv_j\div \mathbf{v}+\sum_{j,k=1}^nv_{j,k}v_{k,j}=-\Delta p\\
\\
\mbox{div}~\mathbf{v}(0,.)=0
\end{array}
\end{equation}
simplifies to
\begin{equation}\label{cauchydiv}
\begin{array}{ll}
\frac{\partial}{\partial t}\mbox{div}~\mathbf{v}-\nu \Delta \div \mathbf{v}+\sum_{j=1}^nv_j\div \mathbf{v}=0\\
\\
\mbox{div}~\mathbf{v}(0,.)=0,
\end{array}
\end{equation}
yielding $\div \mathbf{v}=0$ as the unique solution of (\ref{cauchydiv}). For this reason it suffices to solve the Navier-Stokes equation in its Leray projection form (as is well-known). 
In the next Section we look at the structure of the proof of the main theorems. Then in Section 3 we prove theorem \ref{mainthm1}. 

\begin{cor}\label{mainthm2}
The statements of (\ref{mainthm2}) hold for all numbers $l>0$ and $\nu>0$. Furthermore, the solution is unique in the function space
\begin{equation}
F_{\mbox{ns}}=\left\lbrace \mathbf{g}:\left[0,\infty\right)\times{\mathbb R}^n\rightarrow {\mathbb R}^n|g_i\in C^{1}\left( \left[0,\infty\right),H^s\left( {\mathbb R}^n\right) \right)~\&~g_i(0,.)=h_i \right\rbrace ,
\end{equation}
where $\mathbf{g}=(g_1,\cdots ,g_n)^T$ and $h_i\in H^s$ for $s>n+2\geq 3$. 
\end{cor}

\begin{rem}
For the assumptions of (\ref{mainthm2}) the methods of this paper lead directly to contraction results. Concerning uniqueness the regularity assumptions of the initial data may be weakened, of course. 
\end{rem}

If the viscosity converges to zero, then we get the incompressible Euler equation. A Trotter product formula does not hold but we can apply the Dyson formalism for infinite regular matrix operations (in the sense defined in this paper) directly. However we loose uniqueness for $n\geq 3$ and for $n\geq 3$ there exist singular solutions as well. 

\begin{cor}
For any dimension $n$ there exists a global regular solution branch of the incompressible Euler equation for data $h_i\in h^s\left({\mathbb Z}^n\right), 1\leq i\leq n$ along with $s>n+2$. 
\end{cor}
Here, by a global regular solution branch we mean a global solution where valuations of the $i$th component for $1\leq i\leq n$ at time $t\geq 0$ exist in $h^{s}\left({\mathbb Z}^n\right) $ for all time $t\geq 0$.
Note that an analogous result holds on the whole domain with spatial part on ${\mathbb R}^n$ where viscosity limits for contraction results of the Navier Stokes equation can be used. However, here we consider a different direct approach via a Dyson formalism. A viscosity limit approach via a Trotter product formula seems to be not appropriate in this case.   

\begin{cor}
For dimension $n\geq 3$ there exist singular solutions of the incompressible Euler equation for regular data $0\neq h_i\in h^s\left({\mathbb Z}^n\right), 1\leq i\leq n$ along with $s>n+2$. 
\end{cor}

The lower bound $n=3$ for singular solutions is related to the fact that the class of singular solutions which can be constructed has no analogue in dimension $n\leq 2$. The constructive ansatz of the solution collapses to the trivial solution. 

\section{Structure and ideas of the proof of theorem \ref{mainthm1} and corollary \ref{mainthm2}}

In the proof of theorem \ref{mainthm1} in section 3 below  we first recall  that the incompressible Navier-Stokes equation is formally equivalent to a system of $n$ infinite ODE-systems
\begin{equation}\label{navode2*}
\begin{array}{ll}
\frac{d v_{i\alpha}}{dt}=\sum_{j=1}^n \nu\left( -\frac{4\pi \alpha_j^2}{l^2}\right)v_{i\alpha}
-\sum_{j=1}^n\sum_{\gamma \in {\mathbb Z}^n}\frac{2\pi i \gamma_j}{l}v_{j(\alpha-\gamma)}v_{i\gamma}\\
\\
+2\pi i\alpha_i1_{\left\lbrace \alpha\neq 0\right\rbrace }\frac{\sum_{j,k=1}^n\sum_{\gamma\in {\mathbb Z}^n}4\pi \gamma_j(\alpha_k-\gamma_k)v_{j\gamma}v_{k(\alpha-\gamma)}}{\sum_{i=1}^n4\pi\alpha_i^2},
\end{array} 
\end{equation}
where $\alpha\in {\mathbb Z}^n$, and where the initial data $\mathbf{v}^F_i(0)=\left( v_{i\alpha}(0)\right)_{\alpha \in {\mathbb Z}^n}$ for $1\leq i\leq n$ and $\alpha\in {\mathbb Z}^n$ are given by $n$ infinite vectors
\begin{equation}\label{initdatamod}
\mathbf{v}^F_i(0)=\mathbf{h}^F_i=\left(h_{i\alpha}\right)^T.
\end{equation}
Operations such as $\sum_{\gamma\in {\mathbb Z}^n}v_{k(\alpha-\gamma}v_{j\gamma}$ are interpreted as infinite matrix-vector operations. 
We also consider the equivalent representation with respect to the real basis
\begin{equation}\label{realbasisa}
\left\lbrace \sin\left(\frac{2\pi\alpha}{l}\right),\cos\left(\frac{ 2\pi\alpha}{l}\right) \right\rbrace_{\alpha \in {\mathbb N}^n}.
\end{equation}
In this case we use the notation 
\begin{equation}\label{initdatamodr}
\mathbf{v}^{re,F}_i(0)=\mathbf{h}^{re,F}_i=\left(h^{re}_{i\alpha}\right)^T
\end{equation}
for the initial data in order to indicate that we are referring to data in the real basis 
(\ref{realbasisa}). The structure is quite similar but for each multiindex $\alpha$ we have one equation for the cosine modes and one equation for the sinus modes (cf. \ref{aa}) below. For both types of modes we have operations of the form $\sum_{\gamma\in {\mathbb Z}^n}v_{k(\alpha-\gamma}v_{j\gamma}$ and operations of the form $\sum_{\gamma\in {\mathbb Z}^n}v_{k(\alpha+\gamma}v_{j\gamma}$ which can both be considered as equivalent infinite matrix vector operations. The estimates of these operations via weakly singular elliptic integrals are the same and the transition from the complex system to the real system and vice versa is straightforward. Comparison to the real system makes sure that the procedure written in the more succinct complex notation leads to real solutions. This is clear from an analytical point of view anyway, but from a computational point of view errors of computation should be compared to the real system in order to extract the optimal real approximative solutions of a solution with a possibly complex error. 
For numerical purposes considered later it is also useful to observe the dependence of the different terms on the viscosity $\nu$ and the size of the torus. 
The representation in (\ref{navode2*}) is more convenient than a representation with respect to the real data but we shall see at every step of the argument that an analogous argument also holds for the representation of the infinite ODE systems with respect to the real basis in (\ref{realbasisa}). 
The entries $h_{i\alpha}$ of the infinite vector in (\ref{initdatamod}) denote the Fourier modes of the functions $h_i$ along with $\mathbf{h}\in \left[ C^{\infty}\left({\mathbb T}^n_l\right) \right]^n$. Smoothness of $\mathbf{h}$ translates to polynomial decay of the modes $h_{i\alpha}$. Here we may say that the modes $h_{i\alpha}$ have polynomial decay of order $m>0$ if
\begin{equation}
|\alpha|^mh_{i\alpha}\downarrow 0 \mbox{ as }|\alpha|=\sum_{i=1}^n|\alpha_i|\uparrow \infty
\end{equation}
for a fixed positive integer $m$, and the modes $h_{i\alpha}$ have polynomial decay if they have polynomial of any order $m>0$. We may use the terms 'modes of solution have polynomial decay' and 'solution is smooth' interchangeably in our context.
We may rewrite the equation (\ref{navode2*}) in the form
\begin{equation}\label{navode2*rewr}
\begin{array}{ll}
\frac{d \mathbf{v}^F}{dt}=A^{NS}\left(\mathbf{v}\right) \mathbf{v}^F,
\end{array} 
\end{equation}
where $\mathbf{v}^F=\left(\mathbf{v}^F_1,\cdots ,\mathbf{v}^F_n\right)^T$. Furthermore $A^{NS}\left(\mathbf{v}\right) $ is a $n{\mathbb Z}^n\times n{\mathbb Z}^n$-matrix
\begin{equation}
A^{NS}\left(\mathbf{v}\right) =\left(A^{NS}_{ij}\left(\mathbf{v}\right)\right)_{1\leq i,j\leq n},
\end{equation}
where for $1\leq i,j\leq n$ the entry $A^{NS}_{ij}\left(\mathbf{v}\right) $ is a ${\mathbb Z}^n\times {\mathbb Z}^n$-matrix. We define
\begin{equation}
A^{NS}\left( \mathbf{v}\right)\mathbf{v}^F =\left(\sum_{j=1}^nA^{NS}_{1j}\left(\mathbf{v}\right) \mathbf{v}^F_j ,\cdots,\sum_{j=1}^nA^{NS}_{nj}\left( \mathbf{v}\right)  \mathbf{v}^F_j  \right)^T, 
\end{equation}
where for all $1\leq i\leq n$
\begin{equation}
\sum_{j=1}^nA^{NS}_{ij}\left(\mathbf{v}\right) \mathbf{v}^F_j=\left( \sum_{j=1}^n \sum_{\beta\in {\mathbb Z}^n}A^{NS}_{i\alpha j\beta}\left(\mathbf{v}\right) v_{j\beta}\right)_{\alpha\in {\mathbb Z}^n} . 
\end{equation}
The entries $A^{NS}_{i\alpha j\beta}\left(\mathbf{v}\right)$ of $A^{NS}\left( \mathbf{v}\right)$ are determined by the equation in (\ref{navode2*}) of course. On the diagonal, i.e., for $i=j$ we have the entries 
\begin{equation}
\begin{array}{ll}
\delta_{ij}A^{NS}_{i\alpha j\beta}\left(\mathbf{v}\right)=\delta_{ij\alpha\beta}\sum_{j=1}^n \nu\left( -\frac{4\pi \alpha_j^2}{l^2}\right)
-\delta_{ij}\frac{2\pi i (\alpha_j-\beta_j)}{l}v_{i(\alpha-\beta)}\\
\\+\delta_{ij}2\pi i\alpha_i1_{\left\lbrace \alpha\neq 0\right\rbrace }\frac{\sum_{k=1}^n4\pi \beta_j(\alpha_k-\beta_k)v_{k(\alpha-\beta)}}{\sum_{i=1}^n4\pi\alpha_i^2},
\end{array}
\end{equation}
where $\delta_{ij}=1$ iff $i=j$, $\delta_{ij\alpha\beta}=1$ iff $i=j$ and $\alpha=\beta$,  and zero else denotes the Kronecker $\delta$-function, and off-diagonal we have for $i\neq j$ the entries
\begin{equation}
\begin{array}{ll}
(1-\delta_{ij})A^{NS}_{i\alpha j\beta}\left(\mathbf{v}\right)=\frac{2\pi i (\alpha_j-\beta_j)}{l}v_{i(\alpha-\beta)}
\\
\\2\pi i\alpha_i1_{\left\lbrace \alpha\neq 0\right\rbrace }\frac{\sum_{k=1}^n4\pi \beta_j(\alpha_k-\beta_k)v_{k(\alpha-\beta)}}{\sum_{i=1}^n4\pi\alpha_i^2}
\end{array}
\end{equation}
Here, we remark again that the $i$ in the context $2\pi i$ denotes the complex number $i=\sqrt{-1}$, while the other indices $i$ are integer indices. The next step is to define a contolled system according to the ideas described in the introduction. The additional idea of an auto-control is useful in order to get better upper bounds of the solution. We postpone this and consider first the idea of a simple external control in dual space (which is simple compared to control functions in classical spaces we considered in \cite{KB3} and elsewhere). Recall the main idea: the control function cancels the zero modes in an iterative scheme. If this is possible and a regular limit exists, then it cancels the zero modes in the limit as well. 
In the real basis (\ref{realbasisa}) we may rewrite the equation (\ref{navode2*}) in the form
\begin{equation}\label{navode2*rewr}
\begin{array}{ll}
\frac{d \mathbf{v}^{re,F}}{dt}=A^{r,NS}\left(\mathbf{v}^{re}\right) \mathbf{v}^{re,F},
\end{array} 
\end{equation}
where $\mathbf{v}^{re,F}=\left(\mathbf{v}^{re,F}_1,\cdots ,\mathbf{v}^{re,F}_n\right)^T$. Furthermore $A^{re,NS}\left(\mathbf{v}^r\right) $ is a $2n{\mathbb N}^n\times 2n{\mathbb N}^n$-matrix
\begin{equation}
A^{re,NS}\left(\mathbf{v}\right) =\left(A^{re,NS}_{ij}\left(\mathbf{v}^{re}\right)\right)_{1\leq i,j\leq n} 
\end{equation}
where for $1\leq i,j\leq n$ the entries $A^{re,NS}_{ij}\left(\mathbf{v}^{re}\right) $ can be obtained easily from $A^{NS}_{ij}\left(\mathbf{v}\right) $ or derived independently - we shall do this in the proof below.

Consider the representation in (\ref{navode2*}). The zero modes do not appear in the dissipative term related to the Laplacian, and we have observed that we may represent the Leray projection term without zero mode terms as well. Hence, for the zero modes $v_{i0}$ in the equation in (\ref{navode2*}) we have for $\alpha=0$
\begin{equation}\label{navode2*alpha0}
\begin{array}{ll}
\frac{d v_{i0}}{dt}=
-\sum_{j=1}^n\sum_{\gamma \in {\mathbb Z}^n}\frac{2\pi i \gamma_j}{l}v_{j(-\gamma)}v_{i\gamma}.
\end{array} 
\end{equation}
Assuming that the function $t\rightarrow v_{i0}(t)$ is bounded it is natural to define
\begin{equation}
r^0_0(t):=-v_{i0}(t).
\end{equation}
Formally, this leads to the controlled equation for $v^{r}_{i\alpha},~\alpha\in {\mathbb Z}^n\setminus \{0\}={\mathbb Z}^{n,0}$ of the form
\begin{equation}\label{navode2*r}
\begin{array}{ll}
\frac{d v^r_{i\alpha}}{dt}=\sum_{j=1}^n \nu\left( -\frac{4\pi \alpha_j^2}{l^2}\right)v^r_{i\alpha}
-\sum_{j=1}^n\sum_{\gamma \in {\mathbb Z}^n\setminus \left\lbrace 0,\alpha\right\rbrace }\frac{2\pi i \gamma_j}{l}v^r_{j(\alpha-\gamma)}v^r_{i\gamma}\\
\\
+2\pi i\alpha_i1_{\left\lbrace \alpha\neq 0\right\rbrace }\frac{\sum_{j,k=1}^n\sum_{\gamma\in {\mathbb Z}^n\setminus \left\lbrace 0,\alpha\right\rbrace}4\pi \gamma_j(\alpha_k-\gamma_k)v^r_{j\gamma}v^r_{k(\alpha-\gamma)}}{\sum_{i=1}^n4\pi\alpha_i^2},
\end{array} 
\end{equation}
where the initial data $\mathbf{v}^{r,F}_i(0)=\left( v^r_{i\alpha}(0)\right)_{\alpha \in {\mathbb Z}^n\setminus \left\lbrace 0\right\rbrace }$ for $1\leq i\leq n$ and $\alpha\in {\mathbb Z}^n\setminus \left\lbrace 0\right\rbrace$ are given by $n$ infinite vectors
\begin{equation}
\mathbf{v}^{r,m,F}_i(0)=\mathbf{h}^{r,F}_i=\left(h_{i\alpha}\right)^T_{\alpha\in {\mathbb Z}^n\setminus \left\lbrace 0\right\rbrace}.
\end{equation}
Note that 
\begin{equation}
v_{i0}+r_{0}=0,
\end{equation}
such that we may cancel the zero modes. The justification for this is obtained for each iteration step $m$ below.

The idea for a global scheme is to determine $\mathbf{v}^{r,F}=\lim_{m\uparrow \infty}\mathbf{v}^{r,m,F}$ for a simple control function and a certain iteration
\begin{equation}\label{navode2*rewrlina}
\begin{array}{ll}
\frac{d \mathbf{v}^{r,m,F}}{dt}=A^{NS}\left(\mathbf{v}^{r,m-1}\right) \mathbf{v}^{r,m,F},
\end{array} 
\end{equation}
starting with $\mathbf{v}^{r,0}:=\mathbf{h}$ or with the information of a global solution of the multivariate Burgers equation. In the real basis we write (\ref{navode2*rewrlina}) in the form
 \begin{equation}\label{navode2*rewrlinareal}
\begin{array}{ll}
\frac{d \mathbf{v}^{re,r,m,F}}{dt}=A^{re,NS}\left(\mathbf{v}^{re,r,m-1}\right) \mathbf{v}^{r,m,F},
\end{array} 
\end{equation}
where we shall be more explicit below. The equations in (\ref{navode2*rewrlina}) and in (\ref{navode2*rewrlinareal}) are linear equations, which we can solve by Banach contraction principles in exponentially time weighted norms of time-dependent functions with values in strong Sobolev spaces, where the evaluation of each function component at time $t\geq 0$, i.e., $\mathbf{v}^{r,m,F}_i(t)$ lives in a Sobolev space with $h^s\left( {\mathbb Z}^n\right)$ for $s>n+2$. Spatial regularity of this order is observed if we have data of this order of regularity, where comparison with weakly singular elliptic intergals ensures that certain infinite matrix operations of a solution representations in terms of the data preserve this order of regularity. An exponential time weigh of the norm ensures that the representations are globally valid in time. Note that the equation in (\ref{navode2*rewrlina}) and its real form in (\ref{navode2*rewrlinareal}) correspond to linear partial integro-differential equations. This linearity can be used. In order to fine the solution of the incompressible Navier Stokes equation iterations of this type of linear partial integro-differential equations are considered. We observe that for the functional increments
\begin{equation}
\delta\mathbf{v}^{re,r,m+1,F}:=\mathbf{v}^{re,r,m+1,F}-\mathbf{v}^{re,r,m,F}
\end{equation}
for all $m\geq 0$ we have
\begin{equation}\label{navode2*rewrlinareal2}
\begin{array}{ll}
\frac{d \delta\mathbf{v}^{re,r,m+1,F}}{dt}=A^{re,NS}\left(\mathbf{v}^{re,r,m}\right) \delta\mathbf{v}^{re,r,m+1,F}+A^{re,NS}\left(\delta\mathbf{v}^{re,r,m}\right) \mathbf{v}^{r,m,F}.
\end{array} 
\end{equation}
Time discretization of the equations leads then to sequence of time-homogeneous equations which can be solved by Trotter product formulas. This is the most elementary method where we use Euler schemes. We also consider more sophisticated methods, where the Euler part of the equation is locally solved in time by a Dyson formalism. This leads to higher order schemes with respect to time straightforwardly. In order to derive Trotter product formulas we consider some multiplicative behavior of infinite matrix multiplications with polynomial decay (related to weakly singular integrals). Limits, where the time step size goes to zero, lead to Dyson type formulas for time dependent approximating linearized infinite ODEs. The simplest method in order to implement this program is an Euler time discretization above. There are refinements and higher order schemes and generalisations of the Trotter product formula. Note that the dissipative terms (corresponding to the Laplacian) are not time-dependent even for generalized models with spatially dependent viscosity, and we can integrate the 'Euler part' locally via a Dyson formula. More precisely, if $E^m\equiv E^m(t)$ denotes the matrix of the Euler part of a linearized approximating equation at iteration stage $p\geq 1$ (an equation such as in (\ref{navode2*rewrlina}) with viscosity $\nu=0$), then for given regular data $ \mathbf{v}^{m,E}(t_0),\mathbf{v}^{m-1,E}(t_0)$ at initial time $t_0\geq 0$ we have a well defined time-local solution
\begin{equation}\label{dysonobs}
\begin{array}{ll}
\mathbf{v}^{m,E}(t)=\mathbf{v}^{m,E}(t_0)+\\
\\
\sum_{p=1}^{\infty}\frac{1}{p!}\int_0^tds_1\int_0^tds_2\cdots \int_0^tds_p T_p\left(E^{m-1}(t_1)E^{m-1}(t_2)\cdot \cdots \cdot E^{m-1}(t_p) \right)\mathbf{v}^{m,E}(t_0),
\end{array} 
\end{equation}
where the data $\mathbf{v}^{m-1,E}(t_0)$ appear in the matrices $E^{m-1}(t_i)$. For linear equations with globally regular matrices $E^{m-1}(t_p)$ this representation is even globally valid in time, but without viscosity we loose uniqueness for the nonlinear Euler solution limit for dimension $n \geq 3$ (we shall come back to the reasons for this). This formula for the Euler part can be used for higher order numerical schemes. Note however, that the Euler scheme which we used in a former incarnation of this work is sufficient in order to prove global existence with polynomial decay of the data if the data have polynomial decay.
 The series of Dyson type solutions leads to contraction results in appropriate time-weighted function spaces of functions dependent of time and with values in some function spaces of infinite sequences of modes with a certain appropriate order of polynomial decay. 
 Indeed in the simplest version of a global existence proof we define Trotter product formulas for certain linear infinite equations with time independent coefficients, and then we define a time discretization in order to apply these Trotter product formula. Even the time-dependent linear approximations are then solved by double limits. For each finite approximating system of modes less or equal to some $l>0$ the Trotter product limit considered is first a time discretization limit where the order of modes remains fixed. This leads to a sequence of Dyson type formula limits where we have a polynomial decay of some order in the limit. We give some more details.    

In order to define a global solution scheme for this system of coupled infinite nonlinear ODEs we may start with the solution of the (viscous) multivariate Burgers equation (where we know that a unique global smooth solution exists), i.e., the solution of   
\begin{equation}\label{navode2*burg}
\begin{array}{ll}
\frac{d v^{B}_{i\alpha}}{dt}=\nu\sum_{j=1}^n \left( -\frac{4\pi \alpha_j^2}{l^2}\right)v^{B}_{i\alpha}
-\sum_{j=1}^n\sum_{\gamma \in {\mathbb Z}^n}\frac{2\pi i \gamma_j}{l}v^{B}_{j(\alpha-\gamma)}v^B_{i\gamma},
\end{array} 
\end{equation}
where $\alpha\in {\mathbb Z}^n$, and where the initial data $\mathbf{v}^{BF}_i(0)=\left( v^B_{i\alpha}(0)\right)_{\alpha \in {\mathbb Z}^n}$ for $1\leq i\leq n$ and $\alpha\in {\mathbb Z}^n$ are given by $n$ infinite vectors
\begin{equation}
\mathbf{v}^{BF}_i(0)=\mathbf{h}^F_i=\left(h_{i\alpha}\right)^T.
\end{equation}
We know that the solution $\left( v^B_{i\alpha}\right)_{\alpha\in {\mathbb Z}^n,~1\leq i\leq n}$ is in $h^s\left({\mathbb Z}^n\right)$ for all $s\in {\mathbb R}$ for all time $t\geq 0$. However, without assuming knowledge about the multivariate Burgers equation we may also start with 
\begin{equation}\label{navodeh}
\begin{array}{ll}
\frac{d v^{0}_{i\alpha}}{dt}=\nu\sum_{j=1}^n \left( -\frac{4\pi \alpha_j^2}{l^2}\right)v^{0}_{i\alpha}
-\sum_{j=1}^n\sum_{\gamma \in {\mathbb Z}^n}\frac{2\pi i \gamma_j}{l}h_{j(\alpha-\gamma)}v^0_{i\gamma}\\
\\
+2\pi i\alpha_i1_{\left\lbrace \alpha\neq 0\right\rbrace}\frac{\sum_{j,k=1}^n\sum_{\gamma\in {\mathbb Z}^n}4\pi \gamma_j(\alpha_k-\gamma_k)v^0_{j\gamma}h_{k(\alpha-\gamma)}}{\sum_{i=1}^n4\pi\alpha_i^2},
\end{array} 
\end{equation}
where $\alpha\in {\mathbb Z}^n$, and where the initial data $\mathbf{v}^{0,F}(0)=\left( \mathbf{v}^{0,F}_1(0),\cdots,\mathbf{v}^{0,F}_n(0)\right)$ are given by $\mathbf{v}^{0,F}_i(0)=\left( v^0_{i\alpha}(0)\right)_{\alpha \in {\mathbb Z}^n}=\left( h_{i\alpha}\right)_{\alpha \in {\mathbb Z}^n}$ for $1\leq i\leq n$.
In the equation (\ref{navodeh}) and for $\nu >0$ the zero modes ($|\alpha|=0$) are the only modes where the damping (viscous) term
\begin{equation}
\nu\sum_{j=1}^n \left( -\frac{4\pi \alpha_j^2}{l^2}\right)v^{0}_{i\alpha}
\end{equation}
cancels. The same holds for the equation in (\ref{navode2*}) of course. Therefore we introduce the first approximation of a control function $r^0_{i0}:[0,\infty)\rightarrow {\mathbb R}$ for each $1\leq i\leq n$ such the scheme for $\mathbf{v}^{r,0,F}_i=\mathbf{v}^{0,F}_i+\mathbf{e}_0r^0_{i0}, 1\leq i\leq n$ becomes a system of nonzero modes. At each stage we have to prove that this is possible, of course, i.e., that the solution of the linear approximative problem exists. Here, $\mathbf{v}^{0,F}_i+\mathbf{e}_0r^0_{i0}=\left(v^0_{i\alpha} \right)_{\alpha\in {\mathbb Z}^n}=\left(v^{r,0}_{i\alpha} \right)_{\alpha\in {\mathbb Z}^n}+r^0_{i0}$ is a vector with 
\begin{equation}
v^{r,0}_{i\alpha}:=\left\lbrace \begin{array}{ll}
v^0_{i\alpha} \mbox{ if }\alpha\neq 0,\\
\\
v^0_{i0}+r⁰_{i0}=0 \mbox{ else.}
\end{array}\right.
\end{equation}
Hence, the control function is constructed such that it cancels the zero modes of the original system at each stage of the construction. More precisely, it is constructed as a series $(r^m_0)_m$ where the convergence as $m\uparrow \infty$ follows from properties of the controlled approximative solution functions $\mathbf{v}^{r,m,F}$. Note that it has to be shown that the control function is a) finite at each stage, and b) that it is finite in the limit of all stages. Note that the Leray projection term does not contribute to the zero modes in  (\ref{navodeh}), and this may already indicates this finiteness.
The next step is to analyze the scheme formally defined in the introduction starting from (\ref{controlstart}) to (\ref{controlend}). The effect of this formal scheme is that it is autonomous with respect to the nonzero modes, i.e., only the modes with $|\alpha|\neq 0$ are dynamically active. The fact that the suppression of the zero modes via a control function is possible, i.e., that the controlled system is well-defined, is shown within the proof. Then we get into the heart of the proof. The iteration leads to a sequence $\left( \mathbf{v}^{r,m,F}(t)\right)_{ m\in {\mathbb N}}=\left( \mathbf{v}^{r,m,F}_i(t)\right)_{1\leq i\leq n, m\in {\mathbb N}}$ with
\begin{equation}\label{approxstagem}
 \mathbf{v}^{r,m,F}(t):=T\exp\left(A^{r}_{m}t\right)\mathbf{h}^{r,F}_i,
\end{equation}
where $T$ is a time order operator (as in Dyson's formalism), and the restrictive dual function spaces together with the dissipative features of the operator $A^{r}_m$ will make sure that at each time step $m$ the approximation (\ref{approxstagem}) (corresponding to a linear equation) really makes sense. In the real basis we have an analogous sequence
$\left( \mathbf{v}^{re,r,m,F}(t)\right)_{ m\in {\mathbb N}}=\left( \mathbf{v}^{re,r,m,F}_i(t)\right)_{1\leq i\leq n, m\in {\mathbb N}}$ with
\begin{equation}\label{approxstagem}
 \mathbf{v}^{re,r,m,F}(t):=T\exp\left(A^{re,r}_{m}t\right)\mathbf{h}^{re,r,F}.
\end{equation}
In the following we use the complex notation. In the proof below, we shall supplement the related notation in the real basis and show why the argument holds also in the real basis.
Concerning dissipation features note that we may choose $\nu>0$ arbitrarily as is well-known and as we pointed out in the introduction. Fortunately, we do not need this which indicates that the algorithm can be extended to models with variable positive viscosity- which is desirable from a physical point of view. However, from an algorithmic perspective for the model with constant viscosity it may be useful do consider appropriate constellations of viscosity $\nu >0$ and torus size $l>0$. Note that we have a time-order operator $T$ for all stages $m\geq 1$ since the the infinite matrices $A^{r}_{m}$ depend on time $t$ for $m\geq 1$. In order to make sense of matrix multiplication of infinite unbounded matrices we consider the matrices involved in the computation of the first approximation $\mathbf{v}^{r,0,F}_i(t)=P_0\mathbf{v}^{0,F}_i(t)$ (where $P_0$ is the related projection operator which eliminates the zero mode) as an instructive example. We may rewrite (\ref{navodeh}) in the form
\begin{equation}\label{navodeh2}
\begin{array}{ll}
\frac{d v^{0}_{i\alpha}}{dt}=\nu\sum_{j=1}^n \left(-\frac{4\pi \alpha_j^2}{l^2}\right)v^{0}_{i\alpha}
-\sum_{j=1}^n\sum_{\gamma \in {\mathbb Z}^n}\frac{2\pi i \gamma_j}{l}h_{j(\alpha-\gamma)}v^0_{i\gamma}\\
\\
+2\pi i\alpha_i1_{\left\lbrace \alpha\neq 0\right\rbrace}\frac{\sum_{k=1}^n\sum_{\gamma\in {\mathbb Z}^n}4\pi \gamma_i(\alpha_k-\gamma_k)h_{k(\alpha-\gamma)}}{\sum_{i=1}^n4\pi\alpha_i^2}v^0_{i\gamma}\\
\\
+\sum_{j=1,j\neq i}^n2\pi i\alpha_i1_{\left\lbrace \alpha\neq 0\right\rbrace}\frac{\sum_{k=1}^n\sum_{\gamma\in {\mathbb Z}^n}4\pi \gamma_j(\alpha_k-\gamma_k)h_{k(\alpha-\gamma)}}{\sum_{i=1}^n4\pi\alpha_i^2}v^0_{j\gamma}.
\end{array} 
\end{equation}
Hence this infinite linear system can be written in terms of a matrix $A^r_0=A_0$ with $n\times n$ infinite matrix entries. More precisely,  we have
\begin{equation}
A_{0}=\left(A^{ij}_0\right)_{1\leq i,j\leq n}, 
\end{equation}
where for $1\leq i,j \leq n$ we have
\begin{equation}
A^{ij}_0=\delta_{ij}D^0+\delta_{ij}C^0+L^0_{ij}
\end{equation}
along with the Kronecker $\delta$-function $\delta_{ij}$, and with the infinite matrices
\begin{equation}\label{mat}
\begin{array}{ll}
D^0:=\left( -\nu\delta_{\alpha\beta}\sum_{j=1}^n\frac{4\pi \alpha_j^2}{l^2}\right)_{\alpha,\beta\in {\mathbb Z}^n},\\
\\
C^0_{ij}:=\left( -\sum_{j=1}^n\frac{2\pi i (\alpha_j-\beta_j)}{l}h_{i(\alpha-\beta)}\right)_{\alpha,\beta \in {\mathbb Z}^n},\\
\\
L^0_{ij}=\left( (2\pi i)\alpha_i1_{\left\lbrace \alpha\neq 0\right\rbrace}\frac{\sum_{k=1}^n4\pi \beta_j(\alpha_k-\beta_k)h_{k(\alpha-\beta)}}{\sum_{i=1}^n4\pi\alpha_i^2}\right)_{\alpha,\beta \in {\mathbb Z}^n},
\end{array} 
\end{equation}
corresponding to the the Laplacian, the convection term, and the Leray projection terms respectively. Strictly speaking, we have defined a $n\times n$ matrix of infinite matrices where the Burgers equation terms and the linearized Leray projection terms exist also off diagonal. It is clear how the matrix multiplication may be defined in order to reformulate (\ref{navodeh2}) and we do not dwell on these trivial formalities. We note that
\begin{equation}
\begin{array}{ll}
\frac{\partial \mathbf{v}^{0,F}}{\partial t}=A_0\mathbf{v}^{0,F},\\
\\
\mathbf{v}^{0,F}(0)=\mathbf{h}^F
\end{array}
\end{equation}
along with the vectors $\mathbf{v}^F=\left( \mathbf{v}^F_1,\cdots ,\mathbf{v}^F_n\right)^T$ and $\mathbf{h}^F=\left( \mathbf{h}^F,\cdots ,\mathbf{h}^F\right)^T$. Note that we have off-diagonal terms only because we consider global equations, i.e., equations, which correspond to linear partial integro-differential equations.
Note that the matrices $D^0, C^0_{ij}, L^0_{ij}$ are unbounded even for regular data $h_{i}$ with fast decreasing modes (for the the matrix $C^0_{ij}$ consider constant $\alpha-\beta$ and let $\beta_j$ go to infinity for some $j$). The difficulty to handle unbounded infinite matrices in dual space becomes apparent. However, multiplications of these matrices with the data $\mathbf{h}$ are regular, and because of the special structure of the matrices involved we have matrix-data multiplication which inherits polynomial decay and hence leads to a well-defined iteration of matrices (even in the Euler case).
In the case of an incompressible Navier Stokes equation at this point we can take additional advantage of the dissipative nature of the operator which is indeed the difference to the Euler equation. This is a major motivation for the dissipative Trotter product formula. We describe it here in the comlex notation. We shall supplement this with similar observations in the case of a real basis below. For two matrices
$M=\left( m_{\alpha\beta}\right)_{\alpha,\beta\in {\mathbb Z}^n}$ and $N=\left( n_{\alpha\beta}\right)_{\alpha,\beta\in {\mathbb Z}^n}$ we may formally define the product $P=\left( p_{\alpha\gamma}\right)_{\alpha,\gamma\in {\mathbb Z}^n}=MN$ via
\begin{equation}
p_{\alpha\gamma}=\sum_{\beta\in {\mathbb Z}^n}m_{\alpha\beta}n_{\beta\gamma}
\end{equation}
for all $\alpha,\gamma\in {\mathbb Z}^n$. There are natural spaces for which this definition makes sense (and which we introduce below). In order to apply this natural theory of infinite matrices developed below the next step is to observe that for $k\geq 0$ matrices such as
\begin{equation}\label{mat2}
\begin{array}{ll}
\exp(D^0)\left(C^0_{ij}\right)^k=\left( \exp\left( -\nu\sum_{j=1}^n\frac{4\pi \alpha_j^2}{l^2}\right)\left( \sum_{j=1}^n\frac{2\pi i (\alpha_j-\beta_j=}{l}h_{i(\alpha-\beta)}\right)_{\alpha,\beta \in {\mathbb Z}^n}\right)^k,\\
\\
\exp(D^0)\left( L^0_{ij}\right)^k=
{\Bigg (}\exp\left( -\nu\sum_{j=1}^n\frac{4\pi \alpha_j^2}{l^2}\right)\times\\
\\
\left( 2\pi i\alpha_i1_{\left\lbrace \alpha\neq 0\right\rbrace}\frac{\sum_{k=1}^n4\pi \beta_j(\alpha_k-\beta_k)h_{k(\alpha-\beta)}}{\sum_{i=1}^n4\pi\alpha_i^2}\right)_{\alpha,\beta \in {\mathbb Z}^n}^k{\Bigg )} 
\end{array} 
\end{equation}
have indeed bounded entries if the modes $h_{i\beta}$ decreases sufficiently. More importantly, we shall see that with regular data $h_i\in $ related iterated infinite matrix multiplications applied to the data $h_i,~1\leq i\leq n,~h_i\in h^s\left( {\mathbb Z}^n\right),~s>n+2$ lead to well defined sequences of vectors where evaluations at time $t\geq 0$ live in strong Sobolev spaces $h^s\left( {\mathbb Z}^n\right),~s>n+2$ . Note that for the reasons mentioned it is diffcult to define a formal fundamental solution 
\begin{equation}
\exp\left( (\delta_{ij}D^0+C^0_{ij}+L^0_{ij})t\right)
\end{equation}
even at stage zero of the approximation (where we do not need a time-order operator).
However for regular data $\mathbf{h}^F$ the solution $\mathbf{v}^{0,F}$ of the firs approximation, i.e., the expression
\begin{equation}
\exp\left( (\delta_{ij}D^0+C^0_{ij}+L^0_{ij})t\right)\mathbf{h}^F
\end{equation}
can be defined directly. We do not even need this. We may consider projections of infinite vectors to finite vectors with modes of order less than $l>0$, i.e., projections $P_{v^l}$  which are define on infinite vectors $\left( g_{\alpha} \right)^T_{\alpha\in {\mathbb Z}^n}$ by
\begin{equation}
P_{v^l}\left( g_{\alpha} \right)^T_{\alpha\in {\mathbb Z}^n}:=\left( g_{\alpha} \right)^T_{|\alpha|\leq l}
\end{equation}
 and projections $P_{M^l}$ of infinite matrices $\left(m_{\alpha\beta}\right)_{\alpha\beta \in {\mathbb Z}^n}$ to finite matrices with
\begin{equation}
P_{M^l}\left(  m_{\alpha\beta} \right)^T_{\alpha\beta\in {\mathbb Z}^n}
:=\left( m_{\alpha\beta} \right)^T_{|\alpha|\leq l}
\end{equation}
and prove Trotter product formulas for finite dissipative systems, and then consider limits. 
This will lead us to the crucial observation then of a product formula of Trotter-type of the form
\begin{equation}\label{trotternav}
\begin{array}{ll}
\lim_{l\uparrow \infty}\exp\left( (\delta_{ij}P_{M^l}D^0+P_{M^l}C^0_{ij}+P_{M^l}L^0_{ij})t\right)P_{v^l}\mathbf{h}^F\\
\\
=\lim_{l\uparrow \infty}\lim_{k\uparrow \infty}\left( \exp\left( \delta_{ij}P_{M^l}D^0\frac{t}{k}\right)\exp\left(\left( P_{M^l}C^0_{ij}+P_{M^l}L^0_{ij}\right) \frac{t}{k}\right)\right)^kP_{v_l}\mathbf{h}^F,
\end{array}
\end{equation}
where $(\delta_{ij}D^0+C^0_{ij}+L^0_{ij})$, i.e., the argument of the projection operator $P_{M_l}$, denote 'quadratic' matrices with $\left( n\times {\mathbb Z}^n\right) $ rows and $\left( n\times {\mathbb Z}^n\right) $ columns and $h$ is some regular function. This means that for finite $l>0$ we can prove a Trotter type relation and in the limit the left side equals the solution $\mathbf{v}^{0,F}$ of the equation above and is defined by the right side of (\ref{trotternav}). In the last step the regularity of the data $\mathbf{h}^F$ comes into play.
In this form this formula is useful only for the stage $0$ where the coefficient functions do not depend on time. For stages $m>0$ of the construction we have to take time dependence into account. However, we may define a Euler-type scheme and define substages which produce the approximations of stage $m$ in the limit. As we indicated above we may even set up higher order schemes using a Dyson formalism or use the Dyson formalism in order to compute an Euler equation limit for short time (in this case we have no fundamental solution at all and need the application to the data).  
The formula (\ref{trotternav}) depends on an observation which uses the diagonal structure of $D^0$ (and therefore $\exp\left(D^0\right)$). Here the matrix $C^0_{ij}+L^0_{ij}-\left(C^0_{ij}+L^0_{ij}\right)^T$, i.e., the deficiencies of symmetry in the matrices $C^0_{ij}+L^0_{ij}$, factorize with the correction term of an infinite analog of a special form of a CBH-type formula. It is an interesting fact that this deficiency is in a natural matrix space, we are going to define next.
  
Concerning the natural matrix spaces, in order to make sense of formulas as in (\ref{approxstagem}) for $s>n+ 2$ for  a matrix $M=\left( m_{\alpha\beta}\right)_{\alpha,\beta\in {\mathbb Z}^n}$ we say that
\begin{equation}
M\in M^s_n
\end{equation}
if for all $\alpha,\beta \in {\mathbb Z}^n$
\begin{equation}
|m_{\alpha\beta}|\leq \frac{C}{1+|\alpha-\beta|^{2s}}
\end{equation}
for some $C>0$.
For $s>n$ and a vector $w=(w_{\alpha})_{\alpha\in {\mathbb Z}^n}\in h^s\left({\mathbb Z}\right)$ 
we define the multiplication of the infinite matrix $M$ with $w$ by $Mw=\left(Mw_{\alpha} \right)_{\alpha\in {\mathbb Z}^n}$ along with 
\begin{equation}
Mw_{\alpha}:=\sum_{\beta\in {\mathbb Z}^n}m_{\alpha\beta}w_{\beta}.
\end{equation}
Indeed we observe that for $r,s>n\geq 2$ we have
\begin{equation}
\sum_{\beta\in {\mathbb Z}^n\setminus \left\lbrace 0,\alpha\right\rbrace }\frac{1}{|\alpha-\beta|^{s}\beta^r}\leq \frac{C}{1+|\alpha|^{r+s-n}}
\end{equation}
such that for $s>n$ and $M\in M^s_n$ we have indeed $Mw\in h^r\left({\mathbb Z}^n\right)$ if $w\in h^r\left({\mathbb Z}^n\right)$. This implies that
for $s>n$, $M\in h^s\left({\mathbb Z}^n\times {\mathbb Z}^n\right)$ and $w\in h^r\left({\mathbb Z}^n\right)$ 
\begin{equation}
M^1w:=Mw,~M^{k+1}w=M\left(M^k\right)w 
\end{equation}
is a well defined recursion for $k\geq 1$. Hence for a matrix $M$ which is not time-dependent the analytic vector
\begin{equation}
\exp\left(Mt\right)w:=\sum_{k\geq 0}\frac{M^kt^kw}{k!} 
\end{equation}
is well defined (even globally). The reason that we use the stronger assumption $s>n+2$ is that we have additional mode factors in the matrices which render such that we need slightly more regularity to have inheritage of a certain regularity (inheritage of a certain order of polynomial decay). For a time dependent matrix $A=A(t)$ we formally define the time-ordered exponential '\`{a} la Dyson' to be 
\begin{equation}
\begin{array}{ll}
T\exp(At):=\sum_{m=0}^{\infty}\frac{1}{m!}\int_{[0,t]}dt_1\cdots dt_mTA(t_1)\cdots A(t_m)dt_1\cdots dt_m\\
\\
:=\sum_{m=0}^{\infty}\int_0^tdt_1\int_0^{t_1}dt_2\cdots \int_0^{t_{n-1}}A(t_1)\cdots A(t_m).
\end{array}
\end{equation}
However, in the situation above of the Navier Stokes operator this is just a formal definition which lives in an extremely weak function space (we are not going to define), while especially for the Euler part $E(t):=A(t)|_{\nu =0}$ the expression  
\begin{equation}
\begin{array}{ll}
T\exp(Et)f:=\sum_{m=0}^{\infty}\frac{1}{m!}\int_{[0,t]}dt_1\cdots dt_mTE(t_1)\cdots E(t_m)dt_1\cdots dt_m\\
\\
:=\sum_{m=0}^{\infty}\int_0^tdt_1\int_0^{t_1}dt_2\cdots \int_0^{t_{n-1}}E(t_1)\cdots E(t_m)f.
\end{array}
\end{equation}
makes perfect sense for regular data $f$, i.e. data of polynomial decay of order $s>n+2$. At least this is true for local time while for global time we need the viscosity term in order to establish uniqueness. Combining such a representation for linearized equations at each stage of approximation with a Trotter product formula which adds another damping term via a negative exponential weight we get a natural scheme for the incompressible Navier Stokes equation.
Note that the operator $T$ is the usual Dyson time-order operator which may be defined recursively using the Heavyside function. Note that in this paper the symbol $T$ will sometimes also denote the time horizon but it will be clear from the context which meaning is indended. Note furthermore that schemes with explicit use of the Heavyside functions may make matters a little delicate if stronger regularity with respect to time derivatives is considered - this is another advantage of the simple Euler time discretization. Well, you may check that matrices are smooth on the time diagonals such that these complications of higher order schemes using a Dyson formalism are more of a technical nature. However, you have to be careful at this point and the Euler scheme simplifies the matter a bit. 
In any case, at each stage $m>0$ of the construction we shall define a scheme based on the Trotter product formula such that in the limit the linear approximation expressed formally by 
\begin{equation}
\mathbf{v}^{r,m,F}(t):=T\exp\left(A^{r}_{m}t\right)\mathbf{h}^{F},
\end{equation}
gets a strict sense. This holds also for the uncontrolled linear approximation 
\begin{equation}
\mathbf{v}^{m,F}(t):=T\exp\left(A_{m}t\right)\mathbf{h}^{F},
\end{equation}
where we shall see that these expressions can be well-defined for data with $h_i\in h^s\left({\mathbb Z}\right)$ for $s>n+2$ and for $t\geq 0$. Note that proving the existence of a global limit of the the iteration with respect to $m$ in a regular space also requires $\nu>0$. In order to prove the existence of a limit there are basically three possibilities then. One is to prove the existence of a uniform bound
\begin{equation}\label{uniformbound}
\sup_{t\geq 0}{\big |}\mathbf{v}^{r,m,F}(t){\big |}_{h^s}+\sup_{t\geq 0}{\Big |}\frac{\partial}{\partial t}\mathbf{v}^{r,m,F}(t){\Big |}_{h^s}\leq C
\end{equation}
for some $C>0$ independent of $m$, and then proceed with a compactness arguments a la Rellich. A weaker form of (\ref{uniformbound}) without the time derivative and a strong spatial norm (large $s$) is another variation for this alternative since product formulas for Sobolev norms with $s>\frac{1}{2}n$ and a priori estimates of Schauder type lead to an independent proof of the existence of a regular time derivative in the limit $m\uparrow \infty$. Maybe the latter variation is the most simple one. An alternative is a contraction argument on a ball of an appropriate function space with exponential time weight. Clearly, the radius of the ball will depend on the initial data, dimension and viscosity. This dependence can be encoded in a time weight of a time-weighted norm- as it is known from ODE theory for finite equations. Contraction arguments have the advantage that they lead to uniqueness results naturally. In order to strengthen results concerning the time dependence of regular upper bounds we shall consider auto-controlled schemes, where another time discretization is used and a damping term is introduced via a time dilatation transformation. This latter procedure has the advantage that linear upper bounds and even global uniform upper bounds can be obtained, and it can be combined with both of the former possibilities.  
In any case but concerning the second possibility of contraction especially for $\left( t\rightarrow w(t)\right) \in C\left(\left[0,\infty\right)\times h^s\left({\mathbb Z}^n\right)\right)$,  we may define for some $C>0$ (depending only on $\nu>0$, the dimension $n>0$, and the initial data components $h_i\in C^{\infty}\left({\mathbb T}^n\right)$) the norm
\begin{equation}
|w|^{\mbox{exp}}_{h^s ,C}:=\sup_{t\in [0,\infty)}\exp(-Ct)|w(t)|_{h^s},
\end{equation}
and the norm
\begin{equation}
|w|^{\mbox{exp},1}_{h^s, C}:=\sup_{t\in [0,\infty)}\exp(-Ct)\left( |w(t)|_{h^s}+|D_tw(t)|_{h^s}\right) ,
\end{equation}
where $D_tw(t):=\left(\frac{d}{\partial t}w(t)_{\alpha} \right)^T_{\alpha\in {\mathbb Z}^n}$ denotes the vector of componentwise derivatives with respect to time $t$. 
Then a contraction property
\begin{equation}
\left( \delta\mathbf{v}^{m,F}_i\right)_{1\leq i\leq n} =\left( \mathbf{v}^{r,m,F}_i(t)-\mathbf{v}^{r,m-1,F}_i\right)_{1\leq i\leq n}
\end{equation}
for all $1\leq i\leq n$ with respect to both norms and for $s>n+ 2$ ad $1\leq i\leq n$ can be proved. 
Summarizing we have the following steps:
\begin{itemize}
 \item[i)] in the first step we do some matrix analysis. First, the multiplication of infinite matrices $A_0 $ and $A_m$ with infinite vectors $\mathbf{h}^F$ at approximation stage $m\geq 0$ is well defined. Second matrix multiplication in the matrix space $M^s_n$ is well-defined for $s>n\geq 2$ as is $\exp(M)$ for $M\in M^s_n$. Then we prove a dissipative Trotter product formula for finite systems and apply the result to the first infinite approximation system with solution $\mathbf{v}^{0,F}$ at stage $m=0$ which is related to a linear integro-partial differential equation.
 \item[ii)] In the second step we set up an Euler-type scheme based on the Trotter-product formula which shows that
 for some $\nu>0$ and $s>n+ 2$ the exponential
 \begin{equation}
\mathbf{v}^{r,m,F}_i(t):=\left( T\exp\left(A^{r}_{m}t\right)\mathbf{h}^{F}\right)_i
\end{equation}
is well-defined for all $1\leq i\leq n$, where $\mathbf{v}^{r,m,F}_i(t)\in h^s_l\left({\mathbb Z}^n\right)$ for $t\geq 0$ and $1\leq i\leq n$ for all $m$, i.e., that the linearized equations for $\mathbf{v}^{r,m,F}_i$ which are equivalent to a linear partial integro-differential equation have a global solution.  Here $(.)_i$ indicates that we project to the $i$th component of the $n$ infinite entries of the solution vector at the first stage of construction.
\item[iii)] for $\nu>0$ as in step ii) the limit
\begin{equation}
\mathbf{v}^{r,F}_i(t):=\left( T\exp\left(A^{r}_{\infty}t\right)\mathbf{h}^{F}\right)_i
\end{equation}
exists, where
\begin{equation}
\begin{array}{ll}
 \mathbf{v}^{r,F}_i(t):=\lim_{m\uparrow \infty}\mathbf{v}^{m,F}_i(t)\\
 \\
 =\left( T\exp\left(A^{r}_{\infty}t\right)\mathbf{h}^{F}\right) _i:=\lim_{m\uparrow \infty}\left( T\exp\left(A^{r}_{m}t\right)\mathbf{h}^{F}\right) _i\in h^s_l\left({\mathbb Z}^n\right).
\end{array}
\end{equation}
This limit can be obtained by compactness arguments and by a contraction results in time-weighted regular function spaces. The latter result leads to uniqueness of solutions in the time-weighted function space. In the third step we consider also stronger results for global upper bound for auto-controlled schemes.
\end{itemize}
Concerning the third step (iii) we note that for all $t\geq 0$ and $x\in {\mathbb T}^n_l$ we have
\begin{equation}\label{vrm111}
v^{r,m}_j(t,x):=\sum_{\alpha\in {\mathbb Z}^n}v^{r}_{i\alpha}(t)\exp{(2\pi i\alpha x)}
\end{equation}
for all $1\leq j\leq n$. Hence, in classical Sobolev function spaces we have a compact sequence $\left( v^{r,m,F}_j(t,.)\right)_{m\in {\mathbb N}}$  in higher order Sobolev spaces $H^r$ with $r>s$ by the Rellich embedding theorem on compact manifolds for fixed $t\geq 0$. In this context recall that
\begin{thm}
For any $q>s$, $q,s\in {\mathbb R}$ and any compact manifold $M$ the embedding 
\begin{equation}
j:H^q\left(M\right)\rightarrow H^s\left(M\right)  
\end{equation}
is compact.
\end{thm}
For any $t\geq 0$ we have  a limit $v^r_j(t,.)\in H^{s}\left({\mathbb T}^n\right)\subset C^m$ for $s>m+\frac{n}{2}$ and the fact that the control function $r$ is well-defined and continuous implies
\begin{equation}
v_j(t,.)=v^r_j(t,.)-r(t)\in C^m
\end{equation}
for fixed $t\geq 0$ and all $m\in {\mathbb N}$. Finally we verify that $v_j$ is indeed a classical solution of the original Navier-Stokes equation using the properties we proved for the vector $v^{r,F}$ in the dual formulation. One possibility to do this is to plug in the approximations of the solutions and estimate the remainder pointwise observing that it goes to zero in the strong norm $|w|^{\mbox{exp},1}_{h^s,C}$ for an appropriate constant $C>0$. Using an auto-controlled scheme we can strengthen the results and prove global uniform upper bounds without exponential weights with respect to time.  
Note that for any of the alternative schemes the modes $v^{r}_{i\alpha},~1\leq i\leq n,~\alpha \in {\mathbb Z}^n$ determine classical (controlled) velocity functions $v^{m}_i,~1\leq i\leq n$ ( $v^{r,m}_i,~1\leq i\leq n$) which can be plugged into the incompressible Navier Stokes equation in order to verify that the limit $m\uparrow \infty$ is indeed the solution ('the' as we have strictly positive viscosity). If we plug in the approximation of stage $m\geq 0$ we may use the fact that this approximation satisfies a linear partial integro differential equation with coefficients determined in term of the velocity function of the previous step of approximation. This implies that we have for $1\leq i\leq n$
 \begin{equation}\label{navlerayco}
\begin{array}{ll}
    \frac{\partial v^{m}_i}{\partial t}-\nu \Delta v^m_i+  \sum_{j=1}^nv^m_j \frac{\partial v^m}{\partial x_i}  \\
    \\
    -\sum_{j,k=1}^n\int  K_{,i}(x-y)v^m_{j,k}v^m_{k,j}(t,y)dy\\
    \\
   =-\sum_{j=1}^n\left( v^m_j-v^{m-1}_j\right) \frac{\partial v^m}{\partial x_i} \\ 
 \\
 -\sum_{j,k=1}^n\int  K_{,i}(x-y)\left( v^m_{j,k}-v^{m-1}_{j,k}\right) v^m_{k,j}(t,y)dy,
\end{array}
\end{equation}
such that a construction of a Cauchy sequence $v^{m}_i,~1\leq i\leq n,~m\geq 0$ with $v^m_i(t,.)\in H^{s}$ for $s>n+2$ with a uniform global upper bound $C$ in $H^s$ leads to global existence. Here Sobolov estimates for products may be used. 
\section{Proof of theorem \ref{mainthm1}}
 We consider data
$h_i\in C^{\infty}\left( {\mathbb T}^n_l\right) $ for $1\leq i\leq n$. The special domain of a $n$-torus has the advantage that we may represent approximations of a solution evaluated at a given time $t\geq 0$ in the form of Fourier mode coefficients with respect to a Fourier basis
\begin{equation}
\left\lbrace \exp\left( \frac{2\pi i\alpha x}{l}\right) \right\rbrace_{\alpha \in {\mathbb Z}^n},
\end{equation}
i.e., with respect to an orthonormal basis of $L^2\left({\mathbb T}^n_l \right)$. We consider a real basis below, but let us stick to the complex notation first. We shall see that the considerations of this proof can always be translated to analogous considerations in the real basis. It is natural to start with an expansion of the solution $\mathbf{v}=(v_1,\cdots ,v_n)^T$ in the form
\begin{equation}
\begin{array}{ll}
v_i(t,x)=\sum_{\alpha \in {\mathbb Z}^n}v_{i\alpha}\exp\left(\frac{2\pi i\alpha x}{l}\right),\\
\\
p(t,x)=\sum_{\alpha \in {\mathbb Z}^n}p_{\alpha}\exp\left(\frac{2\pi i\alpha x}{l}\right),
\end{array}
\end{equation}
where the modes $v_{i\alpha},~\alpha\in {\mathbb Z}^n$ and $p_{\alpha},~\alpha\in {\mathbb Z}^n$ depend on time.
Here we have an index $1\leq i\leq n$ which gives rise to some ambiguity with the complex number $2\pi i$, but the latter $i$ always occurs in the context of $\pi$ such that no confusion should arise. 
Plugging this ansatz into the Navier-Stokes equation system (\ref{nav}) we formally get $n$ coupled infinite differential equations of ordinary type for the modes, i.e., for all $1\leq i\leq n$ we have for all $\alpha \in {\mathbb Z}^n$
\begin{equation}\label{navode}
\begin{array}{ll}
\frac{d v_{i\alpha}}{dt}=\nu\sum_{j=1}^n \left( -\frac{4\pi \alpha_j^2}{l^2}\right)v_{i\alpha}
-\sum_{j=1}^n\sum_{\gamma \in {\mathbb Z}^n}\frac{2\pi i \gamma_j}{l}v_{j(\alpha-\gamma)}v_{i\gamma}\\
\\
-2\pi i\alpha_ip_{\alpha},
\end{array} 
\end{equation}
where for each $1\leq i\leq n$ and each $\alpha\in {\mathbb Z}^n$
\begin{equation}
v_{i\alpha}:[0,\infty)\rightarrow {\mathbb R}
\end{equation}
are some time-dependent $\alpha$-modes of velocity, and
\begin{equation}
p_{\alpha}:[0,\infty)\rightarrow {\mathbb R}
\end{equation}
are some time-dependent $\alpha$-modes of pressure to be determined in terms of the velocity modes $v_{i\alpha}$.
Here, $(\alpha -\gamma)$ denotes the subtraction between multiindices, i.e.,
\begin{equation}
(\alpha -\gamma ) =(\alpha_1-\gamma_1,\cdots,\alpha_n-\gamma_n),
\end{equation}
where brackets are added for notational reasons in order to mark separate multiindices.
Next we may eliminate the pressure modes $p_{\alpha}$ using Leray projection, which shows that the pressure $p$ satisfies the Poisson equation
\begin{equation}\label{poisson1}
\Delta p=-\sum_{j,k}v_{j,k}v_{k,j},
\end{equation}
where we use the Einstein abbreviation for differentiation, i.e.,
\begin{equation}
v_{j,k}=\frac{\partial v_j}{\partial x_k}~\mbox{etc.}.
\end{equation}
The Poisson equation in (\ref{poisson1}) is re-expressed by an infinite equation system for the $\alpha$-modes of the form  
\begin{equation}
p_{\alpha}\sum_{i=1}^n\frac{-4\pi^2\alpha_i^2}{l^2}=\sum_{j,k=1}^n\sum_{\gamma\in {\mathbb Z}^n}\frac{4\pi^2 \gamma_j(\alpha_k-\gamma_k)v_{j\gamma}v_{k(\alpha-\gamma)}}{l^2}
\end{equation}
with the formal solution (w.l.o.g. $\alpha\neq 0$ -see below)
\begin{equation}\label{palpha}
p_{\alpha}=-1_{\left\lbrace \alpha\neq 0\right\rbrace }\frac{\sum_{j,k=1}^n\sum_{\gamma\in {\mathbb Z}^n}4\pi^2 \gamma_j(\alpha_k-\gamma_k)v_{j\gamma}v_{k(\alpha-\gamma)}}{\sum_{i=1}^n4\pi^2\alpha_i^2},
\end{equation}
for $\alpha \neq 0$, which is indeed independent of the size of the torus $l$. This means that this term becomes large compared to the second order terms and the convection term for large tori ($l>0$ large while viscosity $\nu >0$ stays fixed). This may happen for approximations of Cauchy problems by equations for large tori for purpose of simulation.  Note that
\begin{equation}
1_{\left\lbrace \alpha\neq 0\right\rbrace}:= \left\lbrace \begin{array}{ll}
1 \mbox{ if }\alpha\neq 0,\\
\\
0 \mbox{ else}. 
\end{array}\right.
\end{equation}
Note that we put $p_{0}=0$ for $\alpha=0$ in (\ref{palpha}). We are free to do so since $(\mathbf{v},p+C)$ is a solution of (\ref{nav}) if $(\mathbf{v},p)$ is.

Plugging into (\ref{navode}) we get
for all $1\leq i\leq n$, and all $\alpha \in {\mathbb Z}^n$
\begin{equation}\label{navode2**}
\begin{array}{ll}
\frac{d v_{i\alpha}}{dt}=\sum_{j=1}^n\nu \left( -\frac{4\pi \alpha_j^2}{l^2}\right)v_{i\alpha}
-\sum_{j=1}^n\sum_{\gamma \in {\mathbb Z}^n}\frac{2\pi i \gamma_j}{l}v_{j(\alpha-\gamma)}v_{i\gamma}\\
\\
+2\pi i\alpha_i1_{\left\lbrace \alpha\neq 0\right\rbrace}\frac{\sum_{j,k=1}^n\sum_{\gamma\in {\mathbb Z}^n}4\pi^2 \gamma_j(\alpha_k-\gamma_k)v_{j\gamma}v_{k(\alpha-\gamma)}}{\sum_{i=1}^n4\pi^2\alpha_i^2}.
\end{array} 
\end{equation}
For $1\leq i\leq n$, this is a system of nonlinear ordinary differential equations of 'infinite dimension', i.e., for infinite modes $v_{i\alpha}$ and $p_{\alpha}$ along with $\alpha \in {\mathbb Z}^n$. Next we rewrite the system in the real basis for torus size $l=1$, i.e in the basis
\begin{equation}\label{realbasis}
\cos\left(2\pi\alpha x\right),~\sin\left(2\pi\alpha x\right),~\alpha\in {\mathbb N}^n,  
\end{equation}
where ${\mathbb N}$ denotes the set set of nonnegative integers, i.e., $0\in {\mathbb N}$.
In this variation of representation on $[0,l]^n$ with periodic boundary condition we start with an expansion of the solution $\mathbf{v}^{re}=(v^{re}_1,\cdots ,v^{re}_n)^T$ in the form
\begin{equation}
\begin{array}{ll}
v^{re}_i(t,x)=\sum_{\alpha \in {\mathbb N}^n}v^{re}_{ci\alpha}\cos\left(2\pi i\alpha x\right)+\sum_{\alpha \in {\mathbb N}^n}v^{re}_{si\alpha}\sin\left(2\pi i\alpha x\right),\\
\\
p^{re}(t,x)=\sum_{\alpha \in {\mathbb N}^n}p^{re}_{c\alpha}\cos\left(2\pi i\alpha x\right)+\sum_{\alpha \in {\mathbb N}^n}p^{re}_{s\alpha}\cos\left(2\pi i\alpha x\right),
\end{array}
\end{equation}
where the modes $v^{re}_{ci\alpha},~v^{re}_{si\alpha},~\alpha\in {\mathbb N}^n$ and $p^{re}_{\alpha},~\alpha\in {\mathbb N}^n$ depend on time. The lower indices $c$ and $s$ indicate that the mode is a coefficient in a cosine series and a sinus series respectively. The size $l=1$ is just in order to simplify notation and for numerical purposes we can write the real system with respect to arbitrary size of the torus as well, of course.
Plugging this ansatz into the Navier-Stokes equation system (\ref{nav}) we formally get again $n$ coupled infinite differential equations of ordinary type for the modes, but this time we get products of sinus functions and cosine functions from the nonlinear terms in a first step which we then re-express in cosine terms. For all $1\leq i\leq n$ we have for all $\alpha \in {\mathbb N}^n$
\begin{equation}\label{aa}
\begin{array}{ll}
\frac{d v^{re}_{ci\alpha}}{dt}\cos(2\pi \alpha x)+\frac{d v^{re}_{si\alpha}}{dt}\sin(2\pi \alpha x)=\nu\sum_{j=1}^n \left( -\frac{4\pi^2 \alpha_j^2}{l^2}\right)v^{re}_{ci\alpha}\cos(2\pi\alpha x)\\
\\
+\nu\sum_{j=1}^n \left( -\frac{4\pi^2 \alpha_j^2}{l^2}\right)v^{re}_{si\alpha}\sin(2\pi\alpha x)\\
\\
-\sum_{j=1}^n\sum_{\gamma \in {\mathbb N}^n,\gamma\leq \alpha}2\pi \gamma_jv^{re}_{cj(\alpha-\gamma)}v^{re}_{ci\gamma}\frac{1}{2}\sin(2\pi\alpha x)\\
\\
-\sum_{j=1}^n\sum_{\gamma \in {\mathbb N}^n,\gamma\geq 0}2\pi \gamma_jv^{re}_{cj(\alpha+\gamma)}v^{re}_{ci\gamma}\frac{1}{2}\sin(2\pi\alpha x)\\
\\
+\sum_{j=1}^n\sum_{\gamma \in {\mathbb N}^n,\gamma\leq \alpha}
2\pi \gamma_jv^{re}_{sj(\alpha-\gamma)}v^{re}_{ci\gamma}\frac{1}{2}\cos(2\pi\alpha x)\\
\\
-\sum_{j=1}^n\sum_{\gamma \in {\mathbb N}^n,\gamma\geq 0}
2\pi \gamma_jv^{re}_{sj(\alpha+\gamma)}v^{re}_{ci\gamma}\frac{1}{2}\cos(2\pi\alpha x)\\
\\
+\sum_{j=1}^n\sum_{\gamma \in {\mathbb N}^n,\gamma\leq \alpha}
2\pi \gamma_jv^{re}_{cj(\alpha-\gamma)}v^{re}_{si\gamma}\frac{1}{2}\sin(2\pi\alpha x)\\
\\
+\sum_{j=1}^n\sum_{\gamma \in {\mathbb N}^n,\gamma\geq 0}
2\pi \gamma_jv^{re}_{cj(\alpha+\gamma)}v^{re}_{si\gamma}\frac{1}{2}\sin(2\pi\alpha x)\\
\\
+\sum_{j=1}^n\sum_{\gamma \in {\mathbb N}^n,\gamma\leq \alpha}
2\pi \gamma_jv^{re}_{sj(\alpha-\gamma)}v^{re}_{si\gamma}\frac{1}{2}\sin(2\pi\alpha x)\\
\\
+\sum_{j=1}^n\sum_{\gamma \in {\mathbb N}^n,\gamma\geq 0}
2\pi \gamma_jv^{re}_{sj(\alpha +\gamma)}v^{re}_{si\gamma}\frac{1}{2}\sin(2\pi\alpha x)\\
\\
+2\pi \alpha_ip^{re}_{c\alpha}\sin(2\alpha x)-2\pi \alpha_ip^{re}_{s\alpha}\cos(2\alpha x),
\end{array} 
\end{equation}
where we assume that in the sum '$\leq$' denotes an appropriate ordering (for example the lexicographic) of ${\mathbb N}^n$, and where we use
\begin{equation}\label{relsc}
\sin\left( 2\pi \beta x\right) \cos\left( 2\pi \gamma x\right)=\frac{1}{2}\sin\left( 2\pi (\beta+\gamma) x\right)+\frac{1}{2}\sin\left( 2\pi (\beta -\gamma)x\right),
\end{equation}
\begin{equation}\label{relss}
\sin\left( 2\pi \beta x\right) \sin\left( 2\pi \gamma x\right)=-\frac{1}{2}\cos\left( 2\pi (\beta+\gamma) x\right)+\frac{1}{2}\cos\left( 2\pi (\beta-\gamma)x\right) ,
\end{equation}
and
\begin{equation}\label{relcc}
\cos\left( 2\pi \beta x\right) \cos\left( 2\pi \gamma x\right)=\frac{1}{2}\cos\left( 2\pi (\beta+\gamma) x\right)+\frac{1}{2}\cos\left( 2\pi (\beta -\gamma)x\right).
\end{equation}
In order to get a system with respect to the full real basis.
Note that for each $1\leq i\leq n$ and each $\alpha\in {\mathbb Z}^n$
\begin{equation}
v^{re}_{ci\alpha}:[0,\infty)\rightarrow {\mathbb R},~v^{re}_{si\alpha}:[0,\infty)\rightarrow {\mathbb R}
\end{equation}
are some time-dependent $\alpha$-modes of velocity, and
\begin{equation}
p^{re}_{c\alpha},p^{re}_{s\alpha}:[0,\infty)\rightarrow {\mathbb R}
\end{equation}
are some time-dependent $\alpha$-modes of pressure to be determined in terms of the velocity modes $v^{re}_{ci\alpha}$ and $v^{re}_{si\alpha}$.

Next we determine these real pressure modes $p^{re}_{c\alpha}$ and $p^{re}_{s\alpha}$.  The Poisson equation in (\ref{poisson1}) is re-expressed by an infinite equation system for thereal $\alpha$-modes of the form  
\begin{equation}
\begin{array}{ll}
p^{re}_{c\alpha}\sum_{i=1}^n(-4\pi^2\alpha_i^2)\cos(2\pi \alpha x)+p^{re}_{s\alpha}\sum_{i=1}^n(-4\pi^2\alpha_i^2)\sin(2\pi \alpha x)\\
\\
=-\sum_{j,k=1}^n\sum_{\gamma\in {\mathbb N}^n,\gamma\leq \alpha}4\pi^2 \gamma_j(\alpha_k-\gamma_k)v_{sj\gamma}v_{sk(\alpha-\gamma)}\frac{1}{2}\cos(2\pi \alpha x)\\
\\
-\sum_{j,k=1}^n\sum_{\gamma\in {\mathbb N}^n,\gamma\geq 0}4\pi^2 \gamma_j(\alpha_k+\gamma_k)v_{sj\gamma}v_{sk(\alpha+\gamma)}\frac{1}{2}\cos(2\pi \alpha x)\\
\\
+2\sum_{j,k=1}^n\sum_{\gamma\in {\mathbb N}^n,\gamma\leq \alpha}4\pi^2 \gamma_j(\alpha_k-\gamma_k)v_{sj\gamma}v_{ck(\alpha-\gamma)}\frac{1}{2}\sin(2\pi \alpha x)\\
\\
+2\sum_{j,k=1}^n\sum_{\gamma\in {\mathbb N}^n,\gamma\geq 0}4\pi^2 \gamma_j(\alpha_k+\gamma_k)v_{sj\gamma}v_{ck(\alpha+\gamma)}\frac{1}{2}\sin(2\pi \alpha x)\\
\\
-\sum_{j,k=1}^n\sum_{\gamma\in {\mathbb N}^n,\gamma\leq \alpha}4\pi^2 \gamma_j(\alpha_k-\gamma_k)v_{cj\gamma}v_{ck(\alpha-\gamma)}\frac{1}{2}\sin(2\pi \alpha x)\\
\\
-\sum_{j,k=1}^n\sum_{\gamma\in {\mathbb N}^n,\gamma\geq 0}4\pi^2 \gamma_j(\alpha_k+\gamma_k)v_{cj\gamma}v_{ck(\alpha +\gamma)}\frac{1}{2}\sin(2\pi \alpha x).
\end{array}
\end{equation}
Hence (w.l.o.g. $\alpha\neq 0$ -see below)
\begin{equation}\label{pcalpha}
\begin{array}{ll}
p^{re}_{c\alpha}=\frac{1}{2}1_{\left\lbrace \alpha\neq 0\right\rbrace }\frac{\sum_{j,k=1}^n\sum_{\gamma\in {\mathbb N}^n,\gamma\leq \alpha}4\pi^2 \gamma_j(\alpha_k-\gamma_k)v_{sj\gamma}v_{sk(\alpha-\gamma)}}{\sum_{i=1}^n4\pi^2\alpha_i^2}\\
\\
+\frac{1}{2}1_{\left\lbrace \alpha\neq 0\right\rbrace }\frac{\sum_{j,k=1}^n\sum_{\gamma\in {\mathbb N}^n,\gamma\geq 0}4\pi^2 \gamma_j(\alpha_k+\gamma_k)v_{sj\gamma}v_{sk(\alpha+\gamma)}}{\sum_{i=1}^n4\pi^2\alpha_i^2}
\end{array}
\end{equation}
for the cosine modes, and
\begin{equation}\label{psalpha}
\begin{array}{ll}
p^{re}_{s\alpha}=-1_{\left\lbrace \alpha\neq 0\right\rbrace }\frac{\sum_{j,k=1}^n\sum_{\gamma\in {\mathbb N}^n,\gamma\leq \alpha}4\pi^2 \gamma_j(\alpha_k-\gamma_k)v_{sj\gamma}v_{ck(\alpha-\gamma)}}{\sum_{i=1}^n4\pi^2\alpha_i^2}\\
\\
-1_{\left\lbrace \alpha\neq 0\right\rbrace }\frac{\sum_{j,k=1}^n\sum_{\gamma\in {\mathbb N}^n,\gamma\geq 0}4\pi^2 \gamma_j(\alpha_k+\gamma_k)v_{sj\gamma}v_{ck(\alpha+\gamma)}}{\sum_{i=1}^n4\pi^2\alpha_i^2}\\
\\
+\frac{1}{2}1_{\left\lbrace \alpha\neq 0\right\rbrace }\frac{\sum_{j,k=1}^n\sum_{\gamma\in {\mathbb N}^n,\gamma\leq \alpha}4\pi^2 \gamma_j(\alpha_k-\gamma_k)v_{cj\gamma}v_{ck(\alpha-\gamma)}}{\sum_{i=1}^n4\pi^2\alpha_i^2}\\
\\
+\frac{1}{2}1_{\left\lbrace \alpha\neq 0\right\rbrace }\frac{\sum_{j,k=1}^n\sum_{\gamma\in {\mathbb N}^n,\gamma\geq 0}4\pi^2 \gamma_j(\alpha_k+\gamma_k)v_{cj\gamma}v_{ck(\alpha+\gamma)}}{\sum_{i=1}^n4\pi^2\alpha_i^2}
\end{array}
\end{equation}
for the sinus modes, and for all $\alpha \neq 0$. Note that $p^{re}_{s\alpha}$ and $p^{re}_{c\alpha}$  are indeed independent of the size of the torus $l$. The equations in (\ref{psalpha}), (\ref{pcalpha}), and (\ref{aa}) determine an equation for the infinite vector of real modes 
\begin{equation}
\left(v^{re}_{si\alpha},v^{re}_{ci\alpha} \right)_{\alpha\in {\mathbb N}^n,1\leq i\leq n} 
\end{equation}
with alternating sinus (subscript $s$) and cosine (subscript $c$) entries.
Let us consider the relation of this real representation to the complex representation which we use in the following because of its succinct form. This relation is one to one, of course, as both representations of functions and equations are dual representations of classical functions and equations. Especially we can rewrite every equation, function, and matrix operation in complex notation in the real notation. This ensures that the operations produce real solutions. There is one issue here, which should be emphasized: the estmates for matrix operations in the complex notation transfer directly to analogous estimates for the corresponding matrix operations related to the real system above, although the matrix operation look partly different. Consider the last to terms in (\ref{psalpha}) for example. We can interpret these terms as finite and infinite matrix operations
\begin{equation}\label{pmult1}
\sum_{\gamma\in {\mathbb N}^n,\gamma\leq \alpha}m^-_{ck\alpha\gamma}v_{cj\gamma}
\end{equation}
and 
\begin{equation}\label{pmult2}
\sum_{\gamma\in {\mathbb N}^n,\gamma\geq 0}m^+_{ck\alpha\gamma}v_{cj\gamma},
\end{equation}
where the matrix elements in (\ref{pmult1}) and (\ref{pmult2}) are defined by equivalence of the matrix-vector multiplication with the pressure terms in
\begin{equation}\label{psalpha1}
\begin{array}{ll}
+\frac{1}{2}1_{\left\lbrace \alpha\neq 0\right\rbrace }\frac{\sum_{j,k=1}^n\sum_{\gamma\in {\mathbb N}^n,\gamma\leq \alpha}4\pi^2 \gamma_j(\alpha_k-\gamma_k)v_{cj\gamma}v_{ck(\alpha-\gamma)}}{\sum_{i=1}^n4\pi^2\alpha_i^2}\\
\\
+\frac{1}{2}1_{\left\lbrace \alpha\neq 0\right\rbrace }\frac{\sum_{j,k=1}^n\sum_{\gamma\in {\mathbb N}^n,\gamma\geq 0}4\pi^2 \gamma_j(\alpha_k+\gamma_k)v_{cj\gamma}v_{ck(\alpha+\gamma)}}{\sum_{i=1}^n4\pi^2\alpha_i^2}
\end{array}
\end{equation}
respectively. The upper bounds estimates for the infinite matrix operations are similar as the infinite matrix operations in complex notation. Note that the matrix operations can be made equivalent indeed by recounting. Furthermore, the finite matrix operations do not alter any convergence properties of the modes for one matrix operation. For iterated matrix operations they can be incorporated naturally in the Trotter product formulas below.  
 
We may take (\ref{navode2**}) as a shorthand notation for an equivalent equation with respect to the full real basis. The equations written in the real basis have structural features which allow us to transfer certain estimates observed for the equations in the complex notation to the equations in the real basis notation. Each step below done in the complex notation is easily transferred to the real notation. However, the complex notation is more convenient, because we do not have two equations for each mode. Therefore, we use the complex notation, and remark that the real systems has the same relevant features for the argument below.  Note that the equation (\ref{navode2**}) has the advantage that this approach leads to more general observations concerning the relations of nonlinear PDEs and infinite nonlinear ODEs. However, it is quite obvious that a similar argument holds for the real system above as well such that we can ensure that we are constructing real solutions, although our notation is complex. This means that we construct real solutions $v_{i\alpha}$ with $v_{i\alpha}(t)\in {\mathbb R}$. Note that it suffices to prove that there is a regular solution to (\ref{navode2**}) and to the corresponding system explicitly written in the real basis, because this translates into a classical regular solution of the incompressible Navier-Stokes equation on the $n$-torus, and this implies the existence of a classical solution of the incompressible Navier-Stokes equation in its original formulation, i.e., without elimination of the pressure. We provided an argument of this well-known implication at the end of section 1. This is not the controlled scheme which we considered in the first section. However, this scheme is identical with the controlled scheme at the first stage $m=0$, so we start with some general considerations which apply to this first stage first and introduce the control function later.
So let us look at the simple scheme in more detail now. 
The first approximation
is the system
\begin{equation}\label{navode3}
\begin{array}{ll}
\frac{d v^0_{i\alpha}}{dt}=\sum_{j=1}^n \left( -\nu\frac{4\pi \alpha_j^2}{l^2}\right)v^0_{i\alpha}
-\sum_{j=1}^n\sum_{\gamma \in {\mathbb Z}^n}\frac{2\pi i \gamma_j}{l}h_{j(\alpha-\gamma)}v^0_{i\gamma}\\
\\
+2\pi i\alpha_i1_{\left\lbrace \alpha\neq 0\right\rbrace}\frac{\sum_{j,k=1}^n\sum_{\gamma\in {\mathbb Z}^n}4\pi^2 \gamma_j(\alpha_k-\gamma_k)v^0_{j\gamma}h_{k(\alpha-\gamma)}}{\sum_{i=1}^n4\pi^2\alpha_i^2}.
\end{array} 
\end{equation}

Let
\begin{equation}
\mathbf{v}^{0,F}=(\mathbf{v}^{0,F}_1,\cdots,\mathbf{v}^{0,F}_n),~\mathbf{v}^{0,F}_i:=\left(v^0_{i\alpha}\right)^T_{\alpha \in {\mathbb Z}^n}.
\end{equation}
Formally, we consider $\mathbf{v}^{0,F}_i$ as an infinite vector, where the upper script $T$ in the componentwise description indicates transposition.
Then equation (\ref{navode3}) may be formally rewritten in the form
\begin{equation}\label{navode4}
\frac{d \mathbf{v}^{0,F}_i}{dt}=A^i_0\mathbf{v}^{0,F}_i+\sum_{j\neq i}L^0_{ij}\mathbf{v}^{0,F}_j,
\end{equation}
with the infinite matrix $A^i_0=\left(a^{i0}_{\alpha\beta} \right)_{\alpha,\beta \in {\mathbb Z}^n}$ and
the entries
\begin{equation}\label{matrix0}
\begin{array}{ll}
a^{i0}_{\alpha\beta}=\delta_{\alpha\beta}\nu\left(-\sum_{j=1}^n\frac{4\pi \alpha_j^2}{l^2} \right)-\sum_{j=1}^n\frac{2\pi i\beta_jh_{j(\alpha-\beta)}}{l}+L_{ii\alpha\beta},
\end{array}
\end{equation}
and where
\begin{equation}
L^0_{ij\alpha\beta}=1_{\left\lbrace \alpha\neq 0\right\rbrace  }2\pi i\alpha_i\frac{\sum_{k=1}^n 4\pi^2 \beta_j(\alpha_k-\beta_k)h_{k(\alpha-\beta)}}{\sum_{i=1}^n4\pi^2 \alpha_i^2}
\end{equation}
for all $1\leq i,j\leq n$
denote coupling terms related to the Leray projection term.
As we observed in the preceding section you may write (\ref{navode4}) in the form
\begin{equation}\label{navode4big}
\frac{d \mathbf{v}^{0,F}}{dt}=A_0\mathbf{v}^{0,F},
\end{equation}
where
\begin{equation}\label{matrix0*}
A_0\mathbf{v}^{0,F}:=
\left( \left( \sum_{j\in \left\lbrace 1,\cdots ,n\right\rbrace ,\beta\in {\mathbb Z}^n}
\left(  \left( \delta_{ij} a^{i0}_{\alpha\beta}\right)
+\overline{\delta}_{ij}L^0_{ij\alpha\beta}\right)v^0_{j\beta}\right)^T_{\alpha \in {\mathbb Z}^n}\right)^T_{ 1\leq i\leq n},
\end{equation}
where we have to insert
\begin{equation}
\overline{\delta}_{ij}=(1-\delta_{ij}),
\end{equation}
since the diagonal terms $L^{0}_{ii}$ are in $a^{i0}$ already,  
and where $A_0=\left(A^{ij}_0\right) $ can be considered as a quadratic matrix with $n^2$ infinite matrix entries
\begin{equation}
A^{ij}_0=A^i_0 \mbox{ for }~i=j,
\end{equation}
and
\begin{equation}
A^{ij}_0=L^0_{ij} \mbox{ for }~i\neq j.
\end{equation}

Note that these matrices which determine the first linear approximation matrix $A_0$ are not time-dependent.
The modes $a^{0}_{\alpha\beta},v^0_{i\beta}$, or at least one set of these modes, have to decay appropriately as $|\alpha|,|\beta|\uparrow \infty$ in order that the definition in (\ref{matrix0*}) makes sense, i.e., leads to finite results in appropriate norms which show that the infinite set of modes in $A_0\mathbf{v}^{0,F}$ belong to a regular function in classical space. 
Since the matrix $A_0$ has constant entries, the formal solution of (\ref{navode4}) is
\begin{equation}\label{sol01}
\mathbf{v}^{0,F}=\exp\left(A_0t\right)\mathbf{h}^{F}.
\end{equation}
As we have positive viscosity $\nu >0$ for the Navier Stokes equation even the fundamental solution (fundamental matrix)
\begin{equation}
\exp\left(A_0t\right)
\end{equation}
can be defined rigorously by a Trotter product formula for regular input $\mathbf{h}^{F}$ in the matrix $A_0$ by the Trotter product formula. Note that the rows of a finite Trotter approximation of the fundamental matrix are as in (\ref{typent}) below, and such that every row of this approximation lives in a regular space $h^s\left({\mathbb Z}^n\right)$ for $s\geq n+2$ if the data $\mathbf{h}^{F}$ of the matrix $A_0$ are in that space. This behavior is preserved as we go to the Trotter product limit. However, for the sake of global existence the weaker observation that we can make sense of  (\ref{sol01}) is essential. Note here that for fixed $\alpha$ and regular  $h_j$ (i.e., $h_j$ has polynomially decaying modes) expressions proportional to
\begin{equation}\label{typent}
\exp\left( -\nu\sum_{i=1}^n\alpha^2_it\right)\exp\left(\gamma_j h_{j(\alpha-\gamma)}\right) \rightarrow  \exp\left( -\nu\sum_{i=1}^n\alpha^2_it\right) ~\mbox{as}~|\gamma|\uparrow \infty,
\end{equation}
such that the matrix multiplication with regular data in $h^s\left({\mathbb Z}^n\right)$ for $s>n+2$ inherits regularity of this order.  
Hence it makes sense to use the dissipative feature on the diagonal terms and a Trotter-type product formula (otherwise we are in trouble because the modulus of diagonal entries increases with the order of the modes, as we remarked before).
Well, in this paper we stress the  constructive point of view and the algorithmic perspective, and by this we mean that we approximation infinite model systems by finite systems on the computer where analysis shows how limits behave. This gives global existence and a computation scheme at the same time. We have seen various ingredients of analysis which allow us to observe how close we are to the 'real' limit solution. One aspect is the polynomial decay of modes and it is clear that finite ODE approximating systems approximate the better the stronger the polynomial decay is. Infinite matrix multiplication is related to weakly singular elliptic integrals and controls the quality of this approximation.  Another aspect is the time approximation. Here we may use the Dyson formalism for the Euler part of the equation (in the Trotter product) in order to obtain higher order schemes.   
For the first approximation we may solve the equation (\ref{sol01}) by approximations via systems of finite modes, i.e., via
\begin{equation}
P_{v^l}\mathbf{v}^{0,F;*}_i=\exp\left(P_{M_l}A_0t\right)P_{v^l}\mathbf{h}^{F}
\end{equation}
where $P_{v^l}$ and $P_{M^l}$ denote projections of vectors and matrices to finite cut-off vectors and finite cut-off matrices of modes of order less or equal to $l$, i.e. modes with multiindices $\alpha=(\alpha_1,\cdots ,\alpha_n)$ which satisfy $|\alpha|=\sum_{i=1}^n\alpha_i\leq l$.
Note that
we used the notation
\begin{equation}
P_{v^l}\mathbf{v}^{0,F;*}_i
\end{equation}
with a star upperscript of the velocity approximation $\mathbf{v}^{0,F;*}_i$ since the projection of the solution of the infinite system is not equal to the solution of the finite projected system in general. The latter is an approximation of the former which becomes equal in the limit as the projection of the infinite system become identity.
Indeed it is the order of polynomial decay of the modes and some properties of infinite matrix times vector multiplications related to the growth behavior of weakly singular integrals which makes it possible to show that
\begin{equation}\label{v0finf}
\begin{array}{ll}
\mathbf{v}^{0,F}_i=\lim_{l\uparrow \infty}P_{v^l}\mathbf{v}^{0,F;*}_i=\lim_{l\uparrow \infty}\lim_{k\uparrow \infty}{\Bigg(} \exp\left(P_{M^l}\left( \delta_{ij}D^0\right) \frac{t}{k}\right)\times\\
\\
\times \exp\left(\left( P_{M^l}\left( \delta_{ij}C^0+L^0_{ij}\right) \right) \frac{t}{k}\right){\Bigg )}^kP_{v^l}\mathbf{h}^{F}_i.
\end{array}
\end{equation}
Note that we even have $\exp\left(\left( \delta_{ij}D^0\right) \frac{t}{k}\right) \exp\left(\left( \left( \delta_{ij}C^0+L^0_{ij}\right) \right)\frac{t}{k}\right)_{ij}\in M^s_n$ (although we do not need this strong behavior). Here the symbol $\left( .\right)_{ij}$ indicates projection onto the infinite ${\mathbb Z}^n\times {\mathbb Z}^n$-block at the $i$th row and the $j$th column of the matrix $A_0$. The right side of (\ref{v0finf}) is well-defined  anyway as due to the matrix $\exp\left(\left( \delta_{ij}D^0\right) \frac{t}{k}\right)$ which has the effect of an multiplication of each row by a function which is exponentially decreasing with respect to time and with respect to the order of the modes, and due to the regularity (order of polynomial decay) of the vector $\mathbf{h}^{F}_i$. The more delicate thing is to prove that the Trotter-type approximation converges to the (approximative) solution at the higher order stages $m\geq 1$ of the iterative construction, where we have time dependent convection and Leray projection terms. We come back to this later in the proof of the dissipative Trotter product formula. Let us start with the description of the other stages $m>0$ of the construction first. 
At stage $m\geq 1$ having computed $\mathbf{v}^{m-1,F}_i=\left( v^{m-1}_{i\alpha}\right)^T_{\alpha\in {\mathbb Z}^n}$ the $m$th approximation
is computed via the system
\begin{equation}\label{navode3m**}
\begin{array}{ll}
\frac{d v^m_{i\alpha}}{dt}=\sum_{j=1}^n \left( -\frac{4\pi \alpha_j^2}{l^2}\right)v^m_{i\alpha}
-\sum_{j=1}^n\sum_{\gamma \in {\mathbb Z}^n}\frac{2\pi i \gamma_j}{l}v^{m-1}_{j(\alpha-\gamma)}v^m_{i\gamma}\\
\\
+2\pi i\alpha_i1_{\left\lbrace \alpha\neq 0\right\rbrace}\frac{\sum_{j,k=1}^n\sum_{\gamma\in {\mathbb Z}^n}4\pi^2 \gamma_j(\alpha_k-\gamma_k)v^{m}_{j\gamma}v^{m-1}_{k(\alpha-\gamma)}}{\sum_{i=1}^n4\pi^2\alpha_i^2}
\end{array} 
\end{equation}
with the same initial data, of course.
We have
\begin{equation}\label{proofodem}
\mathbf{v}^{m,F}_i:=\left(v^m_{i\alpha}\right)^T_{\alpha \in {\mathbb Z}^n},
\end{equation}
and the equation (\ref{navode3m**}) may be formally rewritten in the form
\begin{equation}\label{navode4m}
\frac{d \mathbf{v}^{m,F}}{dt}=A_m\mathbf{v}^{m,F},
\end{equation}
with the infinite matrix $A_m=\left(a^{ijm}_{\alpha\beta} \right)_{\alpha,\beta \in {\mathbb Z}^n}$ and
the (time-dependent!) entries $A_m=\left(A^{ij}_m\right)=\left( a^{ijm}_{\alpha\beta}\right) $, and along with $A^{ii}_m=D^0+C^m_{ii}+L^m_{ii}$ for all $1\leq i\leq n$, and $A^{ij}_m=L^m_{ij}$ for $i\neq j,~1\leq i,j\leq n$, where for all $i\neq j$ we have 
\begin{equation}
\begin{array}{ll}
L^m_{ij\alpha\beta}=2\pi i\alpha_i\frac{\sum_{k=1}^n 4\pi^2 \beta_j(\alpha_k-\beta_k)v^{m-1}_{k(\alpha-\beta)}}{\sum_{i=1}^n4\pi^2 \alpha_i^2}.
\end{array}
\end{equation}
The matrix $C^m_{ii}$ related to the convection term $-\sum_{j=1}^n\sum_{\gamma \in {\mathbb Z}^n}\frac{2\pi i \gamma_j}{l}v^{m-1}_{j(\alpha-\gamma)}v^m_{i\gamma}$ is defined analogously as in the first stage but with coefficients $v^{m-1}_{i\alpha}$ instead of $h_{i\alpha}$.
Again, the equation (\ref{proofodem}) makes sense for regular data and inductively assumed regular coefficient functions $v^{m-1}_i$, because polynomially decay of the  modes $h_{i\alpha}$ and $v_{i\alpha}$ compensate that encoded quadratic growth due to derivatives. For a solution however we have to make sense of exponential functions with exponent $a^{ijm}_{\alpha\beta}$ defined in terms of modes $v^{m-1}_{i\beta}$ and applied to data $\mathbf{h}^F$. Again we shall use a dissipative Trotter product formula in order to have  appropriate decay as $|\alpha|,|\beta|\uparrow \infty$ in order for the solution $v^m_{i\alpha}$. However, this time we have time-dependent coefficients and we need a subscheme at each stage $m$ of the construction in order to deal with time dependent coefficients. Or, at least, a subscheme is a solution to this problem. An alternative is a direct time-local application of the Dyson formalism of the Euler part (equation without viscosity) including the regular data. Note that we need the regular data then because there is no fundamental matrix for the Euler equation. Furthermore we are better time-local since solutions of the Euler equation are globally not unique (and can be even singular). These remarks remind us of the advantages of a simple Euler scheme where we do not meet these problems.     
The formal solution of (\ref{navode4m}) is
\begin{equation}\label{solm}
\mathbf{v}^{m,F}=T\exp\left(A_mt\right)\mathbf{h}^{F},
\end{equation}
where $T\exp(.)$ is a Dyson-type time-order operator $T$ defined above.
Note that for all $m\geq 1$ the functions $\mathbf{v}^{m,F}_i,~1\leq i\leq n$ represent formal solutions of partial integro-differential equation if we rewrite them in the original space. Again in order to make sense of them we shall use the dissipative feature and a Trotter-type product formula at each substage which is a natural time-discretization. It seems that even at this stage we need viscosity $\nu >0$ in order to obtain solutions for these linear equation in this dual context. This assumption is also needed when we consider the limit $m\uparrow \infty$. We shall estimate at each stage $m$ of the construction
\begin{equation}\label{solmh}
\mathbf{v}^{m,F}=T\exp\left(A_mt\right)\mathbf{h}^{F}
\end{equation}
based on the Trotter product formula where we set up an Euler-type scheme in order to deal with time-dependent infinite matrices in the limit of substages at each stage $m$.
Then we shall also discus the alternative of a direct application of the Dyson formalism for the Euler part. Higher order time discretization schemes can be based on this, but this is postponed to the next section on algorithms.

Next we go into the details of this plan. Let us consider some linear algebra of time-independent infinite matrices with fast decaying entries (which can be applied directly at the stage $m=0$). We consider the complex situation first. The relevant structural conditions of the real systems are analogous and are sated as corollaries. The matrices of the 'complex' scheme considered above are $\left( n\times {\mathbb Z}^n\right) \times \left( n\times {\mathbb Z}^n\right)$-matrices, but - for the sake of simplicity of notation - we shall consider some matrix-algebra for ${\mathbb Z}^n \times {\mathbb Z}^n$-matrices. The considerations can be easily adapted to the formally more complicated case.
First let $D=\left(d_{\alpha\beta}\right)_{\alpha\beta\in {\mathbb Z}^n}$ and $E=\left(d_{\alpha\beta}\right)_{\alpha\beta\in {\mathbb Z}^n}$ be two infinite matrices, and define (formally)
\begin{equation}\label{matrixdot}
D\cdot E=\left( f_{\alpha\beta}\right) _{\alpha\beta\in {\mathbb Z}^n}, 
\end{equation}
where
\begin{equation}\label{matrixdot2}
f_{\alpha\beta}=\sum_{\gamma\in{\mathbb Z}^n}d_{\alpha\gamma}e_{\gamma\beta}.
\end{equation}
Next we define a space of matrices such that (\ref{matrixdot}) makes sense. For $s\in {\mathbb R}$ we define
\begin{equation}
M_n^s:=\left\lbrace D=\left(d_{\alpha\beta}\right)_{\alpha\beta\in {\mathbb Z}^n}{\Big |}~d_{\alpha\beta}\in {\mathbb C}~\&~\exists C>0: |d_{\alpha\beta}|\leq \frac{C}{1+|\alpha-\beta|^s}\right\rbrace. 
\end{equation}
In the following we consider rather regular spaces where $s\geq 2+n$. Some results can be optimized with respect to regularity, but our purpose here is full regularity in the end.
Next we have
\begin{lem}
Let $D\in M^s_n$ and $E\in M^r_s$ for some $s,r\geq n+2$. Then
\begin{equation}
D\cdot E\in M^{r+s-n}_n
\end{equation}
\end{lem}
\begin{proof}
For some  $c>0$ we have for $\alpha\neq \beta$
\begin{equation}
\begin{array}{ll}
|f_{\alpha\beta}|=|\sum_{\gamma\in{\mathbb Z}^n}d_{\alpha\gamma}e_{\gamma\beta}|\\
\\
\leq c+\sum_{\gamma\not\in \left\lbrace \alpha,\beta\right\rbrace }\frac{C}{|\alpha-\gamma|^s}\frac{C}{|\gamma-\beta|^r}\\
\\
\leq c+\frac{cC}{|\alpha-\beta|^{r+s-n}}
\end{array}
\end{equation}
The latter inequality is easily obtained b comparison with the integral
\begin{equation}
\int_{{\mathbb R}^n\setminus B_{\alpha\beta}}\frac{dy}{|\alpha-y|^{r}|y-\beta|^{s}}
\end{equation}
where $B_{\alpha\beta}$ is the union of balls of radius $1/2$ around $\alpha$ and around $\beta$. Partial integration of these intergals in polar coordinate form in different cases leads to the conclusion. We observe here that there is some advantage here in the analysis compared to the analysis in classical space because we can avoid the analysis of singularities in discrete space. This advantage can be used to get related estimates via
\begin{equation}
\begin{array}{ll}
\sum_{\gamma\not\in \left\lbrace \alpha,\beta\right\rbrace,~\alpha\neq \beta }\frac{C}{|\alpha-\gamma|^s}\frac{C}{|\gamma-\beta|^r}=\sum_{\gamma'\neq 0,~\alpha\neq \beta }\frac{C}{|\gamma'|^{s+r}}\frac{C|\gamma|'}{|\alpha-\beta|^r}\\
\\
=\sum_{\gamma'\neq 0,~\alpha\neq \beta }\frac{C}{|\gamma'|^{s+r}}\frac{C|\gamma|'}{|\alpha-\beta|^s}
\end{array}
\end{equation}
which leads to upper bounds of the form
\begin{equation}
c+\min\left\lbrace \frac{cC}{|\alpha-\beta|^{r}},\frac{cC}{|\alpha-\beta|^{s}}\right\rbrace 
\end{equation}
for some generic constants $c,C$. This line of argument is an alternative.
\end{proof}
 In the real case the matrices of the scheme are $\left( 2n\times {\mathbb N}^n\right) \times \left( 2n\times {\mathbb N}^n\right)$-matrices. Again the essential multiplication rules may be defined for ${\mathbb N}^n\times {\mathbb N}^n$ matrices, where the considerations can be easily adapted to the formally more complicated case by renumeration. Hence it is certainly sufficient to consider $\left( n\times {\mathbb N}^n\right) \times \left( n\times {\mathbb N}^n\right)$ matrices instead of$\left( 2n\times {\mathbb N}^n\right) \times \left( 2n\times {\mathbb N}^n\right)$.
First let $D=\left(d_{\alpha\beta}\right)_{\alpha\beta\in {\mathbb N}^n}$ and $E=\left(d_{\alpha\beta}\right)_{\alpha\beta\in {\mathbb N}^n}$ be two infinite matrices, and define (formally)
\begin{equation}\label{matrixdotr}
D\cdot E=\left( f_{\alpha\beta}\right) _{\alpha\beta\in {\mathbb N}^n}, 
\end{equation}
where
\begin{equation}\label{matrixdot2r}
f_{\alpha\beta}=\sum_{\gamma\in{\mathbb N}^n}d_{\alpha\gamma}e_{\gamma\beta}.
\end{equation}
The space of matrices such that (\ref{matrixdotr}) makes sense is analogous as before. For $s\in {\mathbb R}$ we define
\begin{equation}
M^{re,s}_n:=\left\lbrace D=\left(d_{\alpha\beta}\right)_{\alpha\beta\in {\mathbb N}^n}{\Big |}~d_{\alpha\beta}\in {\mathbb R}~\&~\exists C>0: |d_{\alpha\beta}|\leq \frac{C}{1+|\alpha-\beta|^s}\right\rbrace. 
\end{equation}
In the following we consider rather regular spaces where $s\geq 2+n$. Some results can be optimized with respect to regularity, but our purpose here is full regularity in the end.
Next we have
\begin{cor}
Let $D\in M^{re,s}_n$ and $E\in M^{re,r}_s$ for some $s,r\geq n+2$. Then
\begin{equation}
D\cdot E\in M^{re,r+s-n}_n
\end{equation}
\end{cor}

For $r=s\geq n+2$ this behavior allows us to define iterations of matrix multiplications recursively according to the matrix rules defined in (\ref{matrixdot}) and (\ref{matrixdot2}), i.e., we may define by recursion 
\begin{equation}
\begin{array}{ll}
A^0=I,~\mbox{where}~I=\left(\delta_{\alpha\beta}\right)_{\alpha\beta\in{\mathbb Z}^n}, \\
\\
A^1=A,\\
\\
A^{k+1}=A\cdot A^k.
\end{array}
\end{equation}
Similar definitions can be considered in the real case, of course. 
In the matrix space $M^s_n$ (resp. $M^{re,s}_n$) we may define exponentials. For our problem this space is too 'narrow' to apply this space directly. However, it is useful to note that we have exponentials in this space. 
\begin{cor}
Let $D\in M^s_n$ (resp. $D\in M^{re,s}_n$) for some $s\geq n+2$. Then
\begin{equation}
\exp(D)=\sum_{k=0}^{\infty}\frac{D^k}{k!}\in M^s_n~(\mbox{resp.}\in M^{re,s}_n).
\end{equation}
\end{cor}
Well what we need is an estimate for the approximative solutions. In this context we note
\begin{cor}
Let $D\in M^s_n$ (resp. $D\in M^{re,s}_n$) for some $s\geq n+2$ and let $\mathbf{h}^F\in h^s\left({\mathbb Z}^n\right) $ (resp. $\mathbf{h}^F\in h^s\left({\mathbb N}^n\right) $). Then
\begin{equation}
\exp(D)\mathbf{h}^F\in h^s\left({\mathbb Z}^n\right)~\left( \mbox{resp}~\in h^s\left({\mathbb N}^n\right)\right) .
\end{equation}
\end{cor}
\begin{proof}
Let 
\begin{equation}
F=\exp(D),~\mbox{where}~F=(f_{\alpha\beta})_{\alpha\beta\in{\mathbb Z}^n},
\end{equation}
and let
\begin{equation}
g_{\alpha}=\sum_{\beta\in {\mathbb Z}^n}f_{\alpha\beta}h_{\beta}.
\end{equation}
Then for some $C,c>0$
\begin{equation}
|g_{\alpha}|=|\sum_{\gamma \in {\mathbb Z}^n}f_{\alpha\beta}h_{\beta}|\leq \sum_{\beta\in {\mathbb Z}^n\setminus \left\lbrace \alpha,0\right\rbrace }\frac{C}{|\alpha-\beta|^s|\beta|^s}\leq \frac{cC}{1+|\alpha|^s}.
\end{equation}
Analogous observations hold in the real case.
\end{proof}
We cannot apply the preceding lemmas directly due to the fact that the diagonal matrix with entries $-\nu\delta_{ij}\sum_{k=1}^n\frac{4\pi^2}{l^2}\alpha_k^2$ is not bounded. Neither is the the matrix related to the convection term. However the multiplication of the dissipative exponential with iterative multiplications of the convection term matrix stay in a regular matrix space such that a multiplication with regular data of polynomial decay lead to regular results. If we want to have a fundamental matrix or uniqueness, then the dissipative term - the smoothing effect of which is obvious in classical space - makes the difference. At first glance, iterations of the matrix lead to matrices which live in weaker and weaker spaces, and we really need the dissipative feature if we do not take the application of the data into account, i.e., the minus signs in the diagonal in mathematical terms, in order to detect the smoothing effect in the exponential form. The irregularity related to a positive sign reminds us of the fact that heat equations cannot be solved backwards in general. However, due to its diagonal structure and its negative sign we can prove a BCH-type formula for infinite matrices. The following has some similarity with Kato's results for semigroups of dissipative operators (cf. \cite{Kat}). However Kato's results (cf. \cite{Kat}) usually require that the domain of the dissipative operator includes the domain of the second operator summand, and this is not true in our formulation for the incompressible Navier-Stokes equation, because the diagonal operator related to the Laplacian, i.e., the diagonal terms $\left(-\delta_{\alpha\beta}\nu\sum_{i=1}^n\alpha_i^2\right)_{\alpha ,\beta\in {\mathbb Z}^n\setminus \left\lbrace 0 \right\rbrace }$,  increase quadratically with the order of the modes $\alpha$. So it seems that the result cannot be applied directly. Anyway we continue with the constructive (algorithmic) view, which means that we consider the behavior of finite mode approximations and then we go to the limit of infinite mode systems later. We have
\begin{lem}\label{lembound}
Let $g^F_l$  be finite vectors of modes of order less or equal to $l>0$ (not to be confused with the torus size which we consider to be equal to $1$ w.l.o.g.). Then for any finite vector $f^F_l=\left(f^l_{\alpha}\right)_{|\alpha|\leq l}$ with finite entries $f^l_{\alpha}$, and in the situation of Lemma \ref{productform} below we have
\begin{equation}
{\big |}\left( \exp\left( C^lt\right)\mathbf{g}^F_l -\exp\left( (A^l+B^l)t\right)\mathbf{g}^F_l\right)_{\alpha}{\big |}\leq {\big |}f^l_{\alpha}t^2{\big |},
\end{equation}
where $\left(.\right)_{\alpha}$ denotes the projection to the  $\alpha$th component of an infinite vector, and $\exp\left( C^lt\right)=\exp\left( A^lt\right)\exp\left( B^lt\right)$.
\end{lem}
This is what we need in the following but let us have a closer look at the Trotter product formula. The reason is that we represent solutions in a double limit where the inner limit is a Trotter product time limit. The polynomial decay behavior of the data and the matrix entries of the problems then ensure that the spatial limits are well-defined (as the upper bound of the modes $l$ goes to infinity). In the following we use the 'complex' notation and formulate results with respect to ${\mathbb Z}^n$. The considerations can be transferred immediately to real systems with multiindices in ${\mathbb N}^n$. First we define
\begin{defi}
A diagonal matrix $\left(\delta_{\alpha \beta}d_{\alpha\beta} \right)_{\alpha,\beta \in {\mathbb Z}^n}$ is called strictly dissipative of order $m>0$ if $d_{\alpha\alpha}<-c|\alpha| ^m$ for all $\alpha\in {\mathbb Z}^n$ and for some constant $c>0$. It is called dissipative if it is dissipative of some order $m$.
\end{defi}
This is a very narrow definition which we use as we consider constant viscosity, but this is sufficient for the purposes of this paper while the generalisation to variable viscosity is rather straightforward.
In order to state the Trotter-product type result we introduce some notation.
\begin{defi}
For all $m\geq 1$ let $\gamma^m=(\gamma^m_1,\cdots \gamma^m_m)$ denote multiindices with $m$ components and with nonnegative integer entries $\gamma^m_i$ for $m\geq 1$ and $1\leq i\leq m$. For each $m\geq 1$ we denote the set of $m$-tuples with nonnegative entries by ${\mathbb N}_{0}^m$. Let $B=\left(b_{\alpha\beta}\right)_{\alpha\beta\in {\mathbb Z}^n}$, denote a quadratic matrix, and let $B^T=\left(b_{\beta\alpha}\right)_{\alpha\beta\in {\mathbb Z}^n}$ be its transposed, and $E$ be some other infinite matrix of the same type. For $m\geq 1$ and $m$-tuples $\gamma^m=(\gamma^m_1,\cdots,\gamma^m_n)$ with nonnegative entries $\gamma^m_j,~1\leq j\leq n$, we introduce some abbreviations for certain iterations of Lie brackets operations of matrices. These are iterations of the matrix $\Delta B:=B-B^T$ and either the matrix $B$ or the matrix $B^T$ in arbitrary order. This gives different expressions dependent on the matrix with which we start, and we define  $I_{\gamma^m}$ (starting with $\left[\Delta B,B \right]$)  and $I^T_{\gamma}$ (starting with $\left[\Delta B,B^T\right]I^T_{\gamma}$ accordingly. First we define $I_{\gamma^1}\left(\Delta B, B\right)=I_{(\gamma^1_1)}(\Delta B,B)$ (starting with $\left[\Delta B,B^T\right]$  for $\gamma_1\geq 0$ recursively). Let
\begin{equation}
\left[E,B\right]_T=EB-B^TE, 
\end{equation}
which is a Lie-bracket type operation with the transposed. 
For $\gamma^1_1=0$ define
\begin{equation}
I_{(\gamma^1_1)}\left[\Delta B,B\right]:=\Delta B, 
\end{equation}
and for $\gamma^1_1>0$ define
\begin{equation}
I_{(\gamma^1_1)}\left[\Delta B,B\right]:=\left[I_{\gamma_1-1}\left[\Delta B,B \right] ,B\right]_T.
\end{equation}
Having defined $I_{\gamma^{m-1}}$ and if $\gamma^m_m=0$ then
\begin{equation}
I_{\gamma^m}\left[\Delta B,B\right]=I_{\gamma^{m-1}}\left[\Delta B,B\right]+I_{\gamma^{m-1}}\left[\Delta B,B^T\right]
\end{equation}
Finally, if $\gamma^m_m>0$, then define
\begin{equation}
I_{\gamma^m}\left[\Delta B,B\right]=\left[ I_{\gamma^{m-1}}\left[\Delta B,B\right],B\right]_T. 
\end{equation}
Similarly, for $\gamma^1_1=0$ define
\begin{equation}
I^T_{(\gamma^1_1)}\left[\Delta B,B\right]:=\Delta B, 
\end{equation}
and for $\gamma^1_1>0$ define
\begin{equation}
I^T_{(\gamma^1_1)}\left[\Delta B,B\right]:=\left[I_{\gamma_1-1}\left[\Delta B,B^T \right] ,B\right]_T.
\end{equation}
Having defined $I_{\gamma^{m-1}}$ and if $\gamma^m_m=0$ then
\begin{equation}
I^T_{\gamma^m}\left[\Delta B,B\right]=I^T_{\gamma^{m-1}}\left[\Delta B,B\right]+I^T_{\gamma^{m-1}}\left[\Delta B,B^T\right]
\end{equation}
Finally, if $\gamma^m_m>0$, then define
\begin{equation}
I^T_{\gamma^m}\left[\Delta B,B\right]=\left[ I^T_{\gamma^{m-1}}\left[\Delta B,B\right],B^T\right]_T. 
\end{equation}
\end{defi}
%
Next we prove a special CBH-formula for finite matrices. We do not need the full force of the lemma \ref{productform} below for our purpose, but it has some interest of its own. The reader who is interested only in the global existence proof may skip it and consider the simplified alternative considerations in order to see that how a Trotter product result can be applied. The results may be generalized but our main purpose in this article is to define a converging algorithm for the incompressible Navier-Stokes equation which provides also a constructive approach to global existence.
We show
\begin{lem}\label{productform}
Define the set of finite modes
\begin{equation}
{\mathbb Z}^n_l:=
\left\lbrace \alpha\in {\mathbb Z}^n||\alpha|\leq l\right\rbrace .
\end{equation}
Let $A^l$ be the cut-off of order $l$ of the dissipative diagonal matrix of order $2$ related to the Laplacian and let $B^l=\left(b_{\alpha\beta}\right)_{\alpha\beta\in {\mathbb Z}^n_l}$ be the cut-off of some other matrix. Next for an arbitrary finite quadratic matrix $N^l=\left( n_{\alpha\beta}\right)_{|\alpha|,|\beta|\leq l}$ let 
\begin{equation}
\exp_m\left(N\right)=\exp\left(N\right)-\left(\sum_{k=0}^{m-1}\frac{N^k}{k!}\right).
\end{equation} 
Then the relation 
\begin{equation}\label{BCH*}
 \exp(A^l)\exp(B^l)=\exp(C^l),
\end{equation}
holds where $C^l$ is of the form
\begin{equation}\label{Ceq*}
\begin{array}{ll}
C^l=A^l+B^l+\frac{1}{2}\exp(2A^l)\Delta B^l+\sum_{m\geq 1}\exp_m\left( A^l\right)\times\\
\\
\times \left(  c_{\beta^m}I_{\beta^m}\left[\Delta  B^l ,B^l \right]_T+c^T_{\beta^m}I^T_{\beta^m}\left[\Delta  B^l ,B^l \right]_T\right) 
\end{array}
\end{equation}
for some constants $c_{\beta^m},c^T_{\beta^m}$ such that the series
\begin{equation}
\sum_{m\geq 1}c_{\beta^m}t^m
\end{equation}
and the series
\begin{equation}
\sum_{m\geq 1}c^T_{\beta^m}t^m
\end{equation}
converges absolutely for all $t\geq 0$.
\end{lem}
\begin{proof}
For simplicity of notation we suppress the upper script $l$ in the following. All matrices $A,B,C,\cdots $ are finite multiindexed matrices of modes of order less or equal to $l$.
The matrix $C=\sum_{i=1}^{\infty}C_i$ is formally determined via power series
\begin{equation}
C(x)=\sum_{i=1}^{\infty}C_ix^i,~C'(x)=\sum_{i=1}^{\infty}iC^l_ix^{i-1}
\end{equation}
by the relation 
\begin{equation}
\sum_{k=0}^{\infty}R^k\left[C'(x),C(x) \right]=A+B+\sum_{k\geq 1}R^k\left[A,B\right]\frac{x^k}{k!},
\end{equation}
where for matrices $E,F$ $\left[E,F\right]=EF-FE$ denotes the Lie bracket and 
\begin{equation}
R^1[E,F]=[E,F],~R^{k+1}[E,F]=\left[ R^k[E,F],F\right] 
\end{equation}
recursively (Lie bracket iteration on the right). Comparing the terms of different orders one gets successively for the first $C$-terms
\begin{equation}
\begin{array}{ll}
C_1=A+B,~C_2=\frac{1}{2}\left[A,B\right],\\
\\
C_3=\frac{1}{12}\left[A,\left[A,B\right] \right]+\frac{1}{12}\left[\left[A,B\right],B\right],\\
\\
C_4=\frac{1}{48}\left[A,\left[\left[A,B\right],B\right]\right]+
\frac{1}{48} \left[ \left[A,\left[A,B\right]\right] ,B\right]\\
\\
C_5=\frac{1}{120}\left[ A,\left[ \left[A,\left[A,B\right]\right] ,B\right]\right]
+\frac{1}{120}\left[ A,\left[ \left[\left[A,B\right],B\right]\right]  ,B\right]\\
\\
-\frac{1}{360}\left[ A,\left[ \left[ \left[A,B\right],B\right] ,B\right]\right]
-\frac{1}{360}\left[ \left[ A,\left[A, \left[A,B\right]\right] \right] ,B\right]\\
\\
-\frac{1}{720}\left[ A,\left[ A,\left[A, \left[A,B\right]\right]\right]\right]
-\frac{1}{720}\left[ \left[\left[\left[A,B\right],B\right],B \right] ,B\right],\cdots .
\end{array}
\end{equation}
Iterated Lie brackets simplify if $A$ is a diagonal matrix. First we 
define left Lie-bracket iterations, i.e.,
\begin{equation}
L^1[E,F]=[E,F],~L^{k+1}[E,F]=\left[ E\left[ L^k[E,F]\right] \right] 
\end{equation}
recursively. Next we study the effect of alternative applications of iterations of the left and right Lie-bracket operation for this specific $A$.
We have
\begin{equation}
\left[ A,B\right]=A\Delta B,
\end{equation}
with $\Delta B=B-B^T$, and 
\begin{equation}
L^k\left[ A,B\right]=A^k2^{k-1}\Delta B,
\end{equation}
Next we have
\begin{equation}
R^k\left[ A,B\right]=AR^{k-1}\left[ \Delta B,B\right]_T
\end{equation}
Next
\begin{equation}
LR^k\left[ A,B\right]=A^2R^{k-1}\left( \left[ \Delta B,B\right]_T+\left[ \Delta B,B^T\right]_T\right) 
\end{equation}
Induction leads to the observation that other summands than $A+B$ of the series for $C$ can be written in terms of expressions of the form 
\begin{equation}
\left( A\right) ^k  I_{\beta^m}\left[\Delta B,B \right]_T
\end{equation}
 and in terms of expressions of the form  
 \begin{equation}
 \left( A\right) ^k  I^T_{\beta^m}\left[\Delta B,B \right]_T
 \end{equation}
For each order $p=k+m$ we have a factor $\frac{1}{p!}$ with $\leq 2^p$ summands. This leads to global convergence of the coefficients $c_{\beta^m},c^T_{\beta^m}$.
\end{proof}

We continue to describe consequences in a framework which is very close to the requirements of the systems related to the incompressible Navier-Stokes equation. As a simple consequence of the preceding lemma we have
\begin{lem}
Let $g^F_l$ be a finite vector of modes of order less or equal to $l>0$ .
In the situation of Lemma \ref{productform} we have
\begin{equation}
\lim_{k\uparrow \infty}\left( \exp\left(A^l\frac{t}{k}\right)\exp\left(B^l\frac{t}{k}\right)\right) ^kg^F_l=\exp\left( (A^l+B^l)t\right)g^F_l.   
\end{equation}
\end{lem}
It is rather straightforward to reformulate the latter result for equations with matrices of type $P_{M^l}A_0$, $P_{M^l}A^{r}_0$, i.e. finite approximations of order $l$ of the $n{\mathbb Z}^n\times n{\mathbb Z}^n$-matrices $A_0$, $A^{r}_0$. Strictly speaking, the matrix of the controlled system is in ${\mathbb Z}^n\setminus \left\lbrace 0\right\rbrace\times {\mathbb Z}^n\setminus \left\lbrace 0\right\rbrace$, but me we add zeros and treat it formally in the same matrix space.  We have
\begin{cor}
Let $B^{0,l}_b=\left( P_{M^l}\left( C^0_i+L^0_{ij}\right)_{ij} \right) $ and \newline
$B^{r,l,0}_b=\left( P_{M^l}\left( C^{r,0}_i+L^{r,0}_{ij}\right)_{ij} \right)$ such that
$A^l_0=D^{0,l}_b+B^{0,l}_b$, $A^{r,l}_0=D^{0,l}_b+B^{r,l}_b$. Then the Trotter product formula
\begin{equation}
\lim_{k\uparrow \infty}\left( \exp\left(D^{0,l}_b\frac{t}{k}\right)\exp\left(B^{0,l}_b\frac{t}{k}\right)\right) ^k=\exp\left( A^l_0t\right), 
\end{equation}
and the Trotter product formula
\begin{equation}
\lim_{k\uparrow \infty}\left( \exp\left(D^{r,l,0}_b\frac{t}{k}\right)\exp\left(B^{r,l,0}_b\frac{t}{k}\right)\right) ^k=\exp\left( A^{r,l}_0t\right) 
\end{equation}
hold.
\end{cor}

Next note that iterations of $B^{0,l}_b$ of order $k\geq 1$ have at most linear growth for constellations of multiindexes where $\alpha-\gamma$ is constant (with the order of the modes as $l\uparrow \infty$). Let is have a closer look at this. Note that $B^{0,l}_b$ is a $n{\mathbb Z}^n\times n{\mathbb Z}^n$ matrix. However, it is sufficient to consider the subblocks $B^{0,l}_{b,ij}$ for fixed $1\leq i,j\leq n$, which are matrices in ${\mathbb Z}^n\times {\mathbb Z}^n$ in order to estimate the growth behavior of iterations of $B^{0,l}_b$. We have
\begin{equation}\label{navode3cit}
\begin{array}{ll}
B^{0,l}_{b,ij\alpha\beta}=
-\sum_{j=1}^n\frac{2\pi i \beta_j}{l}h_{j(\alpha-\beta)}+2\pi i\alpha_i1_{\left\lbrace \alpha\neq 0\right\rbrace}\frac{\sum_{k=1}^n 4\pi \beta_j(\alpha_k-\beta_k)h_{k(\alpha-\beta)}}{\sum_{i=1}^n4\pi^2\alpha_i^2}.
\end{array} 
\end{equation}
We consider the entries for the square. We have
\begin{equation}\label{navode3cit2}
\begin{array}{ll}
\sum_{\beta \in {\mathbb Z}^n}B^{0,l}_{b,ij\alpha\beta}B^{0,l}_{b,ij\beta\gamma}=\\
\\
{\Big (}-\sum_{j=1}^n\frac{2\pi i \beta_j}{l}h_{j(\alpha-\beta)}+2\pi i\alpha_i1_{\left\lbrace \alpha\neq 0\right\rbrace}\frac{\sum_{k=1}^n 4\pi \beta_j(\alpha_k-\beta_k)h_{k(\alpha-\beta)}}{\sum_{i=1}^n4\pi^2\alpha_i^2}{\Big )}\\
\\
{\Big (}-\sum_{j=1}^n\frac{2\pi i \gamma_j}{l}h_{j(\beta-\gamma)}+2\pi i\beta_i1_{\left\lbrace \beta\neq 0\right\rbrace}\frac{\sum_{k=1}^n 4\pi \gamma_j(\beta_k-\gamma_k)h_{k(\beta-\gamma)}}{\sum_{i=1}^n4\pi^2\beta_i^2}{\Big )}.
\end{array} 
\end{equation}
Expanding the latter product (\ref{navode3cit2}) and inspecting the four terms of the expansion we observe that two of these four terms are entries of matrices in the regular matrix space. Only the mixed products of convection terms entries and Leray projection terms entries lead to expressions which are a little less regular. We define an appropriate matrix space
\begin{equation}
\begin{array}{ll}
M_n^{s,\mbox{lin}}:=\\
\\
\left\lbrace D=\left(d_{\alpha\beta}\right)_{\alpha\beta\in {\mathbb Z}^n}{\Big |}~d^l_{\alpha\beta}\in {\mathbb C}~\&~\exists C>0~\forall l: |d^l_{\alpha\beta}|\leq \frac{C+C\left(|\alpha|+|\beta| \right) }{1+|\alpha-\beta|^s}\right\rbrace. 
\end{array}
\end{equation}
By induction we have
\begin{lem}\label{bit}
Given $1\leq i,j\leq n$ for all $k\geq 1$ we have
\begin{equation}
\left( B^{0,l}_{b,ij\alpha\beta}\right)^k\in M_n^{s,\mbox{lin}}
\end{equation}
\end{lem} 
Hence the statement of lemma \ref{bit} is true also for the exponential of $B^{0,l}_{bij}$. We note
\begin{lem}\label{dbexp}
If $D$ is a strictly dissipative diagonal matrix of order $2$, then for given $1\leq i,j\leq n$ we have
\begin{equation}
\exp(D)\exp\left( B^{0,l}_{b,ij\alpha\beta}\right)\in M_n^{s}.
\end{equation}
\end{lem} 
In order to establish Trotter product representations for our iterative linear approximations $v^{r,m,F}$ of a controlled Navier-Stokes equation we consider a matrix space for sequences of finite matrices $\left( D^l\right)_{l}$ where each $D^l$ has modes of order less or equal to $l$. We have already observed that the systems related to the approximations of incompressible Navier Stokes equation operator are matrices in a $n{\mathbb Z}^n\times n{\mathbb Z}^n$ space. In order to prove convergence of Trotter product formulas for these equations it may be useful to define an appropriate matrix space.
\begin{equation}
\begin{array}{ll}
M_{n\times n}^s:={\Big \{} D=(D_{ij})_{1\leq i,j\leq n}=\left(d_{ij\alpha\beta}\right)_{\alpha\beta\in {\mathbb Z}^n,1\leq i,j\leq n}
{\Big |}\\
\\
~d_{ij\alpha\beta}\in {\mathbb C}~\&~\exists C>0: \max_{1\leq i,j\leq n}|d_{ij\alpha\beta}|\leq \frac{C}{1+|\alpha-\beta|^s}{\Big \}}. 
\end{array}
\end{equation}
Next we define a space of sequences $\left( \left( D^l_{ij}\right)_{1\leq i,j\leq n}\right)_l$ of finite matrices which approximate matrices $M_{n\times n}^s$. We define
\begin{equation}
\begin{array}{ll}
M_{n\times n}^{s,\mbox{fin}}:={\Big \{} \left( \left( D^l_{ij}\right)_{1\leq i,j\leq n}\right)_l =\left(d^l_{ij\alpha\beta}\right)_{\alpha\beta\in {\mathbb Z}^n_l,1\leq i,j\leq n}{\Big |}\\
\\
~d^l_{ij\alpha\beta}\in {\mathbb C}~\&~\exists C>0~\forall l: \max_{1\leq i,j\leq n}|d^l_{ij\alpha\beta}|\leq \frac{C}{1+|\alpha-\beta|^s}{\Big \}}. 
\end{array}
\end{equation}
Here for each $l\geq 1$ the set ${\mathbb Z}^n_l$ denotes the set of modes of order less or equal to $l$. 
For our dissipative matrix $D^{0,l}_b$ it follows that for all $t,k$ and
\begin{equation}
\left( \exp\left(D^{0,l}_b\frac{t}{k}\right)\exp\left(B^{0,l}_b\frac{t}{k}\right)\right)^k_l\in M^{s,\mbox{fin}}_{n\times n}
\end{equation}
for $s\geq n$, i.e., the finite approximations of the Trotter product formula are regular matrices indeed.  Hence
\begin{equation}
\lim_{l\uparrow\infty}\lim_{k\uparrow \infty}\left( \exp\left(D^{0,l}_b\frac{t}{k}\right)\exp\left(B^{0,l}_b\frac{t}{k}\right)\right)^k\in M^s_{n\times n}
\end{equation}

Now consider finite approximations of the problem (\ref{navode4big}), i.e. cut-offs of this problem, i.e., a set of problems of finite modes with modes of order less or equal to $l$ with
\begin{equation}\label{navode4bigl}
\frac{d \mathbf{v}^{0,F}_l}{dt}=A^l_0\mathbf{v}^{0,F}_l.
\end{equation}
For each $l$ the solution
\begin{equation}\label{soll}
\mathbf{v}^{0,F}_l=\exp\left(A^l_0t\right)\mathbf{h}^{F}_l
\end{equation}
is globally well-defined via the dissipative Trotter product formula. Our infinite linear algebra lemmas above imply that the sequence $\left( \mathbf{v}^{0,F}_l\right)_{l}$ is a Cauchy sequence for $s>n$ if $\mathbf{h}^{F}_i\in h^s\left({\mathbb Z}^n\right)$ for all $s\in {\mathbb R}$ and $1\leq i\leq n$. Hence we have
\begin{lem}
We consider the torus of length $l=1$ w.l.o.g..
Let $\mathbf{h}^{F}\in h^s\left({\mathbb Z}^n\right)$ for all $s\in {\mathbb R}$.
\begin{equation}\label{sol0}
\begin{array}{ll}
\mathbf{v}^{0,F}_i=\left( \exp\left(A_0t\right)\mathbf{h}^{F}\right)_i\\
\\
=\lim_{l\uparrow \infty}\left( \lim_{k\uparrow \infty}\left( \exp\left(D^{0,l}_b\frac{t}{k}\right)\exp\left(B^{0,l}_b\frac{t}{k}\right)\right) ^k\mathbf{h}^{F}_l\right)_i\in h^s\left({\mathbb Z}^n\right)
\end{array} 
\end{equation}
whenever $\mathbf{h}^{F}_i\in h^s\left({\mathbb Z}^n\right)$ for $s>n$ for $1\leq i\leq n$.
\end{lem}
For the same reason
\begin{cor}
Consider the same situation as in the preceding lemma.
\begin{equation}\label{sol0}
\begin{array}{ll}
\mathbf{v}^{r,0,F}_i=\left( \exp\left(A^r_0t\right)\mathbf{h}^{F}\right)_i\\
\\
=\lim_{l\uparrow \infty}\left( \lim_{k\uparrow \infty}\left( \exp\left(D^{r,0,l}_b\frac{t}{k}\right)\exp\left(B^{r,0,l}_b\frac{t}{k}\right)\right) ^k\mathbf{h}^{F}_l\right)_i\in h^s\left({\mathbb Z}^n\right)
\end{array} 
\end{equation}
whenever $\mathbf{h}^{F}_i\in h^s_l\left({\mathbb Z}^n\right)$ for $s>n$ for $1\leq i\leq n$.
\end{cor}

\begin{proof}
Let $A^{r,l}_0=P_{M^l}A^r_0$ denote the projection of the matrix $A^r_0$ to the finite matrix of modes less or equal to modes of order $l$. For finite vectors $\mathbf{v}^{r,0,F}_{i,l}$ with
\begin{equation}\label{sol00002}
\mathbf{v}^{r,0,F}_{i}=\lim_{l\uparrow \infty}\mathbf{v}^{r,0,F}_{i,l}
\end{equation}
we have
\begin{equation}\label{sol0000l}
\begin{array}{ll}
\mathbf{v}^{r,0,F}_{i,l}=\left( \exp\left(A^{r,l}_0t\right)\mathbf{h}^{F}\right)_i\\
\\
=\left( \lim_{k\uparrow \infty}\left( \exp\left(D^{r,0,l}_b\frac{t}{k}\right)\exp\left(B^{r,0,l}_b\frac{t}{k}\right)\right) ^k\mathbf{h}^{F}_l\right)_i
\end{array} 
\end{equation}
Note that
\begin{equation}
\lim_{k\uparrow \infty}\left( \exp\left(D^{r,0,l}_b\frac{t}{k}\right)\exp\left(B^{r,0,l}_b\frac{t}{k}\right)\right) ^k\mathbf{h}^{F}_l
\end{equation}
is a Cauchy sequence in $M^{s,Fin}_{n\times n}$
Hence, the left side of (\ref{sol00002}) is a limit of a Cauchy sequence in $M^s_n$.
\end{proof}

The stage $m=0$ is special as the matrix $B^{r,0,l}_b$ does not depend on time. For this reason we get similar Trotter formulas for time derivatives using the damping of the dissipative factor $\exp\left(D^{r,0,l}_b\right) $. We shall use first order time derivatives formulas for linear problems with time-independent coefficients later when we approximate time dependent linear problems via time discretizations. We have
\begin{cor}
Recall that $A^{r,l}_0=P_{M^l}A^r_0$ denotes the projection of the matrix $A^r_0$ to the finite matrix of modes less or equal to modes of order $l$.  
\begin{equation}\label{sol02}
\begin{array}{ll}
\frac{d}{d t}\mathbf{v}^{r,0,F}_i=\frac{d}{d t}\left( \exp\left(A^r_0t\right)\mathbf{h}^{F}\right)_i\\
\\
=\lim_{l\uparrow \infty}\left( A^{r,l}_0\lim_{k\uparrow \infty}\left( \exp\left(D^{r,0,l}_b\frac{t}{k}\right)\exp\left(B^{r,0,l}_b\frac{t}{k}\right)\right) ^k\mathbf{h}^{F}_l\right)_i\in h^s\left({\mathbb Z}^n\right)
\end{array} 
\end{equation}
whenever $\mathbf{h}^{F}_i\in h^s_l\left({\mathbb Z}^n\right)$ for $s>n$ for $1\leq i\leq n$.
\end{cor}

Next at stage $m\geq 1$ we cannot apply the results above directly in order to define
\begin{equation}\label{Dysonlimit}
\left( \mathbf{v}^{m,F}\right)=T\exp\left(A_mt\right)\mathbf{h}^{F},
\end{equation}
or in order to define
\begin{equation}\label{Dysonlimitr}
\left( \mathbf{v}^{r,m,F}\right)=T\exp\left(A^r_mt\right)\mathbf{h}^{F}.
\end{equation}
The main difference of the stages $m\geq 1$ to the stage $m=0$ is the time dependence of the coefficients formally expressed by the operators $T$ (\ref{Dysonlimit}) and (\ref{Dysonlimitr}) above. We have already mentioned various methods in order to deal with matter. In principle there are two alternatives: either we may solve the Euler part of the equation separately using a Dyson formalism or we may set up an Euler scheme and prove a Trotter product formula for a subscheme with time-homogeneous subproblems. There are variations for each alternative, of course, but these are the basic possibilities.
We first consider the simpler Euler scheme (second alternative) and postpone the treatment  of the Dyson formalism for the Euler part and related higher order schemes to the end of this section and to the next section on algorithms respectively.  The second difference is that we need to control the zero modes if we want to establish a global scheme.
For the latter reason, next we consider the controlled scheme (the considerations apply to the uncontrolled scheme as well for one iteration step). In each iteration step we construct the solution of an infinite linear ODE in dual space which corresponds to a linear partial integro-differential equation in the original space. For each iteration step we need some subiterations in order to deal with the time-dependence of the coefficients. In order to apply our observations concerning linear algebra of infinite systems and the Trotter product formula, we consider time-discretizations. The time dependent formulas in  (\ref{Dysonlimit}) and (\ref{Dysonlimitr}) have rigorous definition via double limits (with respect to time and with respect to modes) of Trotter product formulas for finite systems. 
There are basically two possibilities to define a time-discretized scheme based on this form of a Trotter product formula. One possibility is to consider the successive linearized global problems at each stage $m$ of the construction, where the matrices $A^r_m$ are known in terms of the entries of the infinite vector $\mathbf{v}^{r,m-1,F}$ which contains information known from the previous step. According to the dissipative Trotter product formula we expect a time-discretization error of order $O(h)$ where $h$ is an upper bound of the time step sizes. 
The other possibility is to apply the dissipative Trotter product formula locally and establish a time local limit $\lim_{m\uparrow \infty}\mathbf{v}^{r,m,F}$ at each time step proceeding by the semi-group property. 

Next we observe that for each stage of approximation the solution 
\begin{equation}\label{Dysonlimit*}
\left( \mathbf{v}^{m,F}\right)=T\exp\left(A^r_mt\right)\mathbf{h}^{F},
\end{equation}
of the linear equation
\begin{equation}\label{Dysonlimitr*}
\left( \mathbf{v}^{r,m,F}\right)=T\exp\left(A^r_mt\right)\mathbf{h}^{F}.
\end{equation}
can be computed rather explicitly if we assume $\mathbf{v}^{r,m,F}_i\in h^{s}\left({\mathbb Z}^n \right)$ for $s>n+2$ inductively. Define the approximative Euler matrix of order $p$ corresponding to the approximation Euclidean part of the equation of order $p$, i.e., on the diagonal $i=j$ define
\begin{equation}
\begin{array}{ll}
\delta_{ij}E^{r,NS}_{i\alpha j\beta}\left(\mathbf{v}^{r,p-1,E}\right)=
-\delta_{ij}\sum_{j=1}^n\frac{2\pi i \beta_j}{l}v^{r,p-1,E}_{j(\alpha-\beta)}\\
\\+\delta_{ij}2\pi i\alpha_i1_{\left\lbrace \alpha\neq 0\right\rbrace }\frac{\sum_{k=1}^n4\pi \beta_j(\alpha_k-\beta_k)v^{r,p-1,E}_{k(\alpha-\beta)}}{\sum_{i=1}^n4\pi\alpha_i^2},
\end{array}
\end{equation}
and off-diagonal, i.e. for $i\neq j$, define the entries
\begin{equation}
(1-\delta_{ij})E^{r,p-1,NS}_{i\alpha j\beta}\left(\mathbf{v}^{r,p-1,E}\right)=2\pi i\alpha_i1_{\left\lbrace \alpha\neq 0\right\rbrace }\frac{\sum_{k=1}^n4\pi \beta_j(\alpha_k-\beta_k)v^{r,p-1,E}_{k(\alpha-\beta)}}{\sum_{i=1}^n4\pi\alpha_i^2}.
\end{equation}
Similar for the uncontrolled Euler system, where we have to include the zero modes. For the Euler system the control makes no difference as we have no viscosity terms. For this reason we consider the uncontrolled Euler system in the following and denote the corresponding uncontrolled Euler matrix, i.e., the matrix corresponding the the incompressible Euler equation just by $E(t)$ for the sake of brevity. Similar the matrix of the iterative approximation of stage $p$ is denoted by $E^p(t)$ in the following. The Euler system is more difficult to solve since we have no fundamental matrix which lives in a space of reasonable regularity. We are interested in two issues: a) whether the approximative Euler system can be solved rather explicitly such that this solution can serve for proving a Trotter product formula solution for approximative Navier Stokes equation systems (and hence for the definition of higher order schemes), and b) whether the solution of the Euler equation is possible for $n\geq 3$ or wether solution of this equation are generically singular (have singularities).   
First we observe that the approximative uncontrolled Euler system with data
\begin{equation}
\begin{array}{ll}
\mathbf{v}^{p-1,E}(t_0)=\left( \mathbf{v}^{p-1,E}(t_0)_1,\cdots ,\mathbf{v}^{p-1,E}(t_0)_n\right)^T,~\\
\\
\mathbf{v}^{p-1,E}_i(t_0)\in h^{s}\left({\mathbb Z}^n\right),~\mbox{ for }1\leq i\leq n,~s>n+2 
\end{array}
\end{equation}
for some initial time $t_0\geq 0$ has a solution for some time $t>0$ of Dyson form
\begin{equation}\label{dysonobseulerrp}
\begin{array}{ll}
\mathbf{v}^{p,E}_1=v(t_0)+\sum_{m=1}^{\infty}\frac{1}{m!}\times\\
\\
\times\int_0^tds_1\int_0^tds_2\cdots \int_0^tds_m T_m\left(E^{p-1}(t_1)E^{p-1}(t_2)\cdot \cdots \cdot E^{p-1}(t_m) \right)\mathbf{v}^{p-1,E}(t_0),
\end{array} 
\end{equation}
where the time order operator is defined as above. The matrices $E^{p-1}(t_m)$ depend on $\mathbf{v}^{p-1,E}$ at higher stages $p$ of approximation and on $\mathbf{v}^{p-1,F}(t_0)$ at the first stage of approximation, and having $\mathbf{v}^{p-1,E}_i(t_0)\in h^s$ for $1\leq\leq n$ and $s>n+2$, and assuming $\mathbf{v}^{p-1,E}_i\in h^s$ for $1\leq i\leq n$ and $s>n+2$, we get 
\begin{equation}\label{Tm}
T_m\left(E^{p-1}(t_1)E^{p-1}(t_2)\cdot \cdots \cdot E^{p-1}(t_m) \right)\mathbf{v}^{p-1,E}(t_0)\in \left[ h^s\left({\mathbb Z}^n\right)\right]^n
\end{equation}
for $s>n+2$ by the upper bounds for matrix vector multiplications considered above. Furthermore, we get the upper bounds $C^mt^m$ for the term in (\ref{Tm}) for some constant $C>0$, and it follows that the solution to the approximative equation in (\ref{dysonobseulerrp}) is well-defined for $t\geq 0$. We denote
\begin{equation}\label{dysonobseulerrp2}
\begin{array}{ll}
T\exp\left(E^{p}(t) \right)v(t_0):=v(t_0)+\sum_{m=1}^{\infty}\frac{1}{m!}\times\\
\\
\times\int_0^tds_1\int_0^tds_2\cdots \int_0^tds_m T_m\left(E^{p-1}(t_1)E^{p-1}(t_2)\cdot \cdots \cdot E^{p-1}(t_m) \right)\mathbf{v}^{p-1,E}(t_0).
\end{array} 
\end{equation}

It is then straightforward to prove the following Trotter product formula for the approximation of order $p$ of the solution of the Navier Stokes equation system.
\begin{cor}
The solution $\mathbf{v}^{p,F}$  of the linear approximative incompressible Navier stokes equation at stage $p\geq 1$ has the representation
\begin{equation}
\begin{array}{ll}
\mathbf{v}^{p,F}(t)=\\
\\
\lim_{k\uparrow \infty}
\left( \exp\left( \left(\delta_{ij}\left(\delta_{\alpha\beta}\nu\sum_{j=1}^n\alpha_j^2 \right)_{\alpha\beta\in {\mathbb Z}^n}  \right)_{1\leq i,j\leq n}\frac{t}{k}\right)
 T\exp\left(\left( E^{p}\left(\frac{ t}{k}\right) \right)\right)\right) ^k \
\end{array}
\end{equation}
\end{cor}
A similar formula holds for the controlled Navier Stokes equation system.
We come back to this formula below in the description of an algorithm.
Next concerning issue b) we first observe that we get a contraction result for the approximations of order $p$ for the Euler equation. We write
\begin{equation}
E^{p-1}(t)=:E(\mathbf{v}^{p-1,E}(t))
\end{equation}
in order to emphasize the dependence of the Euler matrix at time $t\geq 0$ of order $p$ of the approximative Euler velocity $\mathbf{v}^{p-1,E}(t)$ at stage $p-1$.
First we consider the approximative Euler initial data problem of order $p$, i.e., the problem
\begin{equation}
\frac{d\mathbf{v}^{p,E}}{dt}=E(\mathbf{v}^{p-1,E}(t))\mathbf{v}^{p,E},~\mathbf{v}^{p,E}(0)=\mathbf{h}^F
\end{equation}
for data $\mathbf{h}^F$ with components $\mathbf{h}^F_i\in h^s\left({\mathbb Z}^n\right)$ for $1\leq i\leq n$.  
We get
\begin{equation}\label{diffp}
\begin{array}{ll}
\mathbf{v}^{p,E}(t)-\mathbf{v}^{p-1,E}(t)=\int_0^t{\Big (} E(\mathbf{v}^{p-1,E}(s))\delta \mathbf{v}^{p,E}(s)\\
\\
+E(\delta\mathbf{v}^{p-1,E}(s))\mathbf{v}^{p-1,E}(s){\Big )}ds,
\end{array}
\end{equation}
where $\delta \mathbf{v}^{p,E}(s)=
\mathbf{v}^{p,E}(s)-\mathbf{v}^{p-1,E}(s)$.
We denote 
\begin{equation}
E(\delta\mathbf{v}^{p-1,E}(s))=:\left( E(\delta\mathbf{v}^{p-1,E}(s))_{i\alpha j\beta}\right)_{1\leq i,j\leq n,\alpha,\beta\in {\mathbb Z}^n}.
\end{equation}
We have a Lipschitz property for the terms on the right side of (\ref{diffp}), i.e., for a Sobolev exponent $s\geq n+2$ we have for $0\leq t\leq T$ (some horizon $T>0$)
\begin{equation}
\begin{array}{ll}
\sum_{j,\beta}E(\mathbf{v}^{p-1,E}(t))_{i\alpha j\beta}\delta v^{p-1,E}_{j\beta}(t)\\
\\
\leq \sum_{\beta}\frac{C|\beta||\alpha-\beta|}{1+|\alpha-\beta|^{n+2}}\frac{\sup_{0\leq u\leq T}{\big |}\delta \mathbf{v}^{p,E}(u){\big |}_{h^{s}}}{1+|\beta|^{n+2}}\\
\\
\leq \frac{C}{1+|\alpha|^{n+2}}\sup_{0\leq u\leq T}{\big |}\delta \mathbf{v}^{p,E}(u){\big |}_{h^{s}}
\end{array}
\end{equation}
The estimate for the $\alpha$-modes of the second term on the right side of (\ref{diffp}) is similar such that we can sum over $\alpha$ and get for some generic $C>0$ the estimate
\begin{equation}\label{diffp}
\begin{array}{ll}
\sup_{t\in [0,T]}{\Big |} \mathbf{v}^{p,E}(t)-\mathbf{v}^{p-1,E}(t){\Big |}_{h^s}\leq C T\sup_{0\leq u\leq T}{\big |}\delta \mathbf{v}^{p,E}(u){\big |}_{h^{s}}\\
\\
+C T\sup_{0\leq u\leq T}{\big |}\delta \mathbf{v}^{p-1,E}(u){\big |}_{h^{s}}.
\end{array}
\end{equation}
For $T>0$ small enough we have $CT\leq \frac{1}{3}$ we get contraction with a contraction constant $c=\frac{1}{2}$. Hence we have local time contraction, and this implies that we have local existence on the interval $\left[0,\frac{1}{C}\right]$.

Why is it not possible to extend this argument and iterate with respect to time using the semi-group property in order to obtain a global solution of the Euler equation ? Well, it is possible to obtain global solutions of the Euler equation, i.e., for the nonlinear equation
\begin{equation}
\frac{d\mathbf{v}^{E}}{dt}=E(\mathbf{v}^{E}(t))\mathbf{v}^{E},~\mathbf{v}^{E}(0)=\mathbf{h}^F
\end{equation} 
but uniqueness is lost for dimension $n\geq 3$. This seems paradoxical as contraction results lead to uniqueness naturally. However, the problem is that we cannot control a priori the coefficient $\mathbf{v}^{E}(t)$ in the Euler matrix $E(\mathbf{v}^{E}(t))$ globally. We can control it locally in time but not globally. A proof of this is that we may consider the ansatz
\begin{equation}
\mathbf{v}^E(t):=\frac{\mathbf{u}^E(\tau)}{1-t}
\end{equation}
along with 
\begin{equation}
\tau =-\ln\left(1-t \right) 
\end{equation}
for the Euler equation. This leads to
\begin{equation}
 \frac{d \mathbf{v}^E(t)}{dt}=\frac{\mathbf{u}^E(\tau)}{(1-t)^2}
+\frac{\frac{ d \mathbf{u}^E(\tau)}{d\tau}}{(1-t)}\frac{d\tau}{dt},
\end{equation}
where $\frac{d\tau}{dt}=\frac{1}{1-t}$, such that
\begin{equation}\label{eulerinit}
\begin{array}{ll}
\frac{d\mathbf{u}^{E}}{d\tau}=E(\mathbf{u}^{E}(\tau))\mathbf{u}^{E}-\mathbf{u}^{E},\\
\\
\mathbf{u}^{E}(0)=\mathbf{h}^F,
\end{array}
\end{equation} 
where the equation for $t\in [0,1)$ is transformed to an equation for $\tau\in [0,\infty)$. Note that for this ansatz we have used
\begin{equation}
\frac{E(\mathbf{u}^{E}(\tau))\mathbf{u}^{E}(\tau)}{(1-t)^2}=E(\mathbf{v}^{E}(\tau))\mathbf{v}^{E}. \end{equation}
Note that the equation for $\mathbf{u}^E$ has a damping term which helps in order to prove global solutions. Note that any solution of (\ref{eulerinit}) which is well defined on the whole domain (corresponding to the domain $t\in [0,1]$ in original coordinates, and which satisfies
\begin{equation}
\mathbf{u}^{E}(1)\neq 0
\end{equation}
cooresponds to a solution of the Euler equation which becomes singular at time $t=1$. Can such a solution be obtained. It seems difficult to observe this from the equation for $\mathbf{u}^E$ as it stands, but consider a related transformation
\begin{equation}
\mathbf{v}^E(t):=\frac{\mathbf{u}^{\mu,E}(\tau^{\mu})}{\mu-t}
\end{equation}
along with 
\begin{equation}
\tau^{\mu} =-\ln\left(1-\frac{t}{\mu} \right) 
\end{equation}
for the Euler equation. This leads to
\begin{equation}
 \frac{d \mathbf{v}^E(t)}{dt}=\frac{\mathbf{u}^{\mu,E}(\tau^{\mu})}{(\mu-t)^2}
+\frac{\frac{ d \mathbf{u}^{\mu,E}(\tau^{\mu})}{d\tau}}{(\mu-t)}\frac{d\tau^{\mu}}{dt},
\end{equation}
where $\frac{d\tau}{dt}=\frac{1}{1-\frac{t}{\mu}}\frac{1}{\mu}=\frac{1}{\mu-t}$, such that
we get a formally identical equation 
\begin{equation}\label{eulerinit}
\begin{array}{ll}
\frac{d\mathbf{u}^{E}}{d\tau}=E(\mathbf{u}^{\mu,E}(\tau))\mathbf{u}^{\mu,E}-\mathbf{u}^{\mu,E},\\
\\
\mathbf{u}^{\mu,E}(0)=\mathbf{h}^F,
\end{array}
\end{equation} 
but now the equation for $t\in [0,\rho)$ is transformed to an equation for $\tau\in [0,\infty)$. Now if data at time $t=0$ are different from zero, we can always choose $\rho$ small enough such that any regular velocity function solution evaluated at time $t=\rho$ is different from zero at time $t=\rho$, i.e., $\mathbf{v}^E(\rho)\neq 0$ (use a Taylor formula and the equation for the original velocity vector). This shows that singular solutions are generic and proves the Corollary about singular equations. On the other hand we can obtain local time contraction result for the Euler equation in exponentially time weighted norms and repeat this argument. This leads to a global solution branch, i.e., a global solution which is not unique.

Next, let us be a little bit more explicit about the Euler scheme (the lowest time order scheme we propose). In any case, we define a time-discretized subscheme at each stage $m\geq 1$ and use the Trotter product formula for dissipative operators above locally in time in order to show that we get a converging subscheme. The convergence can be based on compactness arguments, and it can be based on global contraction with respect to a time-weighted norm introduced above. In the latter case we need more regularity (stronger polynomial decay) as we shall see. Another possibility is to establish a time-local contraction at each stage $m\geq 1$, i.e. in order to establish a solution $\mathbf{v}^{r,m,F}$ at stage $m$. The time-local contraction can be iterated due to the semi-group property of the operator such that the subscheme becomes global in time. At each iteration step $m$ the time-discretized scheme for each global linear equation works for the uncontrolled scheme and the controlled scheme if we know that the data $\mathbf{v}^{m-1,F}$ (resp. $\mathbf{v}^{r, m-1,F}$) computed at the preceding step are bounded and regular, i.e., $\mathbf{v}^{r,m-1,F}(t)\in h^s\left({\mathbb Z}^n\right)$ for $s\geq n+2$ with $n\geq 3$. Both subschemes have a limit with and without control function. It is in the limit with respect to the stage $m$ that the control function $r$ becomes useful. It also becomes useful for designing algorithms.  Next we write down the subscheme for the uncontrolled subscheme and the controlled subschemes. Later we observe that the control function is globally bounded in order to prove that there is a limit as $m\uparrow \infty$. For each step $m$ of the construction, however, the arguments for the controlled and the uncontrolled scheme are on the same footing. We describe the global linearized scheme, i.e., the scheme based on global linear equations which are equivalent to partial integro-differential equations in the original space. 
The procedure is defined recursively. For $T>0$ arbitrary we define a scheme which converges on $[0,T]$ to the approximative solution $\mathbf{v}^{r,m,F}$. Having defined $\mathbf{v}^{m-1,F}$ at stage $m\geq 1$ we define a series $\mathbf{v}^{r,m,F,p},~p\geq 1$, where $p$ is a natural number index which refers to a global time discretization $t^p_q,~p\geq 1$ and where $q\in \left\lbrace 1,\cdots 2^p\right\rbrace$
along with $t^p_q-t^{p}_{q-1}=T2^{-p}$ for all $p\geq 1$, and $t_0:=0$, such that we get a contractive scheme with respect to a certain time-weighted Sobolev norm.
Next for given $p\geq 1$ and $q\in \left\lbrace 1,\cdots ,2^p\right\rbrace$ we define the equation for $\mathbf{v}^{r,m,F,p,q}$ on the interval $\left[ t^p_{q-1},t^{p}_{q}\right]$, and where the initial data are given by
\begin{equation}
\mathbf{v}^{r,m,F,p,q-1}\left(t^p_{q-1}\right). 
\end{equation}
Here for $q=1$ we have the initial data
\begin{equation}
\mathbf{v}^{r,m,F,p,0}\left(t^p_{0}\right)=\mathbf{h}^{r,F}.
\end{equation}
Note that the vector $\mathbf{h}^{r,F}$ equals the initial data $\mathbf{h}^F$, but without the zero modes, i.e.,
\begin{equation}
\mathbf{h}^{r,F}=\left(h_1^{r,F},\cdots,h^{r,F}_n\right)^T, 
\end{equation}
where for $1\leq i\leq n$ we have
\begin{equation}
h^{r,F}_i=\left(h_{i\alpha} \right)_{\alpha\in {\mathbb Z}^n\setminus \left\lbrace 0 \right\rbrace }.
\end{equation}
Next we define for each $p\geq 1$ the sequence of local equations for $\mathbf{v}^{r,m,F,p,q}$. This leads to a global sequence $\left(\mathbf{v}^{r,m,F,p}\right)_{p\geq 1}$ defined on $[0,T]$ for arbitrary $T>0$. The next goal is to obtain a contraction result with respect to the iteration number $p$ in order to establish the existence of global regular solutions $\mathbf{v}^{r,m,F}$ for the approximative problems
\begin{equation}\label{navode2*rewrm}
\begin{array}{ll}
\frac{d \mathbf{v}^{r,m,F}}{dt}=A^{r,NS}\left(\mathbf{v}^{r,m-1}\right) \mathbf{v}^{r,m,F},
\end{array} 
\end{equation}
where $\mathbf{v}^{r,m,F}=\left(\mathbf{v}^{r,m,F}_1,\cdots ,\mathbf{v}^{r,m,F}_n\right)^T$ and where for each $m\geq 1$ the initial data are given by $\mathbf{h}^{r,F}$. Here we look at $A^{r,NS}$ defined above as an operator such that $A^{r,NS}\left(\mathbf{v}^{r,m-1}\right)$ is obtained by applying this operator to the argument $\mathbf{v}^{r,m-1}$ instead of $\mathbf{v}^{r}$. We shall give the details for the iterations steps $p,q\geq 1$. First we have to obtain $\mathbf{v}^{r,m,F,p,q}$ for each $p\geq 1$ and $q\in \left\lbrace 1,\cdots 2^p\right\rbrace$. On the time interval $\left[ t^p_{q-1},t^{p}_{q}\right]=\left[ \frac{q-1}{2^p}T,\frac{q}{2^p}T\right]$ we have the local Cauchy problem
\begin{equation}\label{navode2*rewrmpq}
\begin{array}{ll}
\frac{d \mathbf{v}^{r,m,F,p,q}}{dt}=A^{r,NS}\left(\mathbf{v}^{r,m-1}\right) \mathbf{v}^{r,m,F,p,q},
\end{array} 
\end{equation}
where $\mathbf{v}^{r,m,F,p,q}=\left(\mathbf{v}^{r,m,F,p,q}_1,\cdots ,\mathbf{v}^{r,m,F,p,q}_n\right)^T$ and where for each $m\geq 1$ the initial data are given by $\mathbf{v}^{r,m,F,p,q-1}\left( t^p_{q-1}\right) $. 
Next we explicitly describe the matrix $A^{r,NS}\left(\mathbf{v}^{r,m-1,p,q}\right) $ which is a $n\left( {\mathbb Z}^n\setminus \left\lbrace 0\right\rbrace  \right) \times n{\mathbb Z}^n\setminus \left\lbrace 0\right\rbrace $-matrix with
\begin{equation}
A^{r,NS}\left(\mathbf{v}^{r,m-1}\right) =\left(A^{r,NS}_{ij}\left(\mathbf{v}^{r,m-1}\right)\right)_{1\leq i,j\leq n} 
\end{equation}
where for $1\leq i,j\leq n$ the entry $A^{r,NS}_{ij}\left(\mathbf{v}^{r,m-1}\right) $ is a ${\mathbb Z}^n\setminus \left\lbrace 0\right\rbrace \times {\mathbb Z}^n\setminus \left\lbrace 0\right\rbrace $-matrix. We have
\begin{equation}
\begin{array}{ll}
A^{r,NS}\left( \mathbf{v}^{r,m-1}\right)\mathbf{v}^{r,m,F,p,q}=\\
\\
\left(\sum_{j=1}^nA^{r,NS}_{1j}\left(\mathbf{v}^{r,m-1}\right) \mathbf{v}^{r,m,F,p,q}_1 ,\cdots,\sum_{j=1}^nA^{r,NS}_{nj}\left( \mathbf{v}^{r,m-1}\right)  \mathbf{v}^{r,m,F,p,q}_n  \right)^T, 
\end{array}
\end{equation}
where for all $1\leq i\leq n$
\begin{equation}
\begin{array}{ll}
\sum_{j=1}^nA^{r,NS}_{ij}\left(\mathbf{v}^{r,m-1}\right) \mathbf{v}^{r,m,F,p,q}_j=\\
\\
\left(\left( \sum_{j=1}^n \sum_{\beta\in {\mathbb Z}^n}A^{r,NS}_{i\alpha j\beta}\left(\mathbf{v}^{r,m-1}\right) v^{r,m,F,p,q}_{j\beta}\right)_{\alpha\in {\mathbb Z}^n} \right)^T_{1\leq i\leq n}. 
\end{array}
\end{equation}
The entries $A^{r,NS}_{i\alpha j\beta}\left(\mathbf{v}^{r,m-1}\right)$  are determined by the equation as follows. On the diagonal, i.e., for $i=j$ and for $\alpha,\beta\neq 0$ we have the entries
\begin{equation}\label{controlmatrix}
\begin{array}{ll}
\delta_{ij}A^{r,NS}_{i\alpha j\beta}\left(\mathbf{v}^{r,m-1}\right)=\delta_{ij}\sum_{j=1}^n \nu\left( -\frac{4\pi \alpha_j^2}{l^2}\right)
-\delta_{ij}\sum_{j=1}^n\frac{2\pi i \beta_j}{l}v^{r,m-1}_{j(\alpha-\beta)}\\
\\+\delta_{ij}2\pi i\alpha_i1_{\left\lbrace \alpha\neq 0\right\rbrace }\frac{\sum_{k=1}^n4\pi \beta_j(\alpha_k-\beta_k)v^{r,m-1}_{k(\alpha-\beta)}}{\sum_{i=1}^n4\pi\alpha_i^2},
\end{array}
\end{equation}
where for $\alpha=\beta$ the terms of the form $v^r_{k(\alpha-\beta)}$ are zero (such that we do not need to exclude these terms explicitly). Furthermore, off-diagonal we have for $i\neq j$ the entries
\begin{equation}
(1-\delta_{ij})A^{r,NS}_{i\alpha j\beta}\left(\mathbf{v}\right)=2\pi i\alpha_i1_{\left\lbrace \alpha\neq 0\right\rbrace }\frac{\sum_{k=1}^n4\pi \beta_j(\alpha_k-\beta_k)v^{r,m-1}_{k(\alpha-\beta)}}{\sum_{i=1}^n4\pi\alpha_i^2}.
\end{equation}

The idea for a global scheme is to determine $\mathbf{v}^{r,F}=\lim_{m\uparrow \infty}\mathbf{v}^{r,m,F}$ for a simple control function (the one which cancels the zero modes at stage $m$ of the iteration) and a certain iteration
\begin{equation}\label{navode2*rewrlin}
\begin{array}{ll}
\frac{d \mathbf{v}^{r,m,F}}{dt}=A^{NS}\left(\mathbf{v}^{r,m-1}\right) \mathbf{v}^{r,m,F},
\end{array} 
\end{equation}
starting with a time-independent matrix $A^{NS}\left(\mathbf{v}^{r,-1}\right) :=A^{NS}\left(\mathbf{h}\right)$ for $m=0$. We shall use the abbreviation
\begin{equation}\label{navode2*rewrlin}
\begin{array}{ll}
A^r_m:=A^{NS}\left(\mathbf{v}^{r,m-1}\right),
\end{array} 
\end{equation} 
which is time-dependent for $m\geq 1$.
 The proof of global regular existence of the incompressible Navier Stokes equation can be obtained the sequence $\left( \mathbf{v}^{r,m,F}\right)_{m\geq 1}$ 
is a Cauchy sequence in $C^1\left((0,T),\left(h^s\left({\mathbb Z}^n\right),\cdots ,h^s\left({\mathbb Z}^n\right) \right)\right) $ with the natural norm (supremum with respect to time). 
An alternative is a contraction result. Contraction results have the advantage that they lead to uniqueness results with respect to the related Banach space. The disadvantage is that we need more regularity (order of polynomial decay) in general if we want to have contraction. 
However, in a first step we have to ensure the existence of the solutions $\mathbf{v}^{r,m,F}$ of the linearized problems in appropriate Banach spaces.
 Next we define a Banach space in order to get a contraction result for the subscheme $\left( \mathbf{v}^{r,m,F,p}\right)_{p\geq 1}$ where the index $p$ is related to the time discretization of size $T\frac{1}{2^p}$ we mentioned above. Note that for each $p\geq 1$ the problem for $\mathbf{v}^{r,m,F,p}$ on $[0,T]$ is defined by $2^p$ recursively defined subproblems for $\mathbf{v}^{r,m,F,p,q}$ for $1\leq q\leq 2^p$ which are defined on the interval $\left[t^p_{q-1},t^p_q\right]$ where the data for the problem for $q\geq 2$ are defined by the final data of the subproblem for $\mathbf{v}^{r,m,F,p-q-1}$ evaluated at $t^p_{q-1}$.

%
Next consider the function space
\begin{equation}
\begin{array}{ll}
B^{n,s}:=\left\lbrace t\rightarrow \mathbf{u}^F(t)=\left(\mathbf{u}^F_1,\cdots ,\mathbf{u}^F_n\right)^T|\forall t\geq 0:~\mathbf{u}^F_i(t)\in h^s\left({\mathbb Z}^n\setminus \left\lbrace 0\right\rbrace \right) \right\rbrace  
\end{array}
\end{equation}
Note that we excluded the zero modes because this Banach space is designed for the controlled equations.
For $\mathbf{u}^F\in B^{n,s}$ define
\begin{equation}
{\Big|}\mathbf{u}^F
{\Big |}^{T,\mbox{exp}}_{h^s,C}:=
 \sup_{t\in \left[0,T\right] }\sum_{i=1}^n\exp\left( -Ct\right) {\Big |}\mathbf{u}^F_i(t){\Big |}_{h^s}.
\end{equation}
In the following we abbreviate
\begin{equation}\label{ans}
A^r_m=A^{NS}\left( \mathbf{v}^{r,m-1,F}\right).
\end{equation}
Especially the evaluation of the right side of (\ref{ans}) at time $t$ is denoted by $A^r_m(t)$. Recall that
\begin{equation}
{\mathbb Z}^{n,0}:={\mathbb Z}^n\setminus \left\lbrace 0\right\rbrace 
\end{equation}
At each substep $p$ we apply the Trotter product representation of the subproblems for $\mathbf{v}^{r,m,F,p,q}$ for $1\leq q\leq 2^p$  $2^p$-times. Note that at each stage of the construction we know that the Trotter product formula is valid in regular function spaces as the coefficients are not time dependent for each of these $2^p$ subproblems.  
For $s>0$ let us assume that
\begin{equation}\label{mminus1ass}
t\rightarrow \mathbf{v}^{r,m-1,F}_i(t) \in C^k\left([0,T],h^{s}\left({\mathbb Z}^{n,0}\right)\right)~\mbox{ for }1\leq i\leq n. 
\end{equation}
Here $C^k\left([0,T],h^{s}\left( {\mathbb Z}^{n,0}\right)\right) $ 
is the function space of time dependent function vectors with image in $ {\mathbb Z}^{n,0}$ which have $k$-times differentiable component functions $t\rightarrow v^{r,m-1}_{i\alpha}(t)$ for 
$1\leq i\leq n$ and $\alpha \in {\mathbb Z}^n\setminus \left\lbrace 0\right\rbrace $. Note that there is no need here to be very restrictive with respect to the degree of the Sobolev norm $s$. We should have $s>n$ in order to prove the existence of anything which exists in a more than distributional sense, of course, and we stick to our assumption $s>n+2$. 
Assuming sufficient regularity at the previous stage $m-1$ of the function $\mathbf{v}^{r,m-1,F}$, we have the Taylor formula 
\begin{equation}
\begin{array}{ll}
\mathbf{v}^{r,m-1,F}(t+h)=\sum_{0\leq p\leq k-1}D^p_t\mathbf{v}^{r,m-1,F}(t)h^p\\
\\
+\frac{h^p}{(k-1)!}\int_0^1(1-\theta)^{k-1}D^k_t\mathbf{v}^{r,m-1,F}(t+\theta h)d\theta 
\end{array}
\end{equation}
for $t\in [0,T-h]$ and $h>0$. Here,
\begin{equation}
D^p_t\mathbf{v}^{r,m-1,F}(t):=\left(\frac{d^p}{dt^p}v^{r,m-1}_{i\alpha}\right)^T_{1\leq i\leq n,\alpha\in {\mathbb Z}^{n,0}}.
\end{equation}
In the following we assume that that $\mathbf{h}^{r,F}_i,~\mathbf{v}^{r,m-1,F}_i\in h^s\left( {\mathbb Z}^{n,0}\right)$ for $s= n+2$ and $1\leq i\leq n$, because this property is inherited for $\mathbf{v}^{r,m,F}_i$ as we go from stage $m-1$ to stage $m$. Note that the matrix-valued function  $t\rightarrow A^r_m(t+h)-A^r_m(t)=A^{r,NS}\left( \mathbf{v}^{r,m-1,F}(t+h)\right)-A^{r,NS}\left( \mathbf{v}^{r,m-1,F}(t)\right)$ applied to the data $\mathbf{h}^{r,F}$ is Lipschitz with respect to a $|.|_{h^{s}}$-norm, i.e., for some finite Lipschitz constant $L>0$ we have
\begin{equation}
\begin{array}{ll}
{\Big|}\left( A^{r,NS}\left( \mathbf{v}^{r,m-1,F}(t+h)\right)-A^{r,NS}\left( \mathbf{v}^{r,m-1,F}(t)\right) \right) \mathbf{h}^{r,F}{\Big |}
_{h^{s}}\\
\\
\leq L |\mathbf{v}^{r,m-1,F}(t+h)-\mathbf{v}^{r,m-1,F}(t)|_{h^{s}}.
\end{array}
\end{equation}
Here, we use the regularity (polynomial decay) of the initial data $\mathbf{h}^{r,F}$ and the assumed regularity (polynomial decay) of $\mathbf{v}^{r,m-1,F}(t)$ in order to compensate have to compensate for quadratic terms  in the Leray projection term approximation above, i.e., the quadratic multiindex terms related to the multiindices $\beta$ in the Leray projection term and to linear multiindex terms related to the convection term. Here the initial data for the subproblems for $\mathbf{v}^{r,m,F,p,q}$ inherit sufficient regularity in order to preserve the Lipschitz property. Let us have a closer look at this.
At stage $m$ of the global iteration we solve an equation of the form (\ref{navode2*rewrm}). Let us describe this for a unit time interval $[0,1]$, i.e., with time horizon $T=1$ for a moment without loss of generality. The generalization of the description  for any finite time horizon $T>0$ is straightforward (we do this below). We do this by a series of time discretizations at time points
 \begin{equation}
 t_i\in \left\lbrace \frac{k}{2^p}=t^p_k{\big |}0\leq k\leq 2^{p}-1\right\rbrace 
 \end{equation}
in order to a apply Trotter product formulas at each substep. The approximation at stage $p$ is denoted by $\mathbf{v}^{r,m,F,p}$ and is determined recursively by $2^p$ time-homogeneous problems on time intervals $\left[ t^p_{k},t^p_{k+1}\right] $. At substep $1\leq q\leq 2^p$ we have computed $\mathbf{v}^{r,m,F,p}(t^p_{q-1})$. We then evaluate the matrix in (\ref{navode2*rewrm}) at time $t^p_{q-1}$ and have the problem
\begin{equation}\label{navode2*rewrm2}
\begin{array}{ll}
\frac{d \mathbf{v}^{r,m,F,p,q}}{dt}=A^{r}_m\left(t^p_{q-1}\right) \mathbf{v}^{r,m,F,p,q},
\end{array} 
\end{equation}
on $\left[ t^p_{q-1},t^p_{q}\right] $, where we take the data from the previous time substep $t^p_{q-1}$ at stage $m$, i.e.,
\begin{equation}\label{navode2*rewrm2data}
\mathbf{v}^{r,m,F,p,q}(t^p_{q-1})=\mathbf{v}^{r,m,F,p,q-1}(t^p_{q-1}).
\end{equation}
The problem described in (\ref{navode2*rewrm2}) and (\ref{navode2*rewrm2data}) can then be solved by a Trotter product formula. Now compare this with with the set of problems at the next stage $p+1$. On the time interval $\left[t^{p}_{q-1},t^p_q\right]$ we have two subproblems at stage $p+1$. Note that $t^{p+1}_{2(q-1)}=t^p_{q-1}$. On the time interval $\left[t^{p}_{q-1},t^{p+1}_{2q-1}\right]$ we have to solve for
\begin{equation}\label{navode2*rewrm21}
\begin{array}{ll}
\frac{d \mathbf{v}^{r,m,F,p+1,2q-1}}{dt}=A^{r}_m\left(t^p_{q-1}\right) \mathbf{v}^{r,m,F,p+1,2q-1},
\end{array} 
\end{equation}
with the initial data from the previous time step, i.e., 
\begin{equation}\label{navode2*rewrm2data21}
\mathbf{v}^{r,m,F,p+1,2q-1}(t^{p+1}_{2(q-1)})=\mathbf{v}^{r,m,F,p+1,2q-2}(t^{p+1}_{2(q-1)}),
\end{equation}
and where we use
\begin{equation}
A^{r}_m\left(t^p_{q-1}\right)=A^{r}_m\left(t^{p+1}_{2(q-1)}\right).
\end{equation}
We then have a second subproblem on the time interval $\left[t^{p+1}_{2q-1},t^{p+1}_{2q}\right]$, where we have to solve for
\begin{equation}\label{navode2*rewrm213}
\begin{array}{ll}
\frac{d \mathbf{v}^{r,m,F,p+1,2q}}{dt}=A^{r}_m\left(t^{p+1}_{2q-1}\right) \mathbf{v}^{r,m,F,p+1,2q},
\end{array} 
\end{equation}
with the initial data 
\begin{equation}\label{navode2*rewrm2data213}
\mathbf{v}^{r,m,F,p+1,2q}(t^{p+1}_{2(q-1)}).
\end{equation}
The regular spaces with polynomial decaying modes makes it possible to generalize observation well known for Euler scheme for finite dimensional systems quite straightforwardly. Especially, a global $O(h)$ (time stepsize $h$) error is a straightforward consequence of the Taylor formula considered above. This may also be used to estimate the difference of solutions $\mathbf{v}^{r,m,F,p+1,2q-1}$ together with $\mathbf{v}^{r,m,F,p+1,2q-2}$ compared to  $\mathbf{v}^{r,m,F,p,q}$.  
Indeed using the infinite linear algebra lemmas above we observe
\begin{lem}
Let $s>n+ 2$ and $T>0$ be given. For $k=2$ assume that $\mathbf{v}^{r,m-1,F}$ is regular as in (\ref{mminus1ass}).
Then for some finite $C>0$
\begin{equation}
\begin{array}{ll}
\sup_{u\in [0,T]}\exp(-Cu){\Big|}\mathbf{v}^{r,m,F,p}(u)-\mathbf{v}^{r,m,F,p-1}(u){\Big |}_{h^s}
\leq \frac{L}{2^{p-1}}.
\end{array}
\end{equation}
\end{lem}
\begin{proof}
For notational reasons the size of the torus is assumed to be one. We remark that for general time horizon $T>0$
\begin{equation}
t^p_{2(q-1)}=2(q-1)2^{-p}T=(q-1)2^{-(p-1)}T=t^{p-1}_{q-1},
\end{equation}
and accordingly
\begin{equation}\label{matdiff}
A^r_{m}(t^p_{2(q-1)})=A^r_{m}(t^{p-1}_{q-1})
\end{equation}
For some vector $\mathbf{C}^{h,p}_{q-1}\in h^s\left({\mathbb Z}^n\right)$ for $s\geq n+2$ we postulate the difference
\begin{equation}\label{chp}
 \mathbf{v}^{r,m,F,p,q-1}\left( t^p_{2(q-1)}\right)- \mathbf{v}^{r,m,F,p-1,q-1}\left( t^{p-1}_{q-1}\right)=\mathbf{C}^{h,p}_{q-1},
\end{equation}
and we consider some properties which the vector $\mathbf{C}^{h,p}_{q}$ inherits from $\mathbf{C}^{h,p}_{q-1}$. 
Next at stage $p\geq 1$ consider the initial data $\mathbf{v}^{r,m,F}\left( t^{p-1}_{q-1}\right) $ of the problem at substep $1\leq q\leq 2^{p-1}$.
We have $t\in \left[t^p_{2(q-1)},t^{p}_{2q}\right]= \left[t^{p-1}_{q-1},t^{p-1}_q\right]$, where it makes sense to consider the the subintervals $\left[t^p_{2(q-1)},t^{p}_{2q-1}\right]$ and $\left[t^p_{2q-1},t^{p}_{2q}\right]$. We may consider $t\in \left[t^p_{2q-1},t^{p}_{2q}\right]$ w.l.o.g. because the following estimate simplifies $t\in \left[t^p_{2q-2},t^{p}_{2q-1}\right]$. For $t\in \left[t^p_{2q-1},t^{p}_{2q}\right]$ we have 
\begin{equation}\label{solm11}
\begin{array}{ll}
{\Big |}\mathbf{v}^{r,m,F,p,2q}_i(t)-\mathbf{v}^{r,m,F,p-1,q}_i(t){\Big |}_{h^s}\\
\\
\leq {\Big |}\mathbf{v}^{r,m,F,p,2q}_i(t)-\mathbf{v}^{r,m,F,p,2q}_i(t^p_{2q-1})-\left( \mathbf{v}^{r,m,F,p-1,q}_i(t)-\mathbf{v}^{r,m,F,p-1,q}_i(t^p_{2q-1})\right) {\Big |}_{h^s}\\
\\
+{\Big |}\mathbf{v}^{r,m,F,p,2q-1}_i(t^p_{2q-1})-\mathbf{v}^{r,m,F,p-1,q}_i(t^{p}_{2q-1}){\Big |}_{h^s},
\end{array}
\end{equation}
where we use $\mathbf{v}^{r,m,F,p,2q}_i(t^p_{2q-1})=\mathbf{v}^{r,m,F,p,2q-1}_i(t^p_{2q-1})$.
Since $t^p_{2(q-1)}=t^{p-1}_{q-1}$ and with (\ref{chp}) above for the last term in (\ref{solm11}) we have
\begin{equation}\label{solm11}
\begin{array}{ll}
{\Big |}\mathbf{v}^{r,m,F,p,2q-1}_i(t^p_{2q-1})-\mathbf{v}^{r,m,F,p-1,q}_i(t^{p}_{2q-1}){\Big |}_{h^s}\\
\\
={\Big |}\left( \exp\left(A^r_m(t^p_{2(q-1)})\left( t^p_{2q-1}-t^p_{2(q-1)}\right) \right)\mathbf{v}^{r,m,F,p,2q}(t^p_{2(q-1)})\right)_i
\\
\\
-\left( \exp\left(A^r_m(t^{p-1}_{q-1})\left( t^p_{2q-1}-t^p_{2(q-1)}\right)\right)\mathbf{v}^{r,m-1,F,p-1,q}(t^{p-1}_{q-1})\right)_i{\Big |}_{h^s}\\
\\
= {\Big |}\exp\left(A^r_m(t^p_{2(q-1)})\left( t^p_{2q-1}-t^p_{2(q-1)}\right)\right)\mathbf{C}^{h,p}_{q-1}{\Big |}_{h^s}\\
\\
\leq
{\Big |}\exp\left(\frac{C}{4^p}\right)\mathbf{C}^{h,p}_{q-1}{\Big |}_{h^s}
\end{array}
\end{equation}
for some finite $C>0$, which depends only on data known at stage $m-1$. Furthermore, for the first term on the right side of (\ref{solm11}) we may use the rough estimate
\begin{equation}\label{solm1122}
\begin{array}{ll}
{\Big |}\mathbf{v}^{r,m,F,p,2q}_i(t)-\mathbf{v}^{r,m,F,p,2q}_i(t^p_{2q-1})-\left( \mathbf{v}^{r,m,F,p-1,q}_i(t)-\mathbf{v}^{r,m,F,p-1,q}_i(t^p_{2q-1})\right) {\Big |}_{h^s}\\
\\
\leq {\Big |}\exp\left(A^r_m(t^p_{2q-1})\left( t-t^p_{2q-1}\right) \right)\mathbf{v}^{r,m,F,p,2q}_i(t^p_{2q-1})-\mathbf{v}^{r,m,F,p,2q}_i(t^p_{2q-1}) \\
\\
-\left(  \exp\left(A^r_m(t^{p}_{2q-1})
\left( t-t^p_{2q-1}\right)\right) 
\mathbf{v}^{r,m,F,p-1,q}_i(t^p_{2q-1})
-\mathbf{v}^{r,m,F,p-1,q}_i(t^p_{2q-1})\right) {\Big |}_{h^s}.
\end{array}
\end{equation}
The right side of (\ref{solm1122}) we observe with (\ref{solm11})
\begin{equation}\label{solm112233}
\begin{array}{ll}
{\Big |}\exp\left(A^r_m(t^p_{2q-1})\left( t-t^p_{2q-1}\right) \right)\mathbf{v}^{r,m,F,p,2q}_i(t^p_{2q-1})-\mathbf{v}^{r,m,F,p-1,q}_i(t^p_{2q-1}) \\
\\
+\mathbf{v}^{r,m,F,p,2q}_i(t^p_{2q-1})-\mathbf{v}^{r,m,F,p-1,q}_i(t^p_{2q-1}) {\Big |}_{h^s}\\
\\
\leq 2{\Big |}\exp\left(\frac{2C}{4^p}\right)\mathbf{C}^{h,p}_{q-1}{\Big |}_{h^s}
\end{array}
\end{equation}
As for each $p\geq 1$ the entries of the sequence $\mathbf{C}^{h,p}_{q}$ are in $O\left(h^2 \right)$ where $h$ denotes the maximal time step size (which is $2^{-p}$ with our choice) the difference to be estimated is in $O(h)$ and we are done. 

\end{proof}
The preceding lemma shows that we have a Cauchy sequence
\begin{equation}
\left( \mathbf{v}^{r,m,F,p}\right)_{p\geq 1}
\end{equation}
with respect to a regular (time weighted) norm and with a limit $\mathbf{v}^{r,m,F}$ with
\begin{equation}
\mathbf{v}^{r,m,F}_i(t)\in h^s\left({\mathbb Z}^{n,0}\right) 
\end{equation}
for all $t\in [0,T]$. Since we have a dissipative term (damping exponential) in the Trotter product formula, similar observations can be made for the time derivative sequence
\begin{equation}
\left( \frac{d}{dt}\mathbf{v}^{r,m,F,p}\right)_{p\geq 1}.
\end{equation}
Note, however, that the order of regularity $s\geq n+2\geq 5$ can be chosen to be as large as we want, and this can be exploited in order to prove regularity with respect to time $t$ for each $\mathbf{v}^{r,m,F}$ via the defining equation of the latter function. We even do not need estimates for products of functions in Sobolev spaces which may be borrowed from classical Sobolev space analysis. Using the lemma above for large $s>0$  we may instead use the regularity implied by infinite matrix products as pointed out above.  
As a consequence of the preceding lemma we note
\begin{lem} For all $m\geq 1$ and $s>n+ 2\geq 5$ the function
\begin{equation}\label{solm}
\left( \mathbf{v}^{r,m,F}\right) _{i}=\left( T\exp\left(A^r_mt\right)\mathbf{h}^{r,F}\right)_i\in h^s_l\left({\mathbb Z}^{n,0}\right), 
\end{equation}
is well-defined, whenever $\mathbf{h}^{F}_i\in h^s_l\left({\mathbb Z}^n\right)$.
\end{lem}
The same holds for the uncontrolled approximations, of course. We note
\begin{cor} For all $m\geq 1$ and $s>n+2$ the function
\begin{equation}\label{solm}
\left( \mathbf{v}^{m,F}\right) _{i}=\left( T\exp\left(A_mt\right)\mathbf{h}^{F}\right)_i\in h^s_l\left({\mathbb Z}^n\right)
\end{equation}
is well-defined, whenever $\mathbf{h}^{F}_i\in h^s_l\left({\mathbb Z}^n\right)$.
\end{cor}
However, it is essential to get a uniformly bounded sequence
\begin{equation}\label{sol0}
\left( \mathbf{v}^{r,m,F}\right)_{m\in {\mathbb N}}=
\left( T\exp\left(A^r_mt\right)\mathbf{h}^{F}_i\in h^s_l\left({\mathbb Z}^n\right)\right)_{m\in {\mathbb N}}.
\end{equation}
for some $\nu>0$ (some $\nu$ is sufficient butt we get the bounded sequence for all $\nu$ which is useful for rather straightforward generalised models with spatially dependent viscosity) . We remarked in the introduction that we can even choose $\nu>0$. Indeed, we observed that we can choose it arbitrarily large (as is also well-known). However this was not needed so far and we shall not need it later on. It is just an useful observation in order check algorithms in the most simple situation via equivalent formulations with rigorous damping. At this point it is useful to consider the controlled sequence $\left( \mathbf{v}^{r,m,F}_i\right)_{m\in {\mathbb N},~1\leq i\leq n}$. Recall that the functions $\mathbf{v}^{r,m,F}$ are designed such that the zero modes are zero $v^{r,m}_{i0}=0$. The control function is just defined this way. Next we show that a uniformly bounded controlled sequence $\left( \mathbf{v}^{r,m,F}_i\right)_{m\in {\mathbb N},~1\leq i\leq n}$ implies uniformly boundedness of the uncontrolled sequence $\left( \mathbf{v}^{m,F}_i\right)_{m\in {\mathbb N},~1\leq i\leq n}$. In order to observe this we go back to (\ref{navode200a}).
At stage $m\geq 1$ it is assumed that $v^{r,m-1}_{i0}=0$. The controlled approximating equation at stage $m$ is obtained from (\ref{navode200a}) by elimination of the zero modes. We have for $1\leq i\leq n$ and $\alpha\neq 0$ the equation
\begin{equation}\label{navode200acontr}
\begin{array}{ll}
\frac{d v^{r,m}_{i\alpha}}{dt}=\sum_{j=1}^n\nu \left( -\frac{4\pi \alpha_j^2}{l^2}\right)v^{r,m}_{i\alpha}
-\sum_{j=1}^n\sum_{\gamma \in {\mathbb Z}^{n}\setminus \left\lbrace  0,\alpha\right\rbrace }\frac{2\pi i \gamma_j}{l}v^{r,m-1}_{j(\alpha-\gamma)}v^{r,m}_{i\gamma}\\
\\
+2\pi i\alpha_i1_{\left\lbrace \alpha\neq 0\right\rbrace}\frac{\sum_{j,k=1}^n\sum_{\gamma\in {\mathbb Z}^n\setminus \left\lbrace  0,\alpha\right\rbrace}4\pi \gamma_j(\alpha_k-\gamma_k)v^{r,m-1}_{j\gamma}v^{r,m}_{k(\alpha-\gamma)}}{\sum_{i=1}^n4\pi\alpha_i^2}.
\end{array} 
\end{equation}
We considered the controlled equation systems as autonomous systems of non-zero modes. However, in order to compare the controlled system with the original one we may define
\begin{equation}\label{zeromcon}
\begin{array}{ll}
\frac{d v^{r,m}_{i0}}{dt}=
-\sum_{j=1}^n\sum_{\gamma \in {\mathbb Z}^{n}\setminus \left\lbrace  0,\alpha\right\rbrace }\frac{2\pi i \gamma_j}{l}v^{r,m-1}_{j(-\gamma)}v^{r,m}_{i\gamma}.
\end{array}
\end{equation} 
Note that $\mathbf{v}^{r,m-1,F}_i,\mathbf{v}^{r,m,F}_i\in h^s\left({\mathbb Z}^{n,0}\right) $ for $s>n\geq 3$ implies that the right side of (\ref{zeromcon}) is bounded by constant finite $C>0$.
Hence we have
\begin{lem}\label{control}
If the sequence $\left( \mathbf{v}^{r,m,F}\right)_{m\geq 1}$ has an upper bound $c>0$ with respect to the $|.|_s=\sum_{i=1}^n|.|_{h^s}$-norm, we have for some finite $C>0$
\begin{equation}
|r_0(t)|\leq \exp(Ct)
\end{equation}
for all $t\geq 0$.
\end{lem}

It remains to show that for some finite  $C>0$ such that for all $1\leq i\leq n$ and $m\geq 0$ we have
\begin{equation}\label{est0vmF}
|\mathbf{v}^{r,m,F}_{i}(t,.)|_{h^s_l}\leq C
\end{equation}

Based on the arguments so far there are several ways to get an uniform bound for the sequence $\left( \mathbf{v}^{r,m,F}\right)_{m\geq 1}$. One possibility is to observe that we have upper bounds
\begin{equation}\label{polynomialgrowth}
\sup_{t\in [0,T]}{\big |}\mathbf{v}^{r,m,F}_i(t){\big |}_{h^s}\leq C^m
\end{equation}
for some $s>n$ and all $1\leq i\leq n$. This implies that we we have contraction on the interval $[0,T]$ with respect to the norm ${\big |}.{\big |}^{\exp,T}_{h^s,C}$ (with appropriate generic constant $C>0$). The result in (\ref{polynomialgrowth}) is obtained by time discretization of (\ref{navode2*rewf}). Then approximations $\left( \mathbf{v}^{r,m,F,p,q}(.)\right)_{1\leq q\leq 2^p } $ as in the construction of the solutions $\mathbf{v}^{r,m,F}(t)$ above can be considered. If for $1\leq i\leq n$ $\mathbf{v}^{r,m,F,p}_i(.):[0,T]\rightarrow h^s({\mathbb Z}^n)$ denotes the function which equals the function $\mathbf{v}^{r,m,F,p,q}_i(.)$ on the intervals $\left[ t^p_q,t^p_{q+1}\right]$ then we get with the 
\begin{equation}\label{polynomialgrowth2}
\sup_{t\in [0,T]}{\big |}\mathbf{v}^{r,m,F,p}_i(t){\big |}_{h^s}\leq C^m
\end{equation}
for a constant $C>0$ independent of the stage $p$ such that (\ref{polynomialgrowth}) is satisfied.

 Another way is via contraction results on certain balls in appropriate function spaces. The radius of such a ball clearly depends on the size of the initial data and on the size of the horizon. However it is sufficient that for each dual Sobolev norm index $s>0$ and for each data size ${\big |}\mathbf{h}^{F}{\big |}_s$ and each horizon size $T>0$ we find a contraction result on an appropriate ball for a related time weighted function space.  Let's look at the details. For arbitrary $T>0$ consider two smooth vector-valued functions on the $n$-torus, i.e., functions of the form 
\begin{equation}
\mathbf{f},\mathbf{g}\in 
\left[ C^{\infty}\left( [0,T],{\mathbb T}^n\right)\right] ^n.
\end{equation}
Consider the equations
\begin{equation}\label{navode2*rewf}
\begin{array}{ll}
\frac{d \mathbf{v}^{r,f,F}}{dt}=A^{r,NS}\left(\mathbf{f}\right) \mathbf{v}^{r,f,F},
\end{array} 
\end{equation}
along with $\mathbf{v}^{r,f,F}(0)=\mathbf{h}^{r,F}$, and
\begin{equation}\label{navode2*rewg}
\begin{array}{ll}
\frac{d \mathbf{v}^{r,g,F}}{dt}=A^{r,NS}\left(\mathbf{g}\right) \mathbf{v}^{r,g,F},
\end{array} 
\end{equation}
along with $\mathbf{v}^{r,g,F}(0)=\mathbf{h}^{r,F}$.
Here we denote $\mathbf{v}^{r,f,F}=\left(\mathbf{v}^{r,f,F}_1,\cdots ,\mathbf{v}^{r,f,F}_n\right)^T$ and similarly for the function $\mathbf{v}^{r,g,F}$. As in or notation above the matrix $A^{r,NS}\left(\mathbf{f}\right) $ is a $n{\mathbb Z}^n_0\times n{\mathbb Z}^n_0$-matrix, where we abbreviate ${\mathbb Z}^n_0={\mathbb Z} \setminus \left\lbrace 0\right\rbrace$, and where 
\begin{equation}
A^{r,NS}\left(\mathbf{f}\right) =\left(A^{r,NS}_{ij}\left(\mathbf{f}\right)\right)_{1\leq i,j\leq n} 
\end{equation}
where for $1\leq i,j\leq n$ the entry $A^{r,NS}_{ij}\left(\mathbf{f}\right) $ is a ${\mathbb Z}^n\times {\mathbb Z}^n$-matrix. We define
\begin{equation}
A^{r,NS}\left( \mathbf{f}\right)\mathbf{v}^{r,f,F} =\left(\sum_{j=1}^nA^{r,NS}_{1j}\left(\mathbf{f}\right) \mathbf{v}^{r,f,F}_1 ,\cdots,\sum_{j=1}^nA^{r,NS}_{nj}\left( \mathbf{f}\right)  \mathbf{v}^{r,f,F}_n  \right)^T, 
\end{equation}
where for all $1\leq i\leq n$
\begin{equation}
\sum_{j=1}^nA^{r,NS}_{ij}\left(\mathbf{f}\right) \mathbf{v}^{r,f,F}_j=\left(\left( \sum_{j=1}^n \sum_{\beta\in {\mathbb Z}^n}A^{r,NS}_{i\alpha j\beta}\left(\mathbf{f}\right) v^{r,f,F}_{j\beta}\right)_{\alpha\in {\mathbb Z}^n} \right)^T_{1\leq i\leq n}. 
\end{equation}
The entries $A^{r,NS}_{i\alpha j\beta}\left(\mathbf{f}\right)$ of $A^{r,NS}\left( \mathbf{v}\right)$ are determined as follows. On the diagonal, i.e., for $i=j$ we have the entries  for $\alpha,\beta\neq 0$
\begin{equation}
\begin{array}{ll}
\delta_{ij}A^{r,NS}_{i\alpha j\beta}\left(\mathbf{f}\right)=\delta_{ij}\sum_{j=1}^n \nu\left( -\frac{4\pi \alpha_j^2}{l^2}\right)
-\delta_{ij}\sum_{j=1}^n\frac{2\pi i \beta_j}{l}f_{j(\alpha-\beta)}\\
\\+\delta_{ij}2\pi i\alpha_i1_{\left\lbrace \alpha\neq 0\right\rbrace }\frac{\sum_{k=1}^n4\pi \beta_j(\alpha_k-\beta_k)f_{k(\alpha-\beta)}}{\sum_{i=1}^n4\pi\alpha_i^2},
\end{array}
\end{equation}
where for $\alpha=\beta$ the terms of the form $f_{k(\alpha-\beta)}$ are zero (such that we do not need to exclude these terms explicitly). Furthermore, off-diagonal we have for $i\neq j$ the entries
\begin{equation}
(1-\delta_{ij})A^{r,NS}_{i\alpha j\beta}\left(\mathbf{f}\right)=2\pi i\alpha_i1_{\left\lbrace \alpha\neq 0\right\rbrace }\frac{\sum_{k=1}^n4\pi \beta_j(\alpha_k-\beta_k)f_{k(\alpha-\beta)}}{\sum_{i=1}^n4\pi\alpha_i^2}.
\end{equation}
The definition of $A^{r,NS}\left(\mathbf{g}\right) $ is analogous. Next for functions $$\mathbf{u}^{r,F}=\left(\mathbf{u}^{r,F}_1,\cdots ,\mathbf{u}^{r,F}_n\right)$$ and for $s\geq n+2$ consider the norm
\begin{equation}
\begin{array}{ll}
{\big |}\mathbf{u}^{r,F}{\big |}^{T,\exp}_{s,C}:=\sum_{i=1}^n{\big |}\mathbf{u}^{r,F}_i{\big |}^{T,\exp}_{h^s,C}.
\end{array}
\end{equation}
Consider a ball of radius $2{\big |}\mathbf{h}^{r,F}{\big |}^{T,\exp}_{s,C}$ around the origin, i.e., consider the ball
\begin{equation}
B_{2{\big |}\mathbf{h}^{r,F}{\big |}^{T,\exp}_{s,C}}:=\left\lbrace \mathbf{u}^{r,F}{\big |}{\big |}\mathbf{u}^{r,F}{\big |}^{T,\exp}_{s,C}\leq 2{\big |}\mathbf{h}^{r,F}{\big |}^{T,\exp}_{s,C}\right\rbrace .
\end{equation}
For $s\geq n+2$ and for data $\mathbf{h}^{r,F}\in h^s\left({\mathbb Z}^n\setminus \left\lbrace 0\right\rbrace  \right) $ the considerations above shows that the Cauchy problem
\begin{equation}
\begin{array}{ll}
\frac{d \mathbf{v}^{r,f,F}}{dt}=A^{r,NS}\left(\mathbf{f}\right) \mathbf{v}^{r,f,F},
\end{array} 
\end{equation}
along with $\mathbf{v}^{r,f,F}(0)=\mathbf{h}^{r,F}$ has a regular solution $\mathbf{v}^{r,f,F}$ with $\mathbf{v}^{r,f,F}(t)\in h^s\left({\mathbb Z}^n\setminus \left\lbrace 0\right\rbrace  \right) $. 
Next we observe that for fixed $\mathbf{u}^{r,F}$ in this ball the linear operator
\begin{equation}
\left( \mathbf{f}-\mathbf{g}\right) \rightarrow \left( A^{r,NS}_{i\alpha j\beta}\left(\mathbf{f}\right)\right) \mathbf{u}^{r,F}- \left( A^{r,NS}_{i\alpha j\beta}\left(\mathbf{g}\right)\right) \mathbf{u}^{r,F}
\end{equation}
is Lipschitz with some Lipschitz constant $L$. Note that  we have
\begin{equation}
\begin{array}{ll}
A^{r,NS}_{i\alpha j\beta}\left(\mathbf{f}\right)-A^{r,NS}_{i\alpha j\beta}\left(\mathbf{g}\right)\\
\\
=-\delta_{ij}\sum_{j=1}^n\frac{2\pi i \beta_j}{l}\left( f_{j(\alpha-\beta)}-g_{j(\alpha-\beta)}\right) \\
\\+2\pi i\alpha_i1_{\left\lbrace \alpha\neq 0\right\rbrace }\frac{\sum_{k=1}^n4\pi \beta_j(\alpha_k-\beta_k)\left( f_{k(\alpha-\beta)}-g_{k(\alpha-\beta)}\right) }{\sum_{i=1}^n4\pi\alpha_i^2}.
\end{array}
\end{equation}
Furthermore, for $\mathbf{g}^{r,F}\in B_{2{\big |}\mathbf{h}^{r,F}{\big |}^{T,\exp}_{s,C}}$  we may assume w.l.o.g. that the Lipschitz constant $L$ is chosen such that
\begin{equation}\label{agu}
\begin{array}{ll}
\sup_{\mathbf{g}^{r,F}\in B_{2{\big |}\mathbf{h}^{r,F}{\big |}^{T,\exp}_{s,C}}}
{\big |}\left( \delta_{ij}A^{r,NS}_{i\alpha j\beta}\left(\mathbf{g}\right)\right)\mathbf{u}^{r,F}{\big |}^{T,\exp}_{h^{s},C}\leq L{\Big |}\mathbf{u}^{r,F}{\big |}^{T,\exp}_{h^{s-2},C}
\end{array}
\end{equation}
where the weaker norm on the right side of (\ref{agu}) is due to the fact that we have to compensate the first term on the right side of 
\begin{equation}
\begin{array}{ll}
\delta_{ij}A^{r,NS}_{i\alpha j\beta}\left(\mathbf{g}\right)=\delta_{ij}\sum_{j=1}^n \nu\left( -\frac{4\pi \alpha_j^2}{l^2}\right)
-\delta_{ij}\sum_{j=1}^n\frac{2\pi i \beta_j}{l}g_{j(\alpha-\beta)}\\
\\+2\pi i\alpha_i1_{\left\lbrace \alpha\neq 0\right\rbrace }\frac{\sum_{k=1}^n4\pi \beta_j(\alpha_k-\beta_k)g_{k(\alpha-\beta)}}{\sum_{i=1}^n4\pi\alpha_i^2}.
\end{array}
\end{equation}
\begin{lem}
Let $T>0$ be arbitrary, and let $\mathbf{f}^{r,F},\mathbf{g}^{r,F}\in B_{2{\big |}\mathbf{h}^{r,F}{\big |}^{T,\exp}_{s,C}}$ for $s\geq n+3$.
For $C\geq 3L$ we have 
\begin{equation}
{\Big|}\mathbf{v}^{r,f,F}-\mathbf{v}^{r,g,F}
{\Big |}^{T,\mbox{exp}}_{h^s,C}\leq \frac{1}{2}{\Big|}\mathbf{f}^F-\mathbf{g}^F
{\Big |}^{T,\mbox{exp}}_{h^s,C}.
\end{equation}
For $C\geq 6L$ we have 
\begin{equation}
{\Big|}\mathbf{v}^{r,f,F}-\mathbf{v}^{r,g,F}
{\Big |}^{T,\mbox{exp},1}_{h^s,C}\leq \frac{1}{2}{\Big|}\mathbf{f}^F-\mathbf{g}^F
{\Big |}^{T,\mbox{exp},1}_{h^s,C}.
\end{equation}
\end{lem}

\begin{proof}
For each $t\in [0,T]$ we have
\begin{equation}\label{navode2*rewcontw1}
\begin{array}{ll}
{\Big|}\mathbf{v}^{r,f,F}(t)-\mathbf{v}^{r,g,F}(t){\Big |}_{h^s}\\
\\
\leq {\Big |}\int_{0}^{t}A^{r,NS}\left(\mathbf{f}\right)(u) \mathbf{v}^{r,f,F}(u)du
-\int_{0}^{t}A^{r,NS}\left(\mathbf{g}\right)(u) \mathbf{v}^{r,g,F}(u)du{\Big |}_{h^s}\\
\\
\leq {\Big |}\left( \int_{0}^{t}A^{r,NS}\left(\mathbf{f}\right)(u)
-\int_{0}^{t}A^{r,NS}\left(\mathbf{g}\right)(u) \right) \mathbf{v}^{r,f,F}(u)du{\Big |}_{h^s}\\
\\
+ {\Big |}\int_{0}^{t}A^{r,NS}\left(\mathbf{g}\right)(u) \left( \mathbf{v}^{r,f,F}(u)ds
-\mathbf{v}^{r,g,F}(u)\right) du{\Big |}_{h^s}\\
\\
\leq LT\sup_{u\in [0,T]}{\Big |}\mathbf{f}^F(u)
-\mathbf{g}^F(u){\Big |}_{h^s}\\
\\
+ LT\sup_{u\in [0,T]}{\Big |} \mathbf{v}^{r,f,F}(u)ds
-\mathbf{v}^{r,g,F}(u){\Big |}_{h^{s-2}}.
\end{array} 
\end{equation}
 It follows that
\begin{equation}\label{navode2*rewcontw2}
\begin{array}{ll}
{\Big|}\mathbf{v}^{r,f,F}(t)-\mathbf{v}^{r,g,F}(t){\Big |}_{h^s}\\
\\
\leq L{\Big |}\mathbf{f}
-\mathbf{g} {\Big |}^{T,\mbox{exp}}_{h^s,C}\int_0^{t}\exp(Cu)du\\
\\
+ L{\Big |} \mathbf{v}^{r,f,F}
-\mathbf{v}^{r,g,F}{\Big |}^{T,\mbox{exp}}_{h^s,C}\int_0^{t}\exp(Cu)du\\
\\
\leq \frac{L\exp(Ct)}{C}{\Big |}\mathbf{f}^F
-\mathbf{g}^F{\Big |}^{T,\mbox{exp}}_{h^s,C}\\
\\
+ \frac{L\exp(Ct)}{C}\sup_{u\geq 0}{\Big |} \mathbf{v}^{r,f,F}(u)ds
-\mathbf{v}^{r,g,F}(.){\Big |}^{T,\mbox{exp}}_{h^{s-2},C}
\end{array} 
\end{equation}
Since
\begin{equation}
{\Big |} \mathbf{v}^{r,f,F}
-\mathbf{v}^{r,g,F}{\Big |}^{T,\mbox{exp}}_{h^{s-2},C}\leq {\Big |} \mathbf{v}^{r,f,F}
-\mathbf{v}^{r,g,F}{\Big |}^{T,\mbox{exp}}_{h^{s},C}
\end{equation}
it follows that 
\begin{equation}\label{navode2*rewcontw2}
\begin{array}{ll}
{\Big|}\mathbf{v}^{r,f,F}(.)-\mathbf{v}^{r,g,F}(.){\Big |}^{T,\mbox{exp}}_{h^s,C}\\
\\
\leq \left(\frac{1}{\left(1-\frac{L}{C}  \right) }\frac{L}{C} \right){\Big |}\mathbf{f}^F(u)
-\mathbf{g}^F(u){\Big |}^{T,\mbox{exp}}_{h^s,C}.
\end{array} 
\end{equation}
For $C=3L$ the result follows. The reasoning for the stronger norm is similar.
\end{proof}

It is clear that this contraction result leads to global existence and uniqueness. Note that for global smooth existence it is sufficient for each $s\geq n+2$ and $T>0$ we find a constant $C>0$ such that
\begin{equation}
{\Big|}\mathbf{v}^{r,m,F}
{\Big |}^{T,\mbox{exp}}_{h^s,C}\leq C
\end{equation}

Uniform upper bounds for the approximative (controlled) solutions $\mathbf{v}^ {r,m,F}$ lead to existence via compactness as well. Note that the infinite vectors  $\mathbf{v}^{r,m,F}_{i}(t)=\left( v^{r,m}_{i\alpha}(t)\right)_{\alpha\in {\mathbb Z}^{n,0}}$
are in $1$-$1$ correspondence with classical functions
\begin{equation}
v^{r,m}_i(t,x)=\sum_{\alpha\in {\mathbb Z}^{n,0}}v^{r,m}_{i\alpha}\exp\left(\frac{2\pi i \alpha x}{l}\right), 
\end{equation}
where $v^{r,m}_i\in H^s\left({\mathbb T}^n\right)$ for $s>n+ 2$. Recall that
\begin{thm}
For $r>s$ and for any compact Riemann manifold $M$ (and especially for $M={\mathbb T}^n_l$) we have a compact embedding
\begin{equation}
e:H^r\left(M\right)\rightarrow  H^s\left(M\right)
\end{equation}
\end{thm}
This means that $\left( v^{r,m}_i\right)_{m\in {\mathbb N}}$ has a convergent subsequence in $H^r\left({\mathbb T}^n_l\right)$ for $r>s$ which corresponds to converging subsequence in the corresponding Sobolev space of infinite vectors of modes. Hence, passing to an appropriate subsequence $\mathbf{v}^{r,m',F}_{i}(t)$ of $\mathbf{v}^{r,m,F}_{i}(t)$ we have a limit
\begin{equation} 
\mathbf{v}^{r,F}_{i}(t)=\lim_{m'\uparrow \infty}\mathbf{v}^{r,m',F}_{i}(t)\in h^r\left({\mathbb Z}^n\right) 
\end{equation}
for $r<s$ (Rellich embedding). Since $s$ is arbitrary this limit exists in $h^r\left({\mathbb Z}^n\right)$ for all $r\in {\mathbb R}$. Hence, for all $1\leq i\leq n$ we have a family $\left( \mathbf{v}^{r,m',F}_i\right)_{m'\in {\mathbb N}}$ which satisfies
\begin{equation}\label{navodermproof}
\begin{array}{ll}
\frac{d v^{r,m'}_{i\alpha}}{dt}=\sum_{j=1}^n\nu \left( -\frac{4\pi \alpha_j^2}{l^2}\right) v^{r,m'}_{i\alpha}
-\sum_{j=1}^n\sum_{\gamma \in {\mathbb Z}^n\setminus \left\lbrace 0,\alpha\right\rbrace }\frac{2\pi i \gamma_j}{l}v^{r,m'-1}_{j(\alpha-\gamma)}v^{r,m'}_{i\gamma}\\
\\
+2\pi i\alpha_i1_{\left\lbrace \alpha\neq 0\right\rbrace}\frac{\sum_{j,k=1}^n\sum_{\gamma\in {\mathbb Z}^n\setminus \left\lbrace 0,\alpha\right\rbrace }4\pi \gamma_j(\alpha_k-\gamma_k)v^{r,m'-1}_{j\gamma} v^{r,m'}_{k(\alpha-\gamma)}}{\sum_{i=1}^n4\pi\alpha_i^2}.
\end{array} 
\end{equation}
If we can prove that the limit of infinite vectors $\left( \frac{d v^{r,m'}_{i\alpha}}{dt}(t)\right)_{\alpha \in {\mathbb Z}^n}$ is continuous in the sense that
\begin{equation}
\begin{array}{ll}
\lim_{m'\uparrow \infty}\left( \frac{d v^{r,m'}_{i\alpha}}{dt}(t)\right)_{\alpha \in {\mathbb Z}^n}\in C\left({\mathbb Z}^n\right)\\
\\
:=\left\lbrace \left(g_{\alpha} \right)|\sum_{\alpha\in {\mathbb Z}^n}f_{\alpha}\exp\left(\frac{2\pi i\alpha x}{l}\right)\in C\left({\mathbb Z}^n\right)   \right\rbrace,
\end{array}
\end{equation}
then we obtain a classical solution.
Inspecting the terms on the right side the assumption that $s> 2+n$ and $n\geq 2$ is more than sufficient in order to get
\begin{equation}
\left(\sum_{j=1}^n\nu \left( -\frac{4\pi \alpha_j^2}{l^2}\right) v^{r,m'}_{i\alpha}\right)_{\alpha\in {\mathbb Z}^n}\in h^{s-2}\left({\mathbb Z}^n\right)\subset h^{n}\left({\mathbb Z}^n\right),
\end{equation}
such that with this assumption we may ensure that $\left( \frac{d v^{r,m'}_{i\alpha}}{dt}\right)_{m\in {\mathbb N}}(t)$ converges in $C\left({\mathbb Z}^n\right) \subset h^r\left({\mathbb Z}^n\right) $ with $r>\frac{1}{2}n$ if we can control the expressions for the convection term and for the Leray projection term appropriately. However, this is easily done with the help of the infinite linear algebra results above. We stick to $s>n+2$ and $n\geq 2$. First for the convection term for each $1\leq i\leq n$ and all $\alpha\in {\mathbb Z}^n$ we consider
\begin{equation}\label{convec}
-\sum_{j=1}^n\sum_{\gamma \in {\mathbb Z}^n\setminus \left\lbrace \alpha\right\rbrace }\frac{2\pi i \gamma_j}{l}v^{r,m'-1}_{j(\alpha-\gamma)}v^{r,m'}_{i\gamma}.
\end{equation}
We observe ${\big |}\gamma_jv^{r,m'}_{i\gamma}{\big |}_{h^{s-1}\left({\mathbb Z}^n\right)}\leq C$
for some constant $C>0$ independent of $m$, hence
\begin{equation}
{\big |}\left( \sum_{\gamma \in {\mathbb Z}^n} 
v^{r,m'-1}_{j(\alpha-\gamma)}\gamma_jv^{r,m'}_{i\gamma}\right)
{\big |}_{\alpha\in {\mathbb Z}^n}\in h^{2s-1-n}\leq C
\end{equation}
for some $C>0$ independent of $m$ such that the limit is in $h^{2}\left({\mathbb Z}^n\right)$. Hence (\ref{convec}) and the limit for $m'\uparrow \infty$ is in $h^{2}\left({\mathbb Z}^n\right)$. Similarly, the Leray projection term
\begin{equation}
2\pi i\alpha_i1_{\left\lbrace \alpha\neq 0\right\rbrace}\frac{\sum_{j,k=1}^n\sum_{\gamma\in {\mathbb Z}^n}4\pi \gamma_j(\alpha_k-\gamma_k)v^{r,m'-1}_{j\gamma} v^{r,m'}_{k(\alpha-\gamma)}}{\sum_{i=1}^n4\pi\alpha_i^2}
\end{equation}
is bounded by a product of two infinite vectors which have some  uniform bound $C>0$ in  $h^{s-1}\left({\mathbb Z}^n\right)$, such that the Leray projection term is safely in $h^{2(s-1)-n}\left({\mathbb Z}^n\right)\subset h^{2}\left({\mathbb Z}^n\right)$ where we did not even take the $|\alpha|^2$ in the denominator corresponding to the Laplacian kernel into account. 
We have shown
\begin{lem}
For $s>n+ 2$ and $n\geq 2$, and for the same $\nu>0$ as above there is a $C>0$ such that for all $1\leq i\leq n$ and $m'\geq 0$ we have
\begin{equation}\label{est0vmF}
{\Big |}\frac{d}{dt}\mathbf{v}^{r,m',F}_{i}(t){\Big |}_{h^s_l}\leq C
\end{equation}
uniformly for $t>0$, and 
\begin{equation} 
\frac{d}{dt}\mathbf{v}^{r,F}_{i}(t)=\lim_{m'\uparrow \infty}\frac{d}{dt}\mathbf{v}^{r,m',F}_{i}(t)\in h^r\left({\mathbb Z}^n\right) \subset C\left({\mathbb Z}^n\right)\subset h^{2}\left({\mathbb Z}^n\right).
\end{equation}
\end{lem}

We conclude
\begin{thm}
The function
\begin{equation}
\mathbf{v}^{r,F}_{i}(t)=\lim_{m'\uparrow \infty}\mathbf{v}^{r,m',F}_{i}(t),~1\leq i\leq n
\end{equation}
satisfies the infinite nonlinear ODE equivalent to the controlled incompressible Navier-Stokes equation on the $n$-torus in a classical sense. Moreover, since the argument above can be repeated with arbitrary large $s>0$ we have that for all $1\leq i\leq n$ the infinite vector $\mathbf{v}^{r,F}_{i}(t)$ and its time derivative are in $h^s\left({\mathbb Z}^n\right)$. Higher order time derivatives also exist in a classical sense by an analogous argument for derivatives of the Navier-Stokes equation.
\end{thm}

Finally we have
\begin{thm}\label{lemma0}
Let $h_i\in C^{\infty}\left({\mathbb T}^n\right)$.
For each $\nu>0$ and $l>0$ and for all $1\leq i\leq n$ and all $t\geq 0$
\begin{equation}
v_i(t,.)\in C^{\infty}\left( {\mathbb T}^n_l \right) . 
\end{equation}
and
\begin{equation}
\mathbf{v}^{F}_i(t)\in h^s\left({\mathbb T}^n_l\right) . 
\end{equation}
for arbitrary $s\in {\mathbb R}$.
\end{thm}

\begin{proof}
We have $v^F_i(t)=v^{r,F}_i(t)-r(t)\in h^s\left({\mathbb Z}^n\right)$ for all $s\in {\mathbb R}$, because $v^{r,F}_i(t)\in h^s\left({\mathbb Z}^n\right)$ for all $s\in {\mathbb R}$, and $r(t)$ is a constant.
The second part follows from Corollary above. Given $s>0$ we can differentiate
\begin{equation}
v_i(t,.)=\sum_{\gamma\in {\mathbb Z}^n}v_{i\gamma}\exp\left( \frac{2\pi i\gamma x}{l}\right) ),
\end{equation}
up to order $m$, where $m$ is the largest integer less than $s$, and get a Fourier series which converges in $L^2\left({\mathbb T}^n_l\right)$.
Hence,
\begin{equation}
v_i(t,.) \in H^{s}\left( {\mathbb T}^n_l \right),
\end{equation}
for all $t\geq 0$. Since, this is true for all for all $s>0$ we the first statement of this lemma is true. 
\end{proof}

%
%
%

Now for a fixed $l>0$ and any $\mathbf{h}^F_1\in h^s\left({\mathbb Z}^n\right)$ for  $s\geq pn+2$ we have
\begin{equation}
\frac{d^p}{dt^p}\mathbf{v}^{m,F}_i(t)\in h^{s-pn}_{l}\left({\mathbb Z}^n\right)
\end{equation}
for all $m\geq 0$ from the uniform bound
\begin{equation}
|\mathbf{v}^{m,F}_i(t)|_{h^s_l}\leq C.
\end{equation}

Next let us sharpen the results. The arguments above can be improved in two directions. First we can derive upper bounds by considering time dilatation transformation which leads to damping terms which serve as an auto-control of the system. This leads to global upper bounds in time. The second improvement is that we can solve the Euler part in the equation at each time step locally by a Dyson formalism, and this leads to the extension of the Trotter product formula on a local time level. This approach also leads naturally to higher order schemes in time as is discussed in the next section.
 
First we consider the auto-control mechanism. Instead of considering a fixed time horizon $T>0$ we consider a time discretization $t_i,i\geq 1$ of the interval $\left[0,\infty \right)$, where we may consider $t_i=i\in {\mathbb N}$ for all $i\geq 1$ and $t_{0}:=0$.
 
We sketched this idea from the equation in (\ref{timedil}) on in the introduction. In order to get uniform global upper bounds in time we shall consider a variation of this idea. Assume that 
\begin{equation}
\mathbf{v}^{F}_i(l-1)\in h^s\left({\mathbb Z}^n\right) , 1\leq i\leq n,~s>n+4
\end{equation}
has been computed (or is given) for $l\geq 1$. We use the regularity order $s>n+3$ because our upper bounds for matrix multiplication and the contraction result with exponentially weighted norms show that
\begin{equation}
\mathbf{v}^{F}_i(t)\in h^s\left({\mathbb Z}^n\right) , 1\leq i\leq n,~s>n+4
\end{equation}
for $t\in [l-1,l]$, or more precisely, that we have
\begin{equation}
|v_{i\alpha}(t)|\leq \frac{C\exp(Ct)}{1+|\alpha|^{n+6}}
\end{equation}
for some generic constant $C>0$ if
\begin{equation}
|v_{i\alpha}(0)|\leq \frac{C}{1+|\alpha|^{n+4}}.
\end{equation}
This looks to be a stronger assumption than necessary, but we intend to simplify the growth estimate and therefore we need the stronger assumption. 
\begin{rem}
For a more sophisticated growth estimate the assumption $s>n+2$ would be enough. However, if we use the infinite ODE directly in order to estimate the growth then we need $s>n+4$ as the matrix rules for the Leray projection term imply that we get regularity $s>2r-n-2$ if we start with regularity of order $r>n+2$. Hence if we start with regularity of order $r=n+3$, then we end up with regularity of order  $2n+6-n-2=n+4$. Now, if we use the infinite mode equation directly, then we end up with regularity of order $n+2$ (or slightly above) as we loose two orders of regularity by a brute force of the Laplacian. If we do a refined estimate with the fundamental matrix, then we can weaken this assumption, of course. 
\end{rem} 
 Then consider the transformation of the local time interval $[l-1,l)$ with time coordinate $t$ to the infinite time interval $\left[0,\infty\right)$ with time coordinated $\tau$, where  
\begin{equation}\label{timedil}
(\tau (t) ,x)=\left(\frac{t-(l-1)}{\sqrt{1-(t-(l-1))^2}},x\right),
\end{equation}
which is a time dilatation effectively and leaves the spatial coordinates untouched.
Then on a time local level, i.e., for $t\in [l-1,l)$ the function $u_i,~1\leq i\leq n$ with 
\begin{equation}\label{uvlin}
\rho(1+(t-(l-1)))u_i(\tau,x)=v_i(t-(l-1) ,x),
\end{equation}
carries all information of the velocity function on this interval.
In the following we think of time $t$ in the form
\begin{equation}
t=t^{l-1,-}:=t-(l-1),
\end{equation}
and then suppress the upper script for simplicity of notation.
The following equations are formally identical with the equations in the introduction but note that we have $t=t^{l-1,-}=t-(l-1)$. Having this in mind we note
\begin{equation}\label{eq1}
\frac{\partial}{\partial t}v_i(t,x)=\rho u_i(\tau,x)+\rho(1+t)\frac{\partial}{\partial \tau}u_i(\tau,x)\frac{d \tau}{d t},
\end{equation}
where 
\begin{equation}\label{eq2}
 \frac{d\tau}{dt}=\frac{1}{\sqrt{1-t^2}^3}.
\end{equation}
Note that the factor $0<\rho <1$ appears quadratically in the nonlinear and only once in the linear term such that these linear terms become smaller compared to the damping term (they get an additional factor $\rho$). 
We denote the inverse of $\tau(t)$ by $t(\tau)$. For the modes of $u_i,~1\leq i\leq n$ we get the equation   
\begin{equation}\label{navode200firsttimedil*}
\begin{array}{ll}
\frac{d u_{i\alpha}}{d\tau}=\sqrt{1-t(\tau)^2}^3
\sum_{j=1}^n\nu \left( -\frac{4\pi \alpha_j^2}{l^2}\right)u_{i\alpha}-\\
\\
\rho(1+t(\tau))\sqrt{1-t(\tau)^2}^3\sum_{j=1}^n\sum_{\gamma \in {\mathbb Z}^n}\frac{2\pi i \gamma_j}{l}u_{j(\alpha-\gamma)}u_{i\gamma}+\\
\\
\rho(1+t(\tau))\sqrt{1-t(\tau)^2}^3\frac{2\pi i\alpha_i1_{\left\lbrace \alpha\neq 0\right\rbrace}\sum_{j,k=1}^n\sum_{\gamma\in {\mathbb Z}^n}4\pi \gamma_j(\alpha_k-\gamma_k)u_{j\gamma}u_{k(\alpha-\gamma)}}{\sum_{i=1}^n4\pi\alpha_i^2}\\
\\
-\sqrt{1-t^2(\tau)}^3(1+t(\tau))^{-1}u_{i\alpha},~t\mbox{ short for }~t^{l-1,-}=t-(l-1).
\end{array} 
\end{equation}
We may consider this equation on the domain $\left[ l-1,l-\frac{1}{2}\right]$ in $t$-coordinates corresponding to a finite time horizon $T$ which is not large, actually, i.e., it is of the size $T=\frac{0.5}{\sqrt{0.75}}=\frac{1}{\sqrt{3}}$ which is only slightly larger than $0.5$. Note that we may use any $t^{l-1,-}\in (0,1)$ giving rise to any finite $T>0$ in the following argument, but it is convenient that $T$ can be chosen so small and we still have only two steps to go to get to the next time step $l$ in original time coordinates $t$!. This makes the auto-control attractive also from an algorithmic perspective. We are interested in global upper bounds for the modes at integer times $t=l$, i.e., we observe the growth of the modes from $v_{i\alpha}(l-1)$ to $v_{i\alpha}(l)$ corresponding to the transition of modes from $vu_{i\alpha}(0)$ to $u_{i\alpha}(T)$ for all $1\leq i\leq n$ and all $\alpha\in {\mathbb Z}^n$. Note that we have
\begin{equation}
t=t^{l-1,-}\in \left[ l-1,l-1+\frac{1}{2}\right]~\mbox{ corresponds to }~\tau\in \left[0,T\right]=\left[0,\frac{1}{\sqrt{3}}\right]
\end{equation}
at each time step where we consider (\ref{navode200firsttimedil*}) on the latter interval. We integrate the equation in (\ref{navode200firsttimedil*}) from $0$ to $T=\frac{1}{\sqrt{3}}$ and have for each mode $\alpha$
\begin{equation}\label{intualg}
\begin{array}{ll}
u_{i\alpha}(T)=u_{i\alpha}(0)+\int_0^{T}{\Bigg (}\sqrt{1-t(\sigma)^2}^3
\sum_{j=1}^n\nu \left( -\frac{4\pi \alpha_j^2}{l^2}\right)u_{i\alpha}(\sigma)d\sigma-\\
\\
\rho(1+t(\sigma))\sqrt{1-t(\sigma)^2}^3\sum_{j=1}^n\sum_{\gamma \in {\mathbb Z}^n}\frac{2\pi i \gamma_j}{l}u_{j(\alpha-\gamma)}(\sigma)u_{i\gamma}(\sigma)+\\
\\
\rho(1+t(\sigma))\sqrt{1-t(\sigma)^2}^3\frac{2\pi i\alpha_i1_{\left\lbrace \alpha\neq 0\right\rbrace}\sum_{j,k=1}^n\sum_{\gamma\in {\mathbb Z}^n}4\pi \gamma_j(\alpha_k-\gamma_k)u_{j\gamma}(\sigma)u_{k(\alpha-\gamma)}(\sigma)}{\sum_{i=1}^n4\pi\alpha_i^2}\\
\\
-\sqrt{1-t^2(\sigma)}^3(1+t(\sigma))^{-1}u_{i\alpha}(\sigma){\Bigg)}d\sigma,~t\mbox{ short for }~t^{l-1,-}=t-(l-1).
\end{array} 
\end{equation}  
The relation in (\ref{uvlin}) translates into a relation of modes
\begin{equation}\label{uvlin*}
\rho(1+(t-(l-1)))u_{i\alpha}=v_{i\alpha}(t-(l-1)),
\end{equation}
such that the assumption 
\begin{equation}
|v_{i\alpha}(l-1)|\leq \frac{C}{1+|\alpha|^{n+4}}
\end{equation}
implies that with $C^{\rho}:=\frac{C^\rho}{\sqrt{3}}$ we have
\begin{equation}
|u_{i\alpha}(0)|\leq \frac{C^{\rho}}{1+|\alpha|^{n+4}}.
\end{equation}
Now, by the contraction argument above, for generic $C$ independent of local time $t\in [l-1-l]$ we have for generic time-independent constants $C,C^{\rho}>0$
\begin{equation}
|v_{i\alpha}(t)|\leq \frac{C\exp(Ct)}{1+|\alpha|^{n+6}}
\end{equation}
implies that with $C^{\rho}:=\frac{C^\rho}{\sqrt{3}}$ we have for $\tau\in [0,T]=\left[0,\frac{1}{\sqrt{3}} \right] $
\begin{equation}
|u_{i\alpha}(\tau)|\leq \frac{C^{\rho}\exp(C^{\rho}\tau)}{1+|\alpha|^{n+6}}.
\end{equation} 
Next we look at each term on the right side of (\ref{intualg}). The gain of two orders of regularity is useful if we want to apply simplified estimates. For the Laplacian term on the right side of (\ref{intualg}) we have
\begin{equation}
\begin{array}{ll}
{\Big |}\int_0^{T}{\Bigg (}\sqrt{1-t(\sigma)^2}^3
\sum_{j=1}^n\nu \left( -\frac{4\pi \alpha_j^2}{l^2}\right)u_{i\alpha}(\sigma)d\sigma{\Big |}\\
\\
\leq T\nu \sup_{\sigma\in[0,T]}{\Big|}\frac{4\pi |\alpha |^2}{l^2}u_{i\alpha}(\sigma){\Big |}\leq T\nu {\Big|}\frac{4\pi }{l^2}\frac{C^{\rho}\exp(C^{\rho}T)}{1+|\alpha|^{n+4}}{\Big |}.
\end{array}
\end{equation}
 Note that this is the only term where we need the stronger assumption of regularity order $s>n+4$. This indicates that estimates with the fundamental matrix show that the weaker assumption of regularity order $s>n+2$ is sufficient if we refine our estimates.
Next consider the convection term.
 \begin{equation}\label{intualglap}
\begin{array}{ll}
{\Big |}\int_{0}^T\rho(1+t(\sigma))\sqrt{1-t(\sigma)^2}^3\sum_{j=1}^n\sum_{\gamma \in {\mathbb Z}^n}\frac{2\pi i \gamma_j}{l}u_{j(\alpha-\gamma)}(\sigma)u_{i\gamma}(\sigma)d\sigma{\Big |}\\
\\
\leq T\rho\frac{3}{2}\sup_{\sigma\in [0,T]}{\Big |}\sum_{j=1}^n\sum_{\gamma \in {\mathbb Z}^n}\frac{2\pi i \gamma_j}{l}u_{j(\alpha-\gamma)}(\sigma)u_{i\gamma}(\sigma){\Big |}\\
\\
\leq T\rho\frac{3}{2}n\frac{2\pi}{l}{\Big |}\frac{\left( C^{\rho}\right)^2\exp(2C^{\rho}T}{1+|\alpha|^{n+4}}{\Big |},
\end{array} 
\end{equation}  
by the contraction argument. We shall choose $\rho>0$ such that a lower bound of the damping term dominates the latter upper bound plus the upper bound of the Leray projection term we are going to estimate next. However there the situation is a little involved here as we need a {\it lower} bound of the damping term. Hence we postpone the choice of $\rho$ and consider the Lery projection term. We have the upper bound
\begin{equation}\label{intualgleray}
\begin{array}{ll}
{\Big |}\rho(1+t(\sigma))\sqrt{1-t(\sigma)^2}^3\frac{2\pi i\alpha_i1_{\left\lbrace \alpha\neq 0\right\rbrace}\sum_{j,k=1}^n\sum_{\gamma\in {\mathbb Z}^n}4\pi \gamma_j(\alpha_k-\gamma_k)u_{j\gamma}(\sigma)u_{k(\alpha-\gamma)}(\sigma)}{\sum_{i=1}^n4\pi\alpha_i^2}{\Big |}\\
\\
\leq {\Big |}\frac{3}{2}\rho2\pi \sum_{j,k=1}^n\sum_{\gamma\in {\mathbb Z}^n}4\pi \gamma_j(\alpha_k-\gamma_k)u_{j\gamma}(\sigma)u_{k(\alpha-\gamma)}(\sigma){\Big |}\\
\\
\leq \frac{3}{2}\rho8\pi^2 n^2{\Big |}\frac{\left( C^{\rho}\right)^2\exp(2C^{\rho}T)}{1+|\alpha|^{n+4}}{\Big |}
\end{array} 
\end{equation}  
for the integrand, hence the upper bound
\begin{equation}
T\frac{3}{2}\rho8\pi^2 n^2{\Big |}\frac{\left( C^{\rho}\right)^2\exp(2C^{\rho}T)}{1+|\alpha|^{n+4}}{\Big |}
\end{equation}
for the integral of the intergrand from $0$ to $T$.

Next we need a lower bound of the damping term. We have
\begin{equation}\label{intualgdamp}
\begin{array}{ll}
{\Big |}\int_0^{T}{\Bigg (}
\sqrt{1-t^2(\tau)}^3(1+t(\tau))^{-1}u_{i\alpha}(\sigma){\Bigg)}d\sigma{\big |}\geq 
{\big |}{\big (}
\frac{1}{2}\int_0^Tu_{i\alpha}(\sigma){\big)}d\sigma{\big |}.
\end{array} 
\end{equation}  
anyway.  According to these estimates we get from
(\ref{intualg})
\begin{equation}\label{intualgest}
\begin{array}{ll}
|u_{i\alpha}(T)|\leq |u_{i\alpha}(0)|+T\nu {\Big|}\frac{4\pi }{l^2}\frac{C^{\rho}\exp(C^{\rho}T)}{1+|\alpha|^{n+4}}{\Big |}
+T\rho\frac{3}{2}n\frac{2\pi}{l}{\Big |}\frac{\left( C^{\rho}\right)^2\exp(2C^{\rho}T)}{1+|\alpha|^{n+4}}{\Big |}+\\
\\
T\frac{3}{2}\rho 8\pi^2 n^2{\Big |}\frac{\left( C^{\rho}\right)^2\exp(2C^{\rho}T)}{1+|\alpha|^{n+4}}{\Big |}
-
\frac{1}{2}\int_0^Tu_{i\alpha}(\sigma){\big)}d\sigma,
\end{array} 
\end{equation}
where we can use a lower bound of the modes in the damping term because of the minus sign. Note that start with an upper bound
\begin{equation}
|u_{i\alpha}(0)|\leq \frac{C^{\rho}}{1+|\alpha|^{n+4}}.
\end{equation}
for all $1\leq i\leq n$ and all modes $\alpha$, and it is sufficient to preserve this bound, i.e, to get 
\begin{equation}
|u_{i\alpha}(T)|\leq \frac{C^{\rho}}{1+|\alpha|^{n+4}}.
\end{equation}
Now for time $t$ consider all modes $\alpha \in M^t_{\leq 1/2}$ with 
\begin{equation}
 M^t_{\leq 1/2}:=\left\lbrace \alpha {\Big |}u_{i\alpha}(t)|\leq\frac{\frac{1}{2}C^{\rho}}{1+|\alpha|^{n+4}}\right\rbrace .
\end{equation}
and all modes $\alpha \in M^t_{> 1/2}$ with 
\begin{equation}
M^t_{>1/2}:=\left\lbrace \alpha |u_{i\alpha}(t)|> \frac{\frac{1}{2}C^{\rho}}{1+|\alpha|^{n+4}}\right\rbrace .
\end{equation}
Now the estimates in (\ref{intualgest}) hold for a time discretization $0=T_0<T_1<T_2< \cdots T_m=T$, i.e., we have for $1\leq l\leq m$
\begin{equation}\label{intualgesttimedis}
\begin{array}{ll}
|u_{i\alpha}(T_{l+1})|\leq |u_{i\alpha}(T_l)|+(T_{l+1}-T_l)\nu {\Big|}\frac{4\pi }{l^2}\frac{C^{\rho}\exp(C^{\rho}(T_{l+1}-T_{l}))}{1+|\alpha|^{n+4}}{\Big |}\\
\\
+(T_{l+1}-T_l)\rho\frac{3}{2}n\frac{2\pi}{l}{\Big |}\frac{\left( C^{\rho}\right)^2\exp(2C^{\rho}(T_{l+1}-T_l))}{1+|\alpha|^{n+4}}{\Big |}\\
\\
+(T_{l+1}-T_l)\frac{3}{2}\rho 8\pi^2 n^2{\Big |}\frac{\left( C^{\rho}\right)^2\exp(2C^{\rho}(T_{l+1}-T_l))}{1+|\alpha|^{n+4}}{\Big |}
-
\frac{1}{2}\int_{T_{l}}^{T_{l+1}}u_{i\alpha}(\sigma){\big)}d\sigma.
\end{array} 
\end{equation}
We may use a time scale $T_{l+1}-T_l$ which is fine enough such that
\begin{equation}
\exp(C^{\rho}(T_{l+1}-T_{l}))\leq 2,
\end{equation}
and have
\begin{equation}\label{intualgesttimed2}
\begin{array}{ll}
|u_{i\alpha}(T_{l+1})|\leq |u_{i\alpha}(T_l)|+(T_{l+1}-T_l)\nu {\Big|}\frac{8\pi }{l^2}\frac{C^{\rho}}{1+|\alpha|^{n+4}}{\Big |}\\
\\
+(T_{l+1}-T_l)\rho 6n\frac{2\pi}{l}{\Big |}\frac{\left( C^{\rho}\right)^2}{1+|\alpha|^{n+4}}{\Big |}+(T_{l+1}-T_l)6\rho 8\pi^2 n^2{\Big |}\frac{\left( C^{\rho}\right)^2}{1+|\alpha|^{n+4}}{\Big |}\\
\\
-
\frac{1}{2}\int_{T_{l}}^{T_{l+1}}u_{i\alpha}(\sigma){\big)}d\sigma.
\end{array} 
\end{equation}
Note that on a small time scale the contraction argument implies that the modes do not change much depending on the time step size. Especially, as we have Lipschitz continuity for the modes in local time we have
\begin{equation}
{\big |}u_{i\alpha}(T_{l+1})-u_{i\alpha}(T_{l}){\big |}\leq L(T_{l+1}-T_l)
\end{equation}
for some Lipschitz constant $L>0$ where this Lipschitz constant can be preserve in the scheme along with the upper bound. For $\alpha \in M^{T_{l-1}}_{> 1/2}$, 
\begin{equation}
\rho \leq \frac{1}{C^{\rho}\left( 24n\frac{2\pi}{l}+96\rho\pi^2 n^2\right) },\mbox{ and }, 
T_{l+1}-T_l \leq \frac{1}{8L},~\nu\frac{8\pi }{l^2}\leq \frac{1}{8}
\end{equation}
(we comment on the latter restriction below), we have (assuming $C^{\rho}\geq 2$ w.l.o.g)
\begin{equation}\label{intualgesttimed2}
\begin{array}{ll}
|u_{i\alpha}(T_{l+1})|\leq |u_{i\alpha}(T_l)|+(T_{l+1}-T_l)\nu {\Big|}\frac{8\pi }{l^2}\frac{C^{\rho}}{1+|\alpha|^{n+4}}{\Big |}\\
\\
+(T_{l+1}-T_l)\rho 6n\frac{2\pi}{l}{\Big |}\frac{\left( C^{\rho}\right)^2}{1+|\alpha|^{n+4}}{\Big |}+(T_{l+1}-T_l)6\rho 8\pi^2 n^2{\Big |}\frac{\left( C^{\rho}\right)^2}{1+|\alpha|^{n+4}}{\Big |}\\
\\
-
\frac{1}{2}{\big (}(T_{l+1}-T_l)u_{i\alpha}(T_l)-L(T_{l+1}-T_l)^2{\big)}\leq \frac{C^{\rho}}{1+|\alpha|^{n+4}}
\end{array} 
\end{equation}
where we use
\begin{equation}
 \frac{\frac{1}{2}C^{\rho}}{1+|\alpha|^{n+4}}\leq u_{i\alpha}(T_l)\leq  \frac{C^{\rho}}{1+|\alpha|^{n+4}}
\end{equation}
A similar estimate with an upper bound for the damping term of the form
\begin{equation}
{\big (}(T_{l+1}-T_l)u_{i\alpha}(T_l)+L(T_{l+1}-T_l)^2{\big)}
\end{equation}
holds for the modes with $\alpha \in M^{T_{l}}_{\leq 1/2}$.
We mentioned the restriction $\nu\frac{8\pi }{l^2}\leq \frac{1}{8}$ which is a restriction due to the Laplacian and seems to be a restriction on the size (lower bound) or on the the viscosity $\nu$ upper bound (high Reynold numbers allowed). Well this is an artificial restriction as it comes from the linear term. As we have remarked, we could have avoided them by representations of the the modes $u_{i\alpha}(T_l)$ in terms of fundamental matrices or even in terms of fundamental matrices for the Laplacian (dual to the simple heat kernel) at the price of a slightly more involved equation. However we could even work directly with the equation for the modes as we have done and still avoid this restriction due to the Laplacian: if we start with the velocity modes $\mathbf{v}^{F}_i$, and then use 
\begin{equation}
\mathbf{v}^{\kappa,F}_i(t^{\kappa})=\mathbf{v}^F(t)
\end{equation}
with $\kappa t^{\kappa}=t$, then the 'spatial terms' of the related equation (analogous time transformation) for $\mathbf{u}^{\kappa,F}_i(\tau)=\mathbf{\kappa v}^F_i(t)$ get a small coefficient $\kappa$ while the damping terms does not. This implies that the damping term can dominate the Laplacian in the estimates above without further restrictions.  The subscheme described for $t\in \left[l-1,l-\frac{1}{2}\right]$ is repeated twice in transformed time coordinates then. Note that all the estimates above are inherited where we can use the semigroup property in order to get global estimates straightforwardly.

The second improvement we mentioned concerns the direct calculation of the Euler part via a Dyson formalism in local time. This is also related to the corollaries of the main statements of this paper consider the Euler equation. As our considerations are closely linked to higher order schemes of the Navier Stokes equation (and of the Euler equation) with respect to time we consider the related properties of the Dyson formalism in the next section.

\section{A converging algorithm}

The detailed description of an algorithm with error estimates, the extension of the scheme to initial-value boundary problems and free boundary problems, and the extension of the to more realistic models with variable viscosity and to compressible models will be treated elsewhere. We only sketch some features of the algorithm which may be interesting to practitioners of simulation and from a numerical point of view.
The constructive global existence proof above leads to algorithmic schemes which converge in strong norms. From an algorithmic point of view  we observe that errors of solutions of finite cut-off systems of the infinite nonlinear infinite ODE representation of the incompressible Navier Stokes equation are controlled
\begin{itemize}
 \item[i)] spatially, i.e., with respect to the order of the modes, by converges of each Fourier representation of a velocity component in the dual Sobolev space $h^s\left({\mathbb Z}^n\right)$ for $s>n+2$ (if the initial data are in this space). This convergence (the error) can be estimated by weakly singular elliptic intergals.
 \item[ii)] An auto-control time dilatation transformation a time $t\in [l-1,l)$ of the form 
 \begin{equation}\label{timetransalg}
 t\rightarrow \tau= \frac{t-(l-1)}{\sqrt{1-(t-(l-1))^2}}
 \end{equation}
 combined with a relation for $t\in [l-1,l]$ at each time step $l\geq 1$ of the form
 \begin{equation}
 \mathbf{v}^F(t)=(1+(t-l-1))\mathbf{u}^F(\tau)
 \end{equation}
produces damping term at each time step $l$ for a locally equivalent equation for the modes of the infinite vector $\mathbf{u}^F$. At first glance it seems that the price to pay for this stabilization is that the time interval $[l-1,l)$ is transformed to an infinite time interval $[0,\infty)$ such that the price for stabilization are infinite computation costs. However, this is not the case. At each time step we may use (\ref{timetransalg}) on a substep interval $\left[l-1,l-\frac{1}{2} \right]$ and this stretches the local time step size only from the length of
\begin{equation}
\left[l-1,l-\frac{1}{2} \right] \mbox{ to the length of } \left[0,\frac{1}{\sqrt{3}}\right], 
\end{equation}
which is a small increase of computation costs of a factor $\frac{2}{\sqrt{3}}$. Having computed the modes $\mathbf{v}^F(t),\mathbf{u}^F(t)$ (via controlled vectors $\mathbf{v}^{r,F}(t),\mathbf{u}^{r,F}(t)$), at time step $t=l-\frac{1}{2}$ we can consider a second transformation 
\begin{equation}\label{timetransalg}
 t\rightarrow \tau= \frac{t-\left( l-\frac{1}{2}\right) }{\sqrt{1-\left( l-\frac{1}{2}\right)^2}}
 \end{equation}
 combined with a relation for $t\in \left[ l-\frac{1}{2},l\right] $ at each time step $l\geq 1$ of the form
 \begin{equation}
 \mathbf{v}^F(t)=\left( 1+\left( t-\left(l-\frac{1}{2}\right)\right) \right) \mathbf{u}^F(\tau)
 \end{equation} 
 \end{itemize}
In the following we describe algorithmic schemes without auto-control (damping). The damping mechanisms is a feature related exclusively to time, and can be treated separately as its main objective is to get stronger global upper bounds with respect to time. This does not touch the local description essentially.  
First, we remark that higher order schemes in time may be based on higher order approximations of the Euler part of the equation. It is interesting that we can solve the Euler part by an iterated Dyson formalism for regular data, although we cannot do this uniquely for dimension $n\geq 3$. We denote the recursively defined solution approximations of the Euler solution part of the controlled scheme by  $\mathbf{v}^{E,r,p},~p\geq 0$, where $\mathbf{v}^{E,r,0}=\mathbf{h}^{r}=\left(h_{i\alpha}\right)_{1\leq i\leq n,\alpha\in {\mathbb Z}^n\setminus \left\lbrace 0\right\rbrace }$ denote the initial data of the controlled scheme.
Consider the Euler part of the controlled matrix for the approximative system of order $p\geq 1$  in (\ref{controlmatrix}) above, i.e., the matrix with viscosity $\nu =0$, which is defined on the diagonal by
\begin{equation}
\begin{array}{ll}
\delta_{ij}E^{r,NS}_{i\alpha j\beta}\left(\mathbf{v}^{E,r,p-1}\right)=
-\delta_{ij}\sum_{j=1}^n\frac{2\pi i \beta_j}{l}v^{E,r,p-1}_{j(\alpha-\beta)}\\
\\+\delta_{ij}2\pi i\alpha_i1_{\left\lbrace \alpha\neq 0\right\rbrace }\frac{\sum_{k=1}^n4\pi \beta_j(\alpha_k-\beta_k)v^{E,r,p-1}_{k(\alpha-\beta)}}{\sum_{i=1}^n4\pi\alpha_i^2},
\end{array}
\end{equation}
and where we have off-diagonal, i.e. for $i\neq j$ we have the entries
\begin{equation}
(1-\delta_{ij})E^{r,NS}_{i\alpha j\beta}\left(\mathbf{v}^E\right)=2\pi i\alpha_i1_{\left\lbrace \alpha\neq 0\right\rbrace }\frac{\sum_{k=1}^n4\pi \beta_j(\alpha_k-\beta_k)v^{E,r,p-1}_{k(\alpha-\beta)}}{\sum_{i=1}^n4\pi\alpha_i^2}.
\end{equation}
Then the approximative  Euler system with data
\begin{equation}
\mathbf{v}^{r,E}(t_0)=\left( \mathbf{v}^{r,E}(t_0)_1,\cdots ,\mathbf{v}^{r,E}(t_0)_n\right)^T,~
\mathbf{v}^{r,E}(t_0)\in h^{s}\left({\mathbb Z}^n\right),~s>n+2 
\end{equation}
 for some initial time $t_0\geq 0$ has a solution for some time $t>0$ of Dyson form
\begin{equation}\label{dysonobseuler}
\begin{array}{ll}
\mathbf{v}^{r,E,p}_1=v(t_0)+\sum_{m=1}^{\infty}\frac{1}{m!}\int_0^tds_1\int_0^tds_2\cdots \int_0^tds_m \\
\\
T_m\left(E^{r,p-1}(t_1)E^{r,p-1}(t_2)\cdot \cdots \cdot E^{r,p-1}(t_m) \right)\mathbf{v}^{r,E}(t_0).
\end{array} 
\end{equation}
In algorithmic scheme we do this integration up to a certain order $p$, and then use this formula in the Trotter product formula developed above. In algorithms we deal with finite cut-off of modes naturally. Let us be a bit more specific concerning these finite mode approximations.

The infinite ODE-system can be approximated by finite ODE systems of with modes of order less or equal to some positive integer $k$ in the sense of a projection on the modes of order $|\alpha|\leq k$. At a fixed time $t\geq 0$ this is a projection from the intersection  $\cap_{s\in {\mathbb R}}h^s\left({\mathbb Z}^n\right)$ of dual Sobolev spaces into the space of (even globally) analytic functions (finite Fourier series). The Trotter product formula for dissipative operators considered above holds for the system of finite modes, of course. More importantly, the resulting schemes based on this Trotter product formula approximate the corresponding scheme of infinite modes if the limit $k\uparrow \infty$ is considered.  In this section we define this approximating scheme of finite modes. It is not difficult to show that the error of this scheme converges to zero as $k\uparrow \infty$. Detailed error estimates which relate the maximal mode of the finite system and dimension $n$ and viscosity $\nu$ to the error in $h^s$ norm are of interest in order to design algorithms. This deserves a closer investigation which will be done elsewhere. In this section we define  algorithms of of different approximation order in time via finite-mode approximations of the non-linear infinite ODEs which are equivalent to the incompressible Navier-Stokes equation. Furthermore some observations for the choice of parameters of the size of the domain $l>0$ and the viscosity $\nu >0$ are made, and we observe reductions to a cosine basis (symmetric data).
We consider the controlled scheme which solves an equivalent equation for the sake of simplicity. From our description in the introduction it is clear how to compute an approximate solution of the incompressible Navier-Stokes equation from these data. Recall the representation of the controlled Navier-Stokes equation in terms of infinite matrices.

It makes sense to formulate an 'algorithm' first for the infinite nonlinear ODE system. Accordingly, in the following the use of the word 'compute' related to substeps in an infinite scheme is not meant in a strict sense. It may be defined in some more strict sense b use of transfinite Turing machines, but what we have in mind is the finite approxmations of an infinite object such that schemes can be implemented eventually.  This is not an algorithm in a strict sense (even the description is not finite), but the projection to a set of finite modes leads to the algorithm immediately once it is described in the infinite set-up.
Consider the limit  $\mathbf{v}^{r,F}=\left(v^r_{i\alpha} \right)_{\alpha \in {\mathbb Z}^n}$ of the construction of the last section, i.e., the global solution of the controlled infinite controlled ODE system, as given. Then we may write this function formally in terms of the Dyson formalism as
\begin{equation}
v^{r,F}_i(t)=T\exp\left(A^{r}t \right) \mathbf{h}^{F}_i\in h^s_l\left({\mathbb Z}^n\right),
\end{equation}
where
\begin{equation}
\begin{array}{ll}
T\exp(A^{r}t):=\sum_{m=0}^{\infty}\frac{1}{k!}\int_{[0,t]}dt_1\cdots dt_kTA^{r}(t_1)\cdots A^r(t_k)dt_1\cdots dt_k\\
\\
:=\sum_{k=0}^{\infty}\int_0^tdt_1\int_0^{t_1}dt_2\cdots \int_0^{t_{k-1}}A^{r}(t_1)\cdots A^{r}(t_k),
\end{array}
\end{equation}
where
\begin{equation}
A^r(t_j):=\left(A^{r,ij}\right)_{1\leq i,j\leq n} 
\end{equation}
and for all $1\leq i,j\leq n$ we have
\begin{equation}
A^{r,ij}=\left(a^{r,ij}_{\alpha\beta} \right)_{\alpha,\beta\in {\mathbb Z}^n},
\end{equation}
where 
\begin{equation}\label{matrixsol}
\begin{array}{ll}
a^{r,ij}_{\alpha\beta}=\delta_{ij}\left( \delta_{\alpha\beta}\nu\left(-\sum_{j=1}^n\frac{4\pi \alpha_j^2}{l^2} \right)-\sum_{j=1}^n\frac{2\pi i\beta_jv_{j(\alpha-\beta)}}{l}\right) +L^v_{ij},
\end{array}
\end{equation}
along with
\begin{equation}
L^v_{ij}=2\pi i\alpha_i\frac{\sum_{k=1}^n 4\pi \beta_j(\alpha_k-\beta_k)v_{k(\alpha-\beta)}}{\sum_{i=1}^n4\pi^2 \alpha_i^2}.
\end{equation}
Note that the modes $v_{i\alpha}$ are time-dependent - for this reason the Dyson time-order operator appears in this formal representation. In order to construct a computable approximation of this formula we need a stable dissipation term for the second order part of the operator. We can achieve this if we may use the Trotter product formula above. because the factor $\exp\left( \left( \delta_{ij}D^r\right) \right)$ is indeed in the regular matrix space $M^s_n$ and it is a damping factor for the controlled system. In order to apply the Trotter product formula to the controlled Navier-Stokes system we may consider a time discretization and apply this formal time step by time step. Starting with $v^{r,t_0}_i:=h_i$ let us assume that we have defined a scheme for $0=T_0<t_1<\cdots<t_k$ and computed the modes $v^{r,t_l}_{i,\alpha}$ for $0\leq l\leq t_k$ 
Then in order to define the scheme recursively on the interval $[t_k,t_{k+1}]$ 
we first consider the natural splitting of the operator and define
\begin{equation}
A^r(t_j)=\left( \delta_{ij}D\right)+\left( B^{r,ij}(t_k)\right),
\end{equation}
where $B^{r,ij}(t_k)=\left( b^{r,t_k,ij}_{\alpha\beta}\right)$ along with 
\begin{equation}\label{matrixscheme}
\begin{array}{ll}
b^{r,t_k,ij}_{\alpha\beta}=
\left( -\sum_{j=1}^n\frac{2\pi i\beta_jv^{r,t_k}_{j(\alpha-\beta)}}{l}\right) +L^{t_k}_{ij},
\end{array}
\end{equation}
and
\begin{equation}
L^{t_k}_{ij}=2\pi i\alpha_i\frac{\sum_{k=1}^n 4\pi \beta_j(\alpha_k-\beta_k)v^{r,t_k}_{k(\alpha-\beta)}}{\sum_{i=1}^n4\pi^2 \alpha_i^2}.
\end{equation}
The most simple infinite scheme we could have in mind is then a scheme  approximating $\mathbf{v}^{r,F}(t)=\left(v^r_{i\alpha}(t)\right)_{\alpha\in {\mathbb Z}^n}$ for arbitrary given $t>0$ in $m$ steps then is 
\begin{equation}
\begin{array}{ll}
\left( \exp\left(\frac{t}{m}\left( \delta_{ij}D\right) \right) \exp\left(\frac{t}{k}\left( B^{r,ij}(t_m)\right) \right)\right)\times \\
\\
\cdots \left( \exp\left(\frac{t}{m}\left( \delta_{ij}D\right) \right) \exp\left(\frac{t}{k}\left( B^{r,ij}(t_0)\right) \right)\right)\mathbf{h}.
\end{array}
\end{equation}
This is indeed a time-dependent Trotter-product type approximation of order $ o\left(\frac{1}{m}\right)\downarrow 0$ as $m\uparrow \infty$, where $o$ denotes the small Landau $o$. Higher order approximations can be achieved taking account of some correction terms in the Trotter product formula. Indeed, consider 'truncations of order $q$' of the term defined in (\ref{Ceq*}) above. For each $t_r,~0\leq r\leq m$  consider
\begin{equation}\label{Ceqorderk}
\begin{array}{ll}
C^q(t_r)=\left( \delta_{ij}D\right) +\left( B^{r,ij}(t_r)\right)+\\
\\
\sum_{p=1}^{q}\frac{1}{p!}\sum_{l\geq 1,~\beta^l\in {\mathbb N}_{10}^l,~ m'+|\beta^l|=p,~l\leq m'+1}\left( \delta_{ij}D\right)^{m'}\times\\
\\
\times I_{\beta^{m'}}\left[\Delta \left(B^{r,ij}(t_r)\right),\left( B^{r,ij}(t_r)\right) \right]_T. 
\end{array}
\end{equation}
According to the Trotter product formula higher order achemes can be obtained if there is a correction term $E$ such that the replacement of $\left( B^{r,ij}(t_m)\right)$ by $\left( B^{r,ij}(t_m)\right)+E$ cancels the correction of order $q$ at each time step which may be defined by
\begin{equation}\label{Ceqorderk}
\begin{array}{ll}
C^q(t_r)_{\mbox{corr}}=C^q(t_r)-\left( \delta_{ij}D\right) +\left( B^{r,ij}(t_r)\right)\\
\\
=\sum_{p=1}^{q}\frac{1}{p!}\sum_{l\geq 1,~\beta^l\in {\mathbb N}_{10}^l,~ m'+|\beta^l|=p,~l\leq m'+1}\left( \delta_{ij}D\right)^{m'}\times\\
\\
\times I_{\beta^{m'}}\left[\Delta \left( B^{r,ij}(t_r)\right),\left( B^{r,ij}(t_r)\right) \right]_T. 
\end{array}
\end{equation}
The explicit analysis of these schemes has a right in its own an will be considered elsewhere.
For simulations we have to make a cut-off leaving only finitely many modes, of course. The formulas established remain true if we consider natural projections to systems of finite modes. We define

\begin{equation}\label{finmo}
\begin{array}{ll}
\left( \exp\left(\frac{t}{m}P_{M^l}\left( \delta_{ij}D\right) \right) \exp\left(\frac{t}{k}P_{M^l}\left( B^{r,ij}(t_m)\right) \right)\right)\times \\
\\
\cdots \left( \exp\left(\frac{t}{m}P_{M^l}\left( \delta_{ij}D\right) \right) \exp\left(\frac{t}{k}P_{M^l}\left( B^{r,ij}(t_0)\right) \right)\right)P_{v^l}\mathbf{h}^F,
\end{array}
\end{equation}
where the operator $P_{v^l}$ is defined by
\begin{equation}
P_{v^l}\mathbf{h}^F=\left(P_{1v^l}\mathbf{h}^F_1,\cdots,P_{nv^l}\mathbf{h}^F_n\right) 
\end{equation}
such that for $1\leq i\leq n$ and  $h_i=\left(h_{i\alpha}\right)_{\alpha\in {\mathbb Z}^n}$
we define
\begin{equation}
P_{iv^l}(v_{\alpha})_{\alpha\in {\mathbb Z}^n}=\left(v^l_{i\alpha}\right)_{|\alpha|\leq k},
\end{equation}
and for infinite matrices $M=(M^{ij})=\left( \left(m^{ij}_{\alpha\beta}\right)_{\alpha,\beta\in {\mathbb Z}^n}\right)$ we define
\begin{equation}
P_{M^l}M=\left(P_{ij,M^l}M^{ij}\right), 
\end{equation}
where
\begin{equation}
P_{ij,M^l}M^{ij}=\left(m^{ij}_{\alpha\beta}\right)_{|\alpha|,|\beta|\leq k}
\end{equation}
In both cases (vector- and matrix-projection) we understand that the order of the modes is preserved of course.
Next it is a consequence of our existence constructive proof that
\begin{thm}
Let $h_i\in C^{\infty}({\mathbb T}_l$, and let $T>0$ be a time horizon for the incompressible Navier-Stokes equation problem on the $n$-torus. Let $s\geq n+2$ Then the finite mode scheme defined in (\ref{finmo}) converges for each $0\leq t\leq T$ to the solution of incompressible Navier-Stokes equation in its infinite nonlinear ODE representation for the Fourier modes as $m,l\uparrow \infty$.
\end{thm}
Next recall that the parameter $\nu>0$ can be chosen arbitrarily. A large parameter $\nu>0$ increases damping and makes the computation more stable. However, in order to approximate a Cauchy problem (the initial value problem on the whole space) we need large $l$ (length of the torus). Note that the damping terms in the computation scheme are of order $\frac{1}{l^2}$, the convection terms are of order $\frac{1}{l}$ and the Leray projection terms are independent of the length of the torus. They become dominant if the size $l$ of the $n$-torus is large. From our constructive existence proof we have
\begin{cor}
The Cauchy problem for the incompressible Navier Stokes equation may be approximated by a computational approximation of order $k$ on the $n$-torus ${\mathbb T}_l$ for $l$ large enough where in a bi-parameter transformation of the original problem $\nu>0$ should be chosen 
\begin{equation}
\nu\gtrsim l^2
\end{equation}
such that the damping term  is not dominated by the Leray projection terms. Furthermore due to the quadratic growth of the moduli of diagonal (damping) terms with the order of the modes compared to linear growth of the convection term and the Leray projection term with the order of the modes the scheme remains stable for larger maximal order $k$ if it is stable for lower maximal order $k$. Furthermore the choice
\begin{equation}
\nu  \geq 2|h|^2_s
\end{equation}
for $s\geq 2+n$ leads to solutions which are uniformly bounded with respect to time. 
\end{cor}

For numerical and computational purposes it is useful to represent initial data in $\cos$ and $\sin$ terms (otherwise computational errors may lead to nonzero imaginary parts of the computational approximations). We have
\begin{lem}\label{lemeven}
Functions on the $n$-torus ${\mathbb T}^n_l$ be represented by symmetric data, i.e. with the basis of even functions 
\begin{equation}
\left\lbrace \cos\left( \frac{2\pi\alpha x}{l}\right) \right\rbrace_{\alpha\in {\mathbb Z}^n}.
\end{equation}
\end{lem}
The reason for the statement of lemma \ref{lemeven} simply is that you may consider a $L^2$- function $f$ on the cube $[-l,l]^n$ (arbitrary prescribed on $[0,l]^n$) with periodic boundary conditions with respect to a general basis for $L^2\left({\mathbb T}^n_l\right)$ such that
\begin{equation}\label{exp}
f(x)=\sum_{\alpha\in {\mathbb Z}^n} f_{\alpha}\exp(\frac{2\pi i\alpha x}{l}),
\end{equation}
and then you observe that $f(\cdots,x_i.\cdots)=f(\cdots,-x_i.\cdots)$ imply that the $\sin$-terms implicit in the representation (\ref{exp}) cancel.
Hence on the $n$-torus which is built from the cube $\left[0,l\right]^n$ we may consider symmetric data without loss of generalization. 

\section{The concept of turbulence}
In the context of the results above let us make some final remarks concerning the concept of turbulence.
There are several authors (cf. \cite{CFNT} and \cite{FJRT} for example), who emphasized that the dynamical properties of the Navier-Stokes equation such as the existence of strange attractors, bifurcation behavior etc. rather than the question of global smooth existence may be important for the concept of turbulence. We share this view, and we want to emphasize from our point of view, why we consider the infinite dynamical systems of velocity  modes to be the interesting object in order to start the study of turbulence. This concept is indeed difficult, and in a perspective of modelling it is always worth to follow this difficulty up to its origins on a logical or at least very elementary level. According to classical (pre-Fregean) logic, a concept is a list of notions, each of which is a list of notions and so on. This is not a definition but something to start with (the definition has some circularity as there is no sharp distinction between a concept and a notion, but this may be part of an inherent difficulty to define the concept of concept). According to Leibniz a concept is clear (german: 'klar') if enough notions are known in order decide the subsumption of a given object under a concept, and a concept is conspicuous (german: 'deutlich') if there is a complete list of notions of that concept. All this is relative to a 'cognoscens' of course. Since Leibniz expected that even for empirical concepts like 'gold' the list of notions may be infinite he expected a conspicuous cognition to be accessible only to an infinite mind and not to human beings. Anyway, according to this concept of concept we may say that our concept of the concept of turbulence may  not even be clear. However, there is some agreement that some specific notions belong to the notions of 'turbulence'. Some aspects of fluid dynamics may be better discussed with respect to more advanced modelling. For example for a realistic discussion of a 'whirl' we may better define a Navier-Stokes equation with a free boundary in one (upper) direction and a gravitational force in the opposite direction (pointing downwards). This free boundary may be analyzed by front fixing on a fixed half space as we did it in the case of an American derivative free boundary problem. Additional effects are created by boundary conditions (the shape  of a river bed, for example). Indeed without boundary conditions dissipative features will dominate in the end.   
First we observe that it is indeed essential to study the dynamics of the modes.
The Navier-Stokes equation is a model of classical physics ($|\mathbf{v}|\leq c$, where $c>0$ denotes the speed of light. Therefore physics should be invariant with respect to the  Galilei transformation. If $\mathbf{v}=(v_i)^T_{1\leq i\leq n}$ is a solution of the  Navier Stokes equation then for a particle which is at $\mathbf{x}_0=(x_{01},\cdots,x_{0n})$ at time $t_0$ we have the trajectory $x(t),~t\geq t_0$ determined by the equation
\begin{equation}
\stackrel{.}{\mathbf{x}}(t)=\mathbf{v}, \mbox{ i.e., }x_i(t)=x_{0i}+\int_{t_0}^tv_i(s)ds,~1\leq i\leq n,
\end{equation}
which means that for the modes we have
\begin{equation}
x_{\alpha i}+\int_{t_0}^t v_{i\alpha}(s)ds,~1\leq i\leq n,
\end{equation}
Hence, rotationality, periodic behavior, or irregular dynamic behaviour due to a superposition of different modes $v_{i\alpha}$ which are 'out of phase' translates to an analogous dynamic behavior for the trajectories with a superposition of a uniform translative movement (which disappears in a moved laboratory and is physically irrelevant therefore.   
up to Galiliei transformation dynamic behavior of the modes translates into equivalent. 
The dynamics of the modes is an infinite ODE which may be studied via bifurcation theory. Since we proved the polynomial decay of modes a natural question is wether a center manifold theorem and the existence of Hopf bifurcations may be proved. The quadratic terms of the modes in the ODE of the incompressible Navier Stokes equation in the dual formulation seem to imply this. Irregular behavior may also be due to several Hopfbifurcations with different periodicity. Bifurcation analysis may lead then to a proof of structural instability, sensitive behaviour and chaos (cf. \cite{A}). Global sensitive behavior and chaos may also be proved via generalisations or Sarkovskii's theorem or via intersection theory of algebraic varieties as proposed in \cite{KD1} and \cite{KD2}. Studying the effects of boundary conditions and free boundary conditions for turbulent behavior may also be crucial. For the Cauchy problem dissipative effects will take over in the long run. For a 'river' modelled by fixed boundary conditions for the second and third velocity components $v_2$ and $v_3$ (river bank) and a free boundary condition (surface) and a fixed boundary condition (bottom) for the first velocity component $v_1$, and a gravitational force in the direction related to the first velocity component, complex dynamical behavior can be due to the boundary effects. It is natural to study such a boundary model by 
the front fixing method considered in \cite{K2}. It is likely that the scheme considered in \cite{KNS} and \cite{K3} can be applied in order to obtain global existence results.
Finally, we discuss some relations of the concept of turbulence, singular solutions of the Euler equation, and bifurcation theory from the point of vie of the Dyson formalism developed in this article. We refer to an update of \cite{KSV} for a deeper discussion which will appear soon.
The qualitative description of turbulence in paragraph 31 of \cite{LL} gives a list of notions of the concept of turbulence associated to high Reynold numbers which consists of $\alpha)$ an extraordinary irregular and disordered change of velocity at every point of a space-time area, $\beta)$ the velocity fluctuates around a mean value, $\gamma)$ the amplitudes of the fluctuations are not small compared to the magnitude of the velocity itself in general, $\delta)$ the fluctuations of the velocity can be found at a fixed point of time, $\epsilon)$ the trajectories of fluid particles are very complicated causing a strong mixture of the fluid. All these notions can be satisfied by proving the existence of Hopf bifurcations an Chenciner bifurcations for the infinite ODE equations
of modes. Proving this seems not to be out of reach except for the notion in $\gamma)$ as the usual methods of bifurcation theory prove the existence of such bifurcation via local topological equivalence normal forms where some lower order terms are sufficient to describe some dynamical properties.  

\begin{center}
 {\bf  Appendix}
\end{center}

\appendix{Convergence criteria of  Euler type Trotter product schemes for Navier Stokes equations}
In this appendix nested Euler type Trotter product schemes of the Navier Stokes equation are defined, where limits of the schemes define regular solutions of the Navier stokes equation depending on the regularity of the data. The critical regularity of convergence for the scheme are data in a Sobolev space of order $n/2+1$, where algorithms do no not converge for lower regularity indicating possible singularities. We support the conjecture that the equations have singular solutions for data which are only in a Sobolev space of order less than $n/2+1$.  Observations of the sketch in \cite{KT} are simplified and sharpened, where explicit upper bound constants are found. This scheme has a first order error with respect to the time discretization size, where higher order schemes can be defined using a Dyson formalism above.

\section{Definition of an Euler-type Trotter-product scheme and statement of results}
We consider global Euler type Trotter product schemes of the incompressible Navier Stokes equation on the $n$-dimensional Torus ${\mathbb T}^n$, where we start with an observation about function spaces.   
A function $g$ in the Sobolev space  $H^s\left({\mathbb T}^n \right)$ for $s\in {\mathbb R}$ can be considered in the analytic basis $\left\lbrace  \exp\left(2\pi i\alpha\right) \right\rbrace_{\alpha\in {\mathbb Z}^n}$, where ${\mathbb Z}$ is the ring of integers as usual and ${\mathbb Z}^n$ is the set of $n$-tuples of integers. The information of the function $g$ is completely encoded in the list of Fourier modes $(g_{\alpha})_{\alpha \in {\mathbb Z}^n}$. 
We say that the infinite vector $(g_{\alpha})_{\alpha\in {\mathbb Z}^n}$ is in the dual Sobolev space $h^s=h^s\left({\mathbb Z}^n\right) $ of order $s\in {\mathbb R}$, if
 \begin{equation}
 \sum_{\alpha \in{\mathbb Z}^n}|g_{\alpha}|^2\left\langle \alpha\right\rangle^{2s}< \infty, 
 \end{equation}
where
\begin{equation}
 \left\langle \alpha\right\rangle :=\left(1+|\alpha|^2 \right)^{1/2}. 
\end{equation}
As we work on the torus for the whole time we suppress the reference to ${\mathbb T}^n$ and, respectively, to ${\mathbb Z}^n$ in the denotation of dual Sobolev spaces in the following.
It is well-known that $g\in H^s$ iff $(g_{\alpha})_{\alpha\in {\mathbb Z}^n}\in h^s$ for $s\in {\mathbb R}$. Next we relate this categorization of regularity by orders $s$ of dual Sobelev spaces
to the regularity criteria of initial data for the Navier stokes equation. For the data $h_i,~1\leq i\leq n$ of an incompressible Navier Stokes equation the regularity condition of the scheme is
\begin{equation}\label{regdata1}
\exists C>0~\forall \alpha\in {\mathbb Z}^n:~{\big |}h_{i\alpha}{\big |}\leq \frac{C}{1+|\alpha|^{n+s}} \mbox{ for some fixed }s>1.
\end{equation}
 Furthermore for $r<\frac{n}{2}+1$ there are data $h_i,~1\leq i\leq n$ $h_i\in h^{r}({\mathbb Z}^n)$
such that the scheme diverges. This does not prove that there are singular solutions (although this may be so) but that the scheme does not work, i.e. has no finite limit. The case $s=1$ is critical, but a certain contractive property of iterated elliptic integrals is lost. It seems that the algorithm defined below is still convergent but this critical case will be considered elsewhere.  Note that for data as in (\ref{regdata1}) we have
\begin{equation}
{\big |}h_{i\alpha}{\big |}^2\leq \frac{C^2}{1+|\alpha|^{2n+2s}}\Rightarrow h_i\in H^{\frac{1}{2}n+s}\left({\mathbb Z}^n\right). 
\end{equation}
On the other hand, if $h_i\in H^{\frac{1}{2}n+s}\left({\mathbb Z}^n\right)$ then
 \begin{equation}
 \sum_{\alpha \in{\mathbb Z}^n}|h_{i\alpha}|^2\left(1+|\alpha|^2 \right)^{n+2s}< \infty, 
 \end{equation}
which means that there must be a $C>0$ such that the condition in (\ref{regdata1}) holds. Here note that the sum over ${\mathbb Z}^n$ is  essentially equivalent to an $n$-dimensional integral such that we need a decay of the modes $|h_{i\alpha}|$ with an exponent which exceeds  $n+2s$ by $n$.  
Especially note that for $n=3$ the critical regularity of the data for convergence is $H^{2.5}$ (resp. $h^{2.5}$ for the dual space). 
Concerning the reason for this specific threshold value $\frac{n}{2}+s$ with $s>1$ for convergence of Euler-type Trotter product schemes we first observe that for some constant $c=c(n)$ depending only on the dimension $n$ we certainly have
\begin{equation}\label{cC}
\sum_{\beta\in {\mathbb Z}^n}\frac{C}{1+|\alpha-\beta|^{m}}\frac{C}{1+|\beta|^{l}}\leq \frac{cC^2}{1+|\alpha|^{m+l-n}}.
\end{equation}
For factors as the Burgers term applied to initial data (which satisfy (\ref{regdata1})) we have orders $m> n+1$ (one spatial derivative) and
$l>n$ such that the right side has a order of decay $>n+1$. Hence, the relation in (\ref{cC}) represents a contractive property for higher order modes $|\alpha|\geq 2$. It becomes contractive for all modes by some spatial scaling (cf. below). Similar for the Leray projection term (which has two spatial derivatives but behaves similar as the Burgers term concerning the decay with respect to modes due to the integration with the dual of first derivatives of the Laplacian kernel) .
We shall use the similar relation in (\ref{cC}) (among other spatial features of the operator), and shall observe that at any stage $N\geq 1$ of a Euler type product scheme $v^N_{i\alpha}(m\delta t^{(N)})_{1\leq i\leq n,~\alpha\in {\mathbb Z}^n,~0\leq m\leq 2^m}$ we have an upper bound of the form
\begin{equation}\label{bb}
{\big |}v^N_{i\alpha}(m\delta t^{(N)}){\big |}\leq \frac{C_0}{1+|\alpha|^{n+s}}
\end{equation}
for all $m\geq 1$ for some finite $C_0>0$ depending only on the dimension $n$, the viscosity $\nu >0$ and the initial data $h_i,~1\leq i\leq n$, and on the time horizon $T>0$. We also provide sharper arguments where we have independence of the upper bound constant $C_0$ of the time horizon $T>0$.  In general we shall have an upper bound as in (\ref{bb}) with $C_0>C$ if $|h_{i\alpha}|\leq \frac{C}{1+|\alpha|^{\frac{n}{2}+s}}$ for some $s>1$  since we cannot prove a strong contraction property for the Navier Stokes operator.  An other spatial effect of the operator that we use is that all nonlinear terms have a spatial derivative. For the spatial transformation
\begin{equation}
v^{r}_i(t,y)=v_i(t,x),~y_i=rx_i,~1\leq i\leq n
\end{equation}
we get $v_{i,j}(t,x)=v^{r}_{i,j}(t,y)r$ and $v_{i,j,j}(t,x)=v^{r}_{i,j,j}(t,y)r^2$ such that the Navier Stokes equation becomes
\begin{equation}\label{Navlerayr}
\left\lbrace \begin{array}{ll}
\frac{\partial v^r_i}{\partial t}-\nu r^2\sum_{j=1}^n v^r_{i,j,j}
+r\sum_{j=1}^n v^r_jv^r_{i,j}=f_i\\
\\ \hspace{1cm}+rL_{{\mathbb T}^n}\left( \sum_{j,m=1}^n\left( v^r_{m,j} v^r_{j,m} \right) (\tau,.)\right) ,\\
\\
\mathbf{v}^r(0,.)=\mathbf{h},
\end{array}\right.
\end{equation}
where $L_{{\mathbb T}^n}$ denotes the Leray prjection on the torus.
These are are natural upper bounds in the scheme related to Burgers terms, and similar relations hold for Leray projection terms. Since $2s-1>s$ we have a strong contraction property for higher order expansion of approximative solutions and this can be used (either directly via weighted norms or in a refined scheme with an auto-control function). 
In order to define a Trotter product scheme for the incompressible Navier Stokes equation models it is convenient to have a time-scaling as well. Note that for a time scale $\rho>0$ with $t=\rho\tau$ where $\rho$ depends only on the initial data, the viscosity and the dimension the resulting Navier Stokes equation is equivalent. For a spatial scaling with parameter $r$ and a time scaling with parameter $\rho$, i.e., for the transformation $v^{\rho,r}_i(\tau,y)=v_i(t,x)$ wit $t=\rho \tau$ and $y_i=rx_i$ we get 
\begin{equation}\label{Navleray}
\left\lbrace \begin{array}{ll}
\frac{\partial v^{\rho,r}_i}{\partial \tau}-\rho r^2\nu\sum_{j=1}^n  v^{\rho,r}_{i,j,j} 
+\rho r\sum_{j=1}^n v^{\rho,r}_jv^{\rho,r}_{i,j}=f_i\\
\\ \hspace{1cm}+\rho rL_{{\mathbb T}^n}\left(  \sum_{j,m=1}^n\left( v^{\rho,r}_{m,j}v^{\rho,r}_{j,m}\right) (\tau,.)\right) ,\\
\mathbf{v}^{\rho,r}(0,.)=\mathbf{h},
\end{array}\right.
\end{equation}
to be solved for $\mathbf{v}^{\rho,r}=\left(v^{\rho,r}_1,\cdots ,v^{\rho,r}_n \right)^T$ on the domain $\left[0,\infty \right)\times {\mathbb T}^n$ is completely equivalent to the usual Navier Stokes equation in its Leray projection form. As
\begin{equation}
v_{i,j}=\frac{\partial v^{\rho,r}_i}{\partial y_j}\frac{d y_j}{d x_j}=rv^{\rho,r}_{i,j}
\end{equation}
it is clear that the Burgers term gets a factor $r$ by the scaling and the relation $v_{i,j,j}=r^2v^{\rho,r}_{i,j,j}$  implies that the viscosity term gets a factor $r^2$ upon the simple spatial scaling transformation. The scaling of the Leray projection term is less obvious. For $\rho=1$ we write $v^{1,r}_i=v^{\rho,r}_i$. The pressure scaling $p^{1,r}(t,y)=p(t,x)$ the related Poisson equation (on the torus ${\mathbb T}^n$) in original coordinates
\begin{equation}
\Delta p=\sum_{j,k=1}^nv_{j,k}v_{k,j}
\end{equation}
transforms to
\begin{equation}
 r^2\Delta p^{1,r}=\sum_{j,k=1}^nr v^{1,r}_{j,k}r v^{1,r}_{k,j}\Leftrightarrow \Delta p^{1,r}=\sum_{j,k=1}^nv^{1,r}_{j,k}v^{1,r}_{k,j}
\end{equation}
such that from
\begin{equation}
p_{,i}=rp^{1,r}_{,i}
\end{equation}
we know that the Leray projection term gets a scaling factor $r$.
\begin{rem} 
Note that the spatial scaling factor behaves differently for some typical Cauchy problems with known singular solutions such as
\begin{equation}\label{w01}
\frac{\partial w_0}{\partial t}-\nu\Delta w_0 +w^2_0,~w(_0)=w_0(0,.).
\end{equation}
Here, a transformation $w^{r}_0(t,y)=w(t,x),~y=rx$ leads to
\begin{equation}\label{w02}
\frac{\partial w^r_0}{\partial t}-\nu r^2\Delta w^r_0 +(w^r_0)^2,~w^r_0(0,.)=g.
\end{equation}
A singular point $(t_s,x_s)$ of a solution $w_0$ of (\ref{w01}) is trasformed to a singular point $(t_s,rx_s)$  of a solution $w^r_0$ of (\ref{w02}). A strong viscosity damping coefficient $r^2\nu$ caused by large $r$ pushes the singular point $(t_s,rx_s)$ towards spatial but it is tthere nevertheless. The example in (\ref{w02}) is a scalar equation but similar considerations hold for systems of equation as well, of course. Similarly, if the function $w_1$ satisfies
\begin{equation}\label{w11}
\frac{\partial w_1}{\partial t}-\nu\Delta w_1 +\left\langle \nabla w_1,\nabla w_1\right\rangle ,~w_1(0,.)=g,
\end{equation}
thenthe function 
\begin{equation}\label{w12}
\frac{\partial w^r_1}{\partial t}-\nu r^2\Delta w^r_1 +r^2\left\langle \nabla w_1,\nabla w_1\right\rangle ,~w_1(0,.)=g.
\end{equation}
Here the viscosity term and the nonlinerar term have the same scaling coefficient $r^2$. Again, a singular point $(t_s,x_s)$ of a solution $w_1$ of (\ref{w01}) is trasformed to a singular point $(t_s,rx_s)$  of a solution $w^r_1$ of (\ref{w02}).    
\end{rem}
We shall observe that iterative elliptic integrals upper bounds as in (\ref{cC}) related to nonlinear terms of the Navier Stokes equation become contractive for each mode, i.e., the relation in (\ref{cC}) below becomes a contractive upper bound for nonlinear terms in the scheme in the sense that for $s>1$   
\begin{equation}\label{cC*}
\sum_{\beta\in {\mathbb Z}^n}\rho r\frac{(n+n^2)C}{1+|\alpha-\beta|^{n+s}}\frac{|\beta|C}{1+|\beta|^{n+s}}\leq \frac{\rho rc_0C^2}{1+|\alpha|^{n+2s-1}},\mbox{where}~n+2s-1>n+s
\end{equation}
for some finite constant $c_0>1 $ which depends on the dimension. 
\begin{rem}
Note that the estimate in (\ref{cC*}) is an estimate for the Burgers term prima facie. However, the Leray projection term has a similar estimates since the quadratic terms $(\alpha_k-\gamma_k)\gamma_j$ get a factor $\frac{\alpha_i}{\sum_{i=1}^2\alpha_i^2}$.
\end{rem}    
We remark that $C>1$ in \ref{cC*} is a constant is an upper bound constant for the velocity modes which is observed to be inherited by the scheme. Note that the assumption $C>1$ is without loss of generality. There are several aramter choices which make viscosity damping dominant.
Note that for the choice
\begin{equation}\label{rpara}
r=\frac{c_0^2C^2}{\nu},~\rho=\frac{\nu}{2c_0^2C^2} 
\end{equation}
the viscosity coefficient $\rho r^2\nu$, i.e. the coefficient of the Laplacian in (\ref{Navleray}), becomes 
\begin{equation}\label{viscpar}
\rho r^2\nu =\frac{c_0^2C^2}{2}
\end{equation}
and the parameter coefficient $\rho r$ of the nonlinear terms in (\ref{Navleray}) satisfies
\begin{equation}\label{nlinpara}
\rho r= \frac{1}{2}.
\end{equation}
Note that for $c_0C>1$ the viscosity damping coefficient $\rho r^2\nu$ in (\ref{viscpar}) becomes dominant compared to the coefficient of the nonlinear term $\rho r= \frac{1}{2}$.   
We mention that the scheme can be applied can be applied to models with external force terms which are functions the functions $f_i: {\mathbb T}^n\rightarrow {\mathbb R},~1\leq i\leq n$ in $H^1$ (at least). For time-dependent forces the situation is more complicated, because forcer terms may cancel the viscosity term. In this paper we consider the case $f_i\equiv 0$ for all $1\leq i\leq n$. 
First we define a solution scheme, and then we state some results concerning convergence and divergence. The main result about convergence is proved in section 2. 

For $1\leq i\leq n$ we write the velocity component $v^{\rho,r}_i=v^{\rho,r}_i(\tau,x)$ for fixed $\tau\geq 0$ in the analytic basis $\left\lbrace \exp\left( \frac{2\pi i\alpha x}{l}\right),~\alpha \in {\mathbb Z}^n\right\rbrace $ such that
\begin{equation}
v^{\rho,r}_i(\tau,x):=\sum_{\alpha\in {\mathbb Z}^n}v^{\rho,r}_{i\alpha}(t)\exp{\left( \frac{2\pi i\alpha x}{l}\right) },
\end{equation}
where $l>0$ measures the size of the torus (sometimes we choose $l=1$ without loss of generality). 
Then the initial value problem in (\ref{Navleray}) is equivalent to an infinite ODE initial value problem for the infinite time dependent vector function of velocity modes $v^{\rho,r}_{i\alpha},~\alpha\in {\mathbb Z}^n,~1\leq i\leq n$, where
\begin{equation}\label{navode200first}
\begin{array}{ll}
\frac{d v^{\rho,r}_{i\alpha}}{d\tau}=\rho r^2\sum_{j=1}^n\nu \left( -\frac{4\pi^2 \alpha_j^2}{l^2}\right)v^{\rho,r}_{i\alpha}
-\rho r\sum_{j=1}^n\sum_{\gamma \in {\mathbb Z}^n}\frac{2\pi i \gamma_j}{l}v^{\rho,r}_{j(\alpha-\gamma)}v^{\rho,r}_{i\gamma}\\
\\
+\rho r2\pi i\alpha_i1_{\left\lbrace \alpha\neq 0\right\rbrace}\frac{\sum_{j,k=1}^n\sum_{\gamma\in {\mathbb Z}^n}4\pi^2 \gamma_j(\alpha_k-\gamma_k)v^{\rho,r}_{j\gamma}v^{\rho,r}_{k(\alpha-\gamma)}}{\sum_{i=1}^n4\pi^2\alpha_i^2},
\end{array} 
\end{equation}
for all $1\leq i\leq n$, and where for all $\alpha\in {\mathbb Z}^n$ we have $v^{\rho,r}_{i\alpha}(0)=h_{i\alpha}$. We denote $\mathbf{v}^{\rho,r,F}=(v^{\rho,r,F}_1,\cdots v^{\rho,r,F}_n)^T$ with $n$ infinite vectors $v^{\rho,r,F}_i=(v^{\rho,r}_{i\alpha})^T_{1\leq i\leq n,~\alpha \in {\mathbb Z}^n}$. The superscript $T$ denotes transposed in accordance to a usual convention about vectors and should not be confused with the time horizon which never appears as a superscript.

As remarked, in this paper we consider the case without force terms, i.e., the case where $f_i=0$ for all $1\leq i\leq n$.
For an arbitrary time horizon $T_0>0$, at stage $N\geq 1$ and with time steps of size $\delta t^{(N)}$ with $2^N\delta t^{(N)}=T_0$ we define an Euler-type Trotter product scheme
with $2^N$ time steps 
$$m\delta t^{(N)}\in \left\lbrace 0,\delta t^{(N)},2\delta t^{(N)},\cdots , 2^N\delta t^{(N)}=T_0\right\rbrace.$$
We put upper script stage number $N$  in brackets in order to avoid confusion with an exponent as we intend to study time-errors of the scheme later.
First, at each stage $N\geq 1$ and at each time step number $m\in \left\lbrace 0,\cdots ,  2^N\right\rbrace$ we consider the Euler step
\begin{equation}\label{navode200first}
\begin{array}{ll}
v^{\rho,r,N}_{i\alpha}((m+1)\delta t^{(N)})=v^{\rho,r,N}_{i\alpha}(m\delta t^{(N)})+\\
\\
\rho r^2\sum_{j=1}^n\nu \left( -\frac{4\pi^2 \alpha_j^2}{l^2}\right)v^{\rho,r,N}_{i\alpha}(m\delta t^{(N)})\delta t^{(N)}-\\
\\
\rho r\sum_{j=1}^n\sum_{\gamma \in {\mathbb Z}^n}\frac{2\pi i \gamma_j}{l}v^{\rho,r,N}_{j(\alpha-\gamma)}(m\delta t)v^{\rho,r,N}_{i\gamma}(m\delta t^{(N)})\delta t^{(N)}+\\
\\
2\pi i\alpha_i1_{\left\lbrace \alpha\neq 0\right\rbrace}\frac{\rho r\sum_{j,k=1}^n\sum_{\gamma\in {\mathbb Z}^n}4\pi^2 \gamma_j(\alpha_k-\gamma_k)v^{\rho,r,N}_{j\gamma}(m\delta t^{(N)})v^{\rho,r,N}_{k(\alpha-\gamma)}(m\delta t^{(N)})}{\sum_{i=1}^n4\pi^2\alpha_i^2}\delta t^{(N)}.
\end{array} 
\end{equation}
At each stage $N\geq 1$ and for all $1\leq i\leq n$ and all modes $\alpha\in {\mathbb Z}^n$ this defines a list of values
\begin{equation}
v^{\rho,r,N}_{i\alpha}(m\delta t^{(N)}),~m\in \left\lbrace 0,1,\cdots,2^N\right\rbrace, 
\end{equation}
such that the whole list of $2^N$ $n{\mathbb Z}^n$-vectors defines an Euler-type approximation of a possible solution of the Navier Stokes equation above (along with zero source term, i.e., $f_i\equiv 0$ for all $1\leq i\leq n$).
For an arbitrary finite time horizon $T^*>0$, and at each stage $N$ we denote the (tentative) approximative solutions by  
\begin{equation}
\mathbf{v}^{\rho,r,N,F}(t):=\left( \left(v^{\rho,r,N}_{i\alpha}(t)\right) _{\alpha\in {\mathbb Z}^n} \right)^T_{1\leq i\leq n}
\end{equation}
for $0\leq t\leq T^*$ (where the upper script $T$ always refers to 'transposed' as the usual vector notation is vertical and not to the time horizon).
The last two terms on the right side of (\ref{navode200first}) correspond to the spatial part of the incompressible Euler equation, where for the sake of simplicity of notation we abbreviate (after some renaming)
\begin{equation}
\begin{array}{ll}
e^{\rho,r,N}_{ij\alpha\gamma}(m dt^N)=-\rho r\frac{2\pi i (\alpha_j-\gamma_j)}{l}v^{\rho,r,N}_{i(\alpha-\gamma)}(m\delta t^{(N)})\\
\\
+\rho r2\pi i\alpha_i1_{\left\lbrace \alpha\neq 0\right\rbrace}
4\pi^2 \frac{\sum_{k=1}^n\gamma_j(\alpha_k-\gamma_k)v^{\rho,r,N}_{k(\alpha-\gamma)}(m\delta t^{(N)})}{\sum_{i=1}^n4\pi^2\alpha_i^2}.
\end{array}
\end{equation}
Note that with this abbreviation (\ref{navode200first}) becomes
\begin{equation}\label{navode200second}
\begin{array}{ll}
v^{\rho,r,N}_{i\alpha}((m+1)\delta t^{(N)})=v^{\rho,r,N}_{i\alpha}(m\delta t^{(N)})+\\
\\
\rho r^2\sum_{j=1}^n\nu \left( -\frac{4\pi^2 \alpha_j^2}{l^2}\right)v^{\rho,r,N}_{i\alpha}(m\delta t^{(N)})\delta t^{(N)}+\\
\\
\sum_{j=1}^n\sum_{\gamma\in {\mathbb Z}^n}e^{\rho,r,N}_{ij\alpha\gamma}(m\delta t^{(N)})v^{\rho,r,N}_{j\gamma}(m\delta t^{(N)})\delta t^{(N)}.
\end{array} 
\end{equation}
As in \cite{KT} we consider Trotter product formulas in order to make use of the damping effect of the diffusion term, which is not obvious in (\ref{navode200second}) where a term with factor $|\alpha|^2$ appears which corresponds to an unbounded Laplacian. At each stage $N$ we get the Trotter product formula stating that for all $T=N\delta t^{(N)}$ we have
\begin{equation}\label{aa}
\begin{array}{ll}
\mathbf{v}^{\rho,r,N,F}(T)\doteq \Pi_{m=0}^{2^N}\left( \delta_{ij\alpha\beta}\exp\left(-\rho r^2\nu \sum_{i=1}^n\alpha_i^2 \delta t^{(N)} \right)\right)\\
\\
\left( \exp\left( \left( \left( e^{N\rho,r,}_{ij\alpha\beta}\right)_{ij\alpha\beta}(m\delta t^{(N)})\right)\delta t^{(N)} \right) \right) \mathbf{h}^F, 
\end{array}
\end{equation}
and where $\doteq$ means that the identity holds up to an error $O(\delta t^{(N)})$, i.e. we have a linear error in $\delta t^{(N)}$ (corresponding to a quadratic error in $\delta t^{(N)}$ at each time step).
Note that equation in (\ref{aa}) still makes sense as the viscosity converges to zero as the first factor on the right side of (\ref{aa}) becomes an infinite identity matrix. This leads to a approximation scheme which we denote by
 \begin{equation}\label{aazz}
\begin{array}{ll}
\mathbf{e}^{\rho,r,N,F}(T)\doteq \Pi_{m=0}^{2^N}
\left( \exp\left( \left( \left( e^{\rho,r,N}_{ij\alpha\beta}\right)_{ij\alpha\beta}(m\delta t^{(N)})\right)\delta t^{(N)} \right) \right) \mathbf{h}^F. 
\end{array}
\end{equation}
 We investigate whether the linear error for the Navier Stokes Trotter product scheme goes to zero as $N\uparrow \infty$ or if we have a blow-up. Note that in (\ref{aa}) the entries in $(\delta_{ij\alpha\beta})$ are Kronecker-$\delta$s which describe the unit $n{\mathbb Z}^n\times n{\mathbb Z}^n$-matrix. The formula in (\ref{aa}) is easily verified by showing that at each time step $m$ 
\begin{equation}
\begin{array}{ll}
\left( \delta_{ij\alpha\beta}\exp\left(-\rho r^2\nu 4\pi^2 \sum_{i=1}^n\alpha_i^2 \delta t^{(N)} \right)\right)\\
\\
\left( \exp\left( \left( \left( e^{\rho,r,N}_{ij\alpha\beta}(m\delta t)^{(N)}\right)_{ij\alpha\beta}\right)\delta t^{(N)}\right)  \right)\mathbf{v}^{\rho,r,N,F}(m\delta t^{(N)})
\end{array}
\end{equation}
(as an approximation of $\mathbf{v}^{\rho,r,N,F}((m+1)\delta t^{(N)})$) solves the equation (\ref{navode200first}) with an error of order in $O\left( \left( \delta t^{(N)}\right)^2\right) $.
Before we continue to describe the scheme we introduce the dual Sobolev spaces which we use in order to measure the the lists of modes at each stage $N\geq 1$ and the limit as $N\uparrow \infty$.
Even the scheme above considered on the usual times scale $\rho=1$ and for small spatial time scalar $r>0$ leads to contraction result with exponentially weighted norms, but from the numerical or algorithmic point of view it is interesting to stabilize the scheme and prove time-independent regular upper bounds.   
Therefore, here we show that for parameters $\rho, r$ as in (\ref{rpara})
and (\ref{viscpar}) for all stages $N$ and all $m\geq 1$ there is an uper bound constant $C$ (depending only on the initial data $\mathbf{h}$, the dimension $n$, and the viscosity $\nu>0$ ) such that 
\begin{equation}\label{bbc}
{\big |}v^{\rho,r,N}_{i\alpha}(m\delta t^{(N)}){\big |}\leq \frac{C}{1+|\alpha|^{n+s}}.
\end{equation}
An explicit upper bound constant is given below in the statement of Theorem \ref{linearboundthm}  below.
 This is shown directly for this scheme using viscosity damping, but we mention that it can be shown also  via comparison functions at each time step. This alternative scheme may be used in order to prove weaker results concerning global solution branches for inviscid limits. It is also interesting form a numerical point of view, since it introduces an additional damping term.  
 At each time step, i.e., on the time interval $[t_0,t_0+a]$ for some $a\in (0,1)$ we compare the value function $v^{\rho,r}_i,\ 1\leq i\leq n$ with a time dilated function $u^{l,\rho,r,t_0}_i:\left[0,\frac{a}{\sqrt{1-a^2}} \right] \times {\mathbb T}^n\rightarrow {\mathbb R},~1\leq i\leq n$ along with
\begin{equation}\label{transformvul}
\lambda (1+\mu (\tau-t_0))u^{l,\rho,r,t_0}_i(s,.)=v^{\rho,r}_i(\tau,.),~s=\frac{\tau-t_0}{\sqrt{1-(\tau-t_0)^2}}.
\end{equation}
Note that in this definition of $u^{l,\rho,r,t_0}_i$ the function $v_i$ is considered on the interval $\left[t_0,t_0+a\right]$ for each $1\leq i\leq n$. The upper script $l$ indicates that the transofrmation in (\ref{transformvul}) is local in time. Alternatively, we may define a global time transformation
\begin{equation}\label{transformvug}
\lambda (1+\mu \tau)u^{g,\rho,r,t_0}_i(s,.)=v^{\rho,r}_i(\tau,.),~s=\frac{\tau-t_0}{\sqrt{1-(\tau-t_0)^2}}.
\end{equation}
Note that $\tau$ becomes $t$ for $\rho=1$. We write
\begin{equation}
u^{*,\rho,r,t_0}_i,~*\in \left\lbrace l,g\right\rbrace, 
\end{equation}
if we want to refer to both transformations at the same time.
Similar comparison functions can be introduced for the incompressible Euler equation as well, of course. The incompressible Euler equation has multiple solutions in general, and even singular solutions. Comparison functions can be used to prove the existence of global solution branches, but also, with a slight modification to prove the existence of singular solutions. Uniqueness arguments have to be added to these techniques in order to show that some evolution equation is deterministic.  As comparisons to the Euler equation are useful in any case, we introduce some related notation. The viscosity limit $\nu\downarrow 0$ of the velocity component functions $v^{\rho,r,\nu}_i\equiv v^{\rho,r,}_i$ is denoted by $e^{\rho,r}_i:=\lim_{\nu\downarrow 0}v^{\rho,r,\nu}_i$ whenever this limit exists pointwise, and where we use the same parameter transformation $e^{\rho,r}_i(\tau,y)=e_i(t,x)$ with $t=t(\tau)$ and $y_i=y(x_i)$ as above. Here $e_i,~1\leq i\leq n$ is a solution of the original Euler equation. We then may use the comparison
\begin{equation}
\lambda (1+\mu (\tau-t_0))u^{l,\rho,r,e,t_0}_i(s,.)=e^{\rho,r}_i(\tau,.),~s=\frac{\tau-t_0}{\sqrt{1-(\tau-t_0)^2}},
\end{equation}
or 
\begin{equation}
\lambda (1+\mu \tau)u^{g,\rho,r,e,t_0}_i(s,.)=e^{\rho,r}_i(\tau,.),~s=\frac{\tau-t_0}{\sqrt{1-(\tau-t_0)^2}},
\end{equation}
in order to argue for global solution branches of the Euler equation for strong data. In general, it is possible to obtain global solution branches for the incompressible Navier Stokes equation from global solution branches of the incompressible Euler equation using the Trotter product approximations inductively, as there is an additional viscosity damping at each time step in the Trotter product formula for the Navier Stokes equation. The time step size of the corresponding Euler scheme transforms as
\begin{equation}
\delta s^{(N)}=\frac{\delta t^{(N)}}{\sqrt{1-(\delta t^{(N)})^2}}
\end{equation}
The scheme for $u^{l\rho,r,t_0}_i,~1\leq i\leq n$ becomes at each stage $N\geq 1$
\begin{equation}\label{navode200secondu}
\begin{array}{ll}
u^{l,\rho,r,N,t_0}_{i\alpha}((m+1)\delta s^{(N)})=u^{l,\rho,r,N,t_0}_{i\alpha}(m\delta s^{(N)})\\
\\
+\rho r^2\mu^{t_0}\sum_{j=1}^n\nu \left( -\frac{4\pi^2 \alpha_j^2}{l^2}\right)u^{l,\rho,r,N,t_0}_{i\alpha}(m\delta s^{(N)})\delta s^{(N)}\\
\\
+\sum_{j=1}^n\sum_{\gamma\in {\mathbb Z}^n}e^{l,\rho,r,N,u,\lambda,t_0}_{ij\alpha\gamma}(m\delta s)u^{l,\rho,r,N,t_0}_{j\gamma}(m\delta t^{(N)})\delta s^{(N)}.
\end{array} 
\end{equation}
where $\mu^{t_0}=\sqrt{1-(.-t_0)^2}^3$ is evaluated at $t_0+m\delta t$ , and
\begin{equation}\label{eujk}
\begin{array}{ll}
e^{l,\rho,r,N,u,\lambda ,t_0}_{ij\alpha\gamma}(m \delta s^{(N)})=-\rho r\lambda\mu^{l,1,t_0}\frac{2\pi i (\alpha_j-\gamma_j)}{l}u^{l,\rho,r,t_0}_{i(\alpha-\gamma)}(m\delta s^{(N)})\\
\\
+\rho r\lambda\mu^{l,1,t_0}\frac{2\pi i\alpha_i1_{\left\lbrace \alpha\neq 0\right\rbrace}
\sum_{k=1}^n4\pi^2\gamma_j(\alpha_k-\gamma_k)u^{l,\rho,r,t_0}_{k(\alpha-\gamma)}(m\delta s^{(N)})}{\sum_{i=1}^n4\pi^2\alpha_i^2}-\mu^{0,t_0}(m\delta s^{(N)})\delta_{ij\alpha\gamma}
\end{array}
\end{equation}
along with $\mu^{l,1,t_0}(\tau):=(1+\mu (\tau-t_0))\sqrt{1-(\tau-t_0)^2}^3$ and $\mu^{l,0,t_0}(\tau):=\frac{\mu\sqrt{1-(\tau-t_0)^2}^3}{1+\mu (\tau-t_0)}$. Again we note that for $\rho=1$ the transformed time coordinates $\tau=\frac{t}{\rho}$ become equal to the original time coordinates $t$, a case we shall mainly consider in the following.
The last term in (\ref{eujk}) is related to the damping term of the equation for the function $u^{t_0}_i,~1\leq i\leq n$. For times $0< t_e$ and with an analogous time discretization we get the Trotter product formula
\begin{equation}\label{trotterlambda}
\begin{array}{ll}
\mathbf{u}^{l,\rho,r,N,F,t_0}(t_e)\doteq \Pi_{m=0}^{2^N}\left( \delta_{ij\alpha\beta}\exp\left(-\rho r^2\nu \sum_{i=1}^n\alpha_i^2 \delta s^{(N)} \right)\right)\times\\
\\
\times \left( \exp\left( \left( \left( e^{l,\rho,r,N,u,\lambda,t_0}_{ij\alpha\beta}\right)_{ij\alpha\beta}(m\delta s^{(N)})\right)\delta s^{(N)}\right)  \right) \mathbf{u}^{l,\rho,r,N,F,t_0}(0).
\end{array}
\end{equation}
In the following remark we describe the role of the parameters.
\begin{rem}
Even if $\rho=1$ and  one of the factor $r$ is small compared to the parameter  $\mu$ of the damping term $-\mu^{l,0,t_0}(m\delta s)\delta_{ij\alpha\gamma}$ in (\ref{eujk}), then the damping term can dominate the growth of the nonlinear terms if a mode of the value function exceeds a certain level. For higher modes and $\nu >0$ the viscosity damping becomes dominant anyway. Consider the transformation in (\ref{transformvu}) above in the case $\rho=1$. Then the transformation is with respect to original time coordinates. We consider this transformation
\begin{equation}
\lambda (1+\mu (t-t_0))u^{1,r,t_0}_i(s,.)=v^{1,r}_i(t,.),~s=\frac{t-t_0}{\sqrt{1-(t-t_0)^2}}.
\end{equation}
on a time interval $t\in [t_0,t_0+a]$ corresponding to $s\in \left[0,\frac{a}{\sqrt{1-a^2}} \right] $. Assume that $\mu=1$. We have
\begin{equation}
\frac{1}{\lambda}{\big |}v^{1,r}_i(t_0,.){\big |}_{H^p}={\big |}u^{1,r,t_0}_i(0,.){\big |}_{H^p}
\end{equation}
The nonlinear growth is described by the nonlinear Euler terms. For each mode $\alpha$ and at time step number $m$ of stage $N$ the nonlinear Euler term growth minus the potential damping is described by (recall that we have chosen $\rho=\mu=1$)
\begin{equation}\label{remeujk}
\begin{array}{ll}
\sum_{j=1}^n\sum_{\gamma\in {\mathbb Z}^n}e^{1,r,N,u,\lambda,t_0}_{ij\alpha\gamma}(m\delta s)u^{N,t_0}_{j\gamma}(m\delta t^{(N)})\delta s^{(N)}=\\
\\
-\sum_{j=1}^n\sum_{\gamma\in {\mathbb Z}^n} r\lambda\mu^{1,t_0}\frac{2\pi i (\alpha_j-\gamma_j)}{l}u^{1,r,t_0}_{i(\alpha-\gamma)}(m\delta s^{(N)})\\
\\
+ r\lambda\mu^{1,t_0}\frac{2\pi i\alpha_i1_{\left\lbrace \alpha\neq 0\right\rbrace}
\sum_{k=1}^n4\pi^2\gamma_j(\alpha_k-\gamma_k)u^{1,r,t_0}_{k(\alpha-\gamma)}(m\delta s^{(N)})}{\sum_{i=1}^n4\pi^2\alpha_i^2}u^{1,r,N,t_0}_{j\gamma}(m\delta s^{(N)})\delta s^{(N)}\\
\\
-\mu^{0,t_0}(m\delta s^{(N)})u^{1,r,N,t_0}_{i\alpha}(m\delta t^{(N)})\delta s^{(N)}.
\end{array}
\end{equation}
where  for $\mu=1$ and $\rho=1$ we have $\mu^{0,t_0}(t-t_0):=\frac{\sqrt{1-(t-t_0)^2}^3}{1+ (t-t_0)}$ (with $t=t(s)$ the inverse of $s=s(t)$). The potential damping in (\ref{remeujk}) becomes relatively strong for example for small $\lambda =r^2 >0$ with $r$ small. A small parameter $r>0$ (keeping $\rho=1$) is sufficient for a linear upper bound with respect to the time horizon $T$, via the global time transformation
\begin{equation}
\lambda (1+\mu t)u^{g,1,r,t_0}_i(s,.)=v^{1,r}_i(t,.),~s=\frac{t-t_0}{\sqrt{1-(t-t_0)^2}}.
\end{equation} 
\end{rem}
Note that the scheme for the Euler equation comparison $u^{*,\rho,r,e,t_0}_i,~1\leq i\leq n$ with $*\in \left\lbrace l,g \right\rbrace$ becomes at each stage $N\geq 1$
\begin{equation}\label{navode200secondu}
\begin{array}{ll}
u^{*,\rho,r,e,N,t_0}_{i\alpha}((m+1)\delta s^{(N)})=u^{*,\rho,r,e,N,t_0}_{i\alpha}(m\delta s^{(N)})\\
\\
+\sum_{j=1}^n\sum_{\gamma\in {\mathbb Z}^n}e^{*,\rho,r,N,u,\lambda,t_0}_{ij\alpha\gamma}(m\delta t)u^{*,\rho,r,e,N,t_0}_{j\gamma}(m\delta s^{(N)})\delta s^{(N)},
\end{array} 
\end{equation}
and that we have an approximation
\begin{equation}\label{trotterlambda}
\begin{array}{ll}
\mathbf{u}^{*,\rho,r,e,N,F,t_0}(t_e)\doteq \\
\\
\Pi_{m=0}^{2^N}\left( \exp\left( \left( \left( e^{*,\rho,r,N,u,\lambda,t_0}_{ij\alpha\beta}\right)_{ij\alpha\beta}(m\delta s^{(N)})\right)\delta s^{(N)}\right)  \right) \mathbf{u}^{*,\rho,r,e,N,F,t_0}(0).
\end{array}
\end{equation}

Note that for the original velocity component functions we have $v_i=v^{1,1}_i$, i.e., $\rho=1$ and $r=1$ and we drop the parameter superscripts in this case. The following theorem hold for all dimensions $n\geq 1$, In the next section we prove 
\begin{thm}\label{linearboundthm}
 If for some constant $C_0>0$ (depending only on the dimension $n$ an the viscosity $\nu >0$) we have 
\begin{equation}\label{regdata}
\forall \alpha\in {\mathbb Z}^n:~{\big |}h_{i\alpha}{\big |}\leq 
\frac{C_0}{1+|\alpha|^{n+s}} \mbox{ for some }s>1,
\end{equation}
then the limit $\mathbf{v}^F(t)=\lim_{n\uparrow \infty}\mathbf{v}^{N,F}(t)$ of the scheme $\mathbf{v}^{N,F}$ described above exists and describes a global regular solutions of the incompressible Navier Stokes equation on the torus, where a time independent constant $C>0$ exists such that
\begin{equation}\label{eq3}
\sup_{t\geq 0}{\big |}v_i(t,.){\big |}_{H^{\frac{n}{2}+1}}\leq C.
\end{equation}
Furthermore the constant $C$ satisfies
\end{thm}
From the proof of theorem \ref{linearboundthm} we shall observe that the value $s=1$ is critical. We have
\begin{thm}\label{linearboundthm2}
If for some constant $C>0$
\begin{equation}\label{regdata}
\forall \alpha\in {\mathbb Z}^n:~{\big |}h_{i\alpha}{\big |}\leq \frac{C}{1+|\alpha|^{n+s}} \mbox{ for }s=1,
\end{equation}
then the limit $\mathbf{v}^F(t)=\lim_{n\uparrow \infty}\mathbf{v}^{N,F}(t)$ of the scheme $\mathbf{v}^{N,F}$ described above exists and describes a global regular solutions of the incompressible Navier Stokes equation on the torus, where for any $T>0$ there exists a constant $C\equiv C(T)>0$ exists such that
\begin{equation}\label{eq3}
\sup_{t\leq T}{\big |}v_i(t,.){\big |}_{H^{\frac{n}{2}+s}}\leq C.
\end{equation}
\end{thm}
Finally, for $s< 1$ we have indication of divergence. We have
\begin{thm}\label{linearboundthm3}
If for some constant $C>0$
\begin{equation}\label{singdata}
\forall \alpha\in {\mathbb Z}^n:~{\big |}h_{i\alpha}{\big |}\geq \frac{C}{1+|\alpha|^{n+s}} \mbox{ for }s<1,
\end{equation}
then for some data in this class the scheme $\left( \mathbf{v}^{N,F}(t)\right)_{N\geq 1}$ diverges, indicating that singularities may occur. 
\end{thm}

\begin{rem}
According to the theorem the critical regularity for global regular existence is $h_i\in H^{\frac{1}{2}n+1}$ for all $1\leq i\leq n$, where we have a uniform upper bound if $h_i\in H^{\frac{1}{2}n+s}$ for $s>1$ for all $1\leq i\leq n$. Theorem \ref{linearboundthm3} only states the divergence of the scheme proposed which does not imply strictly that there are singular solutions for some data $h_i\in H^{\frac{1}{2}n+s},~1\leq i\leq n$ for all $s<1$. However, the methods in \cite{KT} may be used in order to construct such singular solutions. We shall consider this elsewhere. Note that $H^{\frac{1}{2}n+1}$ is a much stronger space than the hypothetical solution space $H^1$ which is proved to imply smoothness in \cite{T}. It is also not identical with the data space $H^2\cap C^2$ on the whole space ${\mathbb R}^n$, where it can be argued that this regularity is sufficient in order to have global regular space on the whole space.
\end{rem}

This paper is a classical interpretation of the scheme considered in \cite{KT}, where the different arguments are used, and the focus is more on the algorithmic perspective. The strong data space with $h_i\in H^s$ for $s>\frac{n}{2}+1$ for all $1\leq i\leq n$ needed for convergence indicates that the conjectures and proofs in \cite{L} and \cite{KP}
concerning the existence of singular solutions may be detected for weaker initial data spaces. In the next section we prove theorem \ref{linearboundthm}. It is a remarkable fact that $H^1$-regularity of the solution is sufficient for global smooth existence (cf. \cite{T}), although it seems that we need stronger data to arrive at this conclusion. This may be related to the fact that many weak schemes fail to converge in spaces of dimension $n=3$. Theorem \ref{linearboundthm2} and Theorem \ref{linearboundthm} follow from analogous observations. 

\section{Proof of the Theorem \ref{linearboundthm}}

In the following we mainly work with a direct scheme (without a time-delay transformation) in order to prove the main result descibed in Theorem \ref{linearboundthm}. Later we shall make some remarks concerning alternative arguments using time delay transformations. 

Assume that $\rho,r>0$ are positive numbers. 
We mention that at each stage $N\geq 1$ we have
\begin{equation}
\forall 1\leq i\leq n~\forall m\geq 1~\forall x:~v^{\rho,r,N}_{i}((m\delta t^{(N)},x)\in {\mathbb R}.
\end{equation}

For the sake of simplicity concerning the torus size $l$ (and without loss of generality) we may consider the case $l=1$ in the following. For any time $T>0$ the limit in (\ref{aa}), i.e., the function 
\begin{equation}\label{aaproof}
\begin{array}{ll}
\lim_{N\uparrow \infty}\mathbf{v}^{\rho,r,N,F}(T)=\lim_{N\uparrow \infty} \Pi_{m=0}^{2^N}\left( \delta_{ij\alpha\beta}\exp\left(-\rho r^2\nu \sum_{i=1}^n\alpha_i^2 \delta t^{(N)} \right)\right)\\
\\
\left( \exp\left( \left( \left( e^{\rho,r,N}_{ij\alpha\beta}\right)_{ij\alpha\beta}(m\delta t^{(N)})\right)\delta t^{(N)} \right) \right) \mathbf{h}^F, 
\end{array}
\end{equation}
is a candidate for a regular solution at time $T>0$ (for time $t\in (0,T)$ a similar formula with a number $\lfloor\frac{t}{T}2^N\rfloor$ describes the solution at time $t$, of course- here $\lfloor .\rfloor$ denotes the Gaussian floor).
We consider functions
\begin{equation}
\mathbf{v}^{\rho,r,F}:=\mathbf{v}^{\rho,r,F}=\left(v^{\rho,r}_{i\alpha}\right)^T_{\alpha\in {\mathbb Z}^n,~1\leq i\leq n}. 
\end{equation}
Here, the list of velocity component modes is denoted by $v^{\rho,r}_i=\left(v^{\rho,r}_{i\alpha} \right)_{\alpha\in {\mathbb Z}^n}$ for all $1\leq i\leq n$, where we identify $n$-tuples of infinite ${\mathbb Z}^n$-tuples of modes with $n{\mathbb Z}^n$-tuples in the obvious way. First for an arbitrary fixed time horizon $T>0$  and $0\leq t\leq T$, $\rho,r>0$ $C>0$ and order $p$ of the dual Sobolev space we define
\begin{equation}
{\big |}\mathbf{v}^{\rho,r, F}{\big |}^{C,\exp}_{h^p}:=\max_{1\leq i\leq n}\sup_{0\leq t\leq T}{\big |}v^{\rho,r}_{i}(t){\big |}_{h^p}\exp(-Ct).
\end{equation}
Accordingly, for discrete time discretization we define
 \begin{equation}\label{ncexp}
{\big |}\mathbf{v}^{\rho,r,N,F}{\big |}^{n,C,\exp}_{h^p}:=\max_{1\leq i\leq n}\max_{m\in  \left\lbrace 0,\cdots 2^N\right\rbrace }{\big |}v^{\rho,r}_{i}(m\delta t)^{(N)}){\big |}_{h^p}\exp(-Cm\delta t)^{(N)}).
\end{equation}
For $C=0$ we write
\begin{equation}
{\big |}\mathbf{v}^{\rho,r,N,F}{\big |}^{n}_{h^p}:=\max_{1\leq i\leq n}\max_{m\in  \left\lbrace 0,\cdots 2^N\right\rbrace }{\big |}v^{\rho,r}_{i}(m\delta t^{(N))}){\big |}_{h^p},
\end{equation}
and, analogously, for continuous time\begin{equation}
{\big |}\mathbf{v}^{\rho,r,F}{\big |}^{n}_{h^p}:=\max_{1\leq i\leq n}\sup_{t\in  [ 0,T] }{\big |}v^{\rho,r}_{i}(m\delta t{(N))}){\big |}_{h^p}.
\end{equation}
For the increment 
\begin{equation}
\delta \mathbf{v}^{\rho,r,N,F}=\mathbf{v}^{\rho,r,N,F}-\mathbf{v}^{\rho,r,N-1,F}
\end{equation}
we define the corresponding norm on the next coarser time scale, i.e.,
\begin{equation}\label{incrementnorm}
{\big |}\delta \mathbf{v}^{\rho,r,N,F}{\big |}^{n}_{h^p}:=\max_{1\leq i\leq n}\max_{m\in  \left\lbrace 0,\cdots 2^{N-1}\right\rbrace }{\big |}\delta v^{\rho,r,N}_{i}(m\delta t^{(N-1)})){\big |}_{h^p}.
\end{equation}
The norms ${\big |}\delta \mathbf{v}^{\rho,r,N,F}{\big |}^{n,C,\exp}_{h^p}$ etc. are defined analogously with a time weight as in (\ref{ncexp}). The increment norm in (\ref{incrementnorm}) measures the maximal deviation of the scheme from a fixed point solution in a given arbitrary large interval $[0,T]$, and the weighted norm may be used in order to show that the scheme is contractive. However, convergence of the scheme follows from the weaker statement that there is a series of error upper bounds $E^{r,N}_p$ of the values ${\big |}\delta \mathbf{v}^{\rho,r,N,F}{\big |}^{n}_{h^p}$ which converges to zero as $N\uparrow \infty$. The argument is simplified by the observation that there is a uniform time-independent global regular upper bound at each stage $N$ of the scheme. We summarize these two facts in Lemma \ref{lemN} and Lemma \ref{lemerr} below which we prove next. 
First we have 
\begin{lem}\label{lemN}
Let $T>0$ be an arbitrary given time horzon. Let $p>\frac{n}{2}+1$ and $\max_{1\leq i\leq n}{\big |}h_i{\big |}_{h^p}=: C^p_h<\infty$. Then for all $N\geq 1$ there exists a $C_N\geq C^p_h$ such that for all 
\begin{equation}\label{CN}
\sup_{t\in [0,T]}{\big |}\mathbf{v}^{\rho,r,N,F}(t,.){\big |}^{n}_{h^p}\leq C_N,
\end{equation}
where $C_N$ is depends only the dimension and the viscosity $\nu$. Especially, for all $N$ the constant $C_N$ is independent of the time horizon $T>0$. 
\end{lem}
\begin{rem}
It is remarkable that at each stage $N$ of the scheme a constant $C_N$ can be chosen which is independent of the time horizon $T>0$ as there is no viscosity damping for the zero modes. Indeed we shall first show that there is a linear time uper bound and then use special features if the scheme in order to construct an upper boud which is independent of the time horizon $T>0$.
\end{rem}

\begin{proof}
 Since  $h_i\in H^p$ for $p>\frac{n}{2}+s$ with $s>1$ we have for all $\alpha \in {\mathbb Z}^n$
\begin{equation}
{\big |}h_{i\alpha}{\big |}\leq \frac{C^{(p)}_{0h}}{1+|\alpha|^{n+s}}
\end{equation}
for some finite $C^{(p)}_{0h} >0$, where we assume 
\begin{equation}
C^{(p)}_{0h} \geq 1~~\mbox{w.l.o.g..}
\end{equation}
From finite iterative application of upper bound estimates of the form (\ref{cC}) we have
\begin{equation}
{\big |}v^{\rho,r,N,F}_{i\alpha}(m\delta t^{(N)}){\big |}\leq \frac{C^N_m}{1+|\alpha|^{n+s}}
\end{equation}
for some finite constants $C^N_m$ and for all $0\leq m\leq 2^N$. We observe the growth (as $N\uparrow \infty$) of the finite constants
\begin{equation}
C^N:=\max_{0\leq m\leq 2^N}C^N_m.
\end{equation}
At time step $m+1$ one Euler-type Trotter product step is described by
\begin{equation}\label{vtrotter}
\begin{array}{ll}
\mathbf{v}^{\rho,r,N,F}((m+1)\delta t^{(N)}):=\left( \delta_{ij\alpha\beta}\exp\left(- \rho r^2\nu 4\pi^2 \sum_{i=1}^n\alpha_i^2 \delta t^{(N)} \right)\right)\\
\\
\left( \exp\left( \left( \left( e^{\rho,r,N}_{ij\alpha\beta}(m\delta t)^{(N)}\right)_{ij\alpha\beta}\right)\delta t^{(N)}\right)  \right)\mathbf{v}^{\rho,r,N,F}(m\delta t^{(N)}),
\end{array}
\end{equation}
where
\begin{equation}
\begin{array}{ll}
e^{\rho,r,N}_{ij\alpha\gamma}(m dt^{(N)})=-\rho r\frac{2\pi i (\alpha_j-\gamma_j)}{l}v^{\rho,r,N}_{i(\alpha-\gamma)}(m\delta t^{(N)})\\
\\
+\rho r2\pi i\alpha_i1_{\left\lbrace \alpha\neq 0\right\rbrace}
4\pi^2 \frac{\sum_{k=1}^n\gamma_j(\alpha_k-\gamma_k)v^{\rho,r,N}_{k(\alpha-\gamma)}(m\delta t^{(N)})}{\sum_{i=1}^n4\pi^2\alpha_i^2}.
\end{array}
\end{equation}
First observe: if an upper bound for a weakened viscosity term $\nu_{vf}(\alpha)$ replacing the original viscosity term factor  $ 4\pi^2 \sum_{i=1}^n\alpha_i^2$ in the viscosity term $$\exp\left(-\rho r^2\nu 4\pi^2 \sum_{i=1}^n\alpha_i^2 \delta t^{(N)} \right)$$ for each mode $\alpha$, where
\begin{equation}\label{nufunction}
\forall \alpha\in {\mathbb Z}^n~~0\leq \nu_{vf}(\alpha)\leq \nu 4\pi^2\sum_{i=1}^n\alpha_i^2
\end{equation}
then this implies the existence an upper bound of the original system.

Second we observe that the statement of the Lemma \ref{lemN} can be proved under the assumption that we have an upper bound for the zero modes, i.e., that we have a finite constant $C_0>0$ such that
\begin{equation}\label{zeromodes}
|v^{\rho,r,N}_{i0}(m\delta t^{(N)})|\leq C_{0}\mbox{ for all }0\leq m\leq 2^N.
\end{equation}
\begin{rem}
The statement in (\ref{zeromodes}) can be verified,i.e., the assumption can be eliminated by introduction of an external control function for the zero modes which is the negative of the Burgers increment $$- \sum_{j,\gamma\neq 0}r\frac{2\pi i (-\gamma_j)}{l}v^{\rho,r,N}_{i(-\gamma)}(m\delta t^{(N)})v^{\rho,r,N}_{j\gamma}(m\delta t^{(N)})\delta t^{(N)}$$ of the zero mode at each time step. This leads to an linear time upper bound. However, we shall eliminate the assumption in (\ref{zeromodes}) in a more sophisticated way below such that we get an upper bound which is independent with respect to time. 
\end{rem}
For the choices of $\rho,r$ in (\ref{rpara}) and (\ref{viscpar}) with $c=C^{(p)}_{h0}$, i.e., for
\begin{equation}\label{rpara}
r=\frac{c_0^2\left( C^{(p)}_{h0}\right) ^2}{\nu},~\rho=\frac{\nu}{2c_0^2\left( C^{(p)}_{h0}\right) ^2} 
\end{equation}
we have $\rho r=\frac{1}{2}$. Then for
\begin{equation}\label{timesize}
\delta t^{(N)}\leq \frac{1}{c(n)\left( C^{(p)}_{h0}\right)^2},~c(n)=4\pi^2(n+n^2)c_0
\end{equation} 
along with $c_0$ the constant in the product rule (\ref{cC}) (which is essentially the finite upper bound constant of an weakly singular elliptic integral) we get an upper bound estimate.
Indeed, assuming inductively with respect to time that $C^N_m\leq C^{(p)}_{h0}$ at time step $m$ and for the  $\alpha$-modes of the Euler term we have the estimate
\begin{equation}\label{vtrottere}
\begin{array}{ll}
{\Bigg |}\left( \left( \exp\left( \left( \left( e^{\rho,r,N}_{ij\alpha\beta}(m\delta t^{(N)})\right)_{ij\alpha\beta}\right)\delta t^{(N)}\right)  \right)\mathbf{v}^{\rho,r,N,F}(m\delta t^{(N)})\right)_{i\alpha}{\Bigg |}\\
\\
\leq 1+\rho r\frac{c(n)\left( C^{(p)}_{h0}\right)^2}{1+|\alpha|^{n+s}}\delta t^{(N)}+\left( \delta t^{(N)} \right)^2,
\end{array}
\end{equation}
where $\left(. \right)_{i\alpha}$ denotes the projection function to the $\alpha$-modeof the $i$th component of the velocity mode vector. Next consider the inductive assumption
\begin{equation}
{\big |}v^{1,r,N,F}_{i\alpha}((m+1)\delta t^{(N)}){\big |}\leq \frac{C^{(p)}_{h0}}{(1+|\alpha|^{n+s})}.
\end{equation}

For the parameter choices of $r$ and $\rho$ in (\ref{rpara}) and the time size in (\ref{timesize}) we compute
\begin{equation}\label{vtrotterialpha}
\begin{array}{ll}
{\big |}v^{1,r,N,F}_{i\alpha}((m+1)\delta t^{(N)}){\big |}\leq \\
\\
{\Bigg|}\exp\left(- \rho r^2\nu 4\pi^2 \sum_{i=1}^n\alpha_i^2 \delta t^{(N)} \right)v^{\rho,r,N,F}(m\delta t^{(N)})_{i\alpha}\times\\
\\
\left(1+\rho r\frac{c(n)\left( C^{(p)}_{h0}\right)^2}{1+|\alpha|^{n+s}}\delta t^{(N)}+\left( \delta t^{(N)} \right)^2\right) {\Bigg |}
\leq \frac{C^{(p)}_{h0}}{(1+|\alpha|^{n+s})},
\end{array}
\end{equation}
where  in the last step we use the alternative that either
\begin{equation}
{\big |}v^{\rho,r,N,F}(m\delta t^{(N)})_{i\alpha}{\big |}\leq \frac{C^{(p)}_{h0}}{2(1+|\alpha|^{n+s})},
\end{equation} 
or
\begin{equation}
v^{\rho,r,N,F}(m\delta t^{(N)})_{i\alpha}\in \left[\frac{C^{(p)}_{h0}}{2(1+|\alpha|^{n+s})},\frac{C^{(p)}_{h0}}{(1+|\alpha|^{n+s})}\right].
\end{equation} 
In order to eliminate the additional assumption concerning the zero modes, i.e., the assumption
\begin{equation}\label{0modes}
\forall 0\leq m\leq 2^N~~v^{\rho,r,N,F}(m\delta t^{(N)})_{i0}=0,
\end{equation}
we consider a related Trotter product scheme and show that the argument above can be applied for this extended scheme. We mentioned that the assumption in (\ref{0modes}) is satisfied for a controlled system with an external control function which forces the zor modes to be zero. It is not difficult to show then that the control function itself is bounded on an arbitrary finite time interval $[0,T]$, and, hence, that a global regular upper bound exists for the uncontrolled system. However, the upper bound constant is then dependent on the time horizon $T>0$. The extended Euler Trotter product schemes allows us to obtain sharper regular upper bounds which are independent of the time horizon $T>0$. 

We define
\begin{equation}
\forall 1\leq j\leq n:~c_{j(0)}=h_{j0}.
\end{equation}
Having defined $c_{j(m)}$ and $\mathbf{v}^{\rho,r,N,F,ext}(m\delta t^{(N)})$ for $m\geq 0$ and $1\leq j\leq n$ an extended Euler Trotter product step at time step number $m+1$ is defined in two steps. First we define
\begin{equation}\label{vtrotterext0}
\begin{array}{ll}
\mathbf{v}^{\rho,r,N,F,ext,0}((m+1)\delta t^{(N)}):=\\
\\
\left( \delta_{ij\alpha\beta}\exp\left(- \rho r^2\nu 4\pi^2 \sum_{i=1}^n\alpha_i^2 \delta t^{(N)} \right)\right)\exp\left(- \rho r 2\pi i \sum_{i=1}^n\gamma_j c_m\delta t^{(N)} \right) \\
\\
\left( \exp\left( \left( \left( e^{\rho,r,N,ext}_{ij\alpha\beta}(m\delta t)^{(N)}\right)_{ij\alpha\beta}\right)\delta t^{(N)}\right)  \right)\mathbf{v}^{\rho,r,N,F,ext}(m\delta t^{(N)}),
\end{array}
\end{equation}
where
\begin{equation}
\begin{array}{ll}
e^{\rho,r,N,ext}_{ij\alpha\gamma}(m dt^{(N)})=-\rho r\frac{2\pi i (\alpha_j-\gamma_j)}{l}v^{\rho,r,N,ext}_{i(\alpha-\gamma)}(m\delta t^{(N)})\\
\\
+\rho r2\pi i\alpha_i1_{\left\lbrace \alpha\neq 0\right\rbrace}
4\pi^2 \frac{\sum_{k=1}^n\gamma_j(\alpha_k-\gamma_k)v^{\rho,r,N,ext}_{k(\alpha-\gamma)}(m\delta t^{(N)})}{\sum_{i=1}^n4\pi^2\alpha_i^2}.
\end{array}
\end{equation}
Then we define
\begin{equation}
c_{m+1}= \sum_{j,\gamma\neq 0}r\frac{2\pi i (-\gamma_j)}{l}v^{\rho,r,N,ext}_{i(-\gamma)}(m\delta t^{(N)})v^{\rho,r,N,ext}_{j\gamma}(m\delta t^{(N)})\delta t^{(N)}
\end{equation}
and
\begin{equation}
\begin{array}{ll}
v^{\rho,r,N,F,ext}((m+1)\delta t^{(N)})_{i\alpha}=v^{\rho,r,N,F,ext,0}((m+1)\delta t^{(N)})_{i\alpha},~\mbox{if~$\alpha\neq 0$}\\
\\
0,,~\mbox{if~$\alpha = 0$}.
\end{array}
\end{equation}
The argument above can then be repeated for the extended scheme.

\end{proof}
Note that the relation
\begin{equation}
D^{\alpha}_xv_i(t,.)=r^{|\alpha|}D^{\alpha}_xv^{\rho,r}_i(\tau,.)
\end{equation}
implies that an upper bound $C_{(m)}$ of $\max_{1\leq i\leq n}\sup_{\tau\in [0,T_{\rho}]}
{\big |}v^{\rho,r}_i(\tau,.){\big |}_{H^m}$ 
implies that the norm $\max_{1\leq i\leq n}
\sup_{t\in [0,T]}{\big |}v_i(t,.){\big |}_{H^m}$ of the original velocity components has an upper bound $\frac{C_{(m)}}{r^m}$.
Converging error upper bounds stated in the next Lemma imply that for the constants $C_N$ in (\ref{CN}) we have for $m\geq \frac{n}{2}+1$
\begin{equation}\label{C}
\forall~N~C_N\leq C=C^{(p)}_{h0}\left( \frac{c(n)^2\left( C^p_{h0}\right)^2}{\nu}\right)^{m},
\end{equation}
where
\begin{equation}
c(n)=(n^2+n)\sum_{\beta\in {\mathbb Z}^n,~\beta\neq \alpha}\frac{2\pi}{(\alpha-\beta)\beta}.
\end{equation}
In order to analyze the error we consider the difference
\begin{equation}
\delta \mathbf{v}^{\rho,r,N+1,F}( 2m\delta t^{(N)})=\mathbf{v}^{\rho,r,N+1,F}( 2m\delta t^{(N)})-  \mathbf{v}^{\rho,r,N,F}(m\delta t^{(N)})
\end{equation}
on the coarser time grid of stage $N$ of the scheme.
\begin{lem}\label{lemerr}
Let $p>\frac{n}{2}+1$ and $\max_{1\leq i\leq n}{\big |}h_i{\big |}_{h^p}=: C_h<\infty$. Given $T_{\rho}>0$ and the parameter choice $r=\frac{c(n)^2\left( C^{(p)}_{h0}\right) ^2}{\nu},~\rho=\frac{\nu}{2c(n)^2\left( C^{(p)}_{h0}\right) ^2} $
let $N\geq 1$ be large enough such that the time step size $\delta t^{(N)}$ satisfies
\begin{equation}\label{timesize}
\delta t^{(N)}=\frac{T_{\rho}}{2^N}\leq \frac{1}{T_{\rho}C},
~\mbox{where $C$ is defined in (\ref{C})}.
\end{equation} 
Then we have a decreasing series of error upper bounds $E^{\rho,r,N}_p$ 
of the values ${\big |}\delta\mathbf{v}^{\rho,r,N,F}{\big |}^{n}_{h^p}$ such that

\begin{equation}
\lim_{N\uparrow \infty}{\big |}\sup_{m\in \left\lbrace 0,\cdots,2^m \right\rbrace  }\delta\mathbf{v}^{\rho,r,N,F}(m\delta t^{N}){\big |}^{n}_{h^p}\leq \lim_{N\uparrow \infty}E^{r,N}_p=0
\end{equation}

\end{lem}
\begin{proof}
For given time $T_{\rho}>0$ and
at time $2(m+1)\delta t^{(N+1)})=(m+1)\delta t^{(N)}$ we compare
\begin{equation}
\mathbf{v}^{1,r,N+1,F}(2(m+1)\delta t^{(N+1)} \mbox{ and }\mathbf{v}^{1,r,N,F}((m+1)\delta t^{(N)}),
\end{equation}
where we assume that the error at time $m\delta t^{(N)}$ the error has an upper bound of form
\begin{equation}
{\big |}v^{1,r,N+1,F}_{i\alpha}(2m\delta t^{(N+1)})-v^{1,r,N,F}_{i\alpha}(m\delta t^{(N)}){\big |}\leq \frac{C^N_{(e)m}\delta t^{(N)}}{1+|\alpha|^{n+s}}
\end{equation}
for $s>1$ and for some finite Euler error constant $0\leq C^N_{(e)m}$ which will be determined inductively with respect to $m\geq 0$ and $N\geq 0$. We have $C^N_{(e)0}=0$. 
For the time step number $2m+1$  we get
\begin{equation}\label{vtrotter}
\begin{array}{ll}
\mathbf{v}^{\rho,r,N+1,F}((2m+1)\delta t^{(N+1)}):=
\left( \delta_{ij\alpha\beta}\exp\left(- c(n)^2\left( C^{(p)}_{h0}\right)^24\pi^2 \sum_{i=1}^n\alpha_i^2 \delta t^{(N+1)} \right)\right)\\
\\
\left( \exp\left( \left( \left( e^{\rho,r,N}_{ij\alpha\beta}(2m\delta t)^{(N+1)}\right)_{ij\alpha\beta}\right)\delta t^{(N+1)}\right)  \right)\mathbf{v}^{\rho,r,N+1,F}(2m\delta t^{(N+1)}),
\end{array}
\end{equation}
where 
\begin{equation}
\begin{array}{ll}
e^{\rho,r,N}_{ij\alpha\gamma}(2m dt^{(N+1)})=-\frac{1}{2}\frac{2\pi i (\alpha_j-\gamma_j)}{l}v^{\rho,r,N}_{i(\alpha-\gamma)}(2m\delta t^{(N+1)})\\
\\
+\pi i\alpha_i1_{\left\lbrace \alpha\neq 0\right\rbrace}
4\pi^2 \frac{\sum_{k=1}^n\gamma_j(\alpha_k-\gamma_k)v^{\rho,r,N+1}_{k(\alpha-\gamma)}(2m\delta t^{(N)})}{\sum_{i=1}^n4\pi^2\alpha_i^2}.
\end{array}
\end{equation}
Here we used the specific parameter choice for $r$ and $\rho$ above.

The size of one step on level $N$ equals the size of two steps on level $N+1$. Using upper bound estimates of the form (\ref{cC}) we consider two steps and estimate
\begin{equation}
\begin{array}{ll}
{\big |}v^{\rho,r,N+1,F}_{i\alpha}((2m+2)\delta t^{(N+1)})-v^{\rho,r,N,F}_{i\alpha}(m\delta t^{(N)}){\big |} \\
\\
\leq \frac{c(n)C^{(p)}_{h0}C_{(e)m}\delta t^{(N)}}{1+|\alpha|^{n+s}},
\end{array}
\end{equation}
where $C$ is the constant in (\ref{C}). Here we observe that 
It follows that for all $m\in \left\lbrace 0,\cdots,2^N\right\rbrace $
\begin{equation}\label{errorv}
\begin{array}{ll}
\left( v^{1,r,N,F}_{i\alpha}(m\delta t^{(N)})\right)_{N\geq 1}
\end{array}
\end{equation}
is a Cauchy sequence and that there is an 
\end{proof}

The argument above proves  Theorem \ref{linearboundthm} where the constant $C>0$ depends on the viscosity constant $\nu$. Next we consider variations of the argument which lead to related results which are not covered by the previous argument, and which use local and global time delay transformation  with a potential damping term. In the case of a time global transformation  we get dependence of the upper bound constant $C$ on the time horizon but independence of the upper bound of the viscosity constant. We also combine viscosity damping with potential damping via local time transformation. Using time horizons $\Delta$ of the local subschemes which are correlated to the viscosity $\nu$ we have another method in order to obtain gobal regular upper bounds. This method will be considered in more detail elsewhere.

Note  the role of the spatial parameter $r>0$ in the  estimates above. For large $r>0$ the parameter coefficent $\rho r^2$ of the viscosity damping term becomes dominant compared to the parameter coefficient $\rho r$ of the nonlinear term. This implies that the viscosity damping becomes dominant compared to possible grwoth caused ny the nonlinear terms whenever a velocity mode exceeds a certain level. However in order to make potential damping dominant this parameter $r$ may be chosen to be small. Since the potential damping term does not depend on the parameter $r$, potential damping becomes dominant if $r$ is chosen small enough compared to a givane time horizon $T>0$.  In this simple case of a global time delay transformation we have dependence of the global uupper bound on the time horizon $T>0$, but the scheme is still global.

Although Euler equations have singular solutions in general, they may have global solution branches in strong spaces as well, and from (\ref{aaproof}). A detailed anaysis is beyond the scope of this paper.  Note that the viscosity limit $\nu \downarrow 0$ of (\ref{aaproof}) leads to an Euler scheme for the incompressible Euler equation with a solution candidate of the form 
\begin{equation}\label{eeproof}
\begin{array}{ll}
\lim_{N\uparrow \infty}\mathbf{e}^{\rho,r,N,F}(T)=\lim_{N\uparrow \infty} \lim_{\nu\downarrow 0}\Pi_{m=0}^{2^N}\left( \delta_{ij\alpha\beta}\exp\left(-\rho r^2\nu \sum_{i=1}^n\alpha_i^2 \delta t^{(N)} \right)\right)\\
\\
\left( \exp\left( \left( \left( e^{\rho,r,N}_{ij\alpha\beta}\right)_{ij\alpha\beta}(m\delta t^{(N)})\right)\delta t^{(N)} \right) \right) \mathbf{h}^F\\
\\
=\lim_{N\uparrow \infty} \Pi_{m=0}^{2^N}
\left( \exp\left( \left( \left( e^{\rho ,r,N}_{ij\alpha\beta}\right)_{ij\alpha\beta}(m\delta t^{(N)})\right)\delta t^{(N)} \right) \right) \mathbf{h}^F,
\end{array}
\end{equation}
where upper bound estimates in a strong norm for the limit $\lim_{N\uparrow \infty}\mathbf{e}^{\rho,r,N,F}(T)$ in (\ref{eeproof}) imply the existence of regular upper bounds for the limit  $\lim_{N\uparrow \infty}\mathbf{v}^{N,F}(T)$ of the Navier stokes equation scheme  in the same strong norm.
A global regular upper bound for a global regular solution branch of the incompressible Navier Stokes equation then leads to a unique solution via a well-known uniqueness result  which holds on the torus as well, i.e.,
for the incompressible Navier Stokes equation a global regular solution branch $v_i,~1\leq i\leq n$ with $v_i\in C^0\left([0,T],H^s\left({\mathbb T}^n\right)\right) $ for $s\geq 2.5$ leads - via Cornwall's inequality-to uniqueness in this space of regularity. More precisely, if $\tilde{v}_i,~1\leq i\leq n$ is another solution of the incompressible Navier Stokes equation, then we have for some constant $C>0$ depending only on the dimension $n$ and the viscosity $\nu>0$ and some integer $p\geq 4$ we have 
\begin{equation}\label{unique}
{\big |}\tilde{v}(t)-v(t){\big |}^2_{L^2}\leq {\big |}\tilde{v}(0)-v(0){\big |}^2_{L^2}
\exp\left(C\int_0^t\left( {\big |}{\big |}v(s){\big |}^p_{L^4}+{\big |}{\big |}v(s){\big |}^2_{L^4}\right)ds  \right). 
\end{equation}
Here, the choice $p=8$ is sufficient in case of dimension $n=3$. Hence, there is no other solution branch with the same $H^q$ data for $q\geq 2.5$ which is the critical level of regularity. Note that (\ref{unique}) does not hold for the Euler equation.

Next we use viscosity damping in cooperation with potential damping via local and global time dilatation in order to prove related global upper bound results. The existence of a global regular upper bound for time delay transformation schemes follows from three facts: a) for some short time interval $\Delta >0$ and some $0<\rho,r,\lambda<1$ and for stages $N\geq 1$ such that time steps are smaller then $\Delta >0$ we have a preservation of upper bounds in strong norms of the comparison functions $u^{*,1,r,N,t_0}_i,~ 1\leq i\leq n$ for $*\in \left\lbrace l,g\right\rbrace$ at each time step which falls into the time interval $\left[0,\Delta\right]$ with respect to transformed time $s$-coordinates, and b) we have a preservation of the upper bound in the limit $N\uparrow \infty$, i.e., a preservation of an upper bound of the $H^p$-norm for $p>\frac{n}{2}+1$ of the functions $u^{*,1,r,t_0}_i(s,.)=\lim_{N\uparrow \infty}u^{*,\rho,r,N,t_0}_i(s,.),~1\leq i\leq n$ for $s\in \left[0,\Delta\right]$ for $*\in {l,g}$,  and c) in the case of a time local time delay transformation, i.e. $*=l$ a upper bound preservation for strong norms can be transferred to the original velocity function components for some $\rho, r,\lambda>0$ and where the time size $\Delta >0$ of the subscheme is related to the viscosity $\nu$. Here for $t\in [t_0,t_0+\Delta]$ we have $v^{\rho,r}_i(t,.)=\lim_{N\uparrow \infty}v^{1,r,N,t_0}_i(t,.)=\lim_{N\uparrow \infty}\lambda (1+\mu(t-t_0)) u^{l,1,r,N,t_0}_i(s,.),~1\leq i\leq n$, and the latter limit function satisfies the incompressible Navier Stokes equation on the time interval $[t_0,t_0+a]$ (recall that $t\in [t_0,t_0+a]$ for $a\in (0,1)$ corresponds to $s\in [0,\Delta]$ with $\Delta=\frac{a}{\sqrt{1-a^2}}$). Here, the potential damping causes a growth term of size $O(\Delta^2)$ which can be offset by the viscosity damping by a choice of a time horizon $\Delta$ of the subscheme which is small compared to $\nu r^2$.  
First we formalize the statements in a) and b) in
\begin{lem}\label{mainlem}
Let a time horizon $T>0$ be given and consider a subscheme with local or global time delay transformation, i.e., let $*\in {l,g}$. Given a time interval length $\Delta \in (0,1)$ and stages $N\geq 1$ large enough such that $\delta t^{(N)}\leq \Delta$, for $m>  \frac{n}{2}+1$ there is a constant $C'>0$ (depending only on the viscosity and the size of the data) such that
\begin{equation}
{\big |}\mathbf{u}^{*,1,r,N,F,t_0}(0){\big |}^n_{h^m}\leq C'\Rightarrow {\big |}\mathbf{u}^{*,1,r,N,F,t_0}(s){\big |}^n_{h^m}\leq C'
\end{equation}
for 
$$s\in \left\lbrace p\delta t^{(N)}|0\leq p\leq 2^{N},~s\leq \Delta\right\rbrace,$$
where transformed time $s\in \left[0,\Delta \right]$ corresponds to original time $t\in \left[t_0,t_0+a\right]$ with $\Delta =\frac{a}{\sqrt{1-a^2}}$. Here, we
define
\begin{equation}
{\big |}\mathbf{u}^{*,1,r,N,F,t_0}(s){\big |}^n_{h^m}:=\max_{1\leq i\leq n}{\big |}\mathbf{u}^{*,1,r,N,F,t_0}_i(s){\big |}_{h^m}
\end{equation}
for $*\in \left\lbrace l,g\right\rbrace$. 
The statement holds in the limit $N\uparrow \infty$ as well, i.e.,   
for $m> \frac{n}{2}+1$ there exists a constant $C'>0$ (depending only on the viscosity and the size of the data) such that for $s\in \left[0,\frac{\Delta}{\sqrt{1-\Delta^2}} \right]$
\begin{equation}
{\big |}\mathbf{u}^{*,1,r,F,t_0}(0){\big |}^n_{h^m}\leq C'\Rightarrow {\big |}\mathbf{u}^{*,1,r,F,t_0}(s){\big |}^n_{h^m}\leq C'
\end{equation}
for $*\in \left\lbrace l,g\right\rbrace$. 
\end{lem}
\begin{rem}
Note that for $s>0$ and $*\in \left\lbrace l,g\right\rbrace$ the inequality
\begin{equation}
|u^{*,1,r}_{i\alpha}|\leq \frac{c}{1+|\alpha|^{n+s}}<\infty
\end{equation}
implies  polynomial decay of order $2n+2s$ of the quadratic modes, where the equivalence
\begin{equation}
u^{*,1,r}_i\in H^m\equiv H^m\left({\mathbb Z}^n\right)  \mbox{iff} \sum_{\alpha\in {\mathbb Z}^n}|u^{*,1,r}_{i\alpha}|^{2m}(1+|\alpha|^{2m})<\infty
\end{equation}
implies that $u^{*,1,r}_i\in H^{\frac{n}{2}+s}$ an vice versa. According to the theorems stated above, $H^{\frac{n}{2}+1}$ is the critical space for regularity. This means that for theorem \ref{linearboundthm} we assume that $s>\frac{n}{2}+1$.
\end{rem}
Next we sketch a proof for lemma \ref{mainlem} in the case $*=l$. The proof in the case of a global time delay transformation is similar with the difference that the spatial transofrmation parametr $r$ may depend on the time horizon.
The scheme for $u^{l,\rho,r,t_0}_i,~1\leq i\leq n$ becomes at each stage $N\geq 1$
\begin{equation}\label{navode200secondu}
\begin{array}{ll}
u^{l,\rho,r,N,t_0}_{i\alpha}((m+1)\delta s^{(N)})=u^{l,\rho,r,N,t_0}_{i\alpha}(m\delta s^{(N)})\\
\\
+\mu^{t_0}\sum_{j=1}^n\rho r^2\nu \left( -\frac{4\pi^2 \alpha_j^2}{l^2}\right)u^{l,\rho,r,N,t_0}_{i\alpha}(m\delta s^{(N)})\delta s^{(N)}\\
\\
+\sum_{j=1}^n\sum_{\gamma\in {\mathbb Z}^n}e^{l,\rho,r,N,u,\lambda,t_0}_{ij\alpha\gamma}(m\delta s)u^{l,\rho,r,N,t_0}_{j\gamma}(m\delta s^{(N)})\delta s^{(N)}.
\end{array} 
\end{equation}
where $\mu^{t_0}=\sqrt{1-(.-t_0)^2}^3$ is evaluated at $t_0+m\delta t^{(N)}$ , and
\begin{equation}\label{eujk*}
\begin{array}{ll}
e^{l,\rho,r,N,u,\lambda ,t_0}_{ij\alpha\gamma}(m \delta s^{(N)})=-\lambda \rho r\mu^{1,t_0}\frac{2\pi i (\alpha_j-\gamma_j)}{l}u^{l,\rho,r,N,t_0}_{i(\alpha-\gamma)}(m\delta s^{(N)})\\
\\
+ \lambda \rho r\mu^{l,1,t_0}\frac{2\pi i\alpha_i1_{\left\lbrace \alpha\neq 0\right\rbrace}
\sum_{k=1}^n4\pi^2\gamma_j(\alpha_k-\gamma_k)u^{l,\rho,r,N,t_0}_{k(\alpha-\gamma)}(m\delta s^{(N)})}{\sum_{i=1}^n4\pi^2\alpha_i^2}-\mu^{l,0,t_0}(m\delta s^{(N)})\delta_{ij\alpha\gamma}
\end{array}
\end{equation}
along with $\mu^{l,1,t_0}(t):=(1+\mu(t-t_0))\sqrt{1-(t-t_0)^2}^3$ and $\mu^{l,0,t_0}(t):=\frac{\mu\sqrt{1-(t-t_0)^2}^3}{1+\mu (t-t_0)}$.
The last term in (\ref{eujk}) is related to the damping term of the equation for the function $u^{l,1,r,t_0}_i,~1\leq i\leq n$. For times $0\leq t_e$ and with an analogous time discretization we get the Trotter product formula
\begin{equation}\label{trotterlambda2}
\begin{array}{ll}
\mathbf{u}^{l,\rho,r,N,F,t_0}(t_e)\doteq \Pi_{m=0}^{2^N}\left( \delta_{ij\alpha\beta}\exp\left(-\rho r^2\nu \sum_{i=1}^n\alpha_i^2 \delta s^{(N)} \right)\right)\times\\
\\
\times \left( \exp\left( \left( \left( e^{l,\rho,r,N,u,\lambda,t_0}_{ij\alpha\beta}\right)_{ij\alpha\beta}(m\delta s^{(N)})\right)\delta s^{(N)}\right)  \right) \mathbf{u}^{N,F,t_0}(0).
\end{array}
\end{equation} 
The related comparison function for the Euler equation velocity satisfies
\begin{equation}\label{trotterlambdae2}
\begin{array}{ll}
\mathbf{u}^{l,\rho,r,e,N,F,t_0}(t_e)\doteq \\
\\
\Pi_{m=0}^{2^N}\left( \exp\left( \left( \left( e^{l,\rho,r,N,u,\lambda,t_0}_{ij\alpha\beta}\right)_{ij\alpha\beta}(m\delta s^{(N)})\right)\delta s^{(N)}\right)  \right) \mathbf{u}^{l,\rho,r,e,N,F,t_0}(0).
\end{array}
\end{equation}
The scheme for the Navier Stokes equation comparison function $\mathbf{u}^{l,\rho,r,N,F,t_0}(t_e)$ has an additional ${\big |}\exp\left(-\rho r^2\nu \sum_{i=1}^n\alpha_i^2 \delta s^{(N)} \right){\big |}\leq 1$ on the diagonal of the infinite matrix $\left( \delta_{ij\alpha\beta}\exp\left(-\rho r^2\nu \sum_{i=1}^n\alpha_i^2 \delta t^{(N)} \right)\right)$ which is effective at each time step of the Trotter scheme. Using this observation and induction with respect to the time step number we get
\begin{lem}
Let $\lambda=1$ and $\rho=1$ for simplicity. For all stages $N\geq 1$, and given  $0\leq t_e<1$  there exists $r>0$ such that
\begin{equation}
{\Big |}\mathbf{u}^{l,\rho,r,e,N,F,t_0}(0){\Big |}^n_{h^p}\leq C^p_h~\Rightarrow ~{\Big |}\mathbf{u}^{l,\rho,r,N,F,t_0}(t_e){\Big |}^n_{h^p}\leq C^p_h
\end{equation}
\end{lem}
We observe that the comparison function $\mathbf{u}^{\rho,r,e,N,F,t_0}(t_e)$ for the Euler equation scheme preserves certain upper bounds at one step within the time interval $[0,\Delta]$ for appropriately chosen parameters $\lambda>0$ and $\mu>\lambda$ depending on $\Delta=\frac{a}{\sqrt{1-a^2}}>0$, where $a$ is the size of the corresponding time interval in original coordinates. We mention here that the parameter $\lambda$ can be useful inorder to strengthen the relative strength of potential damping- for $r=\lambda^2$ the nonlinear terms have a coefficient of order $\sim \lambda^3$ such that the nonlinear growth factor $\sim\frac{1}{\lambda^2}$ of the scaled value functions  is absorbed and the grwoth of the nonlinear terms is of order $\lambda$ while the potential damping term gets a factor $\frac{1}{\lambda}$ via the scaled value function which is elatively strong.  However this it is not really needed in order to prove the existence of a global solution branch of the Euler equation. 
Assume inductively that for some finite constant $C>0$ an for $p>\frac{n}{2}+1$ and a given stage $N\geq 1$ with $\delta s^{(N)}\leq \Delta$ and $m\geq 0$ we have
\begin{equation}
{\Big |}\mathbf{u}^{\rho,r,e,N,F,t_0}(m\delta s^{(N)}){\Big |}^n_{h^p}\leq C^p_{h0}
\end{equation}
The Euler equation scheme for the next time step is 
 \begin{equation}\label{navode200secondue}
\begin{array}{ll}
u^{l,\rho,r,e,N,t_0}_{i\alpha}((m+1)\delta s^{(N)})=u^{l,\rho,r,N,t_0}_{i\alpha}(m\delta s^{(N)})\\
\\
+\sum_{j=1}^n\sum_{\gamma\in {\mathbb Z}^n}{\Big (}-\lambda r\mu^{l,1,t_0}\frac{2\pi i (\alpha_j-\gamma_j)}{l}u^{l,\rho,r,e,N,t_0}_{i(\alpha-\gamma)}(m\delta s^{(N)})\\
\\
+ \lambda r\mu^{l,1,t_0}\frac{2\pi i\alpha_i1_{\left\lbrace \alpha\neq 0\right\rbrace}
\sum_{k=1}^n4\pi^2\gamma_j(\alpha_k-\gamma_k)u^{\rho,r,e,N,t_0}_{k(\alpha-\gamma)}(m\delta s^{(N)})}{\sum_{i=1}^n4\pi^2\alpha_i^2}{\Big )}u^{l,\rho,r,e,N,t_0}_{j\gamma}(m\delta s^{(N)})\delta s^{(N)}\\
\\-\mu^{l,0,t_0}(m\delta s^{(N)})u^{l,\rho,r,e,N,t_0}_{i\alpha}(m\delta s^{(N)})\delta s^{(N)},
\end{array} 
\end{equation}
where we note that  $|\mu^{l,1,t_0}(t)|\leq(1+\mu \Delta)$ and $|\mu^{l,0,t_0}(t)|\geq \frac{\mu\sqrt{1-\Delta^2}^3}{1+\mu \Delta}$ for $t=t(s)\in [t_0,t_0+\Delta]$ and $t\in [t_0,t_0+a]$. We get
 \begin{equation}\label{navode200secondue**}
\begin{array}{ll}
{\Big |}u^{l,\rho,r,e,N,F,t_0}((m+1)\delta s^{(N)}){\Big |}^n_{h^p}\leq {\big |}u^{l,\rho,r,N,t_0}(m\delta s^{(N)}){\big |}_{h^p}^n\\
\\
+{\Big |}{\Big (}\sum_{j=1}^n\sum_{\gamma\in {\mathbb Z}^n}{\Big (}-\lambda r\mu^{1,t_0}\frac{2\pi i (\alpha_j-\gamma_j)}{l}u^{l,\rho,r,e,N,F,t_0}_{i(\alpha-\gamma)}(m\delta s^{(N)})\\
\\
+ \lambda r\mu^{1,t_0}\frac{2\pi i\alpha_i1_{\left\lbrace \alpha\neq 0\right\rbrace}
\sum_{k=1}^n4\pi^2\gamma_j(\alpha_k-\gamma_k)u^{l,\rho,r,e,N,t_0}_{k(\alpha-\gamma)}(m\delta s^{(N)})}{\sum_{i=1}^n4\pi^2\alpha_i^2}{\Big )}\times\\
\\
\times u^{l,\rho,r,e,N,t_0}_{j\gamma}(m\delta s^{(N)})\delta s^{(N)}{\Big )}_{1\leq i\leq n,\alpha \in {\mathbb Z}^n}+O\left(\left( \delta s^{(N)}\right)^2\right) \\
\\-\frac{\mu\sqrt{1-\Delta^2}^3}{1+\mu \Delta}{\Big |}u^{l,\rho,r,e,N,F,t_0}(m\delta s^{(N)})\delta s^{(N)}{\Big |}_{h^p}^n,
\end{array} 
\end{equation}
Hence, the preservation of an upper bound at step $m+1$ follows from
\begin{equation}\label{left}
\begin{array}{ll}
\lambda (1+\mu\Delta){\Big |}{\Big (}\sum_{j=1}^n\sum_{\gamma\in {\mathbb Z}^n}{\Big (}-\frac{2\pi i (\alpha_j-\gamma_j)}{l}u^{l,\rho,r,e,N,F,t_0}_{i(\alpha-\gamma)}(m\delta s^{(N)})\\
\\
+ \frac{2\pi i\alpha_i1_{\left\lbrace \alpha\neq 0\right\rbrace}
\sum_{k=1}^n4\pi^2\gamma_j(\alpha_k-\gamma_k)u^{l,\rho,r,e,N,t_0}_{k(\alpha-\gamma)}(m\delta s^{(N)})}{\sum_{i=1}^n4\pi^2\alpha_i^2}{\Big )} u^{l,\rho,r,e,N,t_0}_{j\gamma}(m\delta s^{(N)})\delta s^{(N)}{\Big )}_{1\leq i\leq n,\alpha \in {\mathbb Z}^n}{\Big |}_{h^p}^n\\
\\
\leq \frac{\mu\sqrt{1-\Delta^2}^3}{1+\mu \Delta}{\Big |}u^{l,\rho,r,e,N,F,t_0}(m\delta s^{(N)})\delta s^{(N)}{\Big |}_{h^p}^n.
\end{array}
\end{equation}
According to the regularity (polynomial decay of the modes) of $u^{e,N,F,t_0}_{i}(m\delta s^{(N)})$ and simple estimates as in (18) the left side of (\ref{left}) has an upper bound
\begin{equation}
\lambda r(1+\mu\Delta)c(n)C^2
\end{equation}
for some constant $c(n)$ which depends only on dimension $n$ such that the inequality in (\ref{left}) follows from
\begin{equation}\label{left2}
\begin{array}{ll}
\lambda r(1+\mu\Delta)c(n)C^2\leq \frac{\mu\sqrt{1-\Delta^2}^3}{1+\mu \Delta}C.
\end{array}
\end{equation}
Here, the number $C>0$ is determined by the data at time $t_0$ and where $r,\lambda >0$ can be chosen such that $|.|_{H^p}$ norms of comparison functions $u^{1,r,t_0}_i,~1\leq i\leq n$ are preserved. This does not imply preservation of the norm for the velocities $v^{1,r}_i,~1\leq i\leq n$ of course. hier
The Euler equation scheme has a limit $N\uparrow \infty$ by standard arguments. Finishing item a) and item b) of the argument we get
\begin{lem}\label{rholemma2}
For all $0\leq t_e<1$ and $p>\frac{n}{2}+1$ and for some $r,\lambda >0$ and $*\in {l,g}$ we get
\begin{equation}
{\Big |}\mathbf{u}^{*,\rho,r,F,t_0}(0){\Big |}^n_{h^p}\leq C\Rightarrow {\Big |}\mathbf{u}^{*,\rho,r,F,t_0}(t_e){\Big |}^n_{h^p}\leq C,
\end{equation}
where $\mathbf{u}^{*,\rho,r,F,t_0}(t_e)=\lim_{N\uparrow \infty} \mathbf{u}^{*,\rho,r,N,F,t_0}(t_e)$.
\end{lem}
 Pure potential damping leads to a growth estimate for the original of the original velocity component functions of order $\Delta^2$ on the time interval $\left[t_0,t_0+a\right]$. This growth can be offset even by a small viscosity damping for nonzero modes with the choice $\Delta =\nu r^2$. Again an extended scheme leads to global regular upper bounds.
\footnotetext[1]{\texttt{{kampen@wias-berlin.de}, joerg.kampen@mathalgorithm.de}.}


\begin{thebibliography}{19}
\baselineskip=12pt


\bibitem{A}{\sc Arnold, V.I.}: {\em Dynamical Systems.} Springer, 1994.


\bibitem{CFNT}
	{\sc Constantin, P., Foias, C., Nicolenko, B., Temam, R.}: {\em Integral manifolds and inertial manifolds for dissipative partial differential equations.} Springer, 1989.
	
	
\bibitem{FJRT}
	{\sc Foias, C., Manley, O., Rosa, R., Temam, R.}: {\em Navier-Stokes equations and turbuelence.} CUP, 2001.
	



\bibitem{FJKT}
	{\sc Foias, C., Jolly, M., Kravchenko, M., Tity, E.}: {\em Navier-Stokes equation, determing modes, dissipative dynamical systems.} arXiv1208,5434v1, August 2012.

\bibitem{KT}	
{\sc Kampen, J.,} {\em Trotter product formulas and global regular upper bounds of the Navier Stokes equation solution}, arXiv1401.2734v3   April. 2014.

\bibitem{KAC}
{\sc Kampen, J.} {\em
On an auto-controlled global existence scheme of the incompressible  Navier Stokes equation}, arxiv13094824v11,  [math.AP], May. 2014.









	
	
\bibitem{K2}
{\sc Kampen, J.,} {\em Global regularity and probabilistic schemes for free boundary surfaces of multivariate American derivatives and their Greeks},  Siam J. Appl. Math. 71, pp. 288-308, 2011.

\bibitem{KKS}
{\sc Kampen, J., Kolodko, A., Schoenmakers, J.}, {\em Monte Carlo Greeks for financial products via approximative transition densities}, Siam J. Sc. Comp., vol. 31 , p. 1-22, 2008.



\bibitem{KSV}
 {\sc Kampen, J.,} {\em Singular vorticity solutions of the incompressible Euler equation via inviscid limits}, arXiv 1308.6082v6,  Mai. 2014.

\bibitem{KAC}
 {\sc Kampen, J.,} {\em On an auto-controlled global existence scheme of the incompressible  Navier Stokes equation  }, arxiv13094824v11,  [math.AP], May. 2014.
 
 
\bibitem{KNS}
	{\sc Kampen, J\"org}: {\em Constructive analysis of the Navier-Stokes equation.} arXiv10044589v6, Juli 2012.

\bibitem{K3}
 {\sc Kampen, J.,} {\em A global scheme for the incompressible Navier-Stokes equation on compact Riemannian manifolds}, arXiv: 1205.4888v5,  Aug 2013.


 

 \bibitem{KB1}
{\sc Kampen, J.} {\em On the multivariate Burgers equation and the incompressible Navier-Stokes equation (part I)}, arXiv:0910.5672v5  [math.AP], 2011.	
	
\bibitem{KB2}
{\sc Kampen, J.} {\em On the multivariate Burgers equation and the incompressible Navier-Stokes equation (part II)}, arXiv:1206.6990v6  [math.AP], Jan., 2013.

\bibitem{KB3}
{\sc Kampen, J.} {\em On the multivariate Burgers equation and the incompressible Navier-Stokes equation (part III)}, arXiv:1210.2602v11 [math.AP], November, 2013. 


\bibitem{KHyp}
{\sc Kampen, J.} {\em On global schemes for highly degenerate Navier Stokes equation systems}, arXiv:13056385.v4 [math.AP] (August, 2013).

\bibitem{KD1}
{\sc Kampen, J.} {\em On fixed points, fixed points of iterated maps and applications}, Nonlinear Analysis, V. 42, p. 502-532, 2000.

\bibitem{KD2}
{\sc Kampen, J.} {\em Determination of homotopy classes of maps on compact surfaces of positive genus $g$ with infinitely many periodic points}, arXiv:1006.2567v1 [math.DS] (July, 2010)


\bibitem{Kat}
{\sc Kato, T.} {\em Perturbation theory of linear operators}, Springer, 1966.

\bibitem{KP}
{\sc Katz, N., Pavlovic, N.} {\em Finite time blow up for a dyadic model of the Euler equations}, Trans. Math. Soc. 357, 695-708, 2005.


\bibitem{LL}
{\sc Landau, L., Lifschitz, E.} {\em Lehrbuch der Theoretischen Physik VI, Hydrodynamik}, Akademie Verlag, Berlin. J., (1978).

\bibitem{L}
{\sc Leray, J.} {\em Sur le Mouvement d'un Liquide Visquex Emplissent l'Espace}, Acta Math. J. (63), 193-248, (1934).



\bibitem{S}
{\sc Schur, J.} {\em \"{U}ber lineare Transformationen in der Theorie der unendlichen Reihen}, J. f. Reine und Angew. Math. (151), 79-111, (1922).

\bibitem{T}
{\sc Tao, T.} {\em Localisation and compactness properties of the Navier Stokes global regularity problem}, arXiv:1108.1165v4  [math.AP], 2011. 

\bibitem{T}
{\sc Trotter, H.} {\em On the product of semigroups of operators}, Proc. AMS 10, p. 545-551, (1959).



 \end{thebibliography}
\end{document}